\documentclass[11pt,reqno]{amsart}
\usepackage[top=1.2in, bottom=1.2in, left=1.2in, right=1.2in]{geometry}               
\usepackage{hyperref,multirow}
\usepackage{graphicx}
\usepackage{enumitem}
\usepackage{amssymb}
\usepackage{epstopdf}
\usepackage{amsmath,amsfonts,amssymb,bm}
\usepackage{amsthm}
\usepackage{float}
\usepackage{algorithm}
\usepackage[]{algpseudocode}
\usepackage[english]{babel}
\usepackage[utf8]{inputenc}
\usepackage[normalem]{ulem}
\usepackage{cancel}
\usepackage{longtable,array}
\usepackage{cite}

\allowdisplaybreaks

\usepackage{titlesec} 
\titleformat{\section}{\vskip10pt\large\bfseries}{\thesection.}{0.5em}{\centering\vspace{5pt}}
\titleformat{\subsection}{\vskip10pt\normalsize\bfseries}{\thesubsection.}{0.5em}{}
\titleformat{\subsubsection}{\vskip10pt\normalsize\bfseries}{\thesubsection.}{0.5em}{}

\usepackage{lipsum,titletoc}
\titlecontents{section} 
[2.5em]                 
{\rmfamily}
{\contentslabel{2.3em}} 
{\hspace*{-2.3em}} {\titlerule*[1pc]{.}\contentspage}

\titlecontents{subsection} 
[4.8em]                 
{\rmfamily}             
{\contentslabel{2.3em}} 
{\hspace*{-2.3em}} {\titlerule*[1pc]{.}\contentspage}

\theoremstyle{definition}

\theoremstyle{plain}
\newtheorem{theorem}{Theorem}[section]
\newtheorem{lemma}{Lemma}[section]
\newtheorem{definition}{Definition}[section]

\newtheorem{remark}{Remark}[section]

\numberwithin{theorem}{section}
\numberwithin{proposition}{section}
\numberwithin{lemma}{section}
\numberwithin{equation}{section}
\usepackage{lipsum}
\usepackage{amsfonts}
\usepackage{graphicx}
\usepackage{epstopdf}
\ifpdf
  \DeclareGraphicsExtensions{.eps,.pdf,.png,.jpg}
\else
  \DeclareGraphicsExtensions{.eps}
\fi

\graphicspath{{./Figures}} 

\usepackage{subfigure}
\usepackage{amssymb}
\usepackage{algorithm}
\usepackage{xcolor}
\newcommand{\bs}[1]{\boldsymbol{#1}}
\newcommand{\uu}[1]  {{\bs{#1}} }

\newcommand{\dep}{{{\bs{d}}}}

\newcommand{\vel}{{\bs\dot{\dep}}} 

\newcommand{\grad}   { \uu{\nabla}    }

\def\d{\mathrm{d}}
\newcommand{\eps}   {\varepsilon}

\newcommand{\stress}{{\bs \sigma}}

\newcommand{\bv}{{\bs v}}

\newcommand{\bx}{{\bs x}}

\newcommand{\bu}{\uu{u}}

\newcommand{\bw}{{\bs w}}
\newcommand{\bn}{{\bs n}}

\def\R{{\mathbb{R}}}

\newcommand{\bV}{{\bs V}}

\newcommand{\bld}[1]{\boldsymbol{#1}}

\newcommand{\bp}{\bld{p}}

\newcommand{\bX}{{\bm X}}

\DeclareMathAlphabet\mathbfcal{OMS}{cmsy}{b}{n}

\floatname{algorithm}{Algorithm}

\algnewcommand{\algorithmicendif}{\textbf{end}}
\algblockdefx[IF]{If}{EndIf}[1]{\algorithmicif\ #1\ \algorithmicthen}{\algorithmicendif}
\algblockdefx[For]{for}{EndFor}[1]{\algorithmicif\ #1\ \algorithmicthen}{\algorithmicendif}

\newcommand{\jump}[1]{\bigl[\hspace{-0.03in}\bigl[#1\bigr]\hspace{-0.03in}\bigr]}

\def\eqd{\stackrel{\mathrm{def}}{=}}
\title[]{Convergence of a C\MakeLowercase{ut}FEM for fluid--structure interaction with a deforming interface
}
\author[M.A. Fernández]{Miguel A. Fernández}
\address{M.A. Fernández: Sorbonne Universit\'e \& CNRS, UMR 7598 LJLL -- Inria, France}
\email{miguel.fernandez@inria.fr}

\author[B. Li]{Buyang Li}
\address{B. Li: Department of Applied Mathematics,
The Hong Kong Polytechnic University, HK}
\email{buyang.li@polyu.edu.hk}

\author[M. Olshanskii]{Maxim Olshanskii}
\address{M. Olshanskii: Department of Mathematics,
University of Houston,
Houston, TX 77204, USA}
\email{maolshanskiy@uh.edu}
\thanks{
The work of B. Li was supported by the National Natural Science Foundation of China (Project No. 12525111) and a fellowship from the Research Grants Council of Hong Kong (Project No. PolyU RFS2324-5S03). 
The work of M. Olshanskii was supported in part by the U.S. National Science Foundation under grant DMS-2408978. 
This work was initiated during a research visit of B.~Li and M. Olshanskii at Inria Paris, partially funded by  Inria.
}

\begin{document}
\maketitle

\begin{abstract}
We present a rigorous error analysis for a semidiscrete immersed
CutFEM applied to fluid--structure interaction (FSI) problems involving an incompressible viscous fluid and an elastic solid. The method employs a Lagrange multiplier to enforce the fluid--solid coupling conditions while fully accounting for the effect of the structural deformation on the fluid domain. For sufficiently regular FSI solutions and finite elements of degree $k\geq3$, we prove convergence of order $k$ for the fluid velocity and solid displacement and convergence of order $k-1$ for the fluid pressure in the energy norms associated with the FSI problem. To the best of our knowledge, this is the first rigorous convergence analysis of an unfitted finite element method for a fully coupled FSI system with a deforming interface. 
\end{abstract}

\tableofcontents 

\section{Introduction}

Fluid--structure interaction (FSI) systems, in which an elastic body is coupled to an
incompressible viscous fluid across a moving interface, arise in many applications in
engineering and biomechanics, including aeroelasticity, parachute dynamics, blood flow
and cardiovascular devices, and biological locomotion; see, e.g.,
\cite{letallec-mouro-01,takizawa-tezduyar-12,Bazilevs2013,kamensky2015immersogeometric,mullert-et-al-10,tian-et-al-14,bergmann-et-al-22}.
Their numerical approximation is challenging for several intertwined reasons. The fluid
subproblem is posed in an Eulerian domain whose boundary is itself part of the unknown
solution, while the solid is naturally described in Lagrangian variables. The coupling
conditions impose both kinematic continuity and balance of stresses on the evolving
fluid--solid interface. In addition, the incompressibility constraint and the added-mass
effect may lead to severe stability restrictions or strong algebraic coupling, especially
when the structure is light or when the fluid and solid densities are comparable.

A classical approach for FSI problems with moderate interface displacement is to use
fitted moving meshes, most often within an arbitrary Lagrangian--Eulerian (ALE)
formulation of the fluid equations; see, e.g.,
\cite{donea1982arbitrary,nomura-hughes-92,stein2003,takizawa2012space,landajuela2017coupling}.
Fitted ALE methods have the advantage that interface conditions can be imposed directly
on a mesh boundary. However, mesh quality may deteriorate under large deformations and rotations. These limitations have motivated the
development of unfitted and immersed approaches, in which the fluid mesh is not required
to conform to the current solid position.

Among such methods, immersed boundary and immersed finite element methods (FEMs) 
\cite{peskin2002,Zhang2004,Mittal2005,boffi-cavallini-gastaldi-11,wang2013modified}
represent the solid by suitable forcing terms or by distributed coupling operators. Fictitious
domain and distributed Lagrange multiplier methods
\cite{GlowinskiDistributedLagrange,Baaijens,DEHART2003103,dossantos2008,cottet-et-al-08,astorino2009,richter-13,boffi2015finite,kamensky2015immersogeometric,BoffiGastaldi,BOILEVINKAYL2019744}
embed the structure into a larger computational domain and enforce the coupling by
constraints or penalties. These techniques are attractive for problems with large motion
because the mesh does not need to follow the interface. At the same time, standard
unfitted formulations may lose accuracy near the interface, since the discrete spaces do
not resolve the discontinuities, jumps, or boundary conditions induced by the moving
solid. CutFEM and related XFEM/Nitsche formulations address this
issue by integrating the equations on the physical, cut domain and by adding stabilization
terms, such as ghost penalties, that restore robustness with respect to arbitrarily small
cut elements; see, e.g.,
\cite{hansbo-hansbo-02,hansbo-hansbo-04,zilian-legay-08,MayerGerstenbergerWall2009,MayerPoppGerstenbergerWall2010,sawada-tezuka-11,burman-hansbo-12,burman-hansbo-14,burman-et-al-14,massing-et-al-14,burman-fernandez-14,alauzet-et-al-15,zonca2018unfitted,ager-et-al-19,GuzmOlsh}.

The numerical analysis of unfitted FEMs for moving-domain and
interface problems has progressed substantially in recent years. For prescribed evolving
domains or interfaces, CutFEM and Eulerian FEMs have been analyzed
for parabolic problems, transport-diffusion problems, and linearized Navier--Stokes type
equations; see, for instance,
\cite{zunino-13,lehrenfeld2019eulerian,NeilanOlsh}. These works
develop many of the tools that are indispensable in moving-geometry analysis, such as
transport identities, norm equivalences between nearby surfaces, and geometric
perturbation estimates. Nevertheless, FSI with a deforming interface presents an additional
difficulty: the interface motion is not prescribed. The interface position is determined by the numerical
solid displacement, and the error in displacement is coupled  with the error in the fluid
solution.

Indeed, in an unfitted FSI discretization, the error in the solid displacement alters the numerical fluid domain, the numerical interface, the surface normal, and the pullback of the interface traction. Consequently, the fluid error estimates depend on the error in the interface position and on suitable Lipschitz bounds for the numerical deformation. Conversely, the error in the solid displacement is driven by the numerical fluid traction and by the defect in the kinematic coupling condition, both of which depend on the fluid velocity, pressure, and Lagrange multiplier. Thus, the geometry error affects the fluid error, while the fluid error feeds back into the geometry error. Closing this loop when the standard energy estimate fails to provide suitable Lipschitz bounds for the displacement is the central analytical difficulty.

There are further complications specific to unfitted FSI. Since the moving interface may
cut the background mesh in an arbitrary way, trace inequalities, inverse inequalities, and
inf--sup constants must be uniform with respect to both the cut position and time. The
ghost-penalty stabilization must control finite element functions on an extension of the
physical fluid domain, while remaining consistent for smooth exact solutions. Interface
terms involve functions living on different surfaces, such as the exact interface, the interpolated interface, and the numerical
interface. The analysis must therefore compare bulk and surface
integrals over several moving domains and must track the dependence of these comparisons
on the displacement error. In addition, the Lagrange multiplier represents the pulled-back
fluid traction and is naturally estimated in negative Sobolev norms, which makes its
coupling with the solid equation particularly delicate.

Much of the rigorous error analysis available for FSI concerns linearized or
small-displacement models in which the deformation of the fluid domain is neglected and
the interface is fixed in time; see, e.g.,
\cite{letallec-mani-00,du-gunzburger-hou-lee-04,astorino-grandmont-10,fernandez-13,burman-fernandez-14,fernandez-mullaert-15,bukac-muha-16,GONZALEZ2018275,bukac-canic-21,burman2023convergence,annese2023splitting,burman2022fully,li-sun-xie-yu-24}.
For genuinely moving-interface FSI, rigorous convergence results are much scarcer.
With the exception of  \cite{von-wahl-richter-23}, which considers the unfitted approximation of a simplified parabolic/rigid-body coupled problem, all the existing analyses for moving interfaces are limited to fitted
mesh approximations, such as \cite{KLLP-2017,Li-Tang-2025,Gao-Li-Tang-2025} for solution-driven surface evolutions and \cite{MR2437776,buyang-fsi-ale} for FSI problems with moving interfaces. A convergence theory for an unfitted
FEM for a fully coupled FSI problem with a deforming interface has
therefore remained open.

The purpose of this paper is to provide such an analysis. We study a semidiscrete
immersed CutFEM for the FSI problem where an incompressible Navier--Stokes fluid is coupled to a linearly elastic
solid. The coupling is imposed by a Lagrange multiplier representing the fluid traction
pulled back to the reference solid interface. The fluid equations are integrated over the
current numerical fluid domain, while ghost-penalty stabilization provides robustness and
extension estimates on the unfitted background mesh. The method fully accounts for the
deformation of the fluid domain and does not rely on a small-displacement or fixed-domain
approximation.

The main contribution is a rigorous a priori error analysis of the proposed semidiscrete CutFEM for the FSI problem with a moving interface.
The proof is based on several new ingredients. First, we introduce an interpolated configuration ${\bs X}_H^*$ of the exact configuration ${\bs X}$ and compare the exact interface $\Gamma[{\bs X}]$, interpolated interface $\Gamma[{\bs X}_H^*]$, and numerical interface $\Gamma[{\bs X}_H]$
through families of intermediate interfaces. Transport formulas with respect to these
auxiliary surface parameters allow us to express domain and interface perturbation errors
in terms of the displacement error. This provides a systematic way to quantify how errors
in the interface position enter the fluid and multiplier equations.

Second, we establish uniform norm equivalences, trace estimates, inverse estimates, and
inf--sup estimates on the moving cut domains. These estimates are uniform with respect
to the position of the interface in the background mesh and with respect to time, as long
as the numerical deformation remains sufficiently close to the exact deformation. The
ghost-penalty terms play a key role in controlling finite element functions on the extended
fluid region needed for these estimates.

Third, we derive estimates for the projected kinematic mismatch 
${\bf P}_{\Gamma[{\bs X}_H]}{\bs e}_{\bs u}
-
\partial_t{\bs e}_{\bs d}\circ{\bs X}_H^{-1}$
on $\Gamma[{\bs X}_H]$, 
where ${\bf P}_{\Gamma[{\bs X}_H]}$ denotes the $L^2$ projection onto the discrete
interface multiplier space. This estimate is crucial for transferring control between the
fluid velocity error ${\bs e}_{\bs u}$ and the solid velocity error ${\bs e}_{\bs d}$ on the moving interface. It is also one of
the mechanisms by which the two-way dependence between fluid error and geometry error
is closed.

Fourth, we obtain stability estimates for the pressure and for the traction multiplier in
the relevant cut-domain and negative Sobolev norms. These estimates require a uniform
unfitted inf--sup condition for the pressure, duality estimates on moving finite element
interface spaces, and carefully constructed extensions from the interface into the solid
finite element space. To avoid loss of order in the pressure-related consistency terms, we
also use a Ritz projection adapted to the cut fluid domain.

Finally, since the validity of the above geometric estimates requires Lipschitz control of
the numerical deformation and of the velocity error, we use a continuation-in-time
argument. We first assume that suitable $W^{1,\infty}$ and $L^\infty$ bounds hold up to a
maximal time $t_*$. Under this assumption we prove energy estimates with constants
independent of $t_*$. For finite elements of degree $k\geq 3$, the resulting convergence
rates, combined with finite element inverse inequalities, imply the very bounds assumed
in the continuation hypothesis. Hence the estimates extend beyond $t_*$ unless
$t_*=T$. This closes the nonlinear argument and yields convergence on the whole time
interval on which the exact solution is smooth.
The resulting estimate, stated in Theorem~\ref{THM}, gives order $k$ convergence in the
natural energy norms for the fluid velocity and solid displacement, together with order $k-1$ convergence for the pressure and the interface traction multiplier. To the
best of our knowledge, this is the first rigorous convergence analysis of an unfitted FEM for a fully coupled FSI system with a genuinely deforming interface.

The rest of the paper is organized as follows. Section~\ref{sect:setting} states the FSI
model and the interface conditions. Section~\ref{section:numerical} introduces the
Lagrange multiplier formulation, the unfitted FEM, the assumptions, and
the main theorem. Section~\ref{sect:interp} defines the interpolated exact solution and
derives the error equation. Section~\ref{section:outline} outlines the
continuation argument. Section~\ref{sect:prelim} collects geometric, trace, and
transport estimates needed in the analysis. The consistency estimates are proved in
Sections~\ref{section:consistency} and~\ref{section:consistency-2}, while the stability
and error estimates are developed in Section~\ref{sect:error}. Numerical experiments are
reported in Section~\ref{sect:num}. Several auxiliary technical results are provided in the
appendices.

\section{Problem setting}\label{sect:setting}
We consider a fluid–structure interaction system in which the fluid is modeled by the incompressible Navier–Stokes equations in Eulerian form, while the immersed solid is described by the linear elasticity equations in Lagrangian form. 
Let $\Omega \subset \mathbb{R}^d$, with $d=2,3$, be a time-independent computational domain and $\widehat\Omega^{\rm s} \subset  \mathbb{R}^d$  the reference configuration of the immersed structure. The boundaries of $\Omega$ and $\widehat{\Omega}^{\rm s}$ are  respectively denoted by $\Gamma^{\rm f}$ and $\widehat\Gamma$. For a given final time $T>0$, the  motion of the solid body is described in terms of the map $\bs X: \widehat\Omega^{\rm s} \times [0,T] \rightarrow \mathbb{R}^d$, so that  ${\bs X}({\bs x},t)$ describes  the position at time $t \in [0,T]$ of the material point ${\bs x}\in \widehat\Omega^{\rm s}$. 
We introduce the solid displacement  $\dep: \widehat\Omega^{\rm s} \times [0,T]\rightarrow \mathbb{R}^d$ as 
$ \dep(\bs x,t) \eqd {\bs X} (\bs x,t) - \bs x $. In what follows, for functions $f$ defined in a space-time domain, we will write $f(t)\eqd f(\cdot,t)$.
 The current configuration of the solid is defined as the instantaneous  image of $\widehat\Omega^{\rm s}$ by the motion map $\bs X$, namely, 
 $$
 \Omega^{\rm{s}}(t) \eqd \Omega^{\rm{s}}[\bs X(t)] \eqd \bs X (\widehat \Omega^{\rm s},t)
 $$
 and it is assumed that $ \Omega^{\rm s}(t) \subsetneq \Omega$  for all $t\in [0,T]$.  The time-dependent fluid domain is hence defined as 
 \begin{equation} \label{eq:omegaft}
 \Omega^{\rm f}(t) \eqd \Omega^{\rm f}[\bs X(t)] \eqd \Omega \backslash \overline{\Omega^{\rm s}(t)} ,
 \end{equation}
 and its non-cylindrical space-time trajectory as 
$$
 \mathcal Q^{\rm f} \eqd \bigcup_{t \in [0,T] }\Omega^{\rm f}(t) \times \{ t \}.
$$
We also introduce the  fluid velocity $\bu: \mathcal Q^{\rm f} \rightarrow \mathbb{R}^d$ and pressure 
$p :\mathcal Q^{\rm f} \rightarrow \mathbb{R} $ fields. 
 The current configuration of the fluid-structure interface is  
$$
\Gamma(t) \eqd \Gamma[\bs X(t)] \eqd\partial \Omega^{\rm s}(t) =  {\bs X}(\widehat\Gamma,t) . 
$$
In particular, it holds $\partial \Omega^{\rm f}(t) = \Gamma(t) \cup \Gamma^{\rm f}$; see Figure \ref{fig:domain}. 

\begin{figure}[h!]
\centering\includegraphics[width=0.9\textwidth]{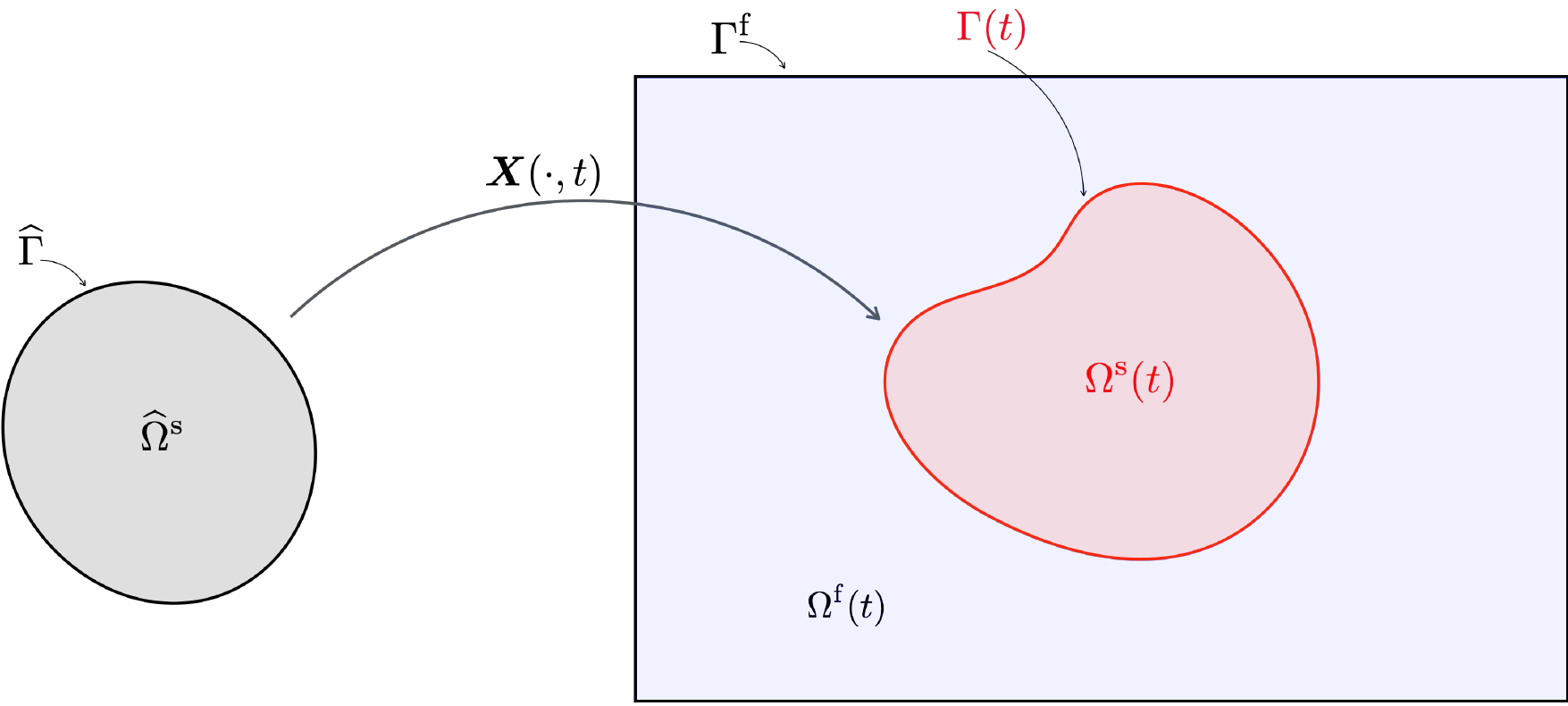}
\caption{Fluid, solid and the solid reference domains, with $\Omega = \Omega^{\rm f}(t) \cup \overline{\Omega^{\rm s}(t)}$. 
}\label{fig:domain}
\end{figure}

The dynamics of the fluid and the solid  are governed by the Navier--Stokes and elasticity equations, respectively, i.e., 
\begin{subequations}\label{eq:coupling-PDE}
\begin{align}
\rho^{\rm f}\left( \partial_t \bu   + \bu\cdot\grad\bu\right)
- \grad\cdot \stress (\bu,p)  
&= \bs 0 
\hspace{10pt}\mbox{in}\,\,\, \Omega^{\rm{f}}(t), 
\label{eq:fluid-u}\\
\grad\cdot \bu 
&= 0 
\hspace{10pt}\mbox{in}\,\,\, \Omega^{\rm{f}}(t), 
\label{eq:fluid-div}\\
\rho^{\rm s} \partial_{tt}{\bs d} - \grad\cdot{\bs\sigma}^{\rm s}({\bs d}) 
&= {\bs 0}
\hspace{10pt}\mbox{in}\,\,\,\widehat{\Omega}^{\rm s}, 
\label{eq:solid-d}
\end{align}
\end{subequations}
where $\rho^{\rm f}$ and $\rho^{\rm s}$ stand for the fluid and solid densities and 
$$
\stress(\bu, p) \eqd  2\mu^{\rm f}\bs D(\bu) -p {\bs I}_{d\times d},\quad 
{\bs\sigma}^{\rm s}({\bs d}) \eqd  2\mu^{\rm s}{\bs D}({\bs d}) + \lambda^{\rm s}(\grad\cdot {\bs d}) {\bs I}_{d\times d} ,
$$ 
respectively  denote the fluid and solid stress tensors. Here, the symbol 
$ \bs D (\bu) \eqd \frac12\big( \grad \bu + (\grad \bu)^{\top}\big)$  stands  for  the symmetric part of the gradient, ${\bs I}_{{d\times d}}$ denotes the  identity matrix in $\mathbb{R}^{d\times d}$,  
$\mu^{\rm f}$ the viscosity of the fluid and  $\mu^{\rm s} ,\lambda^{\rm s}$  the Lam\'e constants of the elastic structure. For simplicity, the analysis is presented after normalizing the physical coefficients to 
$$
\mu^{\rm f}=\mu^{\rm s}=\frac12,\quad \rho^{\rm f} = \rho^{\rm s} = \lambda^{\rm s}=1.
$$ 
This modification only simplifies the corresponding  bilinear forms and does not change the structure of the method. 

We consider the following standard interface conditions for the FSI system in \eqref{eq:coupling-PDE}, where $\bn_{\widehat\Gamma}$ denotes the unit normal vector on $\widehat\Gamma$ pointing to the solid: 
\begin{subequations}\label{eq:coupling-nl}
\begin{align}
{\bs u}&= \partial_t{\bs d} \circ {\bs X}^{-1} \hspace{5pt}  \mbox{on}\quad  \Gamma(t), \label{eq:coupling-nl-a}\\
\det(\grad{\bs X}) \,({\bs\sigma} (\bu,p) \circ{\bs X}) (\grad{\bs X})^{-\top}{\bs n}_{\widehat\Gamma} &= {\bs\sigma}^{\rm s}({\bs d}) {\bs n}_{\widehat\Gamma} \hspace{13.5pt}  \mbox{on}\quad \widehat\Gamma. \label{eq:coupling-nl-b}
\end{align}
\end{subequations}
Moreover, for the sake of simplicity we set 
\begin{equation}
{\bs u} = {\bs 0}\quad \mbox{on}\quad \Gamma^{\rm f} . \label{eq:coupling-nl-c}
\end{equation}
The FSI system \eqref{eq:coupling-PDE}--\eqref{eq:coupling-nl-c} is complemented with the following initial conditions: 
\begin{align}\label{initial-condition}
\begin{aligned}
\bu (0) = \bu_0 &\quad \mbox{in}\quad \Omega^{\rm f}(0),\\
{\bs d}(0) = {\bs d}_0, \quad \partial_t{\bs d}(0) = \vel_0& \quad \mbox{in}\quad  \widehat{\Omega}^{\rm s}.
\end{aligned}
\end{align}
The well-posedness of the complete FSI system \eqref{eq:coupling-PDE}–\eqref{initial-condition} has been analyzed in \cite{Coutand-Shkoller-2005}.

\section{Numerical method}\label{section:numerical}

In this section, we introduce the numerical approximation of the coupled FSI system \eqref{eq:coupling-PDE}--\eqref{initial-condition}. 
The proposed numerical method combines a Lagrange multiplier treatment of the interface coupling with a CutFEM-based fictitious domain formulation in the fluid. We conclude this section by stating the main result of this work. 

\subsection{Integral formulation with Lagrange multiplier}
\label{section:integral-form}

For regular solutions of the FSI problem, the interface condition in \eqref{eq:coupling-nl-b} can be equivalently written in the following integral form through the application of the Piola transformation (see, e.g., \cite{ciarlet-88}): 
\begin{align}\label{interface-weak}
\int_{\Gamma(t)} {\bs\sigma}({\bs u},p){\bs n}_{\Gamma(t)}\cdot{\bs\eta}
&= \int_{\widehat\Gamma} {\bs\sigma}^{\rm s}({\bs d}) {\bs n}_{\widehat\Gamma}\cdot({\bs\eta} \circ{\bs X})
\quad \forall {\bs\eta}\in L^2\big(\Gamma(t)\big)^d ,
\end{align}
where ${\bs n}_{\Gamma(t)}$ denotes the unit normal vector on $\Gamma(t)$ pointing to the solid. 
Let $\widehat{\bs\sigma} : \widehat \Gamma \times [0,T]\rightarrow \mathbb{R}^d $ denote the trace of the fluid traction on the fluid-structure interface, i.e., the normal component of the Cauchy stress pulled back to the reference solid configuration:
\begin{align}\label{def-hat-sigma}
\widehat{\bs\sigma}(t)
\eqd \big[{\bs\sigma}\big({\bs u}(t),p(t)\big)\,{\bs n}_{\Gamma(t)}\big]\circ \bs X(t)
\quad \text{on}\quad  \widehat\Gamma 
\end{align}
for $t\in[0,T]$. 
By using the interface condition in \eqref{eq:coupling-nl-a} and the integral formulation \eqref{interface-weak} of the interface force balance \eqref{eq:coupling-nl-b}, we see that any sufficiently regular solution to the FSI system \eqref{eq:coupling-PDE}–\eqref{eq:coupling-nl} satisfies the following relation for every $t\in(0,T]$: 
\begin{equation}\label{eq:weak-fsi}
\begin{aligned}
&\int_{\Omega^{\rm f}(t)} \partial_t \bu \cdot \bv
+ \int_{\Omega^{\rm f}(t) } (\bu\cdot\bs \nabla)\bu \cdot \bv
+ \int_{\Omega^{\rm f}(t)} \bs D(\bu):\bs D(\bv) \\
&\quad - \int_{\Omega^{\rm f}(t)} p\bs \nabla\cdot \bv
+ \int_{\Omega^{\rm f}(t)} q\bs \nabla\cdot  \bu
+ \int_{\widehat{\Omega}^{\rm s}} \partial_{tt}{\bs d} \cdot \bw
+ \int_{\widehat{\Omega}^{\rm s}} [{\bs D}({\bs d}) : {\bs D}({\bs w}) + (\grad\cdot{\bs d})(\grad\cdot{\bs w}) ] \\
&\quad - \int_{\Gamma(t)} \widehat{\bs\sigma}\circ \bs X^{-1}
\cdot \big(\bv - \bw\circ \bs X^{-1}\big)
+ \int_{\Gamma(t)} \bs\lambda\circ \bs X^{-1}
\cdot \big(\bu - \partial_t{\bs d}\circ \bs X^{-1}\big) = 0 
\end{aligned}
\end{equation}
for all $(\bv,q,\bw,\bs\lambda)\in H^1(\Omega^{\rm f}(t))^d
\times L^2(\Omega^{\rm f}(t))
\times H^1(\widehat{\Omega}^{\rm s})^d
\times L^2(\widehat\Gamma)^d$ with  ${\bs v}\vert_{\Gamma^{\rm f}} = {\bs 0}$. 
One can interpret  $\widehat{\bs\sigma}$ as a Lagrange multiplier enforcing the kinematic coupling condition
${\bs u} = \partial_t{\bs d}\circ {\bs X}^{-1} $ on $\Gamma(t)$.

\begin{remark}\upshape
In the weak formulation~\eqref{eq:weak-fsi}, the Lagrange multiplier $\widehat{\bs \sigma}$ is defined on the reference interface $\widehat{\Gamma}$, whereas the integral enforcing the kinematic interface condition
$\bu = \partial_t \bs d \circ \bs X^{-1}$
is evaluated on the current interface $\Gamma(t)$.
This mixed reference/current formulation is particularly  convenient for the CutFEM spatial semi-discrete approximation in the fluid domain proposed in Section~\ref{section:FEM} and  facilitates the a priori error analysis of the resulting method. Alternatively, one may enforce the coupling entirely on the reference configuration, viz.,
$$
- \int_{\widehat{\Gamma}} \widehat{\bs\sigma}\cdot \big(\bv \circ \bs X - \bw\big)\,\mathrm{d}\widehat{s}
+ \int_{\widehat{\Gamma}} \bs\lambda \cdot \big(\bu \circ \bs X - \partial_t{\bs d}\big)\,\mathrm{d}\widehat{s},
$$
as in~\cite{BoffiGastaldi}.
For that approach, however, a complete error analysis is still not available.
\end{remark}

\subsection{Mesh and finite element spaces}\label{section:mesh}

Let $\widehat{\mathcal T}_H^{\rm s}$ be the set of simplices (triangles in 2D and tetrahedra in 3D) forming a shape-regular and quasi-uniform triangulation of $\widehat\Omega^{\rm s}_H$. Thus 
$$
\widehat\Omega^{\rm s}_H \eqd \bigcup_{\widehat K\in \widehat{\mathcal T}_H^{\rm s}} \widehat K, \quad  
\widehat\Gamma_H \eqd \partial \widehat\Omega^{\rm s}_H
$$ 
approximate the reference domain $\widehat\Omega^{\rm s}$ and the reference interface $\widehat\Gamma$, respectively. We allow $\widehat{\mathcal T}_H^{\rm s}$ to contain  curved simplices that interpolate the interface $\widehat\Gamma $ using polynomial mappings of degree $k\ge 2$ from a reference flat triangle/tetrahedron, 
the setup used to define isoparametric finite elements of degree $k$ (see, e.g.,   \cite{Lenoir-1986}). 
In addition, we denote by ${\mathcal T}_h$ a regular triangulation of the ambient fluid computational domain $\Omega$ (see Figure~\ref{fig:meshes}), with
$
\Omega \eqd \bigcup_{K\in {\mathcal T}_h} K  .
$
\begin{figure}[h!]
        \centering    \includegraphics[width=0.6\textwidth]{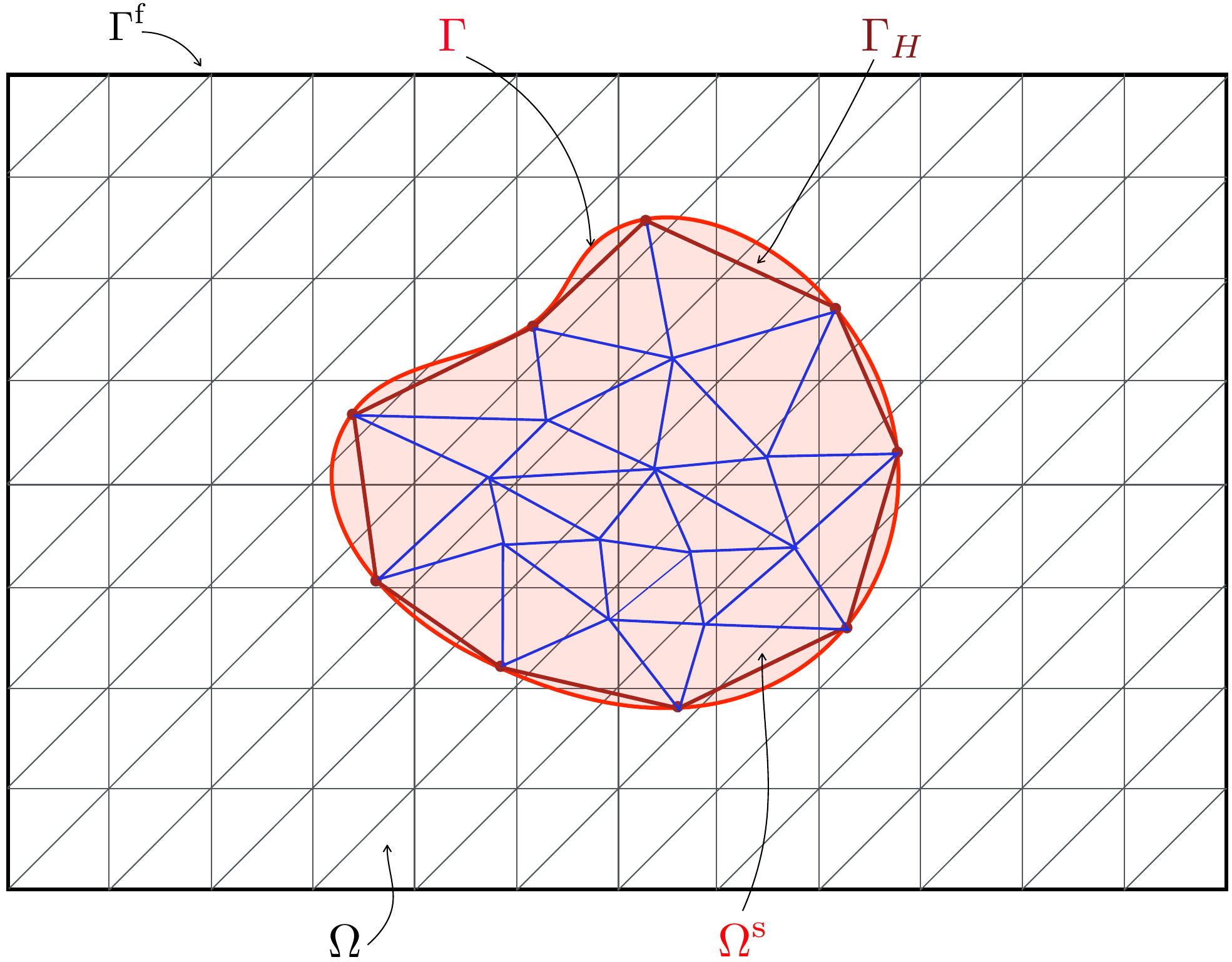}
         \caption{The computational domain with the background mesh, the embedded solid domain, its triangulation,  and the fluid-solid interfaces.
         }
      \label{fig:meshes}
\end{figure}

We consider a finite element method  with isoparametric fitted mesh for the reference solid domain and unfitted mesh for the fluid domain, by introducing the following finite element spaces for the fluid velocity and pressure, the solid displacement, and the pull-back normal stress (pulled back to the reference interface $\widehat\Gamma$):
$$
\begin{aligned}
 {\bs V}_h^{\rm f} \eqd  & \big\{  {\bs v}_h \in H^1_0(\Omega)^d: \bv_h\vert_K \in \mathbb{P}_k(K)^d\quad \forall K\in {\mathcal T}_h \big\} ,\\ 
 Q_h^{\rm f} \eqd  & \big\{ q_h \in H^1(\Omega):  q_h\vert_K \in \mathbb{P}_{k-1}(K)\quad \forall K\in {\mathcal T}_h \big\},\\ 
 \widehat{\bs V}_H^{\rm s} \eqd  & \big\{  \bs w_H \in H^1(\widehat\Omega^{\rm s}_H)^d: \bw_H\vert_{\widehat K} \in \mathbb{P}_k(\widehat K)^d\quad \forall \widehat K\in \widehat{\mathcal T}_H^{\rm s} \big\},\\
 \widehat{\bs \Lambda}_H \eqd  & \big\{  \bs w_H\vert_{\widehat\Gamma_H}: \bw_H \in  \widehat{\bs V}_H^{\rm s}\big\} , 
\end{aligned}
$$
where $\mathbb{P}_k(K)$ stands for the space of polynomials of degree up to $k$ on the simplex $K$. The notation $\bw_H\vert_{\widehat K} \in \mathbb{P}_k(\widehat K)^d$ in defining the solid displacement finite element space $\widehat{\bs V}_H^{\rm s}$ simply means that the pull-back of $\bw_H|_{\widehat K}$ onto a reference flat triangle/tetrahedron is a polynomial of degree up to $k$. 

In what follows, 
we denote by $(\bu_h,p_h, \dep_H, \widehat{\bs \sigma} _H)$ the  
spatial semi-discrete 
finite element approximation of 
$(\bu,p, {\bs d}, \widehat{\bs \sigma} )$, with  
$\big(\bu_h(t),p_h(t), {\bs d}_H(t), \widehat{\bs\sigma} _H(t)\big)\in {\bs V}_h^{\rm f} \times  Q_h^{\rm f} \times \widehat{\bs V}_H^{\rm s} \times \widehat{\bs \Lambda}_H$
for $t \in [0,T]$.
Note that $(\bu_h(t),p_h(t))$ 
is  defined into the whole fluid computational  domain $\Omega$.
Let ${\bs I}$ denote the identity map on $\R^d$, i.e., ${\bs I}(\bs x)= \bs x$ for $\bs x\in\R^d$.
We can then introduce the approximated solid motion map as $\bs X_H \eqd  {\bs I} + {\bs d}_H $. Assuming that for each $t\in [0,T]$ the map $\bs X_H(t)$ is invertible and orientation preserving, we can define the numerical domains: 
\begin{align}
\Omega^{\rm s}[\bs X_H(t)] 
\eqd  \bs X_H\big(\widehat{\Omega}^{\rm s}_H,t\big) ,\quad 
\Omega^{\rm f}[\bs X_H(t)] 
\eqd \Omega \backslash \overline{\Omega^{\rm s}[\bs X_H(t)]}
\end{align}
and the numerical interface 
\begin{align}
\Gamma[\bs X_H(t)] \eqd  \bs X_H\big(\widehat \Gamma_H,t\big) =   \partial \Omega^{\rm s}[\bs X_H(t)] . 
\end{align}

\subsection{Unfitted finite element approximation}\label{section:FEM}

We consider the following  spatial semi-discrete unfitted finite-element method: for $t\in(0,T]$,  find $\big(\bu_h(t),p_h(t), {\bs d}_H(t), \widehat{\bs\sigma} _H(t)\big)\in {\bs V}_h^{\rm f} \times  Q_h^{\rm f} \times \widehat{\bs V}_H^{\rm s} \times \widehat{\bs \Lambda}_H,$ 
satisfying the following relation: 
\begin{equation}
\label{eq:discrete-semi-fsi}
\begin{aligned}
&\int_{\Omega^{\rm f}[\bs X_H]} \partial_t \bs u_h \cdot {\bs v}_h  + \int_{\Omega^{\rm f}[\bs X_H]}   (\bs u_h \cdot \grad) \bu_h \cdot {\bs v}_h 
+   \int_{\Omega^{\rm f}[\bs X_H]} \bs D(\bu_h): \bs D(\bv_h)  \\
& - \int_{\Omega^{\rm f}[\bs X_H]}  p_h \grad\cdot \bv_h + \int_{\Omega^{\rm f}[\bs X_H]}  q_h \grad\cdot \bu_h \\
& +  \int_{\widehat{\Omega}^{\rm s}_H} \partial_{tt}{\bs d}_H \cdot \bw_H + \int_{\widehat{\Omega}^{\rm s}_H} [{\bs D}({\bs d}_H) : {\bs D}({\bs w}_H) + (\grad\cdot{\bs d}_H)(\grad\cdot{\bs w}_H) ] \\
& - \int_{\Gamma[\bs X_H]} \widehat{\bs \sigma} _H \circ {\bs X}_H^{-1} \cdot ( \bv_h - \bw_H \circ {\bs X}_H^{-1}) 
+ \int_{\Gamma[\bs X_H]} \bs \lambda_H \circ {\bs X}_H^{-1} \cdot ( \bu_h - (\partial_t{\bs d}_H) \circ {\bs X}_H^{-1})  \\
& + \epsilon_h g_h^{\rm u} (\partial_t\bu_h,\bv_h ) + g_h^{\rm u} (\bu_h,\bv_h ) + g_h^{\rm p} (p_h,q_h ) 
= 0 
\end{aligned}
\end{equation}
for  all $(\bv_h,q_h, \bw_H, \bs \lambda_H)\in {\bs V}_h^{\rm f} \times  Q_h^{\rm f} \times \widehat{\bs V}_H^{\rm s} \times \widehat{\bs \Lambda}_H$, where $g_h^{\rm u}$ and $g_h^{\rm p}$ are numerical stabilization-extension  forms consisting of global ghost-penalty terms acting on ${\bs V}_h^{\rm f}\times {\bs V}_h^{\rm f}$ and $Q_h^{\rm f}\times Q_h^{\rm f}$ (see, e.g., \cite{burman-10, Burman_Hansbo_Larson_Zahedi_2025,maxim-and-whal-25}):
\begin{equation}
\begin{split}
g_h^{\rm u}(\bu,\bv) 
&\eqd \sum_{F\in {\mathcal F}_h} \sum_{j=1}^{k} h^{2j-1}\int_F \jump{\partial^j_{\bn_F} \bu} \jump{\partial^j_{\bn_F} \bv} ,\\
g_h^{\rm p}(p,q) 
&\eqd \sum_{F\in {\mathcal F}_h} \sum_{j=1}^{k-1} h^{2j+1} \int_F  \jump{\partial^j_{\bn_F} p}\jump{\partial^j_{\bn_F}q} ,
\end{split}
\label{eqn:GPForms}
\end{equation}
with $\jump{\partial^j_{\bn_F} \bv}$ denoting the jump of the $j$th-order normal derivative on the face $F$ in ${\mathcal F}_h $, the collection of all internal faces from ${\mathcal T}_h$, and $\epsilon_h\in(0,h]$ is a small stabilization parameter. The stabilization terms vanish if ${\bs u}$ and $p$ are smooth functions rather than functions from the finite element spaces. 

We choose the discrete initial solid displacement $\dep_H(0)$ to be the interpolation of the exact initial solid displacement $\dep_0$ (see Section~\ref{section:initial-error}). Thus the numerical initial interface $\Gamma[{\bs X}_H(0)]$, which defines the numerical initial fluid domain $\Omega^{\rm f}[{\bs X}_H(0)]$, interpolates the exact initial interface $\Gamma(0)$. 
We assume that the (divergence free) initial velocity ${\bs u}_0$ can be smoothly extended to $\Omega$ and choose ${\bs u}_h(0)={\bs u}_{0,h}$, where $({\bs u}_{0,h},p_{0,h}) \in {\bs V}_h^{\rm f} \times Q_h^{\rm f}$ is the Ritz projection of $({\bs u}_0,0)$ onto the finite element space ${\bs V}_h^{\rm f} \times Q_h^{\rm f}$, defined by 
\begin{equation}
\begin{aligned}\label{def-Ritz-0}
& \int_{\Omega^{\rm f}[\bs X_H]} \bs D({\bs u}_{0,h}-{\bs u}_0): \bs D(\bv_h)  - \int_{\Omega^{\rm f}[\bs X_H]} p_{0,h} \grad\cdot\bv_h + \int_{\Omega^{\rm f}[\bs X_H]}  q_h \grad\cdot {\bs u}_{0,h}\\
&\quad + g_h^{\rm u} ({\bs u}_{0,h},\bv_h ) 
+ g_h^{\rm p} (p_{0,h},q_h ) = 0 
\end{aligned}
\end{equation}
for all  $({\bs v}_h,q_h)\in {\bs V}_h^{\rm f} \times Q_h^{\rm f}$; see Section \ref{section:consistency-2} for its definition and approximation properties. The reason for choosing the initial fluid velocity via Ritz projection is explained in Section \ref{section:initial-error}, and particularly  to have the properties \eqref{ghost-p0h-property}--\eqref{ghost-ep-0} required in the error analysis. 

By the definition of the solid-displacement finite element space, the two finite element functions $\partial_t{\bs d}_H(0)|_{\widehat\Gamma_H}$ and $[\partial_t{\bs d}_H(0)] \circ {\bs X}_H(0)^{-1}|_{\Gamma[{\bs X}_H(0)]}$ have the same pull-back onto the reference flat triangle/tetrahedron. Therefore, on the reference interface $\widehat\Gamma_H$, we can define $\partial_t{\bs d}_H(0)$ via the relation 
\begin{equation}\label{eq:init-nodal-boundary}
[\partial_t{\bs d}_H(0)] \circ {\bs X}_H(0)^{-1}|_{\Gamma[{\bs X}_H(0)]} = {\bf P}_{\Gamma[{\bs X}_H(0)]}{\bs u}_{0,h} ,
\end{equation}
where ${\bf P}_{\Gamma[{\bs X}_H(0)]}$ denotes the $L^2$ projection onto the finite element space ${\bs\Lambda}_H(\Gamma[{\bs X}_H(0)])$ on the interface $\Gamma[{\bs X}_H(0)]$; see its definition in \eqref{FE-space-GammaXH}. At the interior solid nodes, we simply set  $\partial_t{\bs d}_H(0)$ to the interior nodal values of $\vel_0$. 

It should be noted  that the initial solid velocity on the interface $\Gamma[{\bs X}_H(0)]$ is chosen to be the $L^2$ projection of ${\bs u}_{0,h}$ in order to ensure the following compatibility condition at $t=0$:
\begin{align}\label{initial-constraint}
\int_{\Gamma[\bs X_H]} {\bs\lambda}_H \circ {\bs X}_H^{-1} \cdot \big( \bu_h - (\partial_t{\bs d}_H) \circ {\bs X}_H^{-1}\big) = 0
\quad\forall\, {\bs\lambda}_H \in \widehat{\bs\Lambda}_H.
\end{align}
Since \eqref{eq:discrete-semi-fsi} requires this condition to hold for $t \in (0,T]$, the initial values must be compatible with it, which guarantees local well-posedness of \eqref{eq:discrete-semi-fsi}, as shown in Appendix~\ref{section:wellposed}.

\begin{remark}\upshape\upshape 
The energy balance of the scheme can be improved by ``symmetrizing'' the inertia term discretization, e.g., by replacing $\int_{\Omega^{\rm f}[\bs X_H]}   (\bs u_h \cdot \grad) \bu_h \cdot {\bs v}_h $ by the following term in the left-hand side of \eqref{eq:discrete-semi-fsi}:  
$$
\frac12 \int_{\Omega^{\rm f}[\bs X_H]}
\big[({\bs u}_h\cdot\grad) {\bs u}_h \cdot {\bs v}_h 
- ({\bs u}_h\cdot\grad) {\bs v}_h \cdot {\bs u}_h\big] 
+ \frac12 \int_{\Gamma[\bs X_H]} \big(\partial_t{\bs d}_H\cdot{\bs n}_{\Gamma[\bs X_H]}\big) {\bs u}_h \cdot {\bs v}_h . 
$$
However, this is not required by the error analysis below provided the mesh sizes $H$ and $h$ are sufficiently small. 
\end{remark}

\subsection{The main theoretical result}

In this paper, we provide an 
a priori error estimate for the unfitted finite element method
\eqref{eq:discrete-semi-fsi} applied to the fully coupled FSI problem \eqref{eq:coupling-PDE}–\eqref{eq:coupling-nl}. The error bound (see Theorem~\ref{THM} below) controls the
fluid velocity, the solid displacement and velocity, the pressure, and the
interface traction errors in their natural norms, on any time interval over which the
exact solution remains smooth.
 This is proved under the following assumptions on the domain, regularity of solution and mesh sizes.
\medskip

\paragraph{\bf Basic assumptions:}
\begin{enumerate}
  \item[(C1)] The solid reference domain $\widehat{\Omega}^{\mathrm{s}}$ is bounded with a smooth boundary. The computational domain $\Omega$ is either a bounded smooth domain or a convex polygonal (polyhedral) domain such that the $H^2\times H^1$ regularity of the following Stokes problem holds:
  \begin{align}\label{def-v-q-0}
  \left\{
  \begin{aligned}
  -\grad\cdot({\bs D}(\bv) + q{\bs I}) = {\bs f} &\quad \mbox{in}\quad \Omega^{\rm f}(t), \\
  \grad\cdot{\bs v}=0 &\quad\mbox{in}\quad \Omega^{\rm f}(t), \\
  ({\bs D}(\bv) + q{\bs I}){\bs n} = \bs 0 &\quad\mbox{on}\quad \Gamma(t), \\
  {\bs v}= \bs 0 &\quad \mbox{on}\quad \Gamma^{\rm f} .
  \end{aligned}
  \right.
  \end{align}
  Namely, whenever ${\bs f} \in L^2(\Omega^{\rm f}(t))$ the solution of \eqref{def-v-q-0} satisfies the following estimate: 
  \begin{align}\label{H2-H1-estimate-Stokes}
  \|{\bs v}\|_{H^2(\Omega^{\rm f}(t))} + \|q\|_{H^1(\Omega^{\rm f}(t))}
  \leq C\|{\bs f}\|_{L^2(\Omega^{\rm f}(t))} . 
  \end{align}

  \item[(C2)] The solid displacement ${\bs d}:
  \widehat{\Omega}^{\mathrm{s}}\times[0,T]\rightarrow\R^d$ is smooth, and for each $t\in[0,T]$ the mapping
  \[
  \boldsymbol{X}(t) = {\bs I} + {\bs d}(t) : \widehat{\Omega}^{\mathrm{s}}\to \mathbb{R}^d
  \]
  is an orientation-preserving diffeomorphism. Moreover, the inverse deformation gradient is uniformly bounded in time, i.e., 
  \[
  \max_{t\in[0,T]} \|\nabla \boldsymbol{X}^{-1}(t)\|_{L^{\infty}(\Omega^{\mathrm{s}}[\boldsymbol{X}(t)])}
  \;\le\; C_0
  \]
  for some constant $C_0>0$.

  \item[(C3)] The fluid solution $(\boldsymbol{u},p)$ is smooth in the space-time fluid domain  $\mathcal Q^{\rm f}$
  and admits a smooth extension to $\Omega\times[0,T]$. To simplify the notation, we use the same notation $(\bu,p)$ to denote the extended functions  on $\Omega\times[0,T]$.

  \item[(C4)] Mesh conditions: Our analysis requires the mesh sizes $h$ and $H$ in the fluid and solid  satisfy the condition {$c_1 H\le h\le c_2H$}, with some constants $c_1>0$ and $c_2>0$. However, in order to facilitate the distinction between fluid and solid quantities we keep different notations for both mesh parameters.   
\end{enumerate} 

Under condition (C1), the discrete reference domain $\widehat{\Omega}^{\rm s}_H$ constructed by using isoparametric finite elements of degree $k$ has the following approximation property (see \cite{Lenoir-1986}): There exists a globally Lipschitz continuous, piecewise smooth (smooth on each simplex) and bijective map 
$
\widehat{\bs\Phi}_H : \widehat{\Omega}^{\rm s}_H \rightarrow  \widehat{\Omega}^{\rm s}
$
such that $\widehat\Gamma = \widehat{\bs\Phi}_H(\widehat\Gamma_{ H})$ and
\begin{equation}\label{GeomApprox}
\|\widehat{\bs\Phi}_H - {\bs I} \|_{L^{\infty}(\widehat{\Omega}^{\rm s}_H)} +
H\,\| \widehat{\bs\Phi}_H - {\bs I} \|_{W^{1,\infty}(\widehat{\Omega}^{\rm s}_H)}\le C_1 H^{k+1} ,
\end{equation} 
where $C_1$ is a constant depending on the smoothness of the domain $\widehat{\Omega}^{\rm s}$. 

 The main result of the paper is stated in the following theorem.

\begin{theorem}\label{THM}
Under the assumptions (C1)--(C4), with $\epsilon_h\in(0,h]$, there exist positive constants $h_0$ and $H_0$ such that, for $h\le h_0$ and $H\le H_0$, the finite element approximation provided by \eqref{eq:discrete-semi-fsi} with finite elements of degree $k\ge 3$ exists and satisfies the following error bound{\rm:}
\begin{align}\label{Est:main}
&\max_{t\in[0,T]} \left\{ \| {\bs u}_h - {\bs u} \|_{L^2(\Omega^{\rm f}[\bs X_H])}^2 
+ \| {\bs d}_H - {\bs d}\circ\widehat{\bs\Phi}_H \|_{H^1(\widehat{\Omega}^{\rm s}_H)}^2 
+ \| \partial_t\dep_H - \partial_t\dep\circ\widehat{\bs\Phi}_H \|_{L^2(\widehat{\Omega}^{\rm s}_H)}^2 \right\}  \notag\\
&\quad\, 
+ \int_0^{T} \| {\bs u}_h - {\bs u} \|_{H^1(\Omega^{\rm f}[\bs X_H])}^2 \d t \notag\\
&\quad\, 
+ H^2 \int_0^{T} \left( \| p_h - p \|_{L^2(\Omega^{\rm f}[\bs X_H])}^2 
+ \| \widehat{\bs\sigma}_H - \widehat{\bs\sigma}\circ\widehat{\bs\Phi}_H \|_{H^{-\frac12}(\widehat\Gamma_H)}^2 \right) \d t 
\notag\\
&\le C(h^{2k}+H^{2k}) ,  
\end{align}
where the constants $h_0$, $H_0$ and $C$ may depend on $T$ and the norms of the exact solution, but are independent of $\epsilon_h\in(0,h]$. 
\end{theorem}

\begin{remark}\label{remark0_THM}\upshape 
The restriction $k\ge 3$ is technical rather than essential. It is used only in \eqref{requive-kgeq3}, where an inverse estimate for finite element functions is invoked to control the $W^{1,4}$-error between the Lagrange interpolant and the Ritz projection by the corresponding $H^1$-error. If a direct $W^{1,p}$-error estimate for the Ritz projection in the CutFEM setting is established in future work, then the conclusion of Theorem \ref{THM} would immediately extend to the case $k=2$. This is a problem of independent interest and will be considered in a separate work. 
\end{remark}

\begin{remark}\label{remark-epsilonh}\upshape
The inclusion of the small stabilization term
$\epsilon_h g_h^{\rm u} (\partial_t\bu_h,\bv_h )$ in the numerical scheme is primarily for the convenience of proving local well-posedness of the spatial semidiscrete problem \eqref{eq:discrete-semi-fsi} (adding this stabilization term changes the differential algebraic equations to ordinary differential equations, as detailed  in Appendix~\ref{section:wellposed}). The use of stabilization for the discrete time derivative is also supported by the recent analysis of parabolic TraceFEM in \cite{Bouck-Nochetto-Shakipov-Yushutin-2026}, where it is shown to be instrumental in establishing uniform inf--sup stability and favorable conditioning. Although that work concerns stationary surface problems, its conclusions provide further motivation for time-derivative stabilization in unfitted finite element methods. By proving Theorem~3.1 with a constant $C$ independent of $\epsilon_h\in(0,h]$, and then passing to the limit $\epsilon_h\to0$ using a compactness argument for the finite-dimensional-valued numerical solutions (the details are omitted), one obtains existence of numerical solutions and the same error bound for $\epsilon_h=0$.
\end{remark}

\begin{remark}\upshape\label{remark1_THM}\upshape 
For clarity in stating the assumptions and the main theorem, we have focused on the setting where the solid region is fully immersed in the fluid region. This choice allows us to prescribe the external boundary condition $\boldsymbol{u}=\bs 0$ on $\Gamma^{\mathrm{f}}$ and to present the proof within this specific framework. However, the arguments developed in this paper extend directly---without essential modification---to alternative configurations in which the solid is not fully immersed in the fluid. 
\end{remark}


\begin{remark}\upshape\label{remark3_THM}\upshape 
When the interface $\Gamma(t)$ reaches  the external boundary $\partial\Omega$, the solution may develop singularities along the contact line $\Gamma(t)\cap\partial\Omega$. These singularities are analogous to the well-known corner and edge singularities in polygonal or polyhedral domains. Establishing a rigorous convergence theory in this setting---under reduced regularity assumptions on the solution and with locally graded meshes toward corners and edges---is considerably more involved and is deferred to future work.
\end{remark}

The rest of this paper is devoted to the proof of Theorem~\ref{THM}. 

\section{Interpolated solution and error equation}\label{sect:interp}

Throughout the paper, let $C>0$ denote a generic constant that is independent of the mesh sizes $h$ and $H$, of the physical time $t$, and of the critical time $t_*$ introduced in Section~\ref{section:outline} below. We write $A \lesssim B$ to indicate that $A \le CB$ for such a constant $C$, and we use $A \simeq B$ to denote $A \lesssim B \lesssim A$.  Since the analysis involves several configurations, interfaces, and transported error functions, a summary of the main notation used throughout the paper is provided in Appendix~\ref{appendix:notation-summary}. 

\subsection{Interpolated solution}

For $t\in[0,T]$, we define the Lagrange interpolants of the exact solid displacement and the (pulled-back) traction by
\begin{align}\label{def-sigma-s}
\dep_H^*(t)\eqd {\bs I}_H^{\rm s}\big(\dep(t)\circ \widehat{\bs\Phi}_H\big),
\qquad
\widehat{\bs\sigma}_H^*(t)\eqd {\bs I}_H^{\rm s}\big(\widehat{\bs\sigma}(t)\circ \widehat{\bs\Phi}_H\big),
\end{align}
which take values in the finite element spaces $\widehat{\bs V}_H^{\rm s}$ and $\widehat{\bs \Lambda}_H$, respectively. We also use the shorthand
\begin{align*}
 \vel_H^*(t)\eqd\partial_t \dep_H^*(t)
= {\bs I}_H^{\rm s}\big(\partial_t{\bs d}(t)\circ \widehat{\bs\Phi}_H\big),\quad
\vel_H \eqd  \partial_t\dep_H, 
\quad  \vel \eqd  \partial_t\dep . 
\end{align*}

Under the standing assumption that $\dep(t)$ and $\widehat{\bs\sigma}(t)$ are smooth, the Lagrange interpolants satisfy, for all $t\in[0,T]$ (see \cite{Lenoir-1986}):
\begin{align}\label{interpl-error-d}
\begin{aligned}
\| \dep_H^*(t)\circ \widehat{\bs\Phi}_H^{-1} - \dep(t) \|_{L^\infty(\widehat{\Omega}^{\rm s})}
+ H\,\| \dep_H^*(t)\circ \widehat{\bs\Phi}_H^{-1} - \dep(t) \|_{W^{1,\infty}(\widehat{\Omega}^{\rm s})}
&\lesssim H^{k+1},\\
\| \vel_H^*(t)\circ \widehat{\bs\Phi}_H^{-1} - \vel(t) \|_{L^\infty(\widehat{\Omega}^{\rm s})}
+ H\,\| \vel_H^*(t)\circ \widehat{\bs\Phi}_H^{-1} - \vel(t) \|_{W^{1,\infty}(\widehat{\Omega}^{\rm s})}
&\lesssim H^{k+1},\\
\| \widehat{\bs\sigma}_H^*(t)\circ \widehat{\bs\Phi}_H^{-1} - \widehat{\bs\sigma} \|_{L^\infty(\widehat\Gamma)}
&\lesssim H^{k+1}.
\end{aligned}
\end{align}
Consequently, the interpolated motion map ${\bs X}_H^*\eqd \bs I +\dep_H^*$ satisfies
\begin{align}\label{interpl-error-X}
\| {\bs X}_H^*(t)\circ \widehat{\bs\Phi}_H^{-1} - {\bs X}(t) \|_{L^\infty(\widehat{\Omega}^{\rm s})}
+ H\,\| {\bs X}_H^*(t)\circ \widehat{\bs\Phi}_H^{-1} - {\bs X}(t) \|_{W^{1,\infty}(\widehat{\Omega}^{\rm s})}
\lesssim H^{k+1}.
\end{align}
Similarly, for $t\in[0,T]$ let $(\bu_h^*(t),p_h^*(t)) \eqd ({\bs I}_h^{\rm f}\bu(t),\, I_h^{\rm f}p(t))$ be the interpolants of $(\bu(t),p(t))$ onto $ {\bs V}_h^{\rm f}\times Q_h^{\rm f}$. They satisfy the estimates
\begin{align}\label{interpl-u}
\| \bu_h^*(t)-\bu(t) \|_{L^\infty(\Omega)}
+ h\,\| \bu_h^*(t)-\bu(t)\|_{W^{1,\infty}(\Omega)}
\lesssim h^{k+1},\quad \| p_h^*(t)-p(t) \|_{L^\infty(\Omega)}
\lesssim h^{k}.
\end{align} 
We define the interpolated solid and fluid subdomains, and the corresponding interface, by
\begin{equation}\label{def-Omega*}
\Omega^{\rm s}[{\bs X}_H^*(t)] \eqd 
\bs X_H^*(\widehat{\Omega}^{\rm s}_H,t) , 
\quad 
\Omega^{\rm f}[{\bs X}_H^*(t)] \eqd \Omega \backslash \overline{{\Omega}^{\rm s}[{\bs X}_H^*(t)]} ,\quad 
\Gamma[\bs X_H^*(t)] \eqd 
{\bs X}_H^*(\widehat\Gamma_H,t)
\end{equation}
From ${\bs X}(t):\widehat{\Omega}^{\rm s}\to \Omega^{\rm s}(t)$ and $\widehat{\bs\Phi}_H:\widehat{\Omega}^{\rm s}_H\to \widehat{\Omega}^{\rm s}$, we consider the compositional map 
\begin{equation}\label{eq:map-phiHt}
{\bs\Phi}_H(t)
\eqd {\bs X}(t)\circ \widehat{\bs\Phi}_H \circ {\bs X}_H^*(t)^{-1}
:\ \Omega^{\rm s}[{\bs X}_H^*(t)] \longrightarrow \Omega^{\rm s}(t).
\end{equation}
Finally, it follows that the finite element interpolant $(\bs u_h^*, p_h^*, \dep_H^*, \widehat{\bs \sigma}_H^*)$ satisfies the following integral identity for all $t\in [0,T] $ and $({\bs v}_h, q_h, \bs w_H, \bs \lambda_H)\in  {\bs V}_h^{\rm f} \times  Q_h^{\rm f} \times \widehat{\bs V}_H^{\rm s} \times  \widehat{\bs \Lambda}_H$: 
\begin{align}\label{eq:interp-semi-fsi}
&\int_{\Omega^{\rm f}[\bs X_H^*]} \partial_t \bs u_h^* \cdot {\bs v}_h  + \int_{\Omega^{\rm f}[\bs X_H^*]} (\bu_h^* \cdot \grad) \bu_h^* \cdot {\bs v}_h 
+  \int_{\Omega^{\rm f}[\bs X_H^*]} \bs D(\bu_h^*): \bs D(\bv_h)  \notag\\ 
& - \int_{\Omega^{\rm f}[\bs X_H^*]}  p_h^* \grad\cdot\bv_h + \int_{\Omega^{\rm f}[\bs X_H^*]}  q_h \grad\cdot\bu_h^* \notag\\
& +  \int_{\widehat{\Omega}^{\rm s}_H} \partial_t\vel_H^* \cdot \bw_H + 
\int_{\widehat{\Omega}^{\rm s}_H} [{\bs D}({\bs d}_H^*) : {\bs D}({\bs w}_H) + (\grad\cdot{\bs d}_H^*)(\grad\cdot{\bs w}_H) ]  \notag\\
& - \int_{\Gamma[\bs X_H^*]} \widehat{\bs \sigma} _H^* \circ ({\bs X}_H^*)^{-1}\cdot ( \bv_h - \bw_H \circ ({\bs X}_H^*)^{-1})  \notag\\
&
+ \int_{\Gamma[\bs X_H^*]} \bs \lambda_H \circ ({\bs X}_H^*)^{-1}\cdot ( \bu_h^* - \vel_H^* \circ ({\bs X}_H^*)^{-1}) \notag\\
&
+ \epsilon_h g_h^{\rm u} (\partial_t\bu_h^*,\bv_h ) + g_h^{\rm u} (\bu_h^*,\bv_h ) + g_h^{\rm p} (p_h^*,q_h ) 
= R_{h,H}^*({\bs v}_h, q_h, \bs w_H, \bs \lambda_H) ,
\end{align}
where $R_{h,H}^*({\bs v}_h, q_h, \bs w_H, \bs \lambda_H)$ denotes the consistency error of the spatial discretization, given by the following relation 
\begin{subequations}\label{def-RhH}
\begin{align}
R_{h,H}^*({\bs v}_h, q_h, \bs w_H, \bs \lambda_H) \eqd & 
\int_{\Omega^{\rm f}[\bs X_H^*]} \partial_t \bs u_h^* \cdot {\bs v}_h  
- \int_{\Omega^{\rm f}(t)} \partial_t \bs u \cdot {\bs v}_h \label{RhH-a}\\
&+ \int_{\Omega^{\rm f}[\bs X_H^*]}   (\bu_h^* \cdot \grad) \bu_h^* \cdot {\bs v}_h 
- \int_{\Omega^{\rm f}(t)}   (\bu \cdot \grad) \bu \cdot {\bs v}_h \label{RhH-b}\\
& +  \int_{\Omega^{\rm f}[\bs X_H^*]} \bs D(\bu_h^*): \bs D(\bv_h)  -  \int_{\Omega^{\rm f}(t)} \bs D(\bu): \bs D(\bv_h) \label{RhH-c}\\ 
&  - \int_{\Omega^{\rm f}[\bs X_H^*]}  p_h^*\, \grad\cdot \bv_h +  \int_{\Omega^{\rm f}(t)}  p \, \grad\cdot \bv_h \label{RhH-d}\\
& + \int_{\Omega^{\rm f}[\bs X_H^*]}  q_h \grad\cdot \bu_h^* -   \int_{\Omega^{\rm f}(t)}  q_h \grad\cdot \bu \label{RhH-e}\\
& +  \int_{\widehat{\Omega}^{\rm s}_H} \partial_{t}\vel_H^* \cdot \bw_H -   \int_{\widehat{\Omega}^{\rm s}} \partial_{t}\vel \cdot \bw_H \circ \widehat{\bs\Phi}_H^{-1} \label{RhH-f}\\
& +\int_{\widehat{\Omega}^{\rm s}_H} [{\bs D}({\bs d}_H^*) : {\bs D}({\bs w}_H) + (\grad\cdot{\bs d}_H^*)(\grad\cdot{\bs w}_H) ] \notag\\
& - \int_{\widehat{\Omega}^{\rm s}} [{\bs D}({\bs d}) : {\bs D}(\bw_H \circ \widehat{\bs\Phi}_H^{-1}) + (\grad\cdot{\bs d})\grad\cdot (\bw_H \circ \widehat{\bs\Phi}_H^{-1}) ] \label{RhH-g}\\
 & - \int_{\Gamma[\bs X_H^*]} \widehat{\bs \sigma} _H^* \circ ({\bs X}_H^*)^{-1}  \cdot \big( \bv_h - \bw_H  \circ ({\bs X}_H^*)^{-1} \big) \label{RhH-h}\\
& + \int_{\Gamma(t)} \widehat{\bs \sigma} \circ {\bs X}^{-1}  \cdot \big( \bv_h - \bw_H  \circ \widehat{\bs\Phi}_H^{-1} \circ  {\bs X}^{-1} \big) \label{RhH-i}\\
&+ \int_{\Gamma[\bs X_H^*]} \bs \lambda_H \circ ({\bs X}_H^*)^{-1} \cdot \big( \bu_h^* - \vel_H^*  \circ ({\bs X}_H^*)^{-1} \big) \label{RhH-j}\\
&- \int_{\Gamma(t)} \bs \lambda_H \circ \widehat{\bs\Phi}_H^{-1}  \circ {\bs X}^{-1}  \cdot \big( \bu - \vel  \circ {\bs X}^{-1} \big) \label{RhH-k}\\
& + \epsilon_h g_h^{\rm u} (\partial_t\bu_h^*,\bv_h ) + g_h^{\rm u} (\bu_h^*,\bv_h ) + g_h^{\rm p} (p_h^*,q_h ) . \label{RhH-l}
\end{align}
\end{subequations}

\subsection{Estimates of the initial errors}
\label{section:initial-error}

We define the discrete error functions   
\begin{subequations}\label{def-eu-ed}
\begin{align}
&{\bs e}_{\bs u} \eqd {\bs u}_h - {\bs u}_h^*,
\quad {e}_{p} \eqd p_h - p_h^* \quad\mbox{in}\quad\Omega,\\
&{\bs e}_\dep \eqd \dep_H - \dep_H^* \quad\mbox{in}\quad\widehat{\Omega}^{\rm s}_H, 
\quad {\bs e}_{{\bs \sigma} } \eqd {\widehat{\bs \sigma} }_H - {\widehat{\bs \sigma} }_H^* \quad\mbox{on}\quad\widehat{\Gamma}_H .
\end{align}
\end{subequations}
As indicated in Section~\ref{section:FEM}, 
the discrete initial displacement in \eqref{eq:discrete-semi-fsi} is set to  $\dep_H(\cdot,0) =  {\bs I}_H^{\rm s} (\dep_0\circ\widehat{\bs\Phi}_H)$. This guarantees that ${\bs X}_H(0)={\bs X}_H^*(0)$ and ${\bs e}_\dep(0)=\bs 0$. 
The fluid initial velocity is set to ${\bs u}_h(0)={\bs u}_{0,h}$, where $({\bs u}_{0,h},p_{0,h}) \in {\bs V}_h^{\rm f} \times Q_h^{\rm f}$ is the Ritz projection of $({\bs u}_0,0)$   defined by \eqref{def-Ritz-0}. Since ${\bs X}_H(0)={\bs X}_H^*(0)$, the approximation properties of the Ritz projection in Section \ref{section:consistency-2} imply the following error bound at $t=0$: 
\begin{align}\label{uh0-error}
\|{\bs e}_{\bs u}(0)\|_{H^1(\Omega^{\rm f}[\bs X_H(0)])}^2
+ \|p_{0,h}\|_{L^2(\Omega^{\rm f}[{\bs X}_H(0)])}^2 
+ g^{\rm u}_h({\bs e}_{\bs u}(0),{\bs e}_{\bs u}(0)) 
+ g^{\rm p}_h(p_{0,h},p_{0,h})
\lesssim  h^{2k} . 
\end{align}
With this  choice of initial data, after taking $({\bs v}_h,{\bs w}_H,{\bs\lambda}_H)=\bs 0 $ in \eqref{eq:discrete-semi-fsi} and ${\bs v}_h=\bs 0$ in \eqref{def-Ritz-0}, and by subtracting the resulting expressions, we get 
\begin{align}\label{ghost-p0h-property}
g_h^{\rm p} (p_h(0)-p_{0,h}, q_h ) 
=
-\int_{\Omega^{\rm f}[\bs X_H(0)]}  q_h \grad\cdot (\bu_h(0) -\bu_{0,h}) 
= 0 
\quad\forall\, q_h\in Q_h^{\rm f} . 
\end{align}
Therefore, choosing $q_h=e_p(0)\eqd p_h(0)-p_h^*(0)$ in the above relation, using the bound of $ g^{\rm p}_h(p_{0,h},p_{0,h})$ provided by \eqref{uh0-error}, and by noting that the ghost penalty vanishes on smooth functions, we get    
\begin{align}\label{ghost-ep-0}
g_h^{\rm p} \big(e_p(0),e_p(0)\big) 
\lesssim h^{2k}. 
\end{align}
This property will be required in our error analysis below.  

Differentiating in time ${\bs e}_{\bs d}={\bs d}_H-{\bs I}_H^{\rm s}({\bs d}\circ\widehat{\bs\Phi}_H)$ at $t=0$  yields $\partial_t{\bs e}_{\bs d}(0)=\vel_H(0)-{\bs I}_H^{\rm s}({\bs\dot{\bs d}}_0\circ\widehat{\bs\Phi}_H)$. Then, using \eqref{eq:init-nodal-boundary}, the stability of the interface $L^2$-projection, the approximation properties of ${\bs u}_{0,h}$ in \eqref{uh0-error}, and the trace estimate from the interface, we obtain
\begin{multline*}
\| \partial_t\dep_H(0)\circ{\bs X}_H(0)^{-1} - {\bs u}_0  \|_{L^2(\Gamma[{\bs X}_H(0)])} 
= 
\| {\bf P}_{\Gamma[{\bs X}_H(0)]} {\bs u}_{0,h}  - {\bs u}_0  \|_{L^2(\Gamma[{\bs X}_H(0)])} \\ 
\lesssim
\| {\bf P}_{\Gamma[{\bs X}_H(0)]}{\bs u}_0 - {\bs u}_0  \|_{L^2(\Gamma[{\bs X}_H(0)])}
+ \| {\bs u}_{0,h}  - {\bs u}_0 \|_{L^2(\Gamma[{\bs X}_H(0)])} \lesssim H^{k} + h^{k} .
\end{multline*}
On the other hand, the initial data satisfy the kinematic compatibility
condition $\dot{\bs d}_0={\bs u}_0\circ{\bs X}(0)$, so that
${\bs I}_H^{\rm s}(\dot{\bs d}_0\circ\widehat{\bs\Phi}_H)
={\bs I}_H^{\rm s}[{\bs u}_0\circ{\bs X}(0)\circ\widehat{\bs\Phi}_H]$, and it follows that
\begin{multline*}
\|{\bs u}_0 - ({\bs I}_H^{\rm s}({\bs\dot{\bs d}}_0 \circ\widehat{\bs\Phi}_H)) \circ {\bs X}_H(0)^{-1} \|_{L^2(\Gamma[{\bs X}_H(0)])} \\
\lesssim 
\|{\bs u}_0 - {\bs u}_0\circ{\bs X}(0) \circ\widehat{\bs\Phi}_H\circ {\bs X}_H(0)^{-1}\|_{L^2(\Gamma[{\bs X}_H(0)])} \notag\\
+\|[{\bs\dot{\bs d}}_0\circ\widehat{\bs\Phi}_H - {\bs I}_H^{\rm s}({\bs\dot{\bs d}}_0\circ\widehat{\bs\Phi}_H) ]\circ {\bs X}_H(0)^{-1} \|_{L^2(\Gamma[{\bs X}_H(0)])} 
\lesssim H^{k} .
\end{multline*}
The two estimates above imply, after pulling the functions back to the reference interface, 
$$
\|\partial_t{\bs e}_{\bs d}(0)\|_{L^2(\widehat\Gamma_H)}
=
\| \partial_t\dep_H(0)- {\bs I}_H^{\rm s}({\bs\dot{\bs d}}_0\circ\widehat{\bs\Phi}_H)\|_{L^2(\widehat\Gamma_H)}\lesssim H^k + h^{k} .
$$
This and $\partial_t{\bs e}_{\bs d}(0)=\partial_t\dep_H(0) - {\bs I}_H^{\rm s}({\bs\dot{\bs d}}_0\circ\widehat{\bs\Phi}_H)=\bs 0$ at the interior nodes of $\widehat\Omega_H^{\rm s}$ imply that, after using the inverse estimates of finite element functions, 
\begin{align}\label{dtdh0-error}
\|\partial_t{\bs e}_{\bs d}(0)\|_{L^2(\widehat\Omega_H^{\rm s})}
\lesssim H^{k+\frac12} + h^{k}H^{\frac12}. 
\end{align}

\subsection{Error equation}

The difference between \eqref{eq:discrete-semi-fsi} and \eqref{eq:interp-semi-fsi} yields the following error equation: 
\begin{subequations}\label{error-eq}
\begin{align}
& \int_{\Omega^{\rm f}[\bs X_H]} \partial_t {\bs e}_{\bs u} \cdot {\bs v}_h \label{error-eq-a}\\
& + \Big( \int_{\Omega^{\rm f}[\bs X_H]} - \int_{\Omega^{\rm f}[\bs X_H^*]} \Big) \partial_t {\bs u}_h^* \cdot {\bs v}_h \label{error-eq-b}\\
& + \int_{\Omega^{\rm f}[\bs X_H]} {\bs e}_{\bs u} \cdot \grad \bu_h \cdot {\bs v}_h \label{error-eq-c}\\
& + \int_{\Omega^{\rm f}[\bs X_H]} {\bs u}_h^* \cdot \grad {\bs e}_{\bs u}  \cdot {\bs v}_h \label{error-eq-d}\\
& + \Big( \int_{\Omega^{\rm f}[\bs X_H]} - \int_{\Omega^{\rm f}[\bs X_H^*]} \Big) {\bs u}_h^* \cdot \grad {\bs u}_h^* \cdot {\bs v}_h \label{error-eq-e}\\
&
+ \int_{\Omega^{\rm f}[\bs X_H]} \bs D({\bs e}_{\bs u}): \bs D(\bv_h) - \int_{\Omega^{\rm f}[\bs X_H]}  e_p \grad\cdot \bv_h
+ \int_{\Omega^{\rm f}[\bs X_H]}  q_h \grad\cdot {\bs e}_{\bs u}  \label{error-eq-f}\\ 
& + \int_{\Omega^{\rm f}[\bs X_H]} {\bs\sigma}(\bu_h^*,p_h^*): \grad \bv_h 
- \int_{\Omega^{\rm f}[\bs X_H^*]} {\bs\sigma}(\bu_h^*,p_h^*): \grad \bv_h 
\label{error-eq-g} \\
& + \Big( \int_{\Omega^{\rm f}[\bs X_H]} - \int_{\Omega^{\rm f}[\bs X_H^*]} \Big) q_h \grad\cdot \bu_h^* \label{error-eq-h}\\
& +  \int_{\widehat{\Omega}^{\rm s}_H} \partial_{tt} {\bs e}_{\dep} \cdot \bw_H 
+ \int_{\widehat{\Omega}^{\rm s}_H} [{\bs D}({\bs e}_\dep) : {\bs D}({\bs w}_H) + (\grad\cdot {\bs e}_\dep)(\grad\cdot{\bs w}_H) ] \label{error-eq-i}\\ 
& - \int_{\Gamma[\bs X_H]} {\bs e}_{{\bs \sigma} }\circ{\bs X}_H^{-1} \cdot ( \bv_h - \bw_H \circ{\bs X}_H^{-1}) \label{error-eq-j}\\
& - \int_{\Gamma[\bs X_H]} {\widehat{\bs \sigma} }_H^*\circ{\bs X}_H^{-1} \cdot ( \bv_h - \bw_H \circ{\bs X}_H^{-1}) 
\label{error-eq-k}\\
& + \int_{\Gamma[\bs X_H^*]} {\widehat{\bs \sigma} }_H^* \circ({\bs X}_H^*)^{-1}\cdot ( \bv_h - \bw_H\circ({\bs X}_H^*)^{-1}) 
\label{error-eq-l}\\
& + \int_{\Gamma[\bs X_H]} \bs \lambda_H \circ{\bs X}_H^{-1}\cdot ( {\bs e}_{\bs u} - \partial_t{\bs e}_\dep \circ{\bs X}_H^{-1} ) \label{error-eq-m}\\
& + \int_{\Gamma[\bs X_H]} \bs \lambda_H \circ{\bs X}_H^{-1} \cdot \big( \bu_h^*  - \partial_t{\bs d}_H^* \circ{\bs X}_H^{-1} \big) \notag\\
& - \int_{\Gamma[\bs X_H^*]} \bs \lambda_H \circ({\bs X}_H^*)^{-1}\cdot \big( \bu_h^* - \partial_t{\bs d}_H^* \circ({\bs X}_H^*)^{-1}\big) \label{error-eq-n}\\[5pt] 
& + \epsilon_h g_h^{\rm u}(\partial_t{\bs e}_{\bs u},{\bs v}_h)  + g_h^{\rm u} ({\bs e}_{\bs u},\bv_h ) + g_h^{\rm p} (e_p,q_h ) \label{error-eq-o}\\[5pt]
& = - R_{h,H}^*({\bs v}_h, q_h, \bs w_H, \bs \lambda_H) . \label{error-eq-p} 
\end{align}
\end{subequations}

\section{Outline of the proof}\label{section:outline}

In this section, we outline the main idea underlying the proof of Theorem~\ref{THM}. Let $\kappa>0$ be a fixed universal constant, and let $\Omega^{\rm f}_{\rm ext}[\bs X_H] $ be a $2\kappa h$-neighborhood of the numerical fluid region $\Omega^{\rm f}[\bs X_H]$, defined by 
\begin{align}\label{Extended-region}
\Omega^{\rm f}_{\rm ext}[\bs X_H] \eqd \bigcup\left\{K\in {\mathcal T}_h: {\rm dist}(K,\Omega^{\rm f}[{\bs X}_H]) \le 2\kappa h\right\} . 
\end{align}
We begin the proof by assuming that there exists a maximal time $t_* \in(0,T]$  such that
\begin{subequations}\label{ind-1}
\begin{align}
\max_{t\in[0,t_*]}\|{\bs e}_\dep\|_{L^\infty(\widehat{\Omega}^{\rm s}_H)} &\le \kappa\,\min(h,H), \label{ind-1-1}\\[2mm]
\max_{t\in[0,t_*]}\!\left\{\|{\bs e}_\dep\|_{W^{1,\infty}(\widehat{\Omega}^{\rm s}_H)} + \|\partial_t{\bs e}_{\bs d}\|_{L^\infty(\widehat{\Omega}^{\rm s}_H)}
+ \|{\bs e}_{\bs u}\|_{L^\infty(\Omega^{\rm f}_{\rm ext}[\bs X_H])}  \right\} &\le H^{1/4}, \label{ind-1-2}\\[2mm]
\max_{t\in[0,t_*]}\left\{ \|{\bs e}_\dep\|_{H^1(\widehat{\Omega}^{\rm s}_H)} 
+\|\partial_t{\bs e}_\dep\|_{L^2(\widehat{\Omega}^{\rm s}_H)} \right\}&\le H^{3/2}. \label{ind-1-3}
\end{align}
\end{subequations}
Such a positive time $t_*>0$ exists. Indeed, the relation  ${\bs e}_{\bs d}(0)=\bs 0$ and the initial-value estimates in \eqref{uh0-error} and \eqref{dtdh0-error} give
$$
\|{\bs e}_{\bs u}(0)\|_{H^1(\Omega^f_{\rm ext}[X_H(\cdot,0)])}
\lesssim h^k, \quad 
\|\partial_t{\bs e}_{\bs d}(0)\|_{L^2(\widehat\Omega_H^s)}
\lesssim h^k+H^k. 
$$
Consequently, for $k\ge2$ and sufficiently small $h\simeq H$,
finite element inverse estimates imply that all the inequalities
in the continuation hypothesis \eqref{ind-1} hold strictly at $t=0$.
The local well-posedness and continuity in time of \eqref{eq:discrete-semi-fsi}, established in Appendix~\ref{section:wellposed},  imply that they continue to hold over a nontrivial interval $[0,t_*]$. 

Testing the  error equation \eqref{error-eq} with $\boldsymbol{e}_{\bs u}$ and proceeding with appropriate  estimates yield the error bound \eqref{Bound1} for all $t\in [0,t_*]$. 
Since the semidiscrete FEM  \eqref{eq:discrete-semi-fsi} is essentially a system of ODEs (as shown in Appendix~\ref{section:wellposed}) which is locally well-posed from time $t_*$, invoking the time-continuity of the semidiscrete finite element solution, we obtain that the estimates \eqref{ind-1} and \eqref{Bound1}--\eqref{Bound-ep-esigma} remain valid for a strictly larger time interval $[0,t_*+\varepsilon_{h}]$, for some $\varepsilon_{h}>0$, albeit with possibly larger constants (i.e., the entire error analysis of
Sections \ref{section:consistency}--\ref{section:consistency-2} remains valid when \eqref{ind-1} is relaxed by a fixed factor), say
\begin{align}\label{Bound1-sec6}
& \max_{t\in[0,t_*+\varepsilon_{h}]} \left\{ \| {\bs e}_{\bs u} \|_{L^2(\Omega^{\rm f}[\bs X_H])}^2 + \|\partial_t{\bs e}_\dep\|_{L^2(\widehat{\Omega}^{\rm s}_H)}^2 
+ \|{\bs e}_\dep\|_{H^1(\widehat{\Omega}^{\rm s}_H)}^2 \right\} 
+ \int_0^{t_*+\varepsilon_{h}} \| {\bs e}_{\bs u}\|_{1, \Omega[\bs X_H]}^2 \d t \notag\\
&\le 2C_1(h^{2k}+H^{2k}) ,\\
\label{Bound2-sec6}
& H^2\max_{t\in[0,t_*+\varepsilon_{h}]}\|{\bs e}_{\bs u}\|_{1,\Omega[{\bs X}_H]}^2
+H^2\int_0^{t_*+\varepsilon_{h}} \big( \| e_p \|_{0,\Omega[{\bs X}_H]}^2 
+  \| {\bs e}_{{\bs \sigma} } \|_{H^{-\frac12}(\widehat\Gamma_H)}^2 \big) \d t \notag\\
&\le 2C_2(h^{2k}+H^{2k}) ,
\end{align}
where the norms $\| \cdot \|_{1, \Omega[\bs X_H]}$ and $\| \cdot \|_{0, \Omega[\bs X_H]}$ are defined in \eqref{1-0-norm}. 
For sufficiently small mesh parameters $h$ and $H\simeq h$, and finite elements of degree $k\geq 3$ (this is required only in \eqref{requive-kgeq3} as discussed in Remark \ref{remark0_THM}, otherwise $k\geq 2$ would be sufficient), this estimate ensures that (employing finite element inverse estimates in the worst case $d = 3$) 
\begin{align}
\max_{t\in[0,t_*+\varepsilon_{h}]}\|{\bs e}_\dep\|_{L^\infty(\widehat{\Omega}^{\rm s}_H)} 
&\leq 
C\max_{t\in[0,t_*+\varepsilon_{h}]} H^{-\frac12}\|{\bs e}_\dep\|_{H^1(\widehat{\Omega}^{\rm s}_H)} \notag\\
&\leq 
CH^{-\frac12} (h^k+H^k) 
\leq \kappa \min(h,H),\\
\max_{t\in[0,t_*+\varepsilon_{h}]} \|{\bs e}_\dep\|_{W^{1,\infty}(\widehat{\Omega}^{\rm s}_H)} 
&\leq 
C\max_{t\in[0,t_*+\varepsilon_{h}]} H^{-\frac32}\|{\bs e}_\dep\|_{H^1(\widehat{\Omega}^{\rm s}_H)} \notag\\
&\leq 
CH^{-\frac32} (h^k+H^k)  
\leq H^{\frac14} ,\\[5pt]
\max_{t\in[0,t_*+\varepsilon_{h}]} \|\partial_t{\bs e}_\dep\|_{L^\infty(\widehat{\Omega}^{\rm s}_H)} 
&\leq 
C\max_{t\in[0,t_*+\varepsilon_{h}]} H^{-\frac32} \|\partial_t{\bs e}_\dep\|_{L^2(\widehat{\Omega}^{\rm s}_H)} \notag\\
&\leq 
CH^{-\frac32} (h^k+H^k) 
\leq H^{\frac14} ,\\[5pt]
\max_{t\in[0,t_*+\varepsilon_{h}]} \|{\bs e}_{\bs u}\|_{L^\infty(\Omega^{\rm f}_{\rm ext}[\bs X_H])} 
&\leq 
C\max_{t\in[0,t_*+\varepsilon_{h}]} h^{-\frac12} \|{\bs e}_{\bs u}\|_{H^1(\Omega^{\rm f}_{\rm ext}[\bs X_H])} \notag\\
&\leq 
C\max_{t\in[0,t_*+\varepsilon_{h}]} h^{-\frac12} \|{\bs e}_{\bs u}\|_{1,\Omega[{\bs X}_H]} \notag\\
&\leq 
CH^{-1}h^{-\frac12} (h^k+H^k)
\leq H^{\frac14} ,
\end{align}
and
\begin{align}
\max_{t\in[0,t_*+\varepsilon_{h}]}\left\{ \|{\bs e}_\dep\|_{H^1(\widehat{\Omega}^{\rm s}_H)} 
+\|\partial_t{\bs e}_\dep\|_{L^2(\widehat{\Omega}^{\rm s}_H)} \right\} 
\leq C(h^k+H^k) 
\leq H^{\frac32}.
\end{align}
Therefore, the continuation hypothesis \eqref{ind-1} remains valid on $[0,t_*+\varepsilon_{h}]$ with the same constants. 
This leads to a contradiction with the maximality of $t_*$ if $t_*<T$, and we therefore conclude that $t_*=T$, i.e., this bound must hold for the entire time interval $[0,T]$. Consequently, for sufficiently small $h$ and $H$, we have error estimates on the entire time interval $[0,T]$.

\section{Preliminaries}\label{sect:prelim}

\subsection{Norm equivalence for interpolated interfaces}

From ${\bs X}_H^* \eqd {\bs I}+{\bs d}_H^*$ and ${\bs X} \eqd {\bs I}+{\bs d}$, it follows that ${\bs X}_H^*\circ\widehat{\bs\Phi}_H^{-1} - {\bs X} = (\widehat{\bs\Phi}_H^{-1}-{\bs I}) + ({\bs d}_H^*\circ\widehat{\bs\Phi}_H^{-1} - {\bs d})$ and therefore, as a result of \eqref{GeomApprox} and \eqref{interpl-error-d}, we have 
\begin{align}\label{aux865}
\|{\bs X}_H^*\circ\widehat{\bs\Phi}_H^{-1} - {\bs X} \|_{W^{1,\infty}(\widehat{\Omega}^s)} \le C H^k. 
\end{align}

Since the  map ${\bs X}:\widehat{\Omega}^{\rm s}\rightarrow\Omega^{\rm s}(t)$ and its restriction ${\bs X}:\widehat\Gamma\rightarrow\Gamma(t)$ are assumed to be smooth, invertible and orientation preserving, the estimate \eqref{aux865} implies that  the interpolated maps ${\bs X}_H^*:\widehat{\Omega}^{\rm s}_H\rightarrow\Omega^{\rm s}[{\bs X}_H^*]$ and ${\bs X}_H^*:\widehat\Gamma_H \rightarrow\Gamma[{\bs X}_H^*]$ are piecewise smooth on each element, globally invertible and orientation preserving for sufficiently small mesh size $H$ (smaller than some positive constant). Moreover, ${\bs X}_H^*$ satisfies the following estimates: 
\begin{subequations}\label{area-element}
\begin{align}
&c_1^* \le \det(\grad {\bs X}_H^*) \le c_2^*  \\
&c_1^* \le \det(\grad_{\widehat\Gamma_H} {\bs X}_H^* (\grad_{\widehat\Gamma_H} {\bs X}_H^*)^{\rm T}+{\bs n}_{\widehat\Gamma_H}({\bs n}_{\widehat\Gamma_H})^{\rm T}) 
\le c_2^* 
\end{align}
\end{subequations}
for some positive constants $c_1^*$ and $c_2^*$. The first inequality of \eqref{area-element} allows us to control the change of the volume elements under transformation ${\bs X}_H^*$, while the second inequality allows us to control the change of the surface area elements. 
Therefore, for a function $f$ defined on $\Gamma[{\bs X}_H^*]$,  the following norm equivalence holds for $1\le p\le \infty$ (a proof of this result can be found in \cite[Lemma 3.2]{Gao-Li-Tang-2025}): 
\begin{subequations}\label{norm-equiv-interpl-1}
\begin{align}
\|f\|_{L^p(\Gamma[{\bs X}_H^*])} &\simeq \|f\circ {\bs X}_H^*\|_{L^p(\widehat\Gamma_H )} , \\
\|\grad_{\Gamma[{\bs X}_H^*]} f\|_{L^p(\Gamma[{\bs X}_H^*])} &\simeq \|\grad_{\widehat\Gamma_H} (f\circ {\bs X}_H^*)\|_{L^p(\widehat\Gamma_H )} .
\end{align}
\end{subequations}

Since $\grad{\bs X}(t)$ is invertible from the tangent plane of $\widehat\Gamma$ to the tangent plane of $\Gamma(t)$, and the area elements on these two interfaces are equivalent, the results of \eqref{norm-equiv-interpl-1} also imply the following norm-equivalence relations:  
\begin{subequations}\label{norm-equiv-interpl-2}
\begin{align}
\|f \circ {\bs\Phi}_H(t) \|_{L^p(\Gamma[{\bs X}_H^*])} & \simeq \| f \|_{L^p(\Gamma(t))} , \\
\| \grad_{\Gamma[{\bs X}_H^*]} [f \circ {\bs\Phi}_H(t)] \|_{L^p(\Gamma[{\bs X}_H^*])} &\simeq \| \grad_{\Gamma(t)} f \|_{L^p(\Gamma(t))} ,
\end{align}
\end{subequations}
where ${\bs\Phi}_H(t):\Gamma[{\bs X}_H^*]\rightarrow \Gamma(t)$ is defined in \eqref{eq:map-phiHt}. In fact, if we introduce the intermediate mapping 
\begin{align}\label{def-XH-star-theta}
{\bs X}_H^{*,\theta} \eqd (1-\theta) {\bs X} + \theta {\bs X}_H^*\circ \widehat{\bs\Phi}_H^{-1}: \widehat\Omega^s\rightarrow \mathbb{R}^d
\end{align} 
and set $
\Gamma[{\bs X}_H^{*,\theta}] \eqd {\bs X}_H^{*,\theta}(\widehat\Gamma) $, 
then 
\begin{subequations}\label{norm-equiv-interpl-3}
\begin{align}
\|f \circ {\bs X} \circ ({\bs X}_H^{*,\theta})^{-1}\|_{L^p(\Gamma[{\bs X}_H^{*,\theta}])} &\simeq \| f \|_{L^p(\Gamma(t))} , \\
\| \grad_{\Gamma[{\bs X}_H^{*,\theta}]} [f \circ {\bs X} \circ ({\bs X}_H^{*,\theta}])^{-1}] \|_{L^p(\Gamma[{\bs X}_H^{*,\theta}])} &\simeq \| \grad_{\Gamma(t)} f \|_{L^p(\Gamma(t))} 
\end{align}
\end{subequations}
for all $\theta\in[0,1]$.

\begin{remark}\upshape\upshape 
The estimates \eqref{interpl-error-X} and \eqref{ind-1-1} guarantee that $\Omega^{\rm f}(t)$ and $\Gamma[\bs X_H^*]$ are contained in the extended $2\kappa h$-neighborhood $\Omega^{\rm f}_{\rm ext}[\bs X_H]$ of \eqref{Extended-region}, so that a trace inequality from the interface $\Gamma[\bs X_H^*]$ to the bulk region $\Omega^{\rm f}_{\rm ext}[\bs X_H]$ can be used; see Appendix \ref{appendix:trace_inequlaity}. 
\end{remark}

\begin{remark}\upshape\upshape 
For any $t\in[0,t_*]$, the inequality \eqref{ind-1-2} and the norm equivalence \eqref{norm-equiv-interpl-1} imply 
\begin{align}\label{ind-2}
&\|{\bs e}_\dep \circ ({\bs X}_H^*)^{-1}\|_{W^{1,\infty}(\Gamma[{\bs X}_H^*])} \simeq \|{\bs e}_\dep\|_{W^{1,\infty}(\widehat\Gamma_H)} \le H^{\frac14} .
\end{align}
For sufficiently small mesh size $H$, we obtain from \eqref{ind-1-2} and \eqref{ind-2} that 
\begin{equation}\label{aux964}
\max_{t\in[0,t_*]} \left\{  \|{\bs e}_\dep\|_{W^{1,\infty}(\widehat{\Omega}^{\rm s}_H)}  + \|{\bs\dot{\bs e}}_\dep\|_{L^\infty(\widehat{\Omega}^{\rm s}_H)}
+ \|{\bs e}_\dep \circ ({\bs X}_H^*)^{-1}\|_{W^{1,\infty}(\Gamma[{\bs X}_H^*])} \right\}  \le \frac12 .
\end{equation} 
In particular, the condition $\|{\bs e}_\dep \circ ({\bs X}_H^*)^{-1}\|_{W^{1,\infty}(\Gamma[{\bs X}_H^*])}\le \frac12$ guarantees that the map ${\bs I} + {\bs e}_\dep \circ ({\bs X}_H^*)^{-1}:\Gamma[{\bs X}_H^{*}]\rightarrow\mathbb{R}^d$ is invertible, and therefore, the map ${\bs X}_H={\bs X}_H^*+{\bs e}_\dep=({\bs I} + {\bs e}_\dep \circ ({\bs X}_H^*)^{-1})\circ {\bs X}_H^*$ is invertible as well, because it is the composition of two invertible maps. Quantitatively, we have 
\begin{align}\label{W1infty-XH-inv}
&\max_{t\in[0,t_*]} \|({\bs X}_H)^{-1} \|_{W^{1,\infty}(\Gamma[{\bs X}_H])} \notag\\
&\le \max_{t\in[0,t_*]}  \|({\bs I} + {\bs e}_\dep \circ ({\bs X}_H^*)^{-1})^{-1}\|_{W^{1,\infty}(\Gamma[{\bs X}_H])} \|({\bs X}_H^*)^{-1}\|_{W^{1,\infty}(\Gamma[{\bs X}_H^*])} 
\le C.
\end{align}
Moreover, inequality \eqref{aux964} and the interpolation estimates for $\dep_H^*$ , $\vel_H^*$,  ${\bs X}_H^* $ in \eqref{interpl-error-d} and \eqref{aux865} yield the following bound: 
\begin{align}
\max_{t\in[0,t_*]} \left\{ \|\dep_H\|_{W^{1,\infty}(\widehat{\Omega}^{\rm s}_H)} + \|{\bs\dot \dep}_H\|_{L^\infty(\widehat{\Omega}^{\rm s}_H)} 
+ \|{\bs X}_H\|_{W^{1,\infty}(\widehat{\Omega}^{\rm s}_H)} \right\}  
&\le C .
\end{align}
\end{remark}

\subsection{Norm equivalence on the numerical and interpolated interfaces}\label{section:norm-equiv-2}

Let ${\bs X}_H^\theta \eqd (1-\theta) {\bs X}_H^* + \theta {\bs X}_H$ on $\widehat\Gamma_H$ for $\theta\in[0,1]$, so that $\partial_\theta {\bs X}_H^\theta = {\bs X}_H-{\bs X}_H^* = {\bs e}_\dep$ on $\widehat\Gamma_H$. The condition
$\max_{t\in[0,t_*]} \|{\bs e}_\dep \circ ({\bs X}_H^*)^{-1}\|_{W^{1,\infty}(\Gamma[{\bs X}_H^*])}  \le \frac12$
in \eqref{aux964} ensures that ${\bs X}_H^\theta = {\bs X}_H^* + \theta {\bs e}_{\bs d}$ is sufficiently close to ${\bs X}_H^*$ for all $\theta\in[0,1]$, and in particular that norms of functions transported between the corresponding interfaces are equivalent (see, e.g., \cite[Lemma 4.3]{KLLP-2017}). Specifically, the following norm-equivalence relations hold for $t\in[0,t_*]$ and $p\in[1,+\infty]$:
\begin{subequations}\label{norm-equiv-2}
\begin{align}
\| f \circ {\bs X}_H^* \circ({\bs X}_H^\theta)^{-1} \|_{L^p(\Gamma[\bs X_H^\theta])} & \simeq \| f \|_{L^p(\Gamma[\bs X_H^*])} , \\
\| f \circ {\bs X}_H^* \circ({\bs X}_H^\theta)^{-1} \|_{W^{1,p}(\Gamma[\bs X_H^\theta])} & \simeq \| f \|_{W^{1,p}(\Gamma[\bs X_H^*])}  .
\end{align}
\end{subequations} 
In particular, the case $\theta=1$ indicates that the norms of functions transported between $\Gamma[\bs X_H^*]$ and $\Gamma[\bs X_H]$ are equivalent. The norm-equivalence relations in \eqref{norm-equiv-interpl-1} and \eqref{norm-equiv-2} imply the following result for $t\in[0,t_*]$ and $p\in[1,+\infty]$:
\begin{equation}\label{norm-equiv-GH}
\| f \|_{L^p(\Gamma[{\bs X}_H])} \simeq \| f\circ {\bs X}_H\|_{L^p(\widehat\Gamma_H)},\quad
\| f \|_{W^{1,p}(\Gamma[{\bs X}_H])} \simeq \| f\circ {\bs X}_H\|_{W^{1,p}(\widehat\Gamma_H)}.
\end{equation} 
Moreover, $|{\bs X}_H^\theta -{\bs X}_H | = (1-\theta) |{\bs e}_{\bs d}|\leq \kappa h$ and therefore all the intermediate surfaces $\Gamma[{\bs X}_H^\theta]$ are in the extended fluid region $\Omega^{\rm f}_{\rm ext}[{\bs X}_H]$.  

Let ${\bs X}_H^{\#,\theta} \eqd (1-\theta){\bs X} + \theta {\bs X}_H \circ \widehat{\bs\Phi}_H^{-1} : \widehat\Gamma \rightarrow \mathbb{R}^d$ be the map which determines a family of interfaces $\Gamma[{\bs X}_H^{\#,\theta}]$, $\theta\in[0,1]$, intermediate between $\Gamma(t)$ and $\Gamma[{\bs X}_H]$. For sufficiently small $h$ and $H$, the intermediate interfaces $\Gamma[{\bs X}_H^{\#,\theta}]$ can also be shown to lie in the extended fluid region $\Omega^{\rm f}_{\rm ext}[{\bs X}_H]$.  
When $\theta$ moves from $0$ to $1$, the intermediate interface $\Gamma[{\bs X}_H^{\#,\theta}]$ moves from $\Gamma(t)$ to $\Gamma[{\bs X}_H]$ with the following velocity (defined on $\Gamma[{\bs X}_H^{\#,\theta}]$): 
\begin{align}\label{def-ej-theta}
{\bs e}_{\bs X}^{\#,\theta}
&= ({\bs X}_H \circ \widehat{\bs\Phi}_H^{-1} - {\bs X} ) \circ ({\bs X}_H^{\#,\theta})^{-1} \notag\\
&= ({\bs X}_H - {\bs X}_H^*)\circ \widehat{\bs\Phi}_H^{-1} \circ ({\bs X}_H^{\#,\theta})^{-1} + ({\bs X}_H^*\circ \widehat{\bs\Phi}_H^{-1} - {\bs X}) \circ ({\bs X}_H^{\#,\theta})^{-1} \notag\\
&= {\bs e}_{\bs d} \circ \widehat{\bs\Phi}_H^{-1} \circ ({\bs X}_H^{\#,\theta})^{-1} + {\bs e}_{\bs X}^* \circ ({\bs X}_H^{\#,\theta})^{-1} , 
\end{align}
where ${\bs e}_{\bs X}^*=({\bs X}_H^*\circ \widehat{\bs\Phi}_H^{-1} - {\bs X})$. Using an argument similar to the one of \eqref{norm-equiv-interpl-1}, \eqref{norm-equiv-interpl-2}, and \eqref{norm-equiv-2}, we obtain the following norm-equivalence relations for $t \in [0,t_*]$ and $p\in[1,+\infty]$:
\begin{subequations}\label{norm-equiv-X-theta}
\begin{align}
&\| f \circ {\bs X} \circ({\bs X}_H^{\#,\theta})^{-1} \|_{L^p(\Gamma[{\bs X}_H^{\#,\theta}])} \simeq \| f \|_{L^p(\Gamma(t))} \simeq \| f \circ {\bs X}\circ\widehat{\bs\Phi}_H\|_{L^p(\widehat\Gamma_H)} , \\
&\| f \circ {\bs X} \circ({\bs X}_H^{\#,\theta})^{-1} \|_{W^{1,p}(\Gamma[{\bs X}_H^{\#,\theta}])} \simeq \| f \|_{W^{1,p}(\Gamma(t))} 
\simeq \| f \circ {\bs X}\circ\widehat{\bs\Phi}_H\|_{W^{1,p}(\widehat\Gamma_H)} ,\\
&\| f \circ {\bs X} \circ({\bs X}_H^{\#,\theta})^{-1} \|_{H^{s}(\Gamma[{\bs X}_H^{\#,\theta}])} \simeq \| f \|_{H^{s}(\Gamma(t))} 
\simeq \| f \circ {\bs X}\circ\widehat{\bs\Phi}_H\|_{H^{s}(\widehat\Gamma_H)} ,\,\,s\in[0,1],
\end{align}
\end{subequations} 
where the last result follows from interpolating the first two results with $p=2$. 

Since $\grad\cdot{\bs u}$ and all its partial derivatives vanish in  $\overline{\Omega^f(t)}\supset\Gamma(t)$, by using a Taylor expansion of $\grad\cdot{\bs u}$ at a point on $\Gamma(t)$ we obtain 
\begin{subequations}\label{div-u-small}
\begin{align}
|\grad\cdot{\bs u}| &= O({\rm dist}(x,\Gamma(t))^m)\,\,\,\mbox{near}\,\,\,\Gamma(t)\,\,\,\mbox{for any fixed integer $m\geq 1$},
\end{align}
with a constant depending on $m$ and on $\|{\bs u}\|_{C^{m+1}(\overline\Omega)}$.
Since $|{\bs X}_H^\theta\circ\widehat{\bs\Phi}_H^{-1}-{\bs X}|\lesssim h$ and $|{\bs X}_H^{\#,\theta}-{\bs X}|\lesssim h$, as implied by \eqref{aux865} and \eqref{ind-1-1}, it follows that $|\grad\cdot{\bs u}|$ is arbitrarily high-order small near $\Gamma(t)$, i.e., 
\begin{align}\label{div-u-small-2}
|\grad\cdot{\bs u}| &= O(h^m)\,\,\,\mbox{on}\,\,\,\Gamma[{\bs X}_H^\theta]\,\,\,\mbox{and}\,\,\, \Gamma[{\bs X}_H^{\#,\theta}]\,\,\,\mbox{for all $\theta\in[0,1]$ and $m\geq 1$}.
\end{align}
\end{subequations} 

Similarly, let ${\bs X}_H^{*,\theta} \eqd (1-\theta){\bs X} + \theta {\bs X}_H^* \circ \widehat{\bs\Phi}_H^{-1} : \widehat\Gamma \rightarrow \mathbb{R}^d$ be the map which determines a family of interfaces $\Gamma[{\bs X}_H^{*,\theta}]$, $\theta\in[0,1]$, intermediate between $\Gamma(t)$ and $\Gamma[{\bs X}_H^*]$. When $\theta$ moves from $0$ to $1$, the intermediate interface $\Gamma[{\bs X}_H^{*,\theta}]$ moves from $\Gamma(t)$ to $\Gamma[{\bs X}_H^*]$ with the velocity ${\bs e}_{\bs X}^{*,\theta}
= ({\bs X}_H^* \circ \widehat{\bs\Phi}_H^{-1} - {\bs X} ) \circ ({\bs X}_H^{*,\theta})^{-1} $, which satisfies the following estimate:
\begin{align}\label{def-ej-s-theta}
\| {\bs e}_{\bs X}^{*,\theta} \|_{L^\infty(\Gamma[{\bs X}_H^{*,\theta}])} + H \| {\bs e}_{\bs X}^{*,\theta} \|_{W^{1,\infty}(\Gamma[{\bs X}_H^{*,\theta}])} 
\lesssim H^{k+1} .   
\end{align}
For the same reason as in \eqref{norm-equiv-interpl-1}, \eqref{norm-equiv-interpl-2} and \eqref{norm-equiv-2}, the following norm-equivalence relations hold for $t\in[0,t_*]$ and $1\le p\le\infty$: 
\begin{subequations}\label{norm-X-s-theta}
\begin{align}
&\| f \circ {\bs X} \circ({\bs X}_H^{*,\theta})^{-1} \|_{L^p(\Gamma[{\bs X}_H^{*,\theta}])} \simeq \| f \|_{L^p(\Gamma(t))} \simeq \| f \circ {\bs X}\circ\widehat{\bs\Phi}_H\|_{L^p(\widehat\Gamma_H)} , \\
&\| f \circ {\bs X} \circ({\bs X}_H^{*,\theta})^{-1} \|_{W^{1,p}(\Gamma[{\bs X}_H^{*,\theta}])} \simeq \| f \|_{W^{1,p}(\Gamma(t))} 
\simeq \| f \circ {\bs X}\circ\widehat{\bs\Phi}_H\|_{W^{1,p}(\widehat\Gamma_H)} .
\end{align}
\end{subequations} 

\subsection{Transport formulas for parametrized domains and interfaces}

We denote $\vel_H^\theta \eqd [(1-\theta)\vel_H^*+\theta\vel_H] \circ ({\bs X}_H^\theta)^{-1}$. Then the interface $\Gamma[\bs X_H^\theta]$ moves with velocity $\vel_H^\theta$ as time $t$ varies in $(0,t_\star]$, and the interface $\Gamma[\bs X_H^\theta]$ moves with velocity ${\bs e}_\dep^\theta = {\bs e}_{\bs d}\circ ({\bs X}_H^\theta)^{-1}$ as $\theta$ varies in $[0,1]$. Under this setting, the following formulas have been proved in the literature (see, e.g.,  \cite{Dziuk-Elliott-2013,Elliott-Ranner-2021,Lidstrom-2011}) and will be frequently used further in the stability analysis: 
\begin{subequations}\label{dt-int-Omega}
\begin{align}
\frac{\d}{\d t} \int_{\Omega^{\rm f}[\bs X_H^\theta]} g 
&= \int_{\Omega^{\rm f}[\bs X_H^\theta]} \partial_t g + \int_{\Gamma[\bs X_H^\theta]}  g \, \vel_H^\theta \cdot {\bs n}_{\Gamma[\bs X_H^\theta]} \,,
\label{dt-int-Omega-1}\\
\frac{\d}{\d \theta} \int_{\Omega^{\rm f}[\bs X_H^\theta]} g 
&= \int_{\Omega^{\rm f}[\bs X_H^\theta]} \partial_\theta g + \int_{\Gamma[\bs X_H^\theta]} g \, {\bs e}_{\bs d}^\theta\cdot {\bs n}_{\Gamma[\bs X_H^\theta]} \,, 
\label{dt-int-Omega-2}\\
\frac{\d}{\d \theta} \int_{\Gamma[\bs X_H^\theta]} g^\theta 
&= \int_{\Gamma[\bs X_H^\theta]} \partial_\theta^\bullet g^\theta + \int_{\Gamma[\bs X_H^\theta]} g^\theta \, \grad_{\Gamma[\bs X_H^\theta]}\cdot {\bs e}_{\bs d}^\theta \, ,
\label{dt-int-Gamma-1} 
\end{align} 
where $\partial_\theta^\bullet g^\theta$ denotes the material derivative of $g^\theta$ with respect to the velocity ${\bs e}_{\bs d}^\theta$ of the surface $\Gamma[\bs X_H^\theta]$ (velocity with respect to $\theta$). In particular, if $g^\theta \eqd g|_{\Gamma[\bs X_H^\theta]}$ for a function $g$ defined on $\Omega$, then $\partial_\theta^\bullet g = {\bs e}_{\bs d}^\theta\cdot\grad g$ on $\Gamma[\bs X_H^\theta]$; if $g^\theta \eqd g\circ ({\bs X}_H^\theta)^{-1}$ for a function $g$ defined on $\widehat\Gamma_H$, then $\partial_\theta^\bullet g^\theta = 0$ on $\Gamma[\bs X_H^\theta]$. 

We shall also use the following standard evolving-surface identities (see, e.g., \cite{Dziuk-Elliott-2013}): if $\Gamma^\theta$ is a surface which moves with velocity ${\bs V}^\theta$ and $g^\theta$ is a function defined on the surface $\Gamma^\theta$, then
\begin{align*}
\partial_\theta^\bullet {\bs n}_{\Gamma^\theta}
  &=
  -(\nabla_{\Gamma^\theta}{\bs V}^\theta)^{\rm T} {\bs n}_{\Gamma^\theta},\\
|\partial_\theta^\bullet \nabla_{\Gamma^\theta} g^\theta|
  &\lesssim
  |\nabla_{\Gamma^\theta}{\bs V}^\theta|
  |\nabla_{\Gamma^\theta}g^\theta|
  +
  |\nabla_{\Gamma^\theta}\partial_\theta^\bullet g^\theta| .
\end{align*}
In particular, since $\Gamma[{\bs X}_H^\theta]$ moves with velocity ${\bs e}_{\bs d}^\theta$ as $\theta$ changes and moves with velocity $\vel_H^\theta = [(1-\theta)\vel_H^*+\theta\vel_H] \circ ({\bs X}_H^\theta)^{-1}$ as $t$ changes, we have 
\begin{align}\label{dt-int-Gamma-1b} 
\partial_\theta^\bullet {\bs n}_{\Gamma[{\bs X}_H^\theta]}
&=
-(\nabla_{\Gamma[{\bs X}_H^\theta]}{\bs e}_{\bs d}^\theta)^{\rm T} {\bs n}_{\Gamma[{\bs X}_H^\theta]}
\quad\mbox{and}\quad 
\partial_t^\bullet {\bs n}_{\Gamma[{\bs X}_H^\theta]}
= 
-(\nabla_{\Gamma[{\bs X}_H^\theta]}\vel_H^\theta)^{\rm T} {\bs n}_{\Gamma[{\bs X}_H^\theta]}, \\
\label{dt-int-Gamma-1c} 
|\partial_{\theta}^\bullet\nabla_{\Gamma[{\bs X}_H^\theta]}g^\theta|
&\le
C|\nabla_{\Gamma[{\bs X}_H^\theta]}{\bs e}_{\bs d}^\theta|
|\nabla_{\Gamma[{\bs X}_H^\theta]}g^\theta| 
+ |\nabla_{\Gamma[{\bs X}_H^\theta]}\partial_{\theta}^\bullet g^\theta| .
\end{align}

The bulk transport identities \eqref{dt-int-Omega-1}--\eqref{dt-int-Omega-2} are consequences of the Reynolds transport theorem for moving domain (see, for example, \cite{Elliott-Ranner-2021,Lidstrom-2011}). The surface identity \eqref{dt-int-Gamma-1} is the standard surface transport theorem (see \cite{Dziuk-Elliott-2013}). Similarly to \eqref{dt-int-Gamma-1}, the following identity holds for the surface $\Gamma[{\bs X}_H^{*,\theta}]$: 
\begin{align}
\frac{\d}{\d \theta} \int_{\Gamma[{\bs X}_H^{*,\theta}]} g^\theta 
&= \int_{\Gamma[\bs X_H^{*,\theta}]} \partial_\theta^\bullet g^\theta + \int_{\Gamma[\bs X_H^{*,\theta}]} g^\theta \, \grad_{\Gamma[\bs X_H^{*,\theta}]}\cdot {\bs e}_{\bs X}^{*,\theta} , \label{dt-int-Gamma-2} 
\end{align}
where ${\bs e}_{\bs X}^{*,\theta}$ is defined in Section \ref{section:norm-equiv-2} as the velocity of surface $\Gamma[{\bs X}_H^{*,\theta}]$ with respect to $\theta$. 
\end{subequations}

\subsection{Trace inequality at the numerical interface}
\label{section:trace}

Thanks to \eqref{ind-1}, the intermediate interfaces $\Gamma[{\bs X}_H^\theta]$, $\theta\in[0,1]$, are all contained in the extended fluid region $\Omega^{\rm f}_{\rm ext}[\bs X_H]$. Therefore, it is possible to derive a trace inequality which controls the $L^2$ norm of a function on $\Gamma[\bs X_H^\alpha]$ by the $H^1$ norm of the function on $\Omega^{\rm f}_{\rm ext}[\bs X_H]$. 

To this end, we invoke the fundamental theorem of calculus together with the surface Reynolds transport theorem. For a finite element function or smooth function $f$, using the shorthand ${\bs e}_{\bs d}^\theta\eqd{\bs e}_{\bs d}\circ({\bs X}_H^\theta)^{-1}$, we obtain
\begin{align}\label{L2-f-interface}
\begin{aligned}
\int_{\Gamma[\bs X_H^\alpha]} | f |^2 - \int_{\Gamma[\bs X_H^*]} | f |^2 
&= \int_0^{\alpha} \Big( \frac{\d}{\d\theta}\int_{\Gamma[\bs X_H^\theta]} | f |^2\Big) \d\theta \\
&= \int_0^{\alpha} \int_{\Gamma[\bs X_H^\theta]} ( 2f \partial_\theta^\bullet f + |f|^2\grad_{\Gamma[\bs X_H^\theta]}\cdot {\bs e}_{\bs d}^\theta) \d\theta \\
&= \int_0^{\alpha} \int_{\Gamma[\bs X_H^\theta]}( 2f ({\bs e}_{\bs d}^\theta\cdot \grad f) + |f|^2\grad_{\Gamma[\bs X_H^\theta]}\cdot {\bs e}_{\bs d}^\theta) \d\theta .
\end{aligned}
\end{align}
Inequalities \eqref{ind-1}--\eqref{ind-2} and \eqref{norm-equiv-2} imply that 
\begin{equation}\label{aux1006}
\|{\bs e}_{\bs d}^\theta\|_{L^{\infty}(\Gamma[\bs X_H^\theta])} \le C h ,\quad 
\|\grad_{\Gamma[\bs X_H^\theta]}\cdot {\bs e}_{\bs d}^\theta\|_{L^\infty(\Gamma[\bs X_H^\theta])}\le C .
\end{equation}
Using the uniform  trace inequalities (see, e.g.,~\cite{GuzmOlsh}) for finite element functions on the fluid mesh (note that $\Gamma[{\bs X}_H^\theta]$ is contained in the extended fluid region $\Omega^{\rm f}_{\rm ext}[{\bs X}_H]$) one gets:  
$$
\|f\|_{L^2(\Gamma[\bs X_H^\theta])} \le Ch^{-\frac12}\|f\|_{L^2(\Omega^{\rm f}_{\rm ext}[\bs X_H])},\quad 
\|\grad f\|_{L^2(\Gamma[\bs X_H^\theta])} \le Ch^{-\frac12}\|\grad f\|_{L^2(\Omega^{\rm f}_{\rm ext}[\bs X_H])} , 
$$ 
for any finite element function $f$ in either ${\bs V}_h^{\rm f}$ or $Q_h^{\rm f}$, with a constant $C>0$ which is independent of $h$ and the position of the interface $\Gamma[\bs X_H^\theta]$ in the mesh. Therefore, substituting these estimates into \eqref{L2-f-interface} and using \eqref{aux1006} gives
$$
\int_{\Gamma[\bs X_H^\alpha]} | f |^2 - \int_{\Gamma[\bs X_H^*]} | f |^2
\le C \|f\|_{L^2(\Omega^{\rm f}_{\rm ext}[\bs X_H])} \|\grad f\|_{L^2(\Omega^{\rm f}_{\rm ext}[\bs X_H])} 
+  C \int_0^\alpha \int_{\Gamma[\bs X_H^\theta]} | f |^2\d\theta
$$
for $\alpha\in[0,1]$ and for any $f$ in ${\bs V}_h^{\rm f}$ or $Q_h^{\rm f}$.
By applying Gronwall's inequality, we obtain the following result for $f$ in ${\bs V}_h^{\rm f}$ or $Q_h^{\rm f}$:   
\begin{align}\label{L2-f-trace-1}
\begin{aligned}
\max_{\theta\in[0,1]} \| f \|_{L^2(\Gamma[\bs X_H^\theta])}^2
\lesssim \| f \|_{L^2(\Gamma[\bs X_H^*])}^2 + \|f\|_{L^2(\Omega^{\rm f}_{\rm ext}[\bs X_H])} \|\grad f\|_{L^2(\Omega^{\rm f}_{\rm ext}[\bs X_H])} . 
\end{aligned}
\end{align}

Similarly, using the definition of 
$
{\bs X}_H^{*,\theta} 
$
and $\Gamma[{\bs X}_H^{*,\theta}]$
in \eqref{def-XH-star-theta} for $\theta\in [0,1]$, we set 
\begin{equation}\label{def-eXs}
{\bs e}_{\bs X}^*\eqd {\bs X}_H^*\circ \widehat{\bs\Phi}_H^{-1} - {\bs X} \,\,\,\mbox{on}\,\,\, \widehat\Gamma, 
\quad
{\bs e}_{\bs X}^{*,\theta}\eqd{\bs e}_{\bs X}^*\circ ({\bs X}_H^{*,\theta})^{-1} \,\,\,\mbox{on}\,\,\, \Gamma[{\bs X}_H^{*,\theta}] . 
\end{equation}
We then compare $\|f\|_{L^2(\Gamma[\bs X_H^*])}^2$ with $\|f\|_{L^2(\Gamma[\bs X])}^2$ as follows:
\begin{align}\label{L2-f-interface-2}
\begin{aligned}
\int_{\Gamma[\bs X_H^{*,\alpha}]} | f |^2 - \int_{\Gamma[\bs X]} | f |^2 
&= \int_0^\alpha \Big( \frac{\d}{\d\theta}\int_{\Gamma[{\bs X}_H^{*,\theta}]} | f |^2\Big) \d\theta \\
&= \int_0^\alpha \int_{\Gamma[\bs X_H^{*,\theta}]} 2f \partial_\theta^\bullet f + |f|^2\grad_{\Gamma[\bs X_H^{*,\theta}]}\cdot {\bs e}_{\bs X}^{*,\theta} \d\theta \\
&= \int_0^\alpha \int_{\Gamma[\bs X_H^{*,\theta}]} 2f ({\bs e}_{\bs X}^{*,\theta}\cdot \grad f) + |f|^2\grad_{\Gamma[\bs X_H^{*,\theta}]}\cdot {\bs e}_{\bs X}^{*,\theta} \d\theta .
\end{aligned}
\end{align}
From \eqref{interpl-error-X} and \eqref{norm-equiv-interpl-3} it follows that
\begin{equation}\label{aux1048}
\|{\bs e}_{\bs X}^{*,\theta}\|_{L^{\infty}(\Gamma[\bs X_H^{*,\theta}])} \le CH^{k+1} , \quad 
\|\grad_{\Gamma[\bs X_H^{*,\theta}]}\cdot {\bs e}_{\bs X}^{*,\theta}\|_{L^\infty(\Gamma[\bs X_H^{*,\theta}])}\le CH^k .
\end{equation}
Using uniform trace inequalities for finite element functions $f$ on the fluid mesh and noting that $\Gamma[{\bs X}_H^{*,\theta}] \subset \Omega^{\rm f}_{\rm ext}[{\bs X}_H]$, we have
$$
\|f\|_{L^2(\Gamma[\bs X_H^{*,\theta}])} \le Ch^{-\frac12}\|f\|_{L^2(\Omega^{\rm f}_{\rm ext}[\bs X_H])},
\quad
\|\grad f\|_{L^2(\Gamma[\bs X_H^{*,\theta}])} \le Ch^{-\frac12}\|\grad f\|_{L^2(\Omega^{\rm f}_{\rm ext}[\bs X_H])}.
$$ 
Substituting these bounds into \eqref{L2-f-interface-2}, invoking \eqref{aux1048}, and assuming $H^{k+1}\le c\,h$, we obtain
\begin{align}\label{L2-f-trace-err}
\begin{aligned}
&\int_{\Gamma[\bs X_H^{*,\alpha}]} | f |^2 - \int_{\Gamma[\bs X]} | f |^2\\
&\le C \|f\|_{L^2(\Omega^{\rm f}_{\rm ext}[\bs X_H])} \|\grad f\|_{L^2(\Omega^{\rm f}_{\rm ext}[\bs X_H])} 
+ CH^k\int_0^\alpha\int_{\Gamma[\bs X_H^{*,\theta}]} | f |^2 \d\theta 
\quad\mbox{for}\,\,\,\alpha\in[0,1].
\end{aligned}
\end{align}
Applying Gronwall’s inequality yields, for $f\in {\bs V}_h^{\rm f}$ or $f\in Q_h^{\rm f}$, 
\begin{align}\label{L2-f-trace-2}
\begin{aligned}
\max_{\alpha\in[0,1]}\| f \|_{L^2(\Gamma[{\bs X}_H^{*,\alpha}])}^2
\lesssim \| f \|_{L^2(\Gamma(t))}^2 + \|f\|_{L^2(\Omega^{\rm f}_{\rm ext}[\bs X_H])} \|\grad f\|_{L^2(\Omega^{\rm f}_{\rm ext}[\bs X_H])} . 
\end{aligned}
\end{align}

For the smooth interface $\Gamma(t)$, the following trace inequalities hold with a constant $C$ that is uniform in time:
\begin{align} \label{aux1076}
\begin{aligned}
\| f \|_{L^2(\Gamma(t))}^2
&\le C\,\|f\|_{L^2(\Omega^{\rm f}[\bs X])}\, \| f\|_{H^1(\Omega^{\rm f}[\bs X])} \\
&\le C\,\|f\|_{L^2(\Omega^{\rm f}_{\rm ext}[\bs X_H])}\, \| f\|_{H^1(\Omega^{\rm f}_{\rm ext}[\bs X_H])}.
\end{aligned}
\end{align}
Combining \eqref{L2-f-trace-1} and \eqref{L2-f-trace-2}, we conclude that for $f\in{\bs V}_h^{\rm f}$ or $f\in Q_h^{\rm f}$,
\begin{subequations} \label{L2-f-trace}
\begin{align}
\| f \|_{L^2(\Gamma[{\bs X}_H^\theta])}^2
&\le C\,\|f\|_{L^2(\Omega^{\rm f}_{\rm ext}[\bs X_H])}\, \| f\|_{H^1(\Omega^{\rm f}_{\rm ext}[\bs X_H])}.
\end{align}
In particular, evaluating at the endpoints $\theta=0$ and $\theta=1$ yields
\begin{align}
\| f \|_{L^2(\Gamma[{\bs X}_H])}^2
&\le C\,\|f\|_{L^2(\Omega^{\rm f}_{\rm ext}[\bs X_H])}\, \| f\|_{H^1(\Omega^{\rm f}_{\rm ext}[\bs X_H])}, \\
\| f \|_{L^2(\Gamma[{\bs X}_H^*])}^2
&\le C\,\|f\|_{L^2(\Omega^{\rm f}_{\rm ext}[\bs X_H])}\, \| f\|_{H^1(\Omega^{\rm f}_{\rm ext}[\bs X_H])}.
\end{align}
\end{subequations}

\section{Consistency estimates: Part I}\label{section:consistency}

In this section, we establish estimates for the remainder $R_{h,H}^*({\bs v}_h, q_h , {\bs w}_H, 0)$ of the error equation \eqref{error-eq}, whose definition is given in \eqref{def-RhH} (with $\bs \lambda_H=\bs 0$ therein). This result will be used in Sections \ref{section:ep}--\ref{section:dteu} in establishing stability estimates. For simplicity of notation, we define the following penalized $H^1$ and $L^2$ norms:
\begin{subequations}\label{1-0-norm}
\begin{align}
\|{\bs v}_h\|_{1,\Omega[{\bs X}_H]}^2 &\eqd\|{\bs v}_h\|_{H^1(\Omega^{\rm f}[{\bs X}_H])}^2+ g_h^{\rm u}({\bs v}_h,{\bs v}_h),\\
\|{\bs v}_h\|_{0,\Omega[{\bs X}_H]}^2 &\eqd\|{\bs v}_h\|_{L^2(\Omega^{\rm f}[{\bs X}_H])}^2+ h^2 g_h^{\rm u}({\bs v}_h,{\bs v}_h),\\
\|q_h\|_{0,\Omega[{\bs X}_H]}^2 &\eqd\|q_h\|_{L^2(\Omega^{\rm f}[\bs X_H])}^2+ g_h^{\rm p}(q_h,q_h)
\end{align}
\end{subequations}
for $({\bs v}_h,q_h) \in \bs V_h^{\rm f}
\times  Q_h^{\rm f}$.
Owing to \cite[Lemma~5.2]{lehrenfeld2019eulerian}, the $H^1(\Omega^{\rm f}_{\rm ext}[\bs X_H])$ and $L^2(\Omega^{\rm f}_{\rm ext}[\bs X_H])$ norms of a finite element function are bounded by the $H^1$ and $L^2$ norms on $\Omega^{\rm f}[\bs X_H]$ plus a ghost penalty term:
\begin{subequations}\label{H1ext}
\begin{align}
\|\bv_h\|_{H^1(\Omega^{\rm f}_{\rm ext}[\bs X_H])}^2
    &\le C\,  \|{\bs v}_h\|_{1,\Omega[{\bs X}_H]}^2 ,\\
\|\bv_h\|_{L^2(\Omega^{\rm f}_{\rm ext}[\bs X_H])}^2
    &\le C\,  \|{\bs v}_h\|_{0,\Omega[{\bs X}_H]}^2 
\end{align}
\end{subequations}
for all $\bv_h\in {\bs V}_h^{\rm f}$ and 
with a  constant $C$ depending only on $\kappa$ from \eqref{Extended-region}. 
From \eqref{H1ext} the following equivalence relations follow: 
\begin{subequations} \label{0-0-norm}
\begin{align}
\|\bv_h\|_{1,\Omega[{\bs X}_H]}^2 &\simeq\|{\bs v}_h\|_{H^1(\Omega^{\rm f}_{\rm ext}[\bs X_H])}^2+ g_h^{\rm u}(\bv_h,\bv_h) , \label{0-0-norm-a}\\
\|\bv_h\|_{0,\Omega[{\bs X}_H]}^2 &\simeq\|{\bs v}_h\|_{L^2(\Omega^{\rm f}_{\rm ext}[\bs X_H])}^2 + h^2g_h^{\rm u}(\bv_h,\bv_h)   , \label{0-0-norm-a2}\\
\|q_h\|_{0,\Omega[{\bs X}_H]}^2 &\simeq\|q_h\|_{L^2(\Omega^{\rm f}_{\rm ext}[\bs X_H])}^2+ g_h^{\rm p}(q_h,q_h) \label{0-0-norm-b}
\end{align}
\end{subequations}
for $({\bs v}_h,q_h) \in \bs V_h^{\rm f}
\times  Q_h^{\rm f}$.
In addition, we will use the following Korn inequality with ghost-penalty terms: 
\begin{align}\label{Korn-inequlaity}
\|{\bs v}_h\|_{H^1(\Omega^f[{\bs X}_H])}^2 + g_h^{\rm u}({\bs v}_h,{\bs v}_h)
\lesssim
\|{\bs v}_h\|_{L^2(\Omega^f[{\bs X}_H])}^2
+
\|{\bs D}({\bs v}_h)\|_{L^2(\Omega^f[{\bs X}_H])}^2
+
g_h^{\rm u}({\bs v}_h,{\bs v}_h) ,
\end{align}
which is a consequence of the standard CutFEM extension estimate (for example, see 
\cite[Lemma~5.4, equations~(5.38)--(5.40)]{MassingSchottWall2018}
and
\cite[Lemma~4.1, equations~(4.3)--(4.4)]{BurmanClausMassing2015}) and the Korn inequality (for example, see 
\cite[equations~(5.13)--(5.17)]{MassingSchottWall2018}
and
\cite[equations~(4.10)--(4.12)]{BurmanClausMassing2015}).

\begin{lemma}\label{Lemma:RhH}
Under the assumptions of Theorem \ref{THM} and \eqref{ind-1}, the remainder $R_{h,H}^*$ defined in \eqref{def-RhH} satisfies the following estimate{\rm:} 
\begin{align}\label{aux1129}
| R_{h,H}^*({\bs v}_h, q_h , {\bs w}_H, 0) | 
&\lesssim (h^k+H^k) (\|{\bs v}_h\|_{1,\Omega[{\bs X}_H]} + \|q_h\|_{0,\Omega[{\bs X}_H]} 
+ \|{\bs w}_H\|_{H^1(\widehat{\Omega}_H^{\rm s})} ) .
\end{align}
\end{lemma}
\begin{proof}
Let us consider, for instance, the bulk term \eqref{RhH-c} and the interface term \eqref{RhH-h}$+$\eqref{RhH-i}. Then \eqref{RhH-a}--\eqref{RhH-b} and \eqref{RhH-d}--\eqref{RhH-g} can be estimated in the same way as \eqref{RhH-c}, and \eqref{RhH-j}$+$\eqref{RhH-k}=0 because of ${\bs\lambda}_H=0$ here. 

Let $\Omega^{\rm f}[{\bs X}_H^{*,\theta}]$ and $\Gamma[{\bs X}_H^{*,\theta}]$ be the bulk domain and interface, respectively, determined by the map ${\bs X}_H^{*,\theta} \eqd (1-\theta){\bs X} + \theta {\bs X}_H^* \circ \widehat{\bs\Phi}_H^{-1} : \widehat\Omega^{\rm s} \rightarrow \mathbb{R}^d$; see Section \ref{section:norm-equiv-2}. Then, by using the transport formula in \eqref{dt-int-Omega-2}, we have 
\begin{align*}
\eqref{RhH-c}
&= \int_{\Omega^{\rm f}[\bs X_H^*]} \bs D({\bs u}_h^*-{\bs u}): \bs D(\bv_h)  +  \int_0^1\frac{\d}{\d\theta}\int_{\Omega^{\rm f}[\bs X_H^{*,\theta}]} \bs D(\bu): \bs D(\bv_h) \d\theta \notag\\
&= \int_{\Omega^{\rm f}[\bs X_H^*]} \bs D({\bs u}_h^*-{\bs u}): \bs D(\bv_h)  +  \int_0^1\int_{\Gamma[\bs X_H^{*,\theta}]} ({\bs e}_{\bs X}^{*,\theta}\cdot{\bs n}_{\Gamma[\bs X_H^{*,\theta}]}) \bs D(\bu): \bs D(\bv_h) \d\theta  \notag\\
&\lesssim h^k\|{\bs v}_h\|_{H^1(\Omega^{\rm f}[\bs X_H^*])}
+ H^{k+1} \max_{\theta\in[0,1]} \|{\bs v}_h\|_{H^1(\Gamma[\bs X_H^{*,\theta}])} \notag\\
&\lesssim h^k\|{\bs v}_h\|_{H^1(\Omega^{\rm f}_{\rm ext}[\bs X_H])} 
+ H^{k+1}h^{-\frac12}\|{\bs v}_h\|_{H^1(\Omega^{\rm f}_{\rm ext}[\bs X_H])} \quad\mbox{(inverse trace estimate)} \notag\\
&\lesssim (h^k+H^k) \|{\bs v}_h\|_{1,\Omega[\bs X_H]} 
\quad\mbox{(under the condition $H\le h^{\frac12}$)}. 
\end{align*}

Using the notations $\widehat{\bs\sigma}^{*,\theta}\eqd \widehat{\bs\sigma} \circ ({\bs X}_H^{*,\theta})^{-1}$ and ${\bs w}_H^{*,\theta}\eqd {\bs w}_H \circ \widehat{\bs\Phi}_H^{-1} \circ ({\bs X}_H^{*,\theta})^{-1}$, we have 
\begin{subequations}
\begin{align}
\eqref{RhH-h} + \eqref{RhH-i} 
= & -  \int_{\Gamma[{\bs X}_H^*]} (\widehat{\bs \sigma} _H^* -\widehat{\bs\sigma}\circ\widehat{\bs\Phi}_H)\circ \bs ({\bs X}_H^*)^{-1} \cdot \big( \bv_h - \bw_H  \circ (\bs X_H^*)^{-1} \big) \notag\\
& - \Big[ \int_{\Gamma[{\bs X}_H^{*,1}]} \widehat{\bs\sigma}^{*,1} \cdot \big( \bv_h - \bw_H^{*,1} \big) - \int_{\Gamma[{\bs X}_H^{*,0}]} \widehat{\bs\sigma}^{*,0}  \cdot \big( \bv_h - \bw_H^{*,0} \big) \Big] \notag\\
= & -  \int_{\Gamma[{\bs X}_H^*]} (\widehat{\bs \sigma} _H^* -\widehat{\bs\sigma}\circ\widehat{\bs\Phi}_H)\circ \bs ({\bs X}_H^*)^{-1} \cdot \big( \bv_h - \bw_H  \circ (\bs X_H^*)^{-1} \big) \notag\\
& - \int_0^1 \frac{\d}{\d\theta} \int_{\Gamma[{\bs X}_H^{*,\theta}]} \widehat{\bs\sigma}^{*,\theta} \cdot \big( \bv_h - \bw_H^{*,\theta} \big) \d\theta \notag\\
= & -  \int_{\Gamma[{\bs X}_H^*]} (\widehat{\bs \sigma} _H^* -\widehat{\bs\sigma}\circ\widehat{\bs\Phi}_H)\circ \bs ({\bs X}_H^*)^{-1} \cdot \big( \bv_h - \bw_H  \circ (\bs X_H^*)^{-1} \big) \label{RhH-hi-a}\\
& - \int_0^1 \int_{\Gamma[{\bs X}_H^{*,\theta}]} \big[ \widehat{\bs\sigma}^{*,\theta} \cdot \big( \bv_h - \bw_H^{*,\theta} \big) \grad_{\Gamma[{\bs X}_H^{*,\theta}]}\cdot{\bs e}_{\bs X}^{*,\theta} 
+ \widehat{\bs\sigma}^{*,\theta} \cdot ({\bs e}_{\bs X}^{*,\theta} \cdot \grad \bv_h ) \big] \d\theta \label{RhH-hi-b} .
\end{align}
\end{subequations}
Then \eqref{RhH-hi-a} and \eqref{RhH-hi-b} can be bounded by using the estimates of interpolation error: 
$$
\|\widehat{\bs \sigma} _H^* - \widehat{\bs \sigma} \circ \widehat{\bs\Phi}_H\|_{L^\infty(\widehat\Gamma_H)} \lesssim H^{k+1} 
\quad\mbox{and}\quad
\| {\bs e}_{\bs X}^{*,\theta} \|_{W^{1,\infty}(\widehat{\Omega}^{\rm s}_H)} \lesssim H^k .
$$
This leads to the following estimates (using the norm equivalence in \eqref{norm-equiv-interpl-1}): 
\begin{align*}
|\eqref{RhH-hi-a} + \eqref{RhH-hi-b} | 
&\lesssim H^k 
 \max_{\theta\in[0,1]} ( \|{\bs v}_h\|_{L^2(\Gamma[{\bs X}_H^{*,\theta}])} + \|{\bs w}_H\|_{L^2(\widehat\Gamma_H)} ) 
+ H^{k+1} \max_{\theta\in[0,1]} \|\grad {\bs v}_h\|_{L^2(\Gamma[{\bs X}_H^{*,\theta}])} \notag\\
&\le CH^k (\|{\bs v}_h\|_{H^1(\Omega^{\rm f}_{\rm ext}[\bs X_H])}+ \|{\bs w}_H\|_{H^1(\widehat{\Omega}^{\rm s}_H)} ) 
+H^{k+1}h^{-\frac12} \|{\bs v}_h\|_{H^1(\Omega^{\rm f}_{\rm ext}[\bs X_H])} \notag\\
&\le CH^k (\|{\bs v}_h\|_{H^1(\Omega^{\rm f}_{\rm ext}[\bs X_H])}+ \|{\bs w}_H\|_{H^1(\widehat{\Omega}^{\rm s}_H)} ) 
\quad\mbox{(under the condition $H\lesssim h^{\frac12}$)} ,
\end{align*} 
where the second-to-last inequality uses the finite element inverse trace estimate. 

A particular point that requires careful attention concerns the treatment of the velocity ghost-penalty terms in \eqref{RhH-l}. In contrast to most existing CutFEM formulations, where the ghost penalty is applied only locally near the embedded interface, our approach employs a \emph{global} ghost-penalty stabilization (see, e.g., \cite{maxim-and-whal-25}). This distinction slightly complicates the analysis and requires separate estimates for the associated terms.

Recall that the ghost-penalty bilinear form defines a semi-norm on the discrete velocity space. Consequently, it vanishes for sufficiently smooth functions and allows us to apply the Cauchy--Schwarz inequality. Using these properties, we obtain
\[
\begin{split}
  g_h^{\mathrm{u}}(\boldsymbol{u}_h^*, \boldsymbol{v}_h)
  &= g_h^{\mathrm{u}}(\boldsymbol{u}_h^* - \boldsymbol{u}, \boldsymbol{v}_h)
  \le 
  |g_h^{\mathrm{u}}(\boldsymbol{u}_h^* - \boldsymbol{u}, \boldsymbol{u}_h^* - \boldsymbol{u})|^{1/2}
  |g_h^{\mathrm{u}}(\boldsymbol{v}_h, \boldsymbol{v}_h)|^{1/2} \\
  &\le
  |g_h^{\mathrm{u}}(\boldsymbol{u}_h^* - \boldsymbol{u}, \boldsymbol{u}_h^* - \boldsymbol{u})|^{1/2}
  \|\boldsymbol{v}_h\|_{1,\Omega[\boldsymbol{X}_H]} .
\end{split}
\]
To estimate the remaining term, we invoke the standard interpolation properties of finite element functions. In particular,
\[
\begin{split}
  |g_h^{\mathrm{u}}(\boldsymbol{u}_h^* - \boldsymbol{u}, \boldsymbol{u}_h^* - \boldsymbol{u})|
  &= \sum_{F\in \mathcal{F}_h} \sum_{j=1}^{k}
  h^{2j-1} \int_F
  \big|\jump{\partial^j_{\boldsymbol{n}_F} (\boldsymbol{u}_h^* - \boldsymbol{u})}\big|^2 \, \mathrm{d}s \\
  &\lesssim
  \sum_{K\in \mathcal{T}_h} \sum_{j=1}^{k}
  h^{2j-1}\!\left(
  h^{-1}\|\grad^{j}(\boldsymbol{u}_h^* - \boldsymbol{u})\|^2_{L^2(K)}
  + h\|\grad^{j+1}(\boldsymbol{u}_h^* - \boldsymbol{u})\|^2_{L^2(K)}
  \right).
\end{split}
\]
Then, applying standard interpolation estimates on each element and summing over the mesh, we get
\[
|g_h^{\mathrm{u}}(\boldsymbol{u}_h^* - \boldsymbol{u}, \boldsymbol{u}_h^* - \boldsymbol{u})|
\lesssim
h^{2k}\|\boldsymbol{u}\|^2_{H^{k+1}(\Omega)}
\lesssim h^{2k}.
\]
Combining the two bounds yields
\[
g_h^{\mathrm{u}}(\boldsymbol{u}_h^*, \boldsymbol{v}_h)
\lesssim
h^{k}\,\|\boldsymbol{v}_h\|_{1,\Omega[\boldsymbol{X}_H]}.
\]
Similarly, 
\begin{align*}
\left|
\epsilon_hg_h^{\rm u}(\partial_t{\bs u}_h^*,{\bs v}_h)
\right|
=
\epsilon_h
\left|
g_h^{\rm u}(\partial_t{\bs u}_h^*-\partial_t{\bs u},{\bs v}_h)
\right|
\lesssim
\epsilon_h h^k
\|{\bs v}_h\|_{1,\Omega[X_H]}.
\end{align*}
Since $\epsilon_h\lesssim h$, this is of higher order than $h^k\|v_h\|_{1,\Omega[X_H]}$. The other ghost penalty term can be estimated similarly, i.e., 
$$
g_h^{\rm p}(p^*_h,q_h)\lesssim h^k\|q_h\|_{0,\Omega[{\bs X}_H]}.
$$ 
This shows that the ghost-penalty contribution remains of the desired order and consistent with the finite element approximation properties. 
\end{proof}

\section{Error estimates} \label{sect:error}

\subsection{Estimates of \texorpdfstring{${\bf P}_{\Gamma[{\bs X}_H]}{\bs e}_{\bs u} -\partial_t{\bs e}_\dep\circ{\bs X}_H^{-1}$}{x} on \texorpdfstring{$\Gamma[{\bs X}_H]$}{x}}

In this section, we derive estimates for ${\bf P}_{\Gamma[{\bs X}_H]}{\bs e}_{\bs u} -\partial_t{\bs e}_\dep\circ{\bs X}_H^{-1}$ in ${L^p(\Gamma[{\bs X}_H])}$ and ${H^{\frac12}(\Gamma[{\bs X}_H])}$ norms, where ${\bf P}_{\Gamma[\bs X_H]} $ stands for the $L^2$ projection onto the finite element space ${\bs\Lambda}_H(\Gamma[\bs X_H])$ on the interface $\Gamma[\bs X_H]$, defined by 
\begin{align}\label{FE-space-GammaXH}
{\bs\Lambda}_H(\Gamma[\bs X_H])\eqd\{{\bs \lambda}_H \circ {\bs X}_H^{-1} : {\bs \lambda}_H \in \widehat{\bs \Lambda}_H \} . 
\end{align}
For simplicity of notation, we use the following abbreviations:  
\begin{align}\label{notation-lambda}
\begin{aligned}
 {\bs\lambda}_H^\theta &= {\bs\lambda}_H \circ ({\bs X}_H^\theta)^{-1} , 
\quad
{\bs e}_{\bs d}^\theta = {\bs e}_{\bs d} \circ ({\bs X}_H^\theta)^{-1}  , 
\quad
{\bs\dot{\bs e}}_{\bs d}^\theta = \partial_t{\bs e}_{\bs d} \circ ({\bs X}_H^\theta)^{-1} 
&&\mbox{on}\,\,\,\Gamma[{\bs X}_H^{\theta}] \\
{\bs d}_H^{*|\theta} &= {\bs d}_H^* \circ ({\bs X}_H^\theta)^{-1} , 
\quad
{\bs\dot{\bs d}}_H^{*|\theta} =  \partial_t{\bs d}_H^* \circ ({\bs X}_H^\theta)^{-1} 
&&\mbox{on}\,\,\,\Gamma[{\bs X}_H^{\theta}] \\
{\bs\lambda}_H^{*,\theta} &= {\bs\lambda}_H \circ \widehat{\bs\Phi}_H^{-1} \circ ({\bs X}_H^{*,\theta})^{-1} 
\quad\mbox{and}\quad
{\bs e}_{\bs X}^{*,\theta} = {\bs e}_{\bs X}^* \circ ({\bs X}_H^{*,\theta})^{-1}
&&\mbox{on}\,\,\,\Gamma[{\bs X}_H^{*,\theta}] \\
{\bs d}^{*,\theta} &= {\bs d} \circ ({\bs X}_H^{*,\theta})^{-1} , 
\quad
{\bs\dot{\bs d}}^{*,\theta} = \partial_t{\bs d}\circ ({\bs X}_H^{*,\theta})^{-1} 
&&\mbox{on}\,\,\,\Gamma[{\bs X}_H^{*,\theta}] \\
{\bs\dot{\bs d}}^{*,\theta}_H&=\partial_t{\bs d}^*_H\circ
\widehat{\bs\Phi}_H^{-1}\circ({\bs X}^{*,\theta}_H)^{-1}
&&\mbox{on}\,\,\,\Gamma[{\bs X}_H^{*,\theta}] 
\end{aligned}
\end{align}
where ${\bs e}_{\bs d}$ and ${\bs e}_{\bs X}^*$ are defined in \eqref{def-eu-ed} and \eqref{def-eXs}. In particular, 
$$
{\bs\lambda}_H^0={\bs\lambda}_H^{*,1} ,\quad 
\vel_H^{*|0}=\vel_H^{*,1}
\,\,\,\mbox{on}\,\,\, \Gamma[{\bs X}_H^0]=\Gamma[{\bs X}_H^{*,1}] .
$$

By taking ${\bs v}_h=\bs 0$, $q_h=0$, and ${\bs w}_H = \bs 0$ in \eqref{error-eq}, we obtain \eqref{error-eq-m}$+$\eqref{error-eq-n}$=$\eqref{error-eq-p}, which can be rewritten as follows, using the notations introduced above and the definition of $ R_{h,H}^*({\bs 0},0,{\bs 0}, {\bs\lambda}_H)$ in \eqref{def-RhH}:
\begin{subequations}\label{dual-eu-ed}
\begin{align}
&\int_{\Gamma[\bs X_H]} \bs \lambda_H^1 \cdot ({\bf P}_{\Gamma[{\bs X}_H]} {\bs e}_{\bs u} - {\bs\dot{\bs e}}_\dep^1 ) \notag \\
& = - \int_{\Gamma[\bs X_H]} \bs \lambda_H^1  \cdot \big( \bu_h^*  - \vel_H^{*|1}  \big) + \int_{\Gamma[\bs X_H^*]} \bs \lambda_H^0 \cdot \big( \bu_h^* - \vel_H^{*|0} ) \notag\\
&\quad\,  - \int_{\Gamma[\bs X_H^*]} \bs \lambda_H^{*,1} \cdot \big( \bu - \vel^{*,1} \big) 
+ \int_{\Gamma(t)} \bs \lambda_H^{*,0} \cdot \big( \bu - \vel^{*,0} \big) \notag\\
&\quad\,  - \int_{\Gamma[\bs X_H^*]} \bs \lambda_H^{*,1}  \cdot \big( \bu_h^* - \bu + \vel^{*,1} - \vel_H^{*,1} \big)  \notag\\
& = - \int_0^1 \int_{\Gamma[\bs X_H^\theta]} {\bs\lambda}_H^\theta \cdot \Big( 
{\bs e}_\dep^\theta \cdot\grad \bu_h^* 
+ \big( \bu_h^* - \vel_H^{*|\theta} \big) \grad_{\Gamma[\bs X_H^\theta]}\cdot {\bs e}_\dep^\theta \, \Big) \d\theta 
\qquad\mbox{(here \eqref{dt-int-Gamma-1} is used)} \notag\\
&\quad\, - \int_0^1 \int_{\Gamma[\bs X_H^{*,\theta}]} {\bs\lambda}_H^{*,\theta} \cdot \Big( 
{\bs e}_{\bs X}^{*,\theta}\cdot\grad \bu 
+ \big( \bu  - \vel^{*,\theta} \big) \grad_{\Gamma[\bs X_H^{*,\theta}]}\cdot {\bs e}_{\bs X}^{*,\theta} \, \Big) \d\theta 
\quad\mbox{(here \eqref{dt-int-Gamma-1} is used)} \notag\\
&\quad\,  - \int_{\Gamma[\bs X_H^*]} \bs \lambda_H^{*,1}  \cdot \big( \bu_h^* - \bu + \vel^{*,1} - \vel_H^{*,1} \big) \notag\\
& = - \int_0^1 \int_{\Gamma[\bs X_H^\theta]} {\bs\lambda}_H^\theta  \cdot ({\bs e}_\dep^\theta\cdot\grad {\bs u}) \, \d\theta
\label{dual-eu-ed-1} \\
&\quad\, - \int_0^1 \int_{\Gamma[\bs X_H^\theta]} {\bs\lambda}_H^\theta \cdot \Big( 
{\bs e}_\dep^\theta\cdot\grad ({\bs u}_h^*- {\bs u}) 
+ \big( \bu_h^*  - \vel_H^{*|\theta} \big) \grad_{\Gamma[\bs X_H^\theta]}\cdot {\bs e}_\dep^\theta \, \Big) \d\theta \label{dual-eu-ed-2} \\
&\quad\, - \int_0^1 \int_{\Gamma[\bs X_H^{*,\theta}]} {\bs\lambda}_H^{*,\theta} \cdot \Big( 
{\bs e}_{\bs X}^{*,\theta}\cdot\grad \bu 
+ \big( \bu  - \vel^{*,\theta} \big) \grad_{\Gamma[\bs X_H^{*,\theta}]}\cdot {\bs e}_{\bs X}^{*,\theta} \, \Big) \d\theta  \label{dual-eu-ed-3} \\
&\quad\,  - \int_{\Gamma[\bs X_H^*]} \bs \lambda_H^{*,1}  \cdot \big( \bu_h^* - \bu + \vel^{*,1} - \vel_H^{*,1} \big) . \label{dual-eu-ed-4} 
\end{align} 
\end{subequations}
By using the transport formula in \eqref{dt-int-Gamma-1}, we have 
\begin{align*}
\eqref{dual-eu-ed-1} 
&= - \int_{\Gamma[\bs X_H^*]} {\bs\lambda}_H^{0} \cdot ({\bs e}_\dep^{0}\cdot\grad {\bs u}) 
- \int_0^1 \int_0^\theta \frac{\d}{\d\alpha}\int_{\Gamma[\bs X_H^\alpha]} {\bs\lambda}_H^{\alpha} \cdot ({\bs e}_\dep^{\alpha}\cdot\grad {\bs u}) \, \d\alpha \d\theta \notag\\ 
&= - \int_{\Gamma[\bs X_H^*]} {\bs\lambda}_H^0 \cdot ({\bs e}_\dep^0\cdot\grad {\bs u} )
- \int_0^1 \int_0^\theta \int_{\Gamma[\bs X_H^\alpha]} {\bs\lambda}_H^{\alpha} \cdot [{\bs e}_\dep^{\alpha}\cdot({\bs e}_\dep^{\alpha}\cdot\grad^2) {\bs u}] \d\alpha \d\theta \notag\\
&\quad\, - \int_0^1 \int_0^\theta \int_{\Gamma[\bs X_H^\alpha]} {\bs\lambda}_H^{\alpha} \cdot ({\bs e}_\dep^{\alpha}\cdot\grad {\bs u})(\grad_{\Gamma[\bs X_H^\alpha]}\cdot {\bs e}_\dep^{\alpha})\, \d\alpha \d\theta .
\end{align*}
This leads to the following two different types of estimates, by using the norm equivalence relations in \eqref{norm-equiv-2} and the continuation hypothesis \eqref{ind-1}:
\begin{subequations}
\begin{align}\label{dual-eu-ed-1-1} 
\eqref{dual-eu-ed-1} 
&\le C\|{\bs\lambda}_H^0\|_{L^{p'}(\Gamma[\bs X_H^*])} \|{\bs e}_\dep^0\|_{L^p(\Gamma[\bs X_H^*])} 
+ C\|{\bs\lambda}_H^0\|_{L^{p'}(\Gamma[\bs X_H^*])}\|{\bs e}_\dep^0\|_{L^p(\Gamma[\bs X_H^*])} \|{\bs e}_\dep^0\|_{W^{1,\infty}(\Gamma[\bs X_H^*])} \notag\\
&\le C\|{\bs\lambda}_H^0\|_{L^{p'}(\Gamma[\bs X_H^*])} \|{\bs e}_\dep^0\|_{L^p(\Gamma[\bs X_H^*])} 
\end{align}
and 
\begin{align}\label{dual-eu-ed-1-2} 
\eqref{dual-eu-ed-1} 
&\le C\|{\bs\lambda}_H^0\|_{H^{-\frac12}(\Gamma[\bs X_H^*])} \|{\bs e}_\dep^0\|_{H^{\frac12}(\Gamma[\bs X_H^*])} 
+ C\|{\bs\lambda}_H^0\|_{L^2(\Gamma[\bs X_H^*])}\|{\bs e}_\dep^0\|_{L^\infty(\Gamma[\bs X_H^*])} \|{\bs e}_\dep^0\|_{H^1(\Gamma[\bs X_H^*])} \notag\\
&\le C\|{\bs\lambda}_H^0\|_{H^{-\frac12}(\Gamma[\bs X_H^*])} \|{\bs e}_\dep^0\|_{H^{\frac12}(\Gamma[\bs X_H^*])} 
+ CH\|{\bs\lambda}_H^0\|_{L^2(\Gamma[\bs X_H^*])} \|{\bs e}_\dep^0\|_{H^1(\Gamma[\bs X_H^*])} \notag\\
&\le C\|{\bs\lambda}_H^0\|_{H^{-\frac12}(\Gamma[\bs X_H^*])} \|{\bs e}_\dep^0\|_{H^{\frac12}(\Gamma[\bs X_H^*])} ,
\end{align}
\end{subequations}
where we have used the finite element inverse estimate in the fractional order norms; see Appendix \ref{appendix:inverse}.  

In order to estimate \eqref{dual-eu-ed-2}, we compare ${\bs u}_h^*  - \vel_H^{*|\theta} = {\bs u}_h^*  - \vel_H^*\circ ({\bs X}_H^\theta)^{-1}$ with ${\bs u}_h^*  - \vel_H^*\circ ({\bs X}_H^*)^{-1}$ as follows: 
\begin{multline*} 
({\bs u}_h^*  - \vel_H^*\circ ({\bs X}_H^\theta)^{-1})\circ{\bs X}_H^\theta
- ({\bs u}_h^*  - \vel_H^*\circ ({\bs X}_H^*)^{-1})\circ {\bs X}_H^* \\
= \int_0^\theta \frac{\d}{\d\alpha} \big( {\bs u}_h^* \circ {\bs X}_H^\alpha\big) \d\alpha 
= \int_0^\theta ({\bs e}_{\bs d}^\alpha\cdot\grad {\bs u}_h^*)\circ {\bs X}_H^\alpha \d\alpha \quad\mbox{on}\,\,\,\widehat\Gamma_H . 
\end{multline*} 
Therefore, by using \eqref{ind-1-1} and the norm equivalence $\|\grad_{\Gamma[\bs X_H^\theta]} {\bs e}_\dep^\theta\|_{L^p(\Gamma[\bs X_H^\theta])}\simeq \|\grad_{\widehat\Gamma_H} {\bs e}_\dep\|_{L^p(\widehat\Gamma_H)}$, together with the following result (using relation ${\bs u}\circ {\bs X} = \vel$ on $\widehat\Gamma$): 
\begin{align*} 
&\| {\bs u}_h^*  - \vel_H^{*|\theta} \|_{L^\infty(\Gamma[{\bs X}_H^\theta])}
= \|{\bs u}_h^*  - \vel_H^*\circ ({\bs X}_H^\theta)^{-1}\|_{L^\infty(\Gamma[{\bs X}_H^\theta])} \\
&\le \|{\bs u}_h^*  - \vel_H^*\circ ({\bs X}_H^*)^{-1}\|_{L^\infty(\Gamma[{\bs X}_H^*])} 
+ \max_{\alpha\in[0,\theta]} \| {\bs e}_\dep^\alpha \cdot(\grad{\bs u}_h^*)\circ {\bs X}_H^\alpha\|_{L^\infty(\widehat\Gamma_H)} \notag\\
&\le \|{\bs u}_h^* - {\bs u} \|_{L^\infty(\Gamma[{\bs X}_H^*])}  
+ \| {\bs u} - {\bs u}\circ {\bs X} \circ  \widehat{\bs\Phi}_H\circ ({\bs X}_H^*)^{-1} \|_{L^\infty(\Gamma[{\bs X}_H^*])} \notag\\
&\quad\, 
+ \| (\vel \circ \widehat{\bs\Phi}_H - \vel_H^*)\circ ({\bs X}_H^*)^{-1}\|_{L^\infty(\Gamma[{\bs X}_H^*])} 
+ \max_{\alpha\in[0,\theta]} \| {\bs e}_\dep^\alpha \cdot(\grad{\bs u}_h^*)\circ {\bs X}_H^\alpha\|_{L^\infty(\widehat\Gamma_H)} \notag\\
& \lesssim h^{k+1} + H^{k+1} + H 
\quad\mbox{(here $\vel_H^*={\bs I}_H^{\rm s}(\vel\circ\widehat{\bs\Phi}_H)$ and interpolation error estimates are used)} , 
\end{align*}
we have 
\begin{align*}
\eqref{dual-eu-ed-2} &\le C\|{\bs\lambda}_H^0\|_{L^{p'}(\Gamma[\bs X_H^*])} \|{\bs e}_\dep^0\|_{L^p(\Gamma[\bs X_H^*])} h^k 
+ CH\|{\bs\lambda}_H^0\|_{L^{p'}(\Gamma[\bs X_H^*])}\|\grad_{\Gamma[\bs X_H^*]} {\bs e}_\dep^0\|_{L^p(\Gamma[\bs X_H^*])} \\
&\le CH \|{\bs\lambda}_H^0\|_{L^{p'}(\Gamma[\bs X_H^*])} \|{\bs e}_\dep^0\|_{W^{1,p}(\Gamma[\bs X_H^*])} 
\quad\mbox{(under the condition $h^k\lesssim H$).}
\end{align*}
Again, this leads to two different types of estimates (using the finite element inverse estimate):
$$
\eqref{dual-eu-ed-2} 
\le C \|{\bs\lambda}_H^0\|_{L^{p'}(\Gamma[\bs X_H^*])} \|{\bs e}_\dep^0\|_{L^p(\Gamma[\bs X_H^*])} 
$$
and (firstly choosing $p=2$ and then using the finite element inverse estimate)
$$
\eqref{dual-eu-ed-2} 
\le C \|{\bs\lambda}_H^0\|_{H^{-\frac12}(\Gamma[\bs X_H^*])} \|{\bs e}_\dep^0\|_{H^{\frac12}(\Gamma[\bs X_H^*])} .
$$
Moreover, it is not difficult to see that (using norm equivalence among intermediate surfaces)  
\begin{align*}
|\eqref{dual-eu-ed-3}+\eqref{dual-eu-ed-4}|
&\le C\|{\bs\lambda}_H^0\|_{L^{p'}(\Gamma[\bs X_H^*])} (h^{k+1}+H^{k+1})  ,\\
|\eqref{dual-eu-ed-3}+\eqref{dual-eu-ed-4}|
&\le C\|{\bs\lambda}_H^0\|_{H^{-\frac12}(\Gamma[\bs X_H^*])} H^{-\frac12}(h^{k+1}+H^{k+1}) , 
\end{align*}
where the last inequality is obtained by firstly choosing $p=2$ and then using the finite element inverse estimate. 

The above estimates of \eqref{dual-eu-ed-1}--\eqref{dual-eu-ed-4} imply that 
\begin{subequations}
\begin{align}
\int_{\Gamma[\bs X_H]} \bs \lambda_H^1 \cdot ( {\bf P}_{\Gamma[\bs X_H]} {\bs e}_{\bs u} - {\bs\dot{\bs e}}_\dep^1 ) 
&\le C \|{\bs\lambda}_H^0\|_{L^{p'}(\Gamma[\bs X_H^*])} (\|{\bs e}_\dep^0\|_{L^p(\Gamma[\bs X_H^*])} + h^{k+1}+H^{k+1}),
\label{dual-eu-ed-Lp}\\
\int_{\Gamma[\bs X_H]} \bs \lambda_H^1 \cdot ( {\bf P}_{\Gamma[\bs X_H]} {\bs e}_{\bs u} - {\bs\dot{\bs e}}_\dep^1 ) 
&\le C \|{\bs\lambda}_H^0\|_{H^{-\frac12}(\Gamma[\bs X_H^*])} (\|{\bs e}_\dep^0\|_{H^{\frac12}(\Gamma[\bs X_H^*])} + H^{-\frac12}(h^{k+1}+H^{k+1})) ,
\label{dual-eu-ed-H}
\end{align} 
\end{subequations}
where \eqref{dual-eu-ed-Lp}--\eqref{dual-eu-ed-H} imply, via a duality argument (see Lemma \ref{Lemma-A-duality} in Appendix \ref{appendix:duality}) and the norm equivalence in \eqref{norm-equiv-interpl-1},  
\begin{subequations}\label{eu-dot-ed-Lp-H}
\begin{align}
\|{\bf P}_{\Gamma[\bs X_H]} {\bs e}_{\bs u} - {\bs\dot{\bs e}}_\dep^1\|_{L^p(\Gamma[\bs X_H])}
&\le C\|{\bs e}_\dep\|_{L^p(\widehat\Gamma_H)} + C(h^{k+1}+H^{k+1}) \quad\mbox{for}\,\,\, 1\le p\le\infty, \\
\|{\bf P}_{\Gamma[\bs X_H]} {\bs e}_{\bs u} - {\bs\dot{\bs e}}_\dep^1\|_{H^{\frac12}(\Gamma[\bs X_H])}
&\le C\|{\bs e}_\dep\|_{H^{\frac12}(\widehat\Gamma_H)} + CH^{-\frac12}(h^{k+1}+H^{k+1})  .
\end{align}
By the norm equivalence among the intermediate interfaces, we have 
\begin{align}\label{eu-dot-ed-Lp-HH}
\|({\bf P}_{\Gamma[\bs X_H]} {\bs e}_{\bs u})^{\#,\theta} - {\bs\dot{\bs e}}_\dep^{\#,\theta}\|_{H^{\frac12}(\Gamma[\bs X_H^{\#,\theta}])}
& \le C\|{\bs e}_\dep\|_{H^{\frac12}(\widehat\Gamma_H)} 
+ CH^{-\frac12}(h^{k+1}+H^{k+1}) ,
\end{align}
\end{subequations}
where $({\bf P}_{\Gamma[\bs X_H]} {\bs e}_{\bs u})^{\#,\theta}$ denotes the push-forward of ${\bf P}_{\Gamma[\bs X_H]} {\bs e}_{\bs u}$ to the surface $\Gamma[\bs X_H^{\#,\theta}]$ defined in Section \ref{section:norm-equiv-2}. 

The estimates in \eqref{eu-dot-ed-Lp-H} indicate that we can estimate ${\bs\dot{\bs e}}_\dep^1$ by decomposing it into two parts, i.e., 
$$
{\bs\dot{\bs e}}_\dep^1 = ({\bs\dot{\bs e}}_\dep^1 - {\bf P}_{\Gamma[\bs X_H]} {\bs e}_{\bs u}) + {\bf P}_{\Gamma[\bs X_H]} {\bs e}_{\bs u} \,\,\,\mbox{on}\,\,\,\Gamma[{\bs X}_H] .
$$
The estimate of the second part is presented in the next section, where we denote by $({\bf P}_{\Gamma[\bs X_H]} {\bs e}_{\bs u})^{\#,0}$ the push forward of ${\bf P}_{\Gamma[\bs X_H]} {\bs e}_{\bs u}$ from $\Gamma[\bs X_H]$ to $\Gamma[\bs X]$. 

\subsection{Estimates of \texorpdfstring{$\|({\bf P}_{\Gamma[\bs X_H]} {\bs e}_{\bs u})^{\#,0}\|_{H^{\frac12}(\Gamma[\bs X])}$}{x}}
\label{section:PHeu}

Let ${\bs X}_H^{\#,\theta} \eqd (1-\theta){\bs X} + \theta {\bs X}_H \circ \widehat{\bs\Phi}_H^{-1} : \widehat\Gamma \rightarrow \mathbb{R}^d$, as defined in Section \ref
{section:norm-equiv-2}. When $\theta$ varies from $0$ to $1$, the intermediate interface $\Gamma[{\bs X}_H^{\#,\theta}]$ moves from $\Gamma(t)$ to $\Gamma[{\bs X}_H]$ with velocity ${\bs e}_{\bs X}^{\#,\theta}$, whose explicit expression is given in \eqref{def-ej-theta}. 

We denote by $({\bf P}_{\Gamma[\bs X_H]} {\bs e}_{\bs u})^{\#,\theta}$ the push forward of ${\bf P}_{\Gamma[\bs X_H]} {\bs e}_{\bs u}$ from $\Gamma[\bs X_H]$ to $\Gamma[{\bs X}_H^{\#,\theta}]$. In particular, $({\bf P}_{\Gamma[\bs X_H]} {\bs e}_{\bs u})^{\#,1} = {\bf P}_{\Gamma[\bs X_H]} {\bs e}_{\bs u}$. Let ${\bs\lambda}_H^{\#,\theta} = {\bs\lambda}_H \circ \widehat{\bs\Phi}_H^{-1} \circ ({\bs X}_H^{\#,\theta})^{-1}$ be the push forward of ${\bs\lambda}_H$ from $\widehat\Gamma_H$ onto $\Gamma[{\bs X}_H^{\#,\theta}]$. Then 
\begin{align}\label{lambdaH1-lambdaH0}
&\Big| \int_{\Gamma[\bs X_H]} \bs \lambda_H^{\#,1} \cdot {\bs e}_{\bs u} 
- \int_{\Gamma[\bs X]} \bs \lambda_H^{\#,0} \cdot {\bs e}_{\bs u} \Big|  \notag \\
& = \Big| \int_0^1 \frac{\d}{\d\theta}\int_{\Gamma[{\bs X}_H^{\#,\theta}]} \bs \lambda_H^{\#,\theta} \cdot {\bs e}_{\bs u}  \d\theta  \Big|  \notag \\
& = \Big| \int_0^1 \int_{\Gamma[{\bs X}_H^{\#,\theta}]} \bs \lambda_H^{\#,\theta} \cdot 
\Big[{\bs e}_{\bs X}^{\#,\theta} \cdot \grad {\bs e}_{\bs u}+  {\bs e}_{\bs u}  \grad_{\Gamma[{\bs X}_H^{\#,\theta}]}\cdot {\bs e}_{\bs X}^{\#,\theta} \Big] \, \d\theta  \Big|  \notag \\
& \le C \| {\bs\lambda}_H^{\#,\theta} \|_{L^2(\Gamma[{\bs X}_H^{\#,\theta}])} \|{\bs e}_{\bs X}^{\#,\theta} \|_{L^\infty(\Gamma[{\bs X}_H^{\#,\theta}])} \| \grad {\bs e}_{\bs u} \|_{L^2(\Gamma[{\bs X}_H^{\#,\theta}])} \notag\\
&\quad\, 
+ \Big| \int_0^1 \int_{\Gamma[{\bs X}_H^{\#,\theta}]} \bs \lambda_H^{\#,\theta}\cdot 
\Big[ {\bs e}_{\bs u}  \grad_{\Gamma[{\bs X}_H^{\#,\theta}]}\cdot {\bs e}_{\bs X}^{\#,\theta}\Big] \, \d\theta  \Big|  \notag\\
& \le C h^{\frac12} H^{\frac12} \| {\bs\lambda}_H^{\#,\theta} \|_{L^2(\Gamma[{\bs X}_H^{\#,\theta}])} \| \grad {\bs e}_{\bs u} \|_{L^2(\Gamma[{\bs X}_H^{\#,\theta}])} \notag\\
&\quad\, 
+ \Big| \int_0^1 \int_{\Gamma[{\bs X}_H^{\#,\theta}]} \bs \lambda_H^{\#,\theta} \cdot 
\Big[ {\bs e}_{\bs u}  \grad_{\Gamma[{\bs X}_H^{\#,\theta}]}\cdot {\bs e}_{\bs X}^{\#,\theta} \Big] \, \d\theta  \Big|  \notag\\
& \le C \| {\bs\lambda}_H^{\#,0} \|_{H^{-\frac12}(\Gamma[\bs X])} \| \grad {\bs e}_{\bs u} \|_{L^2(\Omega^{\rm f}_{\rm ext}[\bs  X_H])} \notag\\
&\quad\, 
+ \Big| \int_0^1 \int_{\Gamma[{\bs X}_H^{\#,\theta}]} \bs \lambda_H^{\#,\theta} \cdot 
\Big[ {\bs e}_{\bs u}  \grad_{\Gamma[{\bs X}_H^{\#,\theta}]}\cdot {\bs e}_{\bs X}^{\#,\theta} \Big] \, \d\theta  \Big|, 
\end{align} 
where the last inequality follows from the finite element inverse estimate, and the last term in the right-hand side of \eqref{lambdaH1-lambdaH0} can be estimated as follows: 
\begin{align}\label{lambdaH-eu}
&\Big|\int_0^1 \int_{\Gamma[{\bs X}_H^{\#,\theta}]} {\bs \lambda}_H^{\#,\theta} \cdot 
\Big[ {\bs e}_{\bs u}  \grad_{\Gamma[{\bs X}_H^{\#,\theta}]}\cdot {\bs e}_{\bs X}^{\#,\theta} \Big] \, \d\theta \Big| \notag\\
&\le C \max_{\theta\in[0,1]} \|{\bs \lambda}_H^{\#,\theta}\|_{L^2(\Gamma[{\bs X}_H^{\#,\theta}])} 
\|{\bs e}_{\bs u}\|_{L^4(\Gamma[{\bs X}_H^{\#,\theta}])} 
\|\grad_{\Gamma[{\bs X}_H^{\#,\theta}]}\cdot {\bs e}_{\bs X}^{\#,\theta}\|_{L^4(\Gamma[{\bs X}_H^{\#,\theta}])} \notag\\
&\le C\max_{\theta\in[0,1]}  \|{\bs \lambda}_H^{\#,\theta}\|_{L^2(\Gamma[{\bs X}_H^{\#,\theta}])} 
\|{\bs e}_{\bs u}\|_{L^4(\Gamma[{\bs X}_H^{\#,\theta}])} 
(\|\grad_{\Gamma[{\bs X}_H^{\#,\theta}]}\cdot {\bs e}_{\bs d}^{\#,\theta}\|_{L^4(\Gamma[{\bs X}_H^{\#,\theta}])} + H^k) \notag\\
&\hspace{250pt}\mbox{(triangle inequality)}\notag\\
&\le C\|{\bs \lambda}_H^{\#,0}\|_{L^2(\Gamma[\bs X])} 
\|{\bs e}_{\bs u}\|_{H^1(\Omega^{\rm f}_{\rm ext}[{\bs X}_H])} 
(\|\grad {\bs e}_{\bs d}\|_{L^4(\widehat\Gamma_H)} + H^k) \notag\\
&\hspace{250pt}\mbox{(norm equivalence)}\notag\\
&\le CH^{-\frac12}\|{\bs \lambda}_H^{\#,0}\|_{H^{-\frac12}(\Gamma[\bs X])} 
\|{\bs e}_{\bs u}\|_{H^1(\Omega^{\rm f}_{\rm ext}[{\bs X}_H])} 
(H^{-1}\|\grad {\bs e}_{\bs d}\|_{L^2(\widehat{\Omega}_H^{\rm s})} + H^k) \notag\\
&\hspace{250pt}\mbox{(inverse inequality)}\notag\\
&\le CH^{-\frac12}\|{\bs \lambda}_H^{\#,0}\|_{H^{-\frac12}(\Gamma[\bs X])} 
\|{\bs e}_{\bs u}\|_{H^1(\Omega^{\rm f}_{\rm ext}[{\bs X}_H])} H^{\frac12} \notag\\
&\hspace{250pt}\mbox{(here \eqref{ind-1-3} is used)}\notag\\
&= C\|{\bs \lambda}_H^{\#,0}\|_{H^{-\frac12}(\Gamma[\bs X])} 
\|{\bs e}_{\bs u}\|_{H^1(\Omega^{\rm f}_{\rm ext}[{\bs X}_H])} .
\end{align}
Since $\int_{\Gamma[\bs X_H]} \bs \lambda_H^{\#,1} \cdot {\bs e}_{\bs u} = \int_{\Gamma[\bs X_H]} \bs \lambda_H^{\#,1} \cdot ({\bf P}_{\Gamma[\bs X_H]} {\bs e}_{\bs u} )^{\#,1}$, this proves that 
\begin{equation}\label{lambdaH1-01}
\Big| \int_{\Gamma[\bs X_H]} \bs \lambda_H^{\#,1} \cdot ({\bf P}_{\Gamma[\bs X_H]} {\bs e}_{\bs u} )^{\#,1}
- \int_{\Gamma[\bs X]} \bs \lambda_H^{\#,0} \cdot {\bs e}_{\bs u} \Big| 
\le
C\|{\bs \lambda}_H^{\#,0}\|_{H^{-\frac12}(\Gamma(t))} \| {\bs e}_{\bs u}\|_{H^1(\Omega^{\rm f}_{\rm ext}[\bs  X_H])} .
\end{equation} 
Using the fundamental theorem of calculus and the surface transport formula, we get 
\begin{align}\label{lambdaH1-02}
&\Big| \int_{\Gamma[{\bs X}_H]} {\bs\lambda}_H^{\#,1} \cdot ({\bf P}_{\Gamma[\bs X_H]}{\bs e}_{\bs u})^{\#,1} 
- \int_{\Gamma[\bs X]} \bs \lambda_H^{\#,0} \cdot ({\bf P}_{\Gamma[\bs X_H]}{\bs e}_{\bs u})^{\#,0} \Big|  \notag \\
&=
\Big| \int_0^1 \int_{\Gamma[{\bs X}_H^{\#,\theta}]}{\bs\lambda}_H^{\#,\theta} \cdot ({\bf P}_{\Gamma[\bs X_H]}{\bs e}_{\bs u})^{\#,\theta} \grad_{\Gamma[{\bs X}_H^{\#,\theta}]}\cdot {\bs e}_{\bs X}^{\#,\theta} \d\theta \Big| \notag\\
&\le
\max_{\theta\in[0,1]} \|{\bs\lambda}_H^{\#,\theta}\|_{L^2(\Gamma[{\bs X}_H^{\#,\theta}])}
\|({\bf P}_{\Gamma[\bs X_H]}{\bs e}_{\bs u})^{\#,\theta}\|_{L^4(\Gamma[{\bs X}_H^{\#,\theta}])} 
\|\grad_{\Gamma[{\bs X}_H^{\#,\theta}]}{\bs e}_{\bs X}^{\#,\theta}\|_{L^4(\Gamma[{\bs X}_H^{\#,\theta}])} \notag\\
&\le
C\|{\bs\lambda}_H^{\#,0}\|_{L^2(\Gamma(t))}
\|{\bf P}_{\Gamma[\bs X_H]}{\bs e}_{\bs u}\|_{L^4(\Gamma[{\bs X}_H])} H^{\frac12}
\quad\mbox{(similarly to \eqref{lambdaH-eu})}\notag\\
&\le
C\|{\bs\lambda}_H^{\#,0}\|_{H^{-\frac12}(\Gamma(t))}
\|{\bs e}_{\bs u}\|_{H^1(\Omega^{\rm f}_{\rm ext}[{\bs X}_H])} 
\quad\mbox{(finite element inverse estimate)},
\end{align} 
where we have used the following result (see Appendices \ref{appendix:trace_inequlaity} and \ref{appendix:L2-projection}):
\begin{align*}
&\|{\bf P}_{\Gamma[\bs X_H]}{\bs e}_{\bs u}\|_{L^4(\Gamma[{\bs X}_H])} \notag\\
&\le C\|{\bf P}_{\Gamma[\bs X_H]}{\bs e}_{\bs u}\|_{H^{\frac12}(\Gamma[{\bs X}_H])} &&\mbox{(Sobolev embedding)}\notag\\
&\le C\|{\bs e}_{\bs u}\|_{H^{\frac12}(\Gamma[{\bs X}_H])} &&\mbox{(boundedness of ${\bf P}_{\Gamma[\bs X_H]}$ in $H^{\frac12}(\Gamma[{\bs X}_H])$)}\notag\\
&\le C\|{\bs e}_{\bs u}\|_{H^1(\Omega^{\rm f}_{\rm ext}[{\bs X}_H])} &&\mbox{(trace inequality; see Appendix \ref{appendix:trace_inequlaity})}. 
\end{align*}
This proves that, combining the estimates in \eqref{lambdaH1-01}--\eqref{lambdaH1-02} and using the triangle inequality, 
\begin{align}
&\Big|\int_{\Gamma[\bs X]} {\bs\lambda}_H^{\#,0} \cdot ({\bf P}_{\Gamma[\bs X_H]}{\bs e}_{\bs u})^{\#,0}\Big| \notag\\
&\le \Big|\int_{\Gamma[\bs X]} \bs \lambda_H^{\#,0} \cdot {\bs e}_{\bs u} \Big| 
+ C\|{\bs \lambda}_H^{\#,0}\|_{H^{-\frac12}(\Gamma(t))} \| {\bs e}_{\bs u}\|_{H^1(\Omega^{\rm f}_{\rm ext}[\bs  X_H])} \notag \\
&\le 
C\|{\bs \lambda}_H^{\#,0}\|_{H^{-\frac12}(\Gamma(t))} \|{\bs e}_{\bs u}\|_{H^{\frac12}(\Gamma(t))} 
+ C\|{\bs \lambda}_H^{\#,0}\|_{H^{-\frac12}(\Gamma(t))} \| {\bs e}_{\bs u}\|_{H^1(\Omega^{\rm f}_{\rm ext}[\bs  X_H])} \notag\\
&\le 
C\|{\bs \lambda}_H^{\#,0}\|_{H^{-\frac12}(\Gamma(t))} \| {\bs e}_{\bs u}\|_{H^1(\Omega^{\rm f}_{\rm ext}[\bs  X_H])} .
\end{align} 
Therefore, based on the duality argument (see Lemma~\ref{Lemma-A-duality} in Appendix~\ref{appendix:duality}), 
\begin{align}\label{H12-PXeu}
\| ({\bf P}_{\Gamma[\bs X_H]}{\bs e}_{\bs u})^{\#,0}\|_{H^{\frac12}(\Gamma[\bs X])} 
&\le C\| {\bs e}_{\bs u}\|_{H^1(\Omega^{\rm f}_{\rm ext}[\bs  X_H])} . 
\end{align}

\subsection{Stability estimates for \texorpdfstring{$e_{p}$}{x}} \label{section:ep}

In this subsection, we establish stability estimates for $e_{p}$ by utilizing the following inf-sup conditions (see Appendix~\ref{appendix:inf-sup} for a proof): For  $e_p\in Q_h^{\rm f}$ there exists ${\bs v}_h\in  {\bs V}_h^{\rm f}$
such that the following two inequalities hold simultaneously: 
\begin{subequations}\label{inf-sup-ep-L2}
\begin{align}
\|{\bs v}_h\|_{1,\Omega[{\bs X}_H]} 
&\lesssim \| e_p\|_{L^2(\Omega^{\rm f}[{\bs X}_H])},\\
\| e_p\|_{0,\Omega[{\bs X}_H]}^2 
&\lesssim 
\int_{\Omega^{\rm f}[{\bs X}_H]} e_p \,\grad\cdot {\bs v}_h + g^{\rm p}_h(e_p,e_p) .
\end{align}
\end{subequations}
%

By choosing $q_h=-e_p$, ${\bs w}_H=\bs 0$ and ${\bs\lambda}_H=\bs 0$ in \eqref{error-eq}, 
we have 
\begin{subequations}\label{ep-divvh-eq}
\begin{align}
& \int_{\Omega^{\rm f}[\bs X_H]}  e_p \,\grad\cdot \bv_h
 + g_h^{\rm p} (e_p,e_p) \notag\\
= & \int_{\Omega^{\rm f}[\bs X_H]} \partial_t {\bs e}_{\bs u} \cdot {\bs v}_h \label{ep-divvh-a}\\
& + \Big( \int_{\Omega^{\rm f}[\bs X_H]} - \int_{\Omega^{\rm f}[\bs X_H^*]} \Big) \partial_t {\bs u}_h^* \cdot {\bs v}_h \label{ep-divvh-b}\\
& + \int_{\Omega^{\rm f}[\bs X_H]} {\bs e}_{\bs u} \cdot \grad \bu_h \cdot {\bs v}_h \label{ep-divvh-c}\\
& + \int_{\Omega^{\rm f}[\bs X_H]} {\bs u}_h^* \cdot \grad {\bs e}_{\bs u}  \cdot {\bs v}_h \label{ep-divvh-d}\\
& + \Big( \int_{\Omega^{\rm f}[\bs X_H]} - \int_{\Omega^{\rm f}[\bs X_H^*]} \Big) {\bs u}_h^* \cdot \grad {\bs u}_h^* \cdot {\bs v}_h \label{ep-divvh-e}\\
&
+ \int_{\Omega^{\rm f}[\bs X_H]} \bs D({\bs e}_{\bs u}): \bs D(\bv_h) \label{ep-divvh-f}\\ 
& + \int_{\Omega^{\rm f}[\bs X_H]} {\bs\sigma}(\bu_h^*,p_h^*): \grad \bv_h 
- \int_{\Omega^{\rm f}[\bs X_H^*]} {\bs\sigma}(\bu_h^*,p_h^*): \grad \bv_h \label{ep-divvh-g}\\
& - \int_{\Gamma[\bs X_H]} {\bs e}_{\bs\sigma}\circ{\bs X}_H^{-1} \cdot  \bv_h \label{ep-divvh-h}\\
& - \int_{\Gamma[\bs X_H]} {\widehat{\bs \sigma}}_H^*\circ{\bs X}_H^{-1} \cdot \bv_h  
+ \int_{\Gamma[\bs X_H^*]} {\widehat{\bs \sigma} }_H^*\circ({\bs X}_H^*)^{-1}  \cdot \bv_h \label{ep-divvh-i} \\
&- \int_{\Omega^{\rm f}[\bs X_H]} e_p \nabla\cdot  {\bs e}_{\bs u}\label{ep-1505}\\
& - \Big( \int_{\Omega^{\rm f}[\bs X_H]} - \int_{\Omega^{\rm f}[\bs X_H^*]} \Big) e_p \nabla\cdot   \bu_h^* \label{ep-divu-h}\\
& + \epsilon_h g_h^{\rm u} (\partial_t{\bs e}_{\bs u},\bv_h) + g_h^{\rm u} ({\bs e}_{\bs u},\bv_h) + R_{h,H}^*({\bs v}_h, -e_p, \bs 0, \bs 0) , \label{ep-divvh-j} 
\end{align}
\end{subequations}
where
\begin{subequations}\label{ep-divvh-abcdef}
\begin{align}
|\eqref{ep-divvh-a}|
&\le C\|\partial_t {\bs e}_{\bs u}\|_{L^2(\Omega^{\rm f}[\bs X_H])} \|{\bs v}_h\|_{L^2(\Omega^{\rm f}[\bs X_H])} \\
|\eqref{ep-divvh-b}|
&\le C\|{\bs e}_\dep\|_{H^1(\widehat{\Omega}^{\rm s}_H)} \|{\bs v}_h\|_{H^1(\Omega^{\rm f}_{\rm ext}[\bs X_H])} 
\quad\mbox{using \eqref{dt-int-Omega-2}, \eqref{L2-f-trace} and \eqref{norm-equiv-GH}}
\label{ep-divvh-b-est}\\
|\eqref{ep-divvh-c}|
&=
\int_{\Omega^{\rm f}[\bs X_H]} {\bs e}_{\bs u} \cdot \grad \bu_h^* \cdot {\bs v}_h
+ \int_{\Omega^{\rm f}[\bs X_H]} {\bs e}_{\bs u} \cdot \grad {\bs e}_{\bs u} \cdot {\bs v}_h \notag\\ 
&\le C\|{\bs e}_{\bs u}\|_{L^2(\Omega^{\rm f}[{\bs X}_H])} \|{\bs v}_h\|_{L^2(\Omega^{\rm f}[{\bs X}_H])}
+ C\|{\bs e}_{\bs u}\|_{H^1(\Omega^{\rm f}[{\bs X}_H])} \|{\bs v}_h\|_{L^2(\Omega^{\rm f}[{\bs X}_H])} \notag\\
&\hspace{20pt} 
\quad\mbox{where we have used $\|{\bs e}_{\bs u}\|_{L^\infty(\Omega^{\rm f}_{\rm ext}[\bs X_H])} \le 1$, which follows from \eqref{ind-1-2}} \notag\\
&\le C\|{\bs e}_{\bs u}\|_{H^1(\Omega^{\rm f}_{\rm ext}[{\bs X}_H])} \|{\bs v}_h\|_{H^1(\Omega^{\rm f}[{\bs X}_H])} \\
|\eqref{ep-divvh-d}|
&\le C\|{\bs e}_{\bs u}\|_{H^1(\Omega^{\rm f}_{\rm ext}[{\bs X}_H])} \|{\bs v}_h\|_{H^1(\Omega^{\rm f}_{\rm ext}[{\bs X}_H])} \\
|\eqref{ep-divvh-e}|
&\le C\|{\bs e}_\dep\|_{H^1(\widehat{\Omega}^{\rm s}_H)} \|{\bs v}_h\|_{H^1(\Omega^{\rm f}_{\rm ext}[\bs X_H])} 
\quad\mbox{(the same as \eqref{ep-divvh-b-est})}\\
|\eqref{ep-divvh-f}|
&\le C\|{\bs e}_{\bs u}\|_{H^1(\Omega^{\rm f}_{\rm ext}[{\bs X}_H])}\|{\bs v}_h\|_{H^1(\Omega^{\rm f}[\bs X_H])} 
\end{align}
\end{subequations}
and 
\begin{subequations}\label{ep-divvh-g1}
\begin{align}
\eqref{ep-divvh-g} 
&= \int_{\Omega^{\rm f}[\bs X_H]} {\bs\sigma}(\bu,p): \grad \bv_h 
- \int_{\Omega^{\rm f}[\bs X_H^*]} {\bs\sigma}(\bu,p): \grad \bv_h \notag\\
&\quad\,
+ \Big( \int_{\Omega^{\rm f}[\bs X_H]} - \int_{\Omega^{\rm f}[\bs X_H^*]} \Big) ({\bs\sigma}(\bu_h^*,p_h^*)-{\bs\sigma}(\bu,p)): \grad \bv_h \notag\\  
&= -\int_{\Omega^{\rm f}[\bs X_H]} \grad\cdot{\bs\sigma}(\bu,p) \cdot \bv_h 
+ \int_{\Omega^{\rm f}[\bs X_H^*]} \grad\cdot{\bs\sigma}(\bu,p) \cdot \bv_h \label{ep-divvh-g-a}\\
&\quad\, 
+ \int_{\Gamma[\bs X_H]} {\bs\sigma}(\bu,p){\bs n}_{\Gamma[\bs X_H]}\cdot\bv_h 
- \int_{\Gamma[\bs X_H^*]} {\bs\sigma}(\bu,p){\bs n}_{\Gamma[\bs X_H^*]}\cdot\bv_h \label{ep-divvh-g-b}\\
&\quad\, 
+ \Big( \int_{\Omega^{\rm f}[\bs X_H]} - \int_{\Omega^{\rm f}[\bs X_H^*]} \Big) ({\bs\sigma}(\bu_h^*,p_h^*) - {\bs\sigma}(\bu,p)): \grad \bv_h ,\label{ep-divvh-g-c} 
\end{align}
\end{subequations}
where, by using the transport formula in \eqref{dt-int-Omega-2}, the trace inequality in \eqref{L2-f-trace} and the norm equivalence in \eqref{norm-equiv-GH}, 
\begin{align}\label{ep-divvh-g-ac} 
| \eqref{ep-divvh-g-a} + \eqref{ep-divvh-g-c} | 
\le C(\|{\bs e}_\dep\|_{H^1(\widehat{\Omega}^{\rm s}_H)} + CH^k) \|{\bs v}_h\|_{H^1(\Omega^{\rm f}_{\rm ext}[\bs X_H])} 
\end{align}
and \eqref{ep-divvh-g-b} can be estimated with \eqref{ep-divvh-i} together, i.e., 
\begin{subequations}\label{ep-divvh-g-b-i12}
\begin{align}
\eqref{ep-divvh-g-b} + \eqref{ep-divvh-i} 
&= 
\int_{\Gamma[\bs X_H]} {\bs\sigma}(\bu,p){\bs n}_{\Gamma[\bs X_H]}\cdot\bv_h 
- \int_{\Gamma[\bs X_H]} {\widehat{\bs\sigma}}_H^*\circ{\bs X}_H^{-1} \cdot \bv_h \label{ep-divvh-g-b-i1}\\
&\quad\, 
+ \int_{\Gamma[\bs X_H^*]} {\widehat{\bs\sigma}}_H^* \circ({\bs X}_H^*)^{-1} \cdot \bv_h  
- \int_{\Gamma[\bs X_H^*]} {\bs\sigma}(\bu,p){\bs n}_{\Gamma[\bs X_H^*]}\cdot \bv_h .
\label{ep-divvh-g-b-i2}
\end{align}
\end{subequations}
We can estimate \eqref{ep-divvh-g-b-i2} by using the definition ${\widehat{\bs\sigma}}_H^* ={\bs I}_H^{\rm s}[({\bs\sigma}(\bu,p){\bs n})\circ{\bs X} \circ \widehat{\bs\Phi}_H ]$ and the interpolation error estimate in \eqref{aux865}. This gives the following result: 
\begin{align}\label{ep-divvh-g-b-i2e}
|\eqref{ep-divvh-g-b-i2}|
\le CH^k \|{\bs v}_h\|_{L^2(\Gamma[\bs X_H^*])} 
\le CH^k \|{\bs v}_h\|_{H^1(\Omega^{\rm f}_{\rm ext}[\bs X_H])} . 
\end{align}
Moreover, we can estimate \eqref{ep-divvh-g-b-i1} by using relation \eqref{dt-int-Gamma-1} as well as the expressions of $\widehat{\bs\sigma}$ and $\widehat{\bs\sigma}_H^*$ defined in  \eqref{def-hat-sigma} and \eqref{def-sigma-s}, i.e., 
\begin{subequations}\label{ep-divvh-g-b-ji3}
\begin{align}
\eqref{ep-divvh-g-b-i1} 
& = 
\int_{\Gamma[\bs X_H]} [{\bs\sigma}(\bu,p){\bs n}]|_{\Gamma[\bs X_H]}\cdot\bv_h 
- \int_{\Gamma[\bs X_H]} [{\bs\sigma}(\bu,p){\bs n}]\circ{\bs X}\circ\widehat{\bs\Phi}_H \circ{\bs X}_H^{-1} \cdot \bv_h \notag\\
&\quad\, + 
\int_{\Gamma[\bs X_H]} \big(\widehat{\bs\sigma}\circ\widehat{\bs\Phi}_H - {\widehat{\bs\sigma}}_H^* \big)\circ{\bs X}_H^{-1} \cdot\bv_h \notag\\
& = \int_{\Gamma[\bs X_H]} \bv_h\cdot \int_0^1\frac{\d}{\d\theta} \big[\big( {\bs\sigma}(\bu,p)|_{\Gamma[{\bs X}_H^{\#,\theta}]} {\bs n}_{\Gamma[{\bs X}_H^{\#,\theta}]}\big)\circ{\bs X}_H^{\#,\theta}\circ\widehat{\bs\Phi}_H \circ{\bs X}_H^{-1} \big]\d\theta \notag\\ 
&\quad\,\,\mbox{(see the notation in Section \ref{section:PHeu})} \notag\\
&\quad\, + 
\int_{\Gamma[\bs X_H]} \big(\widehat{\bs\sigma}\circ\widehat{\bs\Phi}_H - {\widehat{\bs\sigma}}_H^* \big)\circ{\bs X}_H^{-1} \cdot\bv_h  \notag\\
& = \int_{\Gamma[\bs X_H]} \bv_h\cdot\int_0^1\big[ {\bs e}_{\bs X}^{\#,\theta}\cdot\grad {\bs\sigma}(\bu,p) {\bs n}_{\Gamma[{\bs X}_H^{\#,\theta}]} - {\bs\sigma}(\bu,p) (\grad_{\Gamma[{\bs X}_H^{\#,\theta}]}{\bs e}_{\bs X}^{\#,\theta}){\bs n}_{\Gamma[{\bs X}_H^{\#,\theta}]} \big] \notag\\
&\hspace{85pt} \circ{\bs X}_H^{\#,\theta}\circ\widehat{\bs\Phi}_H \circ{\bs X}_H^{-1}  \,\, \d\theta \notag\\ 
&\quad\, + 
\int_{\Gamma[\bs X_H]} \big(\widehat{\bs\sigma}\circ\widehat{\bs\Phi}_H - {\bs I}_H^{\rm s}[\widehat{\bs\sigma}\circ\widehat{\bs\Phi}_H] \big)\circ{\bs X}_H^{-1} \cdot\bv_h \notag\\ 
&\le C( \|{\bs e}_{\bs X}^{\#,\theta}\circ{\bs X}_H^{\#,\theta}\circ\widehat{\bs\Phi}_H \circ{\bs X}_H^{-1}\|_{L^2(\Gamma[\bs X_H])} + H^k) \|{\bs v}_h\|_{L^2(\Gamma[\bs X_H])} \label{ep-divvh-g-b-j} \\
&\quad\, 
- \int_{\Gamma[\bs X_H]} \bv_h\cdot\int_0^1 \big[{\bs\sigma}(\bu,p) (\grad_{\Gamma[{\bs X}_H^{\#,\theta}]}{\bs e}_{\bs X}^{\#,\theta}){\bs n}_{\Gamma[{\bs X}_H^{\#,\theta}]} \big]\circ{\bs X}_H^{\#,\theta}\circ\widehat{\bs\Phi}_H \circ{\bs X}_H^{-1}  \,\, \d\theta . \label{ep-divvh-g-b-i3} 
\end{align}
\end{subequations}
The term $\|{\bs e}_{\bs X}^{\#,\theta}\circ{\bs X}_H^{\#,\theta}\circ\widehat{\bs\Phi}_H \circ{\bs X}_H^{-1}\|_{L^2(\Gamma[\bs X_H])}$ in \eqref{ep-divvh-g-b-j} can be estimated by using the expression of ${\bs e}_{\bs X}^{\#,\theta}$ in \eqref{def-ej-theta} and the triangle inequality, as well as the norm-equivalence relation in \eqref{norm-equiv-X-theta}, i.e., 
\begin{align*}
\|{\bs e}_{\bs X}^{\#,\theta}\circ{\bs X}_H^{\#,\theta}\circ\widehat{\bs\Phi}_H \circ{\bs X}_H^{-1}\|_{L^2(\Gamma[\bs X_H])}
&\le C\|{\bs e}_{\bs d}\|_{L^2(\widehat\Gamma_H)} 
 + C\|{\bs e}_{\bs X}^*\|_{L^2(\widehat\Gamma)} \\
&\le C\|{\bs e}_{\bs d}\|_{H^1(\widehat\Omega^{\rm s}_H)} + CH^k ,
\end{align*}
where the last inequality follows from the trace inequality from $\widehat\Gamma_H$ and the interpolation error estimate in \eqref{aux865}. Therefore, using the trace inequality from $\Gamma[\bs X_H]$ (see Appendix \ref{appendix:trace_inequlaity}), we have 
\begin{align}\label{ep-divvh-g-b-j1}
|\eqref{ep-divvh-g-b-j}|
\le C(\|{\bs e}_{\bs d}\|_{H^1(\widehat\Omega^{\rm s}_H)}  + H^k)\|{\bs v}_h\|_{H^1(\Omega^{\rm f}_{\rm ext}[\bs X_H])} . 
\end{align}

Since $\widehat{\bs\Phi}_H \circ{\bs X}_H^{-1}= ({\bs X}_H^{\#,\alpha})^{-1}$, for $\alpha=1$ we can estimate \eqref{ep-divvh-g-b-i3} as follows by using the fundamental theorem of calculus, with the notation ${\bs v}_h^{\#,\alpha}\eqd \bv_h\circ{\bs X}_H\circ\widehat{\bs\Phi}_H^{-1}\circ ({\bs X}_H^{\#,\alpha})^{-1}$ (in particular, ${\bs v}_h^{\#,1}=\bv_h$ and ${\bs v}_h^{\#,0}=\bv_h\circ{\bs X}_H\circ\widehat{\bs\Phi}_H^{-1}\circ {\bs X}^{-1}$):
\begin{subequations}
\begin{align}
\eqref{ep-divvh-g-b-i3} 
= & - \int_{\Gamma[\bs X]} {\bs v}_h^{\#,0} \cdot \big[{\bs\sigma}(\bu,p) (\grad_{\Gamma[\bs X]}{\bs e}_{\bs X}^{\#,0}){\bs n}_{\Gamma[\bs X]} \big] \notag\\
&\quad\, - \int_0^1 \frac{\d}{\d\alpha}\Big[\int_{\Gamma[{\bs X}_H^{\#,\alpha}]}\hspace{-5pt} {\bs v}_h^{\#,\alpha} \cdot \hspace{-5pt}\int_0^1\hspace{-5pt} \big[{\bs\sigma}(\bu,p) (\grad_{\Gamma[{\bs X}_H^{\#,\alpha\theta}]}{\bs e}_{\bs X}^{\#,\alpha\theta}){\bs n}_{\Gamma[{\bs X}_H^{\#,\alpha\theta}]} \big]\notag\\
&\hspace{65pt}\circ\hspace{-1pt}{\bs X}_H^{\#,\alpha\theta}\hspace{-1pt}\circ\hspace{-1pt}({\bs X}_H^{\#,\alpha})^{-1} \d\theta \Big]\d\alpha \notag\\
= & - \int_{\Gamma[\bs X]} {\bs v}_h^{\#,0} \cdot \big[{\bs\sigma}(\bu,p) (\grad_{\Gamma[\bs X]}{\bs e}_{\bs X}^{\#,0}){\bs n}_{\Gamma[\bs X]} \big] \label{sigma-e-n-1} \\ 
&\quad\, + \int_0^1\hspace{-5pt}\int_0^1\hspace{-5pt} \int_{\Gamma[{\bs X}_H^{\#,\alpha}]} |{\bs v}_h^{\#,\alpha}| \, O(|{\bs e}_{\bs X}^{\#,\alpha}|^2\hspace{-1pt} + |{\bs e}_{\bs X}^{\#,\alpha\theta}|^2) \, \d\theta\d\alpha \notag\\
&\quad\, + \int_0^1\hspace{-5pt}\int_0^1\hspace{-5pt} \int_{\Gamma[{\bs X}_H^{\#,\alpha}]} |{\bs v}_h^{\#,\alpha}| \, O(|\grad_{\Gamma[{\bs X}_H^{\#,\alpha}]} {\bs e}_{\bs X}^{\#,\alpha}|^2\hspace{-1pt} + |\grad_{\Gamma[{\bs X}_H^{\#,\alpha\theta}]} {\bs e}_{\bs X}^{\#,\alpha\theta}|^2) \, \d\theta\d\alpha , \label{sigma-e-n-2}
\end{align} 
\end{subequations}
where the expression in \eqref{sigma-e-n-2} can be obtained by using the transport formulas in \eqref{dt-int-Gamma-1} and \eqref{dt-int-Gamma-1c}. Note that \eqref{sigma-e-n-1} can be estimated by integration by parts on the smooth interface $\Gamma(t)$, i.e., 
\begin{align}
|\eqref{sigma-e-n-1}|
&\le C\|{\bs v}_h^{\#,0}\|_{H^{\frac12}(\Gamma(t))} 
\|{\bs e}_{\bs X}^{\#,0}\|_{H^{\frac12}(\Gamma[\bs X])} \le C\|{\bs v}_h\|_{H^{\frac12}(\Gamma[{\bs X}_H])} 
(\|{\bs e}_{\bs d}\|_{H^{\frac12}(\widehat\Gamma_H)} + H^k) \notag\\
&\le C\|{\bs v}_h\|_{H^1(\Omega^{\rm f}_{\rm ext}[{\bs X}_H])} 
(\|{\bs e}_{\bs d}\|_{H^1(\widehat\Omega^{\rm s}_H)} + H^k) .
\end{align}
Moreover, \eqref{sigma-e-n-2} can be estimated by using 
\begin{align}
|\eqref{sigma-e-n-2}|
&\le C\int_0^1\int_0^1\|{\bs v}_h^{\#,\alpha}\|_{L^4(\Gamma[{\bs X}_H^{\#,\alpha}])} 
\|{\bs e}_{\bs X}^{\#,\alpha}\|_{H^1(\Gamma[{\bs X}_H^{\#,\alpha}])} 
\|{\bs e}_{\bs X}^{\#,\alpha}\|_{W^{1,4}(\Gamma[{\bs X}_H^{\#,\alpha}])} \d\alpha\d\theta \notag\\
&\quad\, 
+ C\int_0^1\int_0^1\|{\bs v}_h^{\#,\alpha}\|_{L^4(\Gamma[{\bs X}_H^{\#,\alpha}])} \|{\bs e}_{\bs X}^{\#,\alpha\theta}\|_{H^1(\Gamma[{\bs X}_H^{\#,\alpha}])} 
\|{\bs e}_{\bs X}^{\#,\alpha\theta}\|_{W^{1,4}(\Gamma[{\bs X}_H^{\#,\alpha}])} \d\alpha\d\theta\notag\\
&\le C\|{\bs v}_h\|_{L^4(\Gamma[{\bs X}_H])} 
(\|{\bs e}_{\bs d}\|_{H^1(\widehat\Gamma_H)} + H^k)
(\|{\bs e}_{\bs d}\|_{W^{1,4}(\widehat\Gamma_H)} + H^k) \notag\\
&\le C\|{\bs v}_h\|_{H^{\frac12}(\Gamma[{\bs X}_H])} 
(H^{-\frac12}\|{\bs e}_{\bs d}\|_{H^1(\widehat\Omega^{\rm s}_H)} + H^k)
(H^{-1}\|{\bs e}_{\bs d}\|_{H^1(\widehat\Omega^{\rm s}_H)} + H^k) \notag\\
&\le C\|{\bs v}_h\|_{H^1(\Omega^{\rm f}_{\rm ext}[{\bs X}_H])} 
(\|{\bs e}_{\bs d}\|_{H^1(\widehat\Omega^{\rm s}_H)} + H^{k+\frac12}) ,
\end{align}
where the last inequality follows from the trace inequality from $\Gamma[{\bs X}_H]$ (see Appendix \ref{appendix:trace_inequlaity}) and \eqref{ind-1-3}. These estimates of \eqref{sigma-e-n-1} and \eqref{sigma-e-n-2} show that
\begin{align}\label{ep-divvh-g-b-j2}
\eqref{ep-divvh-g-b-i3} 
&\le C\|{\bs v}_h\|_{H^1(\Omega^{\rm f}_{\rm ext}[{\bs X}_H])} 
(\|{\bs e}_{\bs d}\|_{H^1(\widehat\Omega^{\rm s}_H)} + H^{k}) .
\end{align} 
Substituting \eqref{ep-divvh-g-b-j1} and \eqref{ep-divvh-g-b-j2} into \eqref{ep-divvh-g-b-ji3}, we obtain 
\begin{align}\label{ep-divvh-g-b-i1e}
|\eqref{ep-divvh-g-b-i1}|
\le C
(\|{\bs e}_{\bs d}\|_{H^1(\widehat\Omega^{\rm s}_H)} + H^{k}) \|{\bs v}_h\|_{H^1(\Omega^{\rm f}_{\rm ext}[{\bs X}_H])} .
\end{align}
Then, substituting \eqref{ep-divvh-g-b-i1e} and \eqref{ep-divvh-g-b-i2e} into \eqref{ep-divvh-g-b-i12}, we obtain 
\begin{align}
|\eqref{ep-divvh-g-b} + \eqref{ep-divvh-i} |
\le 
C(\|{\bs e}_{\bs d}\|_{H^1(\widehat\Omega^{\rm s}_H)} + H^{k}) \|{\bs v}_h\|_{H^1(\Omega^{\rm f}_{\rm ext}[{\bs X}_H])} .
\end{align}
Furthermore, substituting this and \eqref{ep-divvh-g-ac} into \eqref{ep-divvh-g1}, we obtain 
\begin{align}\label{ep-divvh-g-a-1}
|\eqref{ep-divvh-g} + \eqref{ep-divvh-i}|
\le 
C(\|{\bs e}_{\bs d}\|_{H^1(\widehat\Omega^{\rm s}_H)} + H^{k}) \|{\bs v}_h\|_{H^1(\Omega^{\rm f}_{\rm ext}[{\bs X}_H])} . 
\end{align}
Using the duality between $H^{-\frac12}(\Gamma[{\bs X}_H])$ and $H^{\frac12}(\Gamma[{\bs X}_H])$ and trace inequality from $\Gamma[{\bs X}_H]$ (see Appendix \ref{appendix:trace_inequlaity}), 
\begin{align}\label{ep-divvh-h-0.5}
|\eqref{ep-divvh-h}|
&\le C\|{\bs e}_{\bs\sigma}\circ{\bs X}_H^{-1}\|_{H^{-\frac12}(\Gamma[{\bs X}_H])} \|{\bs v}_h\|_{H^{\frac12}(\Gamma[{\bs X}_H])} \notag\\
&\le C\|{\bs e}_{\bs\sigma}\circ{\bs X}_H^{-1}\|_{H^{-\frac12}(\Gamma[{\bs X}_H])} \|{\bs v}_h\|_{H^1(\Omega^{\rm f}_{\rm ext}[{\bs X}_H])} . 
\end{align}
Using the Cauchy–Schwarz inequality, we have 
\begin{align}
|\eqref{ep-1505}| 
&\le \|e_p\|_{L^2(\Omega^{\rm f}_{\rm ext}[{\bs X}_H])}\|{\bs e}_{\bs u}\|_{H^1(\Omega^{\rm f}_{\rm ext}[\bs X_H])} . 
\end{align}
Using the triangle inequality, the interpolation error estimate in \eqref{interpl-u}, and the transport formula in \eqref{dt-int-Omega-2}, we have 
\begin{align*}
|\eqref{ep-divu-h}|
&\le \Big|\Big( \int_{\Omega^{\rm f}[\bs X_H]} - \int_{\Omega^{\rm f}[\bs X_H^*]} \Big) e_p \grad\cdot (\bu_h^*-\bu)\Big|
+ \Big|\Big( \int_{\Omega^{\rm f}[\bs X_H]} - \int_{\Omega^{\rm f}[\bs X_H^*]} \Big) e_p \grad\cdot \bu\Big| \\ 
&\le C H^k \|e_p\|_{L^2(\Omega^{\rm f}_{\rm ext}[{\bs X}_H])} + \max_{\theta\in[0,1]}\Big|\int_{\Gamma[{\bs X}_H^\theta]}e_p \, (\grad\cdot\bu) \, {\bs e}_{\bs d}\circ({\bs X}_H^\theta)^{-1}\cdot{\bs n}_{\Gamma[\bs X_H^\theta]} \Big| .
\end{align*}
On $\Gamma[{\bs X}_H^\theta]$ it holds that $$\grad\cdot\bu = \grad\cdot\bu - (\grad\cdot\bu)\circ {\bs X}\circ \widehat{\bs\Phi}_H\circ ({\bs X}_H^{\theta})^{-1} = O(|{\bs I}-{\bs X}\circ \widehat{\bs\Phi}_H\circ ({\bs X}_H^{\theta})^{-1}|) . $$ 
Since 
$$
\|{\bs I}-{\bs X}\circ \widehat{\bs\Phi}_H\circ ({\bs X}_H^{\theta})^{-1}\|_{L^\infty(\Gamma[{\bs X}_H^{\theta}])}
\le \|{\bs e}_{\bs d}\|_{L^\infty(\hat\Gamma_H)} + CH^k 
\le CH \quad\mbox{according to \eqref{ind-1-1}}, 
$$
it follows that 
\begin{align}\label{ep-divu-h-est}
|\eqref{ep-divu-h}| 
&\le C H^k \|e_p\|_{L^2(\Omega^{\rm f}_{\rm ext}[{\bs X}_H])} + CH\|e_p\|_{L^2(\Gamma[{\bs X}_H])}\|{\bs e}_{\bs d}\|_{L^2(\Gamma[{\bs X}_H])} \notag\\
&\le C(H^k + \|{\bs e}_{\bs d}\|_{H^1(\widehat{\Omega}^{\rm s}_H)})\|e_p\|_{L^2(\Omega^{\rm f}_{\rm ext}[{\bs X}_H])} , 
\end{align}
where the last inequality uses the inverse trace inequality.

The velocity ghost-penalty contribution in \eqref{ep-divvh-j} can be estimated by
\begin{subequations}\label{eq:pressure-ghost-estimate}
\begin{align}
\left|
g_h^{\rm u}({\bs e}_{\bs u},{\bs v}_h)
\right|
&\le
g_h^{\rm u}({\bs e}_{\bs u},{\bs e}_{\bs u})^{1/2}
g_h^{\rm u}({\bs v}_h,{\bs v}_h)^{1/2} 
\lesssim
\|{\bs e}_{\bs u}\|_{1,\Omega[{\bs X}_H]}
\|{\bs v}_h\|_{1,\Omega[{\bs X}_H]} ,
\label{eq:pressure-velocity-ghost-estimate} \\
\left|\epsilon_h
g_h^{\rm u}(\partial_t{\bs e}_{\bs u},{\bs v}_h)
\right|
&\lesssim
\epsilon_h
g_h^{\rm u}(\partial_t{\bs e}_{\bs u},\partial_t{\bs e}_{\bs u})^{1/2}
\|{\bs v}_h\|_{1,\Omega[{\bs X}_H]}. 
\label{eq:pressure-mass-ghost-estimate}
\end{align}
\end{subequations}
The remainder in \eqref{ep-divvh-j} can be estimated by using the estimate in Lemma \ref{Lemma:RhH}, i.e., 
\begin{align}\label{eq:defect-estimate}
|R_{h,H}^*({\bs v}_h, -e_p, 0, 0)| \le C(h^k+H^k)( \|{\bs v}_h\|_{1,\Omega[{\bs X}_H]} + \|e_p\|_{0,\Omega[{\bs X}_H]}). 
\end{align}
Combining the estimates in \eqref{ep-divvh-abcdef} and \eqref{ep-divvh-g-a-1}--\eqref{eq:defect-estimate} and choosing a ${\bs v}_h$ satisfying \eqref{inf-sup-ep-L2}, we obtain (for sufficiently small $h$ and $H$)
\begin{align}\label{ep-L2}
\| e_p \|_{0,\Omega[{\bs X}_H]}
&\le C(\|{\bs e}_{\bs u}\|_{1,\Omega[\bs X_H]} + \|{\bs e}_\dep\|_{H^1(\widehat{\Omega}^{\rm s}_H)} ) 
+ C\epsilon_h
g_h^{\rm u}(\partial_t{\bs e}_{\bs u},\partial_t{\bs e}_{\bs u})^{1/2}\notag\\
&\quad\, +C(\|\partial_t {\bs e}_{\bs u}\|_{L^2(\Omega^{\rm f}[\bs X_H])} + \|{\bs e}_{\bs\sigma}\circ{\bs X}_H^{-1}\|_{H^{-\frac12}(\Gamma[{\bs X}_H])}) + C(h^k+H^k).
\end{align}

\subsection{Stability estimates for \texorpdfstring{${\bs e}_{\bs\sigma}\circ{\bs X}_H^{-1}$}{x}} \label{section:esigma}

To abbreviate the notation, for the functions ${\bs e}_{\bs\sigma}$, ${\bs w}_H$ and ${\widehat{\bs\sigma}}_H^{*}$ which are originally defined on $\widehat\Gamma_H$, we define the following transported functions on $\Gamma[{\bs X}_H^{\#,\theta}]$ and $\Gamma[{\bs X}_H^{*,\theta}]$:
\begin{align}\label{notation-sigma}
\begin{aligned}
{\bs e}_{\bs\sigma}^{\#,\theta} &\eqd {\bs e}_{\bs\sigma} \circ \widehat{\bs\Phi}_H^{-1} \circ({\bs X}_H^{\#,\theta})^{-1}, 
&&{\bs w}_H^{\#,\theta} \eqd {\bs w}_H\circ \widehat{\bs\Phi}_H^{-1} \circ({\bs X}_H^{\#,\theta})^{-1} ,\\ 
{\bs e}_{\bs d}^{\#,\theta} &\eqd {\bs e}_{\bs d}\circ \widehat{\bs\Phi}_H^{-1} \circ({\bs X}_H^{\#,\theta})^{-1},
&&\,\,\,{\bs e}_{\bs d}^{*,\theta} \eqd {\bs e}_{\bs d}\circ \widehat{\bs\Phi}_H^{-1} \circ({\bs X}_H^{*,\theta})^{-1} , \\
{\widehat{\bs\sigma}}_H^{\#,\theta} &\eqd {\widehat{\bs\sigma}}_H^{*}\circ \widehat{\bs\Phi}_H^{-1} \circ({\bs X}_H^{\#,\theta})^{-1} ,
&&\,\,{\widehat{\bs\sigma}}_H^{*,\theta} \eqd {\widehat{\bs\sigma}}_H^{*}\circ \widehat{\bs\Phi}_H^{-1} \circ({\bs X}_H^{*,\theta})^{-1} , 
\end{aligned}
\end{align}
where
\begin{align*}
{\bs X}_H^{\#,\theta}
&\eqd (1-\theta){\bs X} + \theta {\bs X}_H \circ \widehat{\bs\Phi}_H^{-1} : \widehat\Gamma \rightarrow \mathbb{R}^d ,\\
{\bs X}_H^{*,\theta} 
&\eqd (1-\theta) {\bs X} + \theta {\bs X}_H^*\circ \widehat{\bs\Phi}_H^{-1}:\widehat\Gamma\rightarrow \mathbb{R}^d .
\end{align*}

With these abbreviations, we establish stability estimates for ${\bs e}_{\bs\sigma}^{\#,1}\eqd{\bs e}_{\bs\sigma}\circ{\bs X}_H^{-1}$ by utilizing the following inf-sup condition (see Lemma \ref{Lemma:H-half} in Appendix \ref{appendix:duality}):  
\begin{align}
\| {\bs e}_{\bs\sigma}^{\#,1} \|_{H^{-\frac12}(\Gamma[{\bs X}_H])} 
\le C\sup_{{\substack{{\bs w}_H\in  \widehat{\bs V}_H^{\rm s}  \\ \|{\bs w}_H\|_{H^1(\widehat{\Omega}^{\rm s}_H)} \le 1}}} 
\int_{\Gamma[{\bs X}_H]} {\bs e}_{\bs\sigma}^{\#,1} \cdot {\bs w}_H^{\#,1} . \label{inf-sup-esigma} 
\end{align}
In order to utilize \eqref{inf-sup-esigma}, we can choose ${\bs v}_h=0$, $q_h=0$ and ${\bs\lambda}_H=0$ in \eqref{error-eq}. This leads to the following relation:
\begin{subequations}\label{esigma-test}
\begin{align}
&\int_{\Gamma[\bs X_H]} {\bs e}_{\bs\sigma}^{\#,1} \cdot \bw_H^{\#,1} \notag\\  
=\, & -\int_{\widehat{\Omega}^{\rm s}_H} \partial_{tt}{\bs e}_{\dep} \cdot \bw_H 
- \int_{\widehat{\Omega}^{\rm s}_H} {\bs D}({\bs e}_\dep):{\bs D}(\bw_H) - \int_{\widehat{\Omega}^{\rm s}_H} (\grad\cdot{\bs e}_\dep) (\grad\cdot \bw_H) \label{esigma-test-a}\\ 
& - \int_{\Gamma[\bs X_H]} {\widehat{\bs\sigma}}_H^{\#,1}\cdot \bw_H^{\#,1} + \int_{\Gamma[\bs X]} {\widehat{\bs\sigma}}_H^{\#,0} \cdot \bw_H^{\#,0}  \label{esigma-test-b}\\[5pt]
& - \int_{\Gamma[\bs X]} {\widehat{\bs\sigma}}_H^{*,0} \cdot \bw_H^{*,0} + \int_{\Gamma[\bs X_H^*]} {\widehat{\bs\sigma}}_H^{*,1} \cdot \bw_H^{*,1}  \label{esigma-test-c}\\[5pt] 
& - R_{h,H}^*(\bs 0,  0, \bs w_H, \bs 0) , \label{esigma-test-d} 
\end{align}
\end{subequations}
which holds for all ${\bs w}_H\in \widehat{\bs V}^{\rm s}_H$ defined on $\widehat\Omega^{\rm s}_H$. 
Clearly, 
\begin{align}
|\eqref{esigma-test-a}|\lesssim (\|\partial_{tt}{\bs e}_{\dep}\|_{L^2(\widehat{\Omega}^{\rm s}_H)} + \|\grad {\bs  e}_\dep\|_{L^2(\widehat{\Omega}^{\rm s}_H)}) \|{\bs w}_H\|_{H^1(\widehat\Omega^{\rm s}_H)} , 
\end{align}
and
\begin{subequations}\label{esigma-test-b12}
\begin{align}
\eqref{esigma-test-b}
&= - \int_0^1\d\theta\frac{\d}{\d\theta}\int_{\Gamma[\bs X_H^{\#,\theta}]} {\widehat{\bs\sigma}}_H^{\#,\theta}\cdot \bw_H^{\#,\theta} \notag\\ 
&= - \int_0^1\d\theta\int_{\Gamma[\bs X_H^{\#,\theta}]} {\widehat{\bs\sigma}}_H^{\#,\theta}\cdot\bw_H^{\#,\theta}\grad_{\Gamma[\bs X_H^{\#,\theta}]}\cdot{\bs e}_{\bs X}^{\#,\theta} \notag\\
&= - \int_{\Gamma[\bs X]} {\widehat{\bs\sigma}}_H^{\#,0} \cdot {\bs w}_H^{\#,0} \grad_{\Gamma[\bs X]}\cdot{\bs e}_{\bs X}^{\#,0} \label{esigma-test-b1}\\
&\quad\, 
- \int_0^1\int_0^1\d\alpha\d\theta \frac{\d}{\d\alpha} \int_{\Gamma[\bs X_H^{\#,\alpha\theta}]} {\widehat{\bs\sigma}}_H^{\#,\alpha\theta}\cdot {\bs w}_H^{\#,\alpha\theta} \grad_{\Gamma[\bs X_H^{\#,\alpha\theta}]}\cdot{\bs e}_{\bs X}^{\#,\alpha\theta}
\label{esigma-test-b2}
\end{align}
\end{subequations} 
where \eqref{esigma-test-b1} can be estimated via integration by parts on the smooth interface $\Gamma[\bs X]$, i.e.,
\begin{align}
|\eqref{esigma-test-b1}|
&\lesssim \|{\bs w}_H^{\#,0}\|_{H^{\frac12}(\Gamma[\bs X])} 
\| {\bs e}_{\bs X}^{\#,0} \|_{H^{\frac12}(\Gamma[\bs X])} \notag\\
&\lesssim \|{\bs w}_H\|_{H^{\frac12}(\widehat\Gamma_H)} 
\big(\| {\bs e}_{\bs d}\|_{H^{\frac12}(\widehat\Gamma_H)} 
+ \| {\bs e}_{\bs X}^*\|_{H^{\frac12}(\widehat\Gamma)} \big)\quad\mbox{(here \eqref{def-ej-theta} and \eqref{norm-equiv-X-theta} are used)}\notag\\
&\lesssim \|{\bs w}_H\|_{H^1(\widehat\Omega^{\rm s}_H)} 
\big( \| {\bs e}_{\bs d}\|_{H^1(\widehat\Omega^{\rm s}_H)} + H^k \big) \quad\,\mbox{(trace inequality)},
\end{align}
and \eqref{esigma-test-b2} can be estimated as follows, after applying the formulas in \eqref{dt-int-Gamma-1} and \eqref{dt-int-Gamma-1c}:
\begin{align}
|\eqref{esigma-test-b2}|
&\lesssim \int_0^1\int_0^1\d\alpha\d\theta \int_{\Gamma[\bs X_H^{\#,\alpha\theta}]} |{\widehat{\bs\sigma}}_H^{\#,\alpha\theta}\cdot {\bs w}_H^{\#,\alpha\theta}| |\grad_{\Gamma[\bs X_H^{\#,\alpha\theta}]}{\bs e}_{\bs X}^{\#,\alpha\theta}|^2 \notag\\
&\lesssim \int_0^1\int_0^1\d\alpha\d\theta \|{\bs w}_H^{\#,\alpha\theta}\|_{L^4(\Gamma[\bs X_H^{\#,\alpha\theta}])} \|{\bs e}_{\bs X}^{\#,\alpha\theta}\|_{H^1(\Gamma[\bs X_H^{\#,\alpha\theta}])} 
\|{\bs e}_{\bs X}^{\#,\alpha\theta}\|_{W^{1,4}(\Gamma[\bs X_H^{\#,\alpha\theta}])} \notag\\
&\lesssim \|{\bs w}_H\|_{L^4(\widehat\Gamma_H)} 
(\|{\bs e}_{\bs d}\|_{H^1(\widehat\Gamma_H)} +\|{\bs e}_{\bs X}^*\|_{H^1(\widehat\Gamma_H)} )
(\|{\bs e}_{\bs d}\|_{W^{1,4}(\widehat\Gamma_H)} +\|{\bs e}_{\bs X}^*\|_{W^{1,4}(\widehat\Gamma_H)})\notag\\
&\qquad \mbox{(here \eqref{def-ej-theta} and \eqref{norm-equiv-X-theta} are used)}\notag\\
&\lesssim \|{\bs w}_H\|_{H^{\frac12}(\widehat\Gamma_H)} 
(H^{-\frac12}\|{\bs e}_{\bs d}\|_{H^1(\widehat\Omega^{\rm s}_H)} + H^k)
(H^{-1}\|{\bs e}_{\bs d}\|_{H^1(\widehat\Omega^{\rm s}_H)} +H^k)\notag\\
&\qquad\mbox{(here Sobolev embedding and inverse trace estimate are used)}\notag\\
&\lesssim \|{\bs w}_H\|_{H^1(\widehat\Omega^{\rm s}_H)} 
\big(\|{\bs e}_{\bs d}\|_{H^1(\widehat\Omega^{\rm s}_H)} + H^k \big)\notag\\
&\qquad\mbox{(here $\|{\bs e}_{\bs d}\|_{H^1(\widehat\Omega^{\rm s}_H)} \le H^{\frac32}$ in \eqref{ind-1-3} is used)} . 
\end{align}
Substituting these estimates into \eqref{esigma-test-b12}, we obtain 
\begin{align}
|\eqref{esigma-test-b}|
&\lesssim \big( \|{\bs e}_{\bs d}\|_{H^1(\widehat\Omega^{\rm s}_H)} + H^k \big)
\|{\bs w}_H\|_{H^1(\widehat\Omega^{\rm s}_H)} .
\end{align}
The term \eqref{esigma-test-c} can be estimated in the same way, i.e., 
\begin{align}
|\eqref{esigma-test-c}|
&\lesssim H^k \|{\bs w}_H\|_{H^1(\widehat\Omega^{\rm s}_H)} .
\end{align}
Then, substituting these estimates of \eqref{esigma-test-a}, \eqref{esigma-test-b} and \eqref{esigma-test-c} into \eqref{esigma-test}, and estimating \eqref{esigma-test-d} with the consistency estimate $|R_{h,H}^*(0, 0, \bs w_H, 0)|\lesssim (h^k+H^k)\|{\bs w}_H\|_{H^1(\widehat\Omega^{\rm s}_H)}$ from Lemma \ref{Lemma:RhH}, we obtain the following result: 
\begin{align}\label{esigma-wH-pair}
\Big|\int_{\Gamma[\bs X_H]} {\bs e}_{\bs\sigma}^{\#,1} \cdot \bw_H^{\#,1} \Big|
\lesssim 
\|\partial_{tt}{\bs e}_{\dep}\|_{L^2(\widehat{\Omega}^{\rm s}_H)} \|{\bs w}_H\|_{L^2(\widehat\Omega^{\rm s}_H)}
+ (\|{\bs e}_\dep\|_{H^1(\widehat{\Omega}^{\rm s}_H)} 
+ h^k+H^k) \|{\bs w}_H\|_{H^1(\widehat\Omega^{\rm s}_H)} . 
\end{align}
By applying the inf-sup conditions in \eqref{inf-sup-esigma}, we derive that 
\begin{align}\label{esigma-H-half}
\| {\bs e}_{\bs\sigma}^{\#,1}\|_{H^{-\frac12}(\Gamma[{\bs X}_H])} 
\lesssim 
\|\partial_{tt}{\bs e}_{\dep}\|_{L^2(\widehat{\Omega}^{\rm s}_H)} 
+ \|{\bs e}_\dep\|_{H^1(\widehat{\Omega}^{\rm s}_H)} 
+ (h^k+H^k) . 
\end{align} 
Then, by the norm equivalence in \eqref{norm-equivl-Hs}, we have
\begin{align}\label{esigma-H-half-hat}
\| {\bs e}_{\bs\sigma}\|_{H^{-\frac12}(\widehat\Gamma_H)} 
\lesssim 
\|\partial_{tt}{\bs e}_{\dep}\|_{L^2(\widehat{\Omega}^{\rm s}_H)} 
+ \|{\bs e}_\dep\|_{H^1(\widehat{\Omega}^{\rm s}_H)} 
+ (h^k+H^k) . 
\end{align}

Moreover, by extending ${\bs w}_H$ from $\widehat\Gamma_H$ to $\widehat{\Omega}^{\rm s}_H$ as a finite element function which is zero at the interior nodes of $\widehat{\Omega}^{\rm s}_H$, this extended finite element function satisfies the following estimates: 
\begin{align}
\| \bw_H \|_{L^2(\widehat{\Omega}^{\rm s}_H)} 
&\le CH^{\frac12} \| \bw_H \|_{L^2(\widehat\Gamma_H)}  
\le C \| \bw_H \|_{L^{\frac43}(\widehat\Gamma_H)} , \label{L2-int-bd}\\ 
\| \grad \bw_H \|_{L^2(\widehat{\Omega}^{\rm s}_H)} 
&\le CH^{-1} \| \bw_H \|_{L^2(\widehat{\Omega}^{\rm s}_H)} 
\le CH^{-1} \| \bw_H \|_{L^{\frac43}(\widehat\Gamma_H)} ,
\end{align} 
where the last inequality uses \eqref{L2-int-bd}. 
This proves that, in view of \eqref{esigma-wH-pair}, 
\begin{align*}
\Big|\int_{\Gamma[\bs X_H]} {\bs e}_{\bs\sigma}^{\#,1} \cdot \bw_H^{\#,1} \Big|
&\lesssim 
[\|\partial_{tt}{\bs e}_{\dep}\|_{L^2(\widehat{\Omega}^{\rm s}_H)} 
+ H^{-1}(\|{\bs e}_\dep\|_{H^1(\widehat{\Omega}^{\rm s}_H)} 
+ h^k+H^k)] \| \bw_H \|_{L^{\frac43}(\widehat\Gamma_H)} \notag \\
&\lesssim 
[\|\partial_{tt}{\bs e}_{\dep}\|_{L^2(\widehat{\Omega}^{\rm s}_H)} 
+ H^{-1}(\|{\bs e}_\dep\|_{H^1(\widehat{\Omega}^{\rm s}_H)} 
+ h^k+H^k)] \| \bw_H^{\#,1} \|_{L^{\frac43}(\Gamma[\bs X_H])} . 
\end{align*}
Since the $L^2$ projection ${\bf P}_{\Gamma[\bs X_H]}$ is bounded on $L^p(\Gamma[\bs X_H])$ for all $1\le p\le \infty$ (see \cite[Appendix B]{Li-Tang-2025}, where $p\ge 2$ is proved and $1\le p\le 2$ follows from duality), it follows that 
\begin{align}\label{esigma-L4-1} 
\Big|\int_{\Gamma[\bs X_H]} {\bs e}_{\bs\sigma}^{\#,1} \cdot \bw \Big|
&= \Big|\int_{\Gamma[\bs X_H]} {\bs e}_{\bs\sigma}^{\#,1} \cdot {\bf P}_{\Gamma[\bs X_H]}\bw \Big| \notag\\
&\lesssim 
[\|\partial_{tt}{\bs e}_{\dep}\|_{L^2(\widehat{\Omega}^{\rm s}_H)} 
+ H^{-1}(\|{\bs e}_\dep\|_{H^1(\widehat{\Omega}^{\rm s}_H)} 
+ h^k+H^k)] \| {\bf P}_{\Gamma[\bs X_H]}\bw \|_{L^{\frac43}(\Gamma[\bs X_H])} \notag\\
&\lesssim 
[\|\partial_{tt}{\bs e}_{\dep}\|_{L^2(\widehat{\Omega}^{\rm s}_H)} 
+ H^{-1}(\|{\bs e}_\dep\|_{H^1(\widehat{\Omega}^{\rm s}_H)} 
+ h^k+H^k)] \|\bw \|_{L^{\frac43}(\Gamma[\bs X_H])} ,
\end{align}
which holds for all $\bw\in L^{\frac43}(\Gamma[\bs X_H])$. Therefore, by the duality between $L^4(\Gamma[{\bs X}_H])$ and $L^{\frac43}(\Gamma[\bs X_H])$ and the norm equivalence relation $\|{\bs e}_{\bs\sigma}^{\#,1}\|_{L^4(\Gamma[{\bs X}_H])} \simeq \|{\bs e}_{\bs\sigma}\|_{L^4(\widehat\Gamma_H)}$ as shown in \eqref{norm-equiv-GH}, we derive that
\begin{align}\label{esigma-L4} 
\| {\bs e}_{{\bs\sigma}}\|_{L^4(\widehat\Gamma_H)} 
\lesssim  
\|\partial_{tt} {\bs e}_{\dep}\|_{L^2(\widehat{\Omega}^{\rm s}_H)} + H^{-1} \|{\bs e}_\dep\|_{H^1(\widehat{\Omega}^{\rm s}_H)} 
+ H^{-1} (h^k+H^k) . 
\end{align} 

\subsection{Stability estimates for \texorpdfstring{$ \partial_t {\bs e}_{\bs u}$}{x} and \texorpdfstring{$\partial_{tt}{\bs e}_{\dep}$}{x}} \label{section:dteu}

In this subsection, we prove estimates of $ \partial_t {\bs e}_{\bs u}$ and $\partial_{tt}{\bs e}_{\dep}$ for $t\in[0,t_*]$. To this end, we choose ${\bs v}_h = \partial_t {\bs e}_{\bs u}$, $q_h=0$, ${\bs w}_H=\partial_{tt}{\bs e}_\dep$ and ${\bs\lambda}_H=\bs 0$ in \eqref{error-eq}. This leads to the following equation:
\begin{subequations}\label{error-eq-dte}
\begin{align} 
& \int_{\Omega^{\rm f}[\bs X_H]} |\partial_t {\bs e}_{\bs u} |^2 
\label{err-a-dte}\\
&\quad\, + \Big( \int_{\Omega^{\rm f}[\bs X_H]} - \int_{\Omega^{\rm f}[\bs X_H^*]} \Big) \partial_t {\bs u}_h^* \cdot \partial_t {\bs e}_{\bs u}\label{err-b-dte}\\
&\quad\, + \int_{\Omega^{\rm f}[\bs X_H]} {\bs e}_{\bs u} \cdot \grad \bu_h \cdot \partial_t {\bs e}_{\bs u} \label{err-c-dte}\\
&\quad\, + \int_{\Omega^{\rm f}[\bs X_H]} {\bs u}_h^* \cdot \grad {\bs e}_{\bs u}  \cdot \partial_t {\bs e}_{\bs u} \label{err-d-dte}\\
&\quad\, + \Big( \int_{\Omega^{\rm f}[\bs X_H]} - \int_{\Omega^{\rm f}[\bs X_H^*]} \Big) {\bs u}_h^* \cdot \grad {\bs u}_h^* \cdot \partial_t {\bs e}_{\bs u} \label{err-e-dte}\\
&\quad\, + \int_{\Omega^{\rm f}[\bs X_H]} \bs D({\bs e}_{\bs u}): \bs D(\partial_t {\bs e}_{\bs u}) \label{err-f-dte}\\ 
&\quad\, - \int_{\Omega^{\rm f}[\bs X_H]}  e_p \,\grad\cdot \partial_t {\bs e}_{\bs u} \label{err-g-dte}\\ 
&\quad\, + \int_{\Omega^{\rm f}[\bs X_H]} {\bs\sigma}(\bu_h^*,p_h^*): \grad \partial_t {\bs e}_{\bs u} 
- \int_{\Omega^{\rm f}[\bs X_H^*]} {\bs\sigma} (\bu_h^*,p_h^*): \grad \partial_t {\bs e}_{\bs u}
\label{err-h-dte}\\ 
&\quad\, +  \int_{\widehat{\Omega}^{\rm s}_H} |\partial_{tt}{\bs e}_{\dep} |^2 
+ \int_{\widehat{\Omega}^{\rm s}_H} \big[ {\bs D}({\bs e}_\dep):{\bs D}(\partial_{tt}{\bs e}_\dep) + (\grad\cdot {\bs e}_\dep ) (\grad\cdot \partial_{tt}{\bs e}_\dep ) \big] \label{err-i-dte}\\ 
&\quad\, - \int_{\Gamma[\bs X_H]} {\bs e}_{{\bs \sigma} } \circ{\bs X}_H^{-1} \cdot ( \partial_t {\bs e}_{\bs u} - \partial_{tt}{\bs e}_\dep \circ{\bs X}_H^{-1} ) \label{err-j-dte}\\ 
&\quad\, - \int_{\Gamma[\bs X_H]} {\widehat{\bs \sigma} }_H^*\circ{\bs X}_H^{-1} \cdot (\partial_t {\bs e}_{\bs u} - \partial_{tt}{\bs e}_\dep \circ{\bs X}_H^{-1}  ) 
\label{err-k-dte}\\ 
&\quad\, + \int_{\Gamma[\bs X_H^*]} {\widehat{\bs \sigma} }_H^* \circ({\bs X}_H^*)^{-1} \cdot ( \partial_t {\bs e}_{\bs u} - \partial_{tt}{\bs e}_\dep \circ({\bs X}_H^*)^{-1}) 
\label{err-l-dte}\\ 
&\quad\, + \epsilon_h g_h^{\rm u} (\partial_t{\bs e}_{\bs u}, \partial_t{\bs e}_{\bs u})   + g_h^{\rm u} ({\bs e}_{\bs u}, \partial_t{\bs e}_{\bs u})  \label{err-m-dte} \\
& = -R_{h,H}^*(\partial_t {\bs e}_{\bs u}, 0, \partial_{tt}{\bs e}_\dep, 0) . \label{err-n-dte} 
\end{align}
\end{subequations}

\noindent{\bf Part 1: Estimates for (\ref{err-a-dte})+(\ref{err-f-dte})+(\ref{err-i-dte})+(\ref{err-m-dte}).}\\
By using the transport formula in \eqref{dt-int-Omega-1}, we have 
\begin{align}
&\int_0^s(\eqref{err-a-dte}+\eqref{err-f-dte}+\eqref{err-i-dte}+\eqref{err-m-dte})\d t \notag\\ 
&=
\int_0^s\int_{\Omega^{\rm f}[\bs X_H]} |\partial_t {\bs e}_{\bs u} |^2 \d t 
+ \int_0^s \epsilon_h g_h^{\rm u} (\partial_t{\bs e}_{\bs u}, \partial_t{\bs e}_{\bs u}) \d t 
+ \int_0^s\int_{\widehat{\Omega}^{\rm s}_H} |\partial_{tt} {\bs e}_{\dep} |^2 \d t \notag\\
&\quad\,
+ \frac12\int_{\Omega^{\rm f}[\bs X_H(\cdot,s)]} |{\bs D}({\bs e}_{\bs u}(s))|^2 -  \frac12\int_{\Omega^{\rm f}[\bs X_H(\cdot,0)]} |{\bs D}({\bs e}_{\bs u}(0))|^2 \notag\\
&\quad\, 
- \int_0^s\frac12\int_{\Gamma[{\bs X}_H]} |{\bs D}({\bs e}_{\bs u})|^2 [(\partial_t{\bs d}_H)\circ {\bs X}_H^{-1} ]\cdot {\bs n}_{\Gamma[{\bs X}_H]} \d t \notag\\
&\quad\, + \int_0^s \int_{\widehat{\Omega}^{\rm s}_H} \big[ {\bs D}({\bs e}_\dep):{\bs D}(\partial_{tt}{\bs e}_\dep) + (\grad\cdot {\bs e}_\dep ) (\grad\cdot \partial_{tt}{\bs e}_\dep ) \big] \d t \notag\\
&\quad\, 
+ \frac12  g_h^{\rm u} ({\bs e}_{\bs u}(s),{\bs e}_{\bs u}(s)) 
- \frac12  g_h^{\rm u} ({\bs e}_{\bs u}(0),{\bs e}_{\bs u}(0)) \notag\\ 
&\ge 
\int_0^s\int_{\Omega^{\rm f}[\bs X_H]} |\partial_t {\bs e}_{\bs u} |^2 \d t 
+ \int_0^s \epsilon_h g_h^{\rm u} (\partial_t{\bs e}_{\bs u}, \partial_t{\bs e}_{\bs u}) \d t 
+ \frac12\int_0^s\int_{\widehat{\Omega}^{\rm s}_H} |\partial_{tt} {\bs e}_{\dep} |^2 \d t \notag\\
&\quad\,
+ \frac12\int_{\Omega^{\rm f}[\bs X_H(\cdot,s)]} |{\bs D}({\bs e}_{\bs u}(s))|^2 \notag\\
&\quad\, 
- C\int_0^s \Big( h^{-1}\|{\bs e}_{\bs u}\|_{H^1(\Omega^{\rm f}_{\rm ext}[{\bs X}_H])}^2 + H^{-2}\|\grad{\bs e}_{\bs d}\|_{L^2(\widehat\Omega^{\rm s}_H)}^2 \Big) \d t \notag\\ 
&\quad\, 
+ \frac12 g_h^{\rm u} ({\bs e}_{\bs u}(s),{\bs e}_{\bs u}(s)) 
- Ch^{2k} , 
\label{err-afim-dte}
\end{align}
where we have used the inverse trace inequality and $\|\partial_t{\bs d}_H\|_{L^\infty(\widehat\Omega^{\rm s}_H)}\lesssim 1$, which follows from $\|\partial_t{\bs e}_\dep\|_{L^\infty(\widehat{\Omega}^{\rm s}_H)}\le H^{1/4}$ in \eqref{ind-1-2}, and 
\begin{align*}
\int_{\Omega^{\rm f}[\bs X_H(\cdot,0)]} |{\bs D}({\bs e}_{\bs u}(0))|^2
+ g_h^{\rm u} ({\bs e}_{\bs u}(0),{\bs e}_{\bs u}(0)) 
&\lesssim 
\|{\bs e}_{\bs u}(0)\|_{H^1(\Omega^{\rm f}[\bs X_H(\cdot,0)])}^2
+ g^{\rm u}_h({\bs e}_{\bs u}(0),{\bs e}_{\bs u}(0)) 
\lesssim h^{2k} 
\end{align*}
as shown in \eqref{uh0-error}.

\noindent{\bf Part 2: Estimates for (\ref{err-b-dte})+(\ref{err-c-dte})+(\ref{err-d-dte})+(\ref{err-e-dte}).}\\
By using the transport formula in \eqref{dt-int-Omega-2}, we can estimate \eqref{err-b-dte} as follows, using the formula in \eqref{dt-int-Gamma-1b}: 
\begin{align}
\eqref{err-b-dte}
&= \frac{\d}{\d t} \Big[ \Big( \int_{\Omega^{\rm f}[\bs X_H]} - \int_{\Omega^{\rm f}[\bs X_H^*]} \Big) \partial_{t} {\bs u}_h^* \cdot  {\bs e}_{\bs u} \Big] \notag\\
&\quad\, 
- \Big( \int_{\Omega^{\rm f}[\bs X_H]} - \int_{\Omega^{\rm f}[\bs X_H^*]} \Big) \partial_{tt} {\bs u}_h^* \cdot  {\bs e}_{\bs u} \notag\\
&\quad\, 
- \Big( \int_{\Gamma[\bs X_H]} {\bs\dot{\bs d}}_H^1 \cdot{\bs n}_{\Gamma[\bs X_H]} \partial_{t} {\bs u}_h^* \cdot  {\bs e}_{\bs u}
-\int_{\Gamma[\bs X_H^*]} {\bs\dot{\bs d}}_H^0 \cdot{\bs n}_{\Gamma[\bs X_H^*]} \partial_{t} {\bs u}_h^* \cdot  {\bs e}_{\bs u} \Big) \notag\\
&= \frac{\d}{\d t} \int_0^1\d\theta \int_{\Gamma[\bs X_H^\theta]} ({\bs e}_{\bs d}^\theta\cdot{\bs n}_{\Gamma[\bs X_H^\theta]}) (\partial_{t} {\bs u}_h^* \cdot  {\bs e}_{\bs u}) \notag\\
&\quad\, 
- \int_0^1\d\theta\int_{\Gamma[\bs X_H^\theta]} ({\bs e}_{\bs d}^\theta\cdot{\bs n}_{\Gamma[\bs X_H^\theta]}) (\partial_{tt} {\bs u}_h^* \cdot {\bs e}_{\bs u}) \notag\\
&\quad\, 
- \int_0^1\d\theta\int_{\Gamma[\bs X_H^\theta]} [{\bs\dot{\bs e}}_{\bs d}^\theta \cdot{\bs n}_{\Gamma[\bs X_H^\theta]} - ({\bs\dot{\bs d}}_H^\theta)^{\top} (\grad_{\Gamma[\bs X_H^\theta]} {\bs e}_{\bs d}^\theta) {\bs n}_{\Gamma[\bs X_H^\theta]} ] \partial_{t} {\bs u}_h^* \cdot  {\bs e}_{\bs u} \notag\\
&\quad\, 
- \int_0^1\d\theta\int_{\Gamma[\bs X_H^\theta]} ({\bs\dot{\bs d}}_H^\theta \cdot {\bs n}_{\Gamma[\bs X_H^\theta]})({\bs e}_{\bs d}^\theta\cdot\grad) (\partial_{t}{\bs u}_h^* \cdot {\bs e}_{\bs u}) \notag\\
&\quad\, 
- \int_0^1\d\theta\int_{\Gamma[\bs X_H^\theta]} ({\bs\dot{\bs d}}_H^\theta \cdot {\bs n}_{\Gamma[\bs X_H^\theta]})(\partial_{t}{\bs u}_h^* \cdot  {\bs e}_{\bs u}) \grad_{\Gamma[\bs X_H^\theta]}\cdot {\bs e}_{\bs d}^\theta.
\end{align}
Since ${\bs\dot{\bs d}}_H=\partial_t{\bs d}_H^* + \partial_t{\bs e}_{\bs d}$ and $\|\partial_t{\bs e}_{\bs d}\|_{L^\infty(\widehat\Omega^{\rm s}_H)}\lesssim 1$, as shown in \eqref{ind-1-2}, it follows that 
$$
\|{\bs\dot{\bs d}}_H^\theta\|_{L^\infty(\Gamma[{\bs X}_H^\theta])}\lesssim 1 . 
$$
Since ${\bs e}_{\bs d}(0)=\bs 0$, by using the trace inequalities (see Appendix \ref{appendix:trace_inequlaity})
\begin{align*}
&\|{\bs e}_{\bs d}\|_{L^2(\widehat\Gamma_H)}
\lesssim \|{\bs e}_{\bs d}\|_{H^1(\widehat\Omega^{\rm s}_H)} , 
\quad 
\|{\bs e}_{\bs u}\|_{L^2(\Gamma[\bs X_H^\theta])}
\lesssim \|{\bs e}_{\bs u}\|_{H^1(\Omega^{\rm f}_{\rm ext}[{\bs X}_H])} 
\end{align*} 
and the finite element inverse estimates 
\begin{align*}
\|{\bs\dot{\bs e}}_{\bs d}^\theta\|_{L^2(\Gamma[{\bs X}_H^\theta])}
&\lesssim \|\partial_t{\bs e}_{\bs d}^\theta\|_{L^2(\widehat\Gamma_H)} 
\lesssim H^{-\frac12}\|\partial_t{\bs e}_{\bs d}\|_{L^2(\widehat\Omega^{\rm s}_H)} , \\
\|{\bs e}_{\bs d}^\theta\|_{H^1(\Gamma[{\bs X}_H^\theta])}
&\lesssim \|{\bs e}_{\bs d}\|_{H^1(\widehat\Gamma_H)} 
\lesssim H^{-\frac12}\|{\bs e}_{\bs d}\|_{H^1(\widehat\Omega^{\rm s}_H)} , \\
\|{\bs e}_{\bs u}\|_{H^1(\Gamma[{\bs X}_H^\theta])}
&\lesssim h^{-\frac12}\|{\bs e}_{\bs u}\|_{H^1(\Omega^{\rm f}_{\rm ext}[{\bs X}_H])} ,  
\end{align*} 
we obtain 
\begin{align}\label{estimate-err-b-dte}
\Big|\int_0^s\eqref{err-b-dte}\d t\Big|
&\lesssim \int_0^1 \|{\bs e}_{\bs d}(s)\|_{L^2(\widehat\Gamma_H)}
\|{\bs e}_{\bs u}(s)\|_{L^2(\Gamma[{\bs X}_H^\theta(\cdot,s)])} \d\theta\notag\\
&\quad\,
+ \int_0^s \|{\bs e}_{\bs d}\|_{H^1(\widehat\Omega^{\rm s}_H)}
\|{\bs e}_{\bs u}\|_{H^1(\Omega^{\rm f}_{\rm ext}[{\bs X}_H])} \d t\notag\\
&\quad\,
+ \int_0^s H^{-\frac12} \|\partial_t{\bs e}_{\bs d}\|_{L^2(\widehat\Omega^{\rm s}_H)}
\|{\bs e}_{\bs u}\|_{H^1(\Omega^{\rm f}_{\rm ext}[{\bs X}_H])} \d t
\notag\\
&\quad\,
+ \int_0^s (h^{-\frac12}+H^{-\frac12}) \|{\bs e}_{\bs d}\|_{H^1(\widehat\Omega^{\rm s}_H)} 
\|{\bs e}_{\bs u}\|_{H^1(\Omega^{\rm f}_{\rm ext}[{\bs X}_H])} \d t
\notag\\
&\le \varepsilon
\big( \|{\bs e}_{\bs u}(s)\|_{H^1(\Omega^{\rm f}_{\rm ext}[{\bs X}_H(\cdot,s)])}^2 + \|\grad{\bs e}_{\bs d}(s)\|_{L^2(\widehat\Omega^{\rm s}_H)}^2 \big) \notag\\
&\quad\,
+ C_\varepsilon\int_0^s 
\|{\bs e}_{\bs u}\|_{H^1(\Omega^{\rm f}_{\rm ext}[{\bs X}_H])}^2 \d t\notag\\
&\quad\,
+ C_\varepsilon\int_0^s 
\big[ (h^{-1}+H^{-1})\|{\bs e}_{\bs d}\|_{H^1(\widehat\Omega^{\rm s}_H)}^2 
+ H^{-1}\|\partial_t{\bs e}_{\bs d}\|_{L^2(\widehat\Omega^{\rm s}_H)}^2 \big]\d t , 
\end{align}
where the last inequality follows from using the interpolation inequality
\eqref{L2-f-trace} and the corresponding result for $\widehat\Omega^{\rm s}_H$, as well as the following result:
\begin{align}
\|{\bs e}_{\bs d}(s)\|_{L^2(\widehat\Omega^{\rm s}_H)}^2
&= 2\int_0^s \int_{\widehat\Omega^{\rm s}_H} {\bs e}_{\bs d} \cdot \partial_t{\bs e}_{\bs d} \d t \notag\\ 
&\lesssim \int_0^s \|\partial_t{\bs e}_{\bs d}\|_{L^2(\widehat\Omega^{\rm s}_H)}^2 \d t 
+ \int_0^s \|{\bs e}_{\bs d}\|_{L^2(\widehat\Omega^{\rm s}_H)}^2 \d t . 
\end{align}
The term \eqref{err-e-dte} can be estimated in the same way, i.e., 
\begin{align}\label{estimate-err-e-dte}
|\eqref{err-e-dte}|
&\lesssim \varepsilon\|\partial_{t}{\bs e}_{\bs u}\|_{L^2(\Omega^{\rm f}[\bs X_H])}^2 
+C_{\varepsilon}h^{-1}\|{\bs e}_{\bs d}\|_{H^1(\widehat\Omega^{\rm s}_H)}^2 .
\end{align}

By using decomposition ${\bs u}_h={\bs u}_h^*+{\bs e}_{\bs u}$ in \eqref{err-c-dte}, as well as the estimates $\|{\bs u}_h^*\|_{W^{1,\infty}(\Omega^{\rm f}_{\rm ext}[{\bs X}_H])}\lesssim 1$ and $\|{\bs e}_{\bs u}\|_{L^\infty(\Omega^{\rm f}_{\rm ext}[{\bs X}_H])}\lesssim 1$ from \eqref{ind-1-2}, we have 
\begin{align}\label{dte-estimates-cdf}
| \eqref{err-c-dte} + \eqref{err-d-dte} | 
&\le \varepsilon \int_{\Omega^{\rm f}[\bs X_H]} |\partial_t {\bs e}_{\bs u} |^2 + C_{\varepsilon} \|{\bs e}_{\bs u}\|_{H^1(\Omega^{\rm f}_{\rm ext}[{\bs X}_H])}^2 .
\end{align}

\noindent{\bf Part 3: Estimates for (\ref{err-h-dte})+(\ref{err-k-dte})+(\ref{err-l-dte}).}\\
By decomposing ${\bs\sigma}(\bu_h^*,p_h^*)$ into ${\bs\sigma}(\bu_h^*,p_h^*)- {\bs\sigma}(\bu,p) + {\bs\sigma}(\bu,p)$ and applying integration by parts, we have 
\begin{subequations}\label{err-h-dte-1} 
\begin{align}
\eqref{err-h-dte} 
&= \int_{\Omega^{\rm f}[\bs X_H]} ({\bs\sigma}(\bu_h^*,p_h^*) - {\bs\sigma}(\bu,p)): \grad \partial_t {\bs e}_{\bs u} \label{err-h-dte-a} \\
&\quad\, - \int_{\Omega^{\rm f}[\bs X_H^*]} ({\bs\sigma}(\bu_h^*,p_h^*) - {\bs\sigma}(\bu,p)): \grad \partial_t {\bs e}_{\bs u} \label{err-h-dte-b}\\ 
&\quad\, + \int_{\Omega^{\rm f}[\bs X_H^*]} \hspace{-5pt}\grad\cdot{\bs\sigma}(\bu,p)\cdot \partial_t {\bs e}_{\bs u} - \int_{\Omega^{\rm f}[\bs X_H]}\hspace{-5pt} \grad\cdot {\bs\sigma}(\bu,p)\cdot \partial_t {\bs e}_{\bs u} \label{err-h-dte-c}\\
&\quad\, + \int_{\Gamma[\bs X_H]} {\bs\sigma}(\bu,p){\bs n}\cdot \partial_t {\bs e}_{\bs u} - \int_{\Gamma[\bs X_H^*]} {\bs\sigma}(\bu,p){\bs n}\cdot \partial_t {\bs e}_{\bs u}\label{err-h-dte-d} , 
\end{align}
\end{subequations}
where $\eqref{err-h-dte-a}$ and $\eqref{err-h-dte-b}$ can be estimated by using the transport formula in \eqref{dt-int-Omega-1}, the interpolation error estimates in \eqref{interpl-u} and the finite element inverse estimate. For example, 
\begin{align}
\Big|\int_0^s \eqref{err-h-dte-a} \d t \Big|
&= \int_{\Omega^{\rm f}[\bs X_H]} ({\bs\sigma}(\bu_h^*, p_h^*) - {\bs\sigma}(\bu,p)): \grad {\bs e}_{\bs u} \Big|_{t=0}^{t=s} \notag\\
&\quad\, 
-\int_0^s \int_{\Omega^{\rm f}[\bs X_H]} ({\bs\sigma}(\partial_t \bu_h^*,\partial_t p_h^*) - {\bs\sigma}(\partial_t \bu,\partial_t p)): \grad {\bs e}_{\bs u} \d t\notag\\
&\quad\,
- \int_0^s \int_{\Gamma[\bs X_H]} ({\bs\sigma}(\bu_h^*, p_h^*) - {\bs\sigma}(\bu,p)): \grad {\bs e}_{\bs u} ({\bs\dot{\bs d}}_H \cdot {\bs n}_{\Gamma[\bs X_H]}) \d t \notag\\
&\lesssim h^k\|{\bs e}_{\bs u}(s)\|_{H^1(\Omega^{\rm f}_{\rm ext}[\bs X_H])}  
+ h^k\|{\bs e}_{\bs u}(0)\|_{H^1(\Omega^{\rm f}_{\rm ext}[\bs X_H])}  \notag\\
&\quad\, 
+ \int_0^s h^k (\|{\bs e}_{\bs u}\|_{H^1(\Omega^{\rm f}_{\rm ext}[\bs X_H])} + \| \grad {\bs e}_{\bs u}\|_{L^2(\Gamma[\bs X_H])}) \d t \notag\\
&\le \varepsilon
 \|{\bs e}_{\bs u}(s)\|_{H^1(\Omega^{\rm f}_{\rm ext}[{\bs X}_H(\cdot,s)])}^2 
+ C_{\varepsilon}\int_0^s 
\|{\bs e}_{\bs u}\|_{H^1(\Omega^{\rm f}_{\rm ext}[{\bs X}_H])}^2 \d t 
+ C_{\varepsilon}h^{2k-1}, 
\end{align}
where we have used the result $\|{\bs e}_{\bs u}(0)\|_{H^1(\Omega^{\rm f}_{\rm ext}[\bs X_H])} \lesssim h^k$. The estimate for $\int_0^s \eqref{err-h-dte-b} \d t$ is similar (details are omitted), i.e., 
\begin{align}
\Big|\int_0^s \eqref{err-h-dte-b} \d t \Big|
\le \varepsilon
 \|{\bs e}_{\bs u}(s)\|_{H^1(\Omega^{\rm f}_{\rm ext}[{\bs X}_H(\cdot,s)])}^2 
+ C_{\varepsilon}\int_0^s 
\|{\bs e}_{\bs u}\|_{H^1(\Omega^{\rm f}_{\rm ext}[{\bs X}_H])}^2 \d t 
+ C_{\varepsilon}h^{2k-1}. 
\end{align}
Moreover, \eqref{err-h-dte-c} can be estimated in the same way as \eqref{err-b-dte}, i.e.,
\begin{align}
\Big|\int_0^s \eqref{err-h-dte-c} \d t\Big|
&\le \varepsilon
\big( \|{\bs e}_{\bs u}(s)\|_{H^1(\Omega^{\rm f}_{\rm ext}[{\bs X}_H(\cdot,s)])}^2 + \|\grad{\bs e}_{\bs d}(s)\|_{L^2(\widehat\Omega^{\rm s}_H)}^2 \big) \notag\\
&\quad\,
+ C_\varepsilon\int_0^s 
\|{\bs e}_{\bs u}\|_{H^1(\Omega^{\rm f}_{\rm ext}[{\bs X}_H])}^2 \d t\notag\\
&\quad\,
+ C_\varepsilon\int_0^s 
\big[ (h^{-1}+H^{-1})\|{\bs e}_{\bs d}\|_{H^1(\widehat\Omega^{\rm s}_H)}^2 
+ H^{-1}\|\partial_t{\bs e}_{\bs d}\|_{L^2(\widehat\Omega^{\rm s}_H)}^2 \big]\d t .
\end{align}
Furthermore, \eqref{err-h-dte-d} can be estimated in combination with \eqref{err-k-dte} and \eqref{err-l-dte}, i.e., 
\begin{subequations}\label{err-hjk-dte}
\begin{align}
&\int_0^s [ \eqref{err-h-dte-d} + \eqref{err-k-dte} + \eqref{err-l-dte} ] \d t \notag\\
&= \int_0^s  \int_{\Gamma[\bs X_H]} ({\bs\sigma}(\bu,p){\bs n} - {\widehat{\bs\sigma}}_H^*\circ{\bs X}_H^{-1}) \cdot \partial_t {\bs e}_{\bs u} \d t\label{err-hjk-dte-a}\\
&\quad\, 
- \int_0^s \int_{\Gamma[\bs X_H^*]} ({\bs\sigma}(\bu,p){\bs n} - {\widehat{\bs \sigma} }_H^* \circ({\bs X}_H^*)^{-1}) \cdot \partial_t {\bs e}_{\bs u} \d t\label{err-hjk-dte-b}\\
&\quad\,  
+ \int_0^s \Big( \int_{\Gamma[\bs X_H]} {\widehat{\bs\sigma}}_H^*\circ{\bs X}_H^{-1} \cdot \partial_{tt}{\bs e}_\dep \circ{\bs X}_H^{-1} - \int_{\Gamma[\bs X_H^*]} {\widehat{\bs \sigma} }_H^* \circ({\bs X}_H^*)^{-1} \cdot \partial_{tt}{\bs e}_\dep \circ({\bs X}_H^*)^{-1} \Big)\d t .\label{err-hjk-dte-c}
\end{align}
\end{subequations}
Since ${\widehat{\bs\sigma}}_H^*= {\bs I}_H^{\rm s}[ ({\bs\sigma}(\bu,p){\bs n})|_{\Gamma(t)} \circ{\bs X}\circ\widehat{\bs\Phi}_H] $, it follows that 
\begin{align*} 
&[({\bs\sigma}(\bu,p){\bs n})|_{\Gamma[{\bs X}_H]} - {\widehat{\bs\sigma}}_H^*\circ{\bs X}_H^{-1} ] \circ{\bs X}_H\circ \widehat{\bs\Phi}_H^{-1} \\ 
&= ({\bs\sigma}(\bu,p){\bs n})|_{\Gamma[{\bs X}_H]}\circ{\bs X}_H\circ \widehat{\bs\Phi}_H^{-1} - ({\bs\sigma}(\bu,p){\bs n})|_{\Gamma(t)} \circ{\bs X} + O(H^{k+1})  \\
&= \int_0^1 \frac{\d}{\d\theta}\Big[({\bs\sigma}(\bu,p){\bs n})|_{\Gamma[{\bs X}_H^{\#,\theta}]}\circ{\bs X}_H^{\#,\theta} \Big] \d\theta + O(H^{k+1}) \\ 
&= \int_0^1 \Big[ ({\bs e}_{\bs X}^{\#,\theta}\cdot\grad{\bs\sigma}(\bu,p)){\bs n}|_{\Gamma[{\bs X}_H^{\#,\theta}]}
- {\bs\sigma}(\bu,p)(\grad_{\Gamma[{\bs X}_H^{\#,\theta}]}{\bs e}_{\bs X}^{\#,\theta}) {\bs n}_{\Gamma[{\bs X}_H^{\#,\theta}]}\Big]\circ{\bs X}_H^{\#,\theta} \d\theta + O(H^{k+1}) \\
&= \int_0^1 \Big[ ({\bs e}_{\bs d}^{\#,\theta}\cdot\grad{\bs\sigma}(\bu,p)){\bs n}|_{\Gamma[{\bs X}_H^{\#,\theta}]}
- {\bs\sigma}(\bu,p)(\grad_{\Gamma[{\bs X}_H^{\#,\theta}]}{\bs e}_{\bs d}^{\#,\theta}) {\bs n}_{\Gamma[{\bs X}_H^{\#,\theta}]} \Big]\circ{\bs X}_H^{\#,\theta} \d\theta + O(H^{k}) ,
\end{align*} 
and therefore, using the norm equivalence relation $\|{\bs e}_{\bs d}^{\#,\theta}\|_{H^1(\Gamma[{\bs X}_H^{\#,\theta}])}\simeq \|{\bs e}_{\bs d}\|_{H^1(\widehat\Gamma_H)}$, 
\begin{align*}
|\eqref{err-hjk-dte-a}|
&\lesssim \int_0^s (\|{\bs e}_{\bs d}\|_{H^1(\widehat\Gamma_H)} + H^{k})\|\partial_t{\bs e}_{\bs u}\|_{L^2(\Gamma[{\bs X}_H])} \d t\notag\\ 
&\le \int_0^s \big( H^{-\frac12}\|{\bs e}_{\bs d}\|_{H^1(\widehat\Omega^{\rm s}_H)} + H^{k} \big) h^{-\frac12} \|\partial_t{\bs e}_{\bs u}\|_{L^2(\Omega^{\rm f}[{\bs X}_H])} \d t
\quad\mbox{(inverse estimate)}\notag\\
&\le 
\varepsilon \int_0^s\|\partial_t{\bs e}_{\bs u}\|_{L^2(\Omega^{\rm f}[{\bs X}_H])}^2 \d t
+ C_{\varepsilon} \int_0^sh^{-1}H^{-1}\|{\bs e}_{\bs d}\|_{H^1(\widehat\Omega^{\rm s}_H)}^2 \d t + C_{\varepsilon}  h^{-1} H^{2k} .
\end{align*}
The estimate for \eqref{err-hjk-dte-b} can be obtained via integration by parts in time after using decomposition $\partial_t {\bs e}_{\bs u}=\partial_t^\bullet {\bs e}_{\bs u}-(\vel_H^*\circ({\bs X}_H^*)^{-1})\cdot\grad {\bs e}_{\bs u}$, with $\partial_t^\bullet {\bs e}_{\bs u}$ denoting the material derivative with respect to the velocity of surface $\Gamma[{\bs X}_H^*]$, i.e.,  
\begin{align*}
|\eqref{err-hjk-dte-b}| 
&= \bigg| - \int_{\Gamma[\bs X_H^*]} ({\bs\sigma}(\bu,p){\bs n} - {\widehat{\bs \sigma} }_H^* \circ({\bs X}_H^*)^{-1}) \cdot {\bs e}_{\bs u} \Big|_{t=0}^{t=s} \notag\\
&\quad\, 
+ \int_0^s \int_{\Gamma[\bs X_H^*]} \partial_t^\bullet({\bs\sigma}(\bu,p){\bs n} - {\widehat{\bs \sigma} }_H^* \circ({\bs X}_H^*)^{-1}) \cdot {\bs e}_{\bs u}  \d t \notag\\
&\quad\, 
+ \int_0^s \int_{\Gamma[\bs X_H^*]} ({\bs\sigma}(\bu,p){\bs n} - {\widehat{\bs \sigma} }_H^* \circ({\bs X}_H^*)^{-1}) \cdot [(\vel_H^*\circ({\bs X}_H^*)^{-1})\cdot\grad {\bs e}_{\bs u}] \d t \bigg| \notag\\
&\lesssim H^{k+1}(\|{\bs e}_{\bs u}(s)\|_{L^2(\Gamma[\bs X_H^*(\cdot,s)])} 
+ \|{\bs e}_{\bs u}(0)\|_{L^2(\Gamma[\bs X_H^*(\cdot,0)])}) + H^{k+1}\int_0^s\|{\bs e}_{\bs u}\|_{H^1(\Gamma[\bs X_H^*])} \d t \notag\\
&\lesssim H^{k+1}(\|{\bs e}_{\bs u}(s)\|_{H^1(\Omega^{\rm f}_{\rm ext}[\bs X_H(\cdot,s)])} 
+ \|{\bs e}_{\bs u}(0)\|_{H^1(\Omega^{\rm f}_{\rm ext}[\bs X_H(\cdot,0)])}) \notag\\
&\quad\, + h^{-\frac12}H^{k+1}\int_0^s\|{\bs e}_{\bs u}\|_{H^1(\Omega^{\rm f}_{\rm ext}[\bs X_H])} \d t \notag\\
&\le \varepsilon \|{\bs e}_{\bs u}(s)\|_{H^1(\Omega^{\rm f}_{\rm ext}[\bs X_H(\cdot,s)])}^2 
+ C_\varepsilon \int_0^s\|{\bs e}_{\bs u}\|_{H^1(\Omega^{\rm f}_{\rm ext}[\bs X_H])}^2 \d t 
+ C_\varepsilon (h^{-1}H^{2k+2} + H^{2k+1}).
\end{align*}
Moreover, \eqref{err-hjk-dte-c} can be estimated by using the transport formula in \eqref{dt-int-Gamma-1} and the norm equivalence relation which converts norms on $\Gamma[{\bs X}_H^\theta]$ to norms on $\widehat\Gamma_H$, i.e.,
\begin{align}\label{err-hjk-dte-cc}
|\eqref{err-hjk-dte-c}|
&\lesssim \int_0^s \|{\bs e}_{\bs d}\|_{H^1(\widehat\Gamma_H)}\|\partial_{tt} {\bs e}_\dep\|_{L^2(\widehat\Gamma_H)} \d t \notag\\
&\lesssim \int_0^s H^{-1}\|{\bs e}_{\bs d}\|_{H^1(\widehat\Omega^{\rm s}_H)}\|\partial_{tt} {\bs e}_\dep\|_{L^2(\widehat\Omega^{\rm s}_H)} \d t 
\quad\mbox{(inverse estimate is used)}\notag\\
&\le \varepsilon\int_0^s\|\partial_{tt} {\bs e}_\dep\|_{L^2(\widehat\Omega_H^{\rm s})}^2 \d t
+ C_\varepsilon \int_0^s H^{-2}\|{\bs e}_{\bs d}\|_{H^1(\widehat\Omega^{\rm s}_H)}^2 \d t.
\end{align}
By substituting \eqref{err-hjk-dte}--\eqref{err-hjk-dte-cc} into \eqref{err-h-dte-1}, we obtain 
\begin{align}
&\Big| \int_0^s  [\eqref{err-h-dte} + \eqref{err-k-dte} + \eqref{err-l-dte}] \d t \Big| \notag\\
&\le \varepsilon
\big( \|{\bs e}_{\bs u}(s)\|_{H^1(\Omega^{\rm f}_{\rm ext}[{\bs X}_H(\cdot,s)])}^2 + \|\grad{\bs e}_{\bs d}(s)\|_{L^2(\widehat\Omega^{\rm s}_H)}^2 \big) \notag\\
&\quad\, 
+ \varepsilon \int_0^s \big(  \|\partial_t{\bs e}_{\bs u}\|_{L^2(\Omega^{\rm f}[{\bs X}_H])}^2 
+ \|\partial_{tt} {\bs e}_\dep\|_{L^2(\widehat\Omega^{\rm s}_H)}^2 \big) \d t \notag\\
&\quad\, 
+ C_\varepsilon \int_0^s \big((h^{-1}H^{-1}+H^{-2}) \|{\bs e}_{\bs d}\|_{H^1(\widehat\Omega^{\rm s}_H)}^2  
+ H^{-1} \|\partial_t{\bs e}_{\bs d}\|_{L^2(\widehat\Omega^{\rm s}_H)}^2 + \|{\bs e}_{\bs u}\|_{H^1(\Omega^{\rm f}_{\rm ext}[{\bs X}_H])}^2  \big) \d t \notag\\
&\quad\, + C_\varepsilon (h^{2k-1} + h^{-1}H^{2k}) . 
\end{align}

\noindent{\bf Part 4: Estimates for (\ref{err-g-dte}).}\\
We estimate (\ref{err-g-dte}) by choosing ${\bs v}_h = 0$, ${\bs w}_H= 0$ and ${\bs\lambda}_H=0$ in \eqref{error-eq}. This leads to
\begin{align}\label{eq-div-eu}
&\int_{\Omega^{\rm f}[\bs X_H]}  q_h \grad\cdot {\bs e}_{\bs u}  
+ \Big( \int_{\Omega^{\rm f}[\bs X_H]} - \int_{\Omega^{\rm f}[\bs X_H^*]} \Big) q_h \grad\cdot  \bu_h^* 
+ g_h^{\rm p} (e_p, q_h) \notag\\
&= 
- R_{h,H}^*(0, q_h, 0, 0) \notag\\
&\lesssim (h^k+H^k) \|q_h\|_{0,\Omega[{\bs X}_H]} \quad\mbox{(here Lemma \ref{Lemma:RhH} is used)}. 
\end{align}
Taking time derivative of the first equality in \eqref{eq-div-eu} yields the following relation: 
\begin{subequations}\label{dteq-div-eu}
\begin{align}
&\int_{\Omega^{\rm f}[\bs X_H]} q_h \grad\cdot \partial_t{\bs e}_{\bs u}  \notag\\
&\quad\, 
+ \int_{\Gamma[\bs X_H]} q_h (\grad\cdot {\bs e}_{\bs u}) (\partial_t\dep_H\circ \bs X_H^{-1})\cdot{\bs n}_{\Gamma[\bs X_H]} \label{dteq-div-eu-a}\\
&\quad\, 
+ \int_{\Omega^{\rm f}[\bs X_H]} q_h \grad\cdot \partial_t\bu_h^* 
- \int_{\Omega^{\rm f}[\bs X_H^*]} q_h \grad\cdot \partial_t\bu_h^* \label{dteq-div-eu-b}\\
&\quad\, 
+ \int_{\Gamma[\bs X_H]} q_h (\grad\cdot \bu_h^*) (\partial_t\dep_H\circ \bs X_H^{-1})\cdot{\bs n}_{\Gamma[\bs X_H]} \notag\\
&\quad\, 
- \int_{\Gamma[\bs X_H^*]} q_h (\grad\cdot \bu_h^*) (\partial_t\dep_H^*\circ ({\bs X}_H^*)^{-1})\cdot{\bs n}_{\Gamma[\bs X_H^*]} \label{dteq-div-eu-c}\\
&\quad\, 
+ g_h^{\rm p} (\partial_t e_p, q_h) = - \Big(\frac{\d}{\d t} R_{h,H}^*\Big) (0, q_h, 0, 0) . 
\end{align}
\end{subequations}
Since $|\grad\cdot{\bs u}_h^*| \le |\grad\cdot({\bs u}_h^* - {\bs u} ) | + |\grad\cdot {\bs u}| =  O(h^k)$ on $\Gamma[{\bs X}_H]$, as shown in \eqref{interpl-u} and \eqref{div-u-small}, and similarly, $\grad\cdot\partial_t{\bs u}_h^* =  O(h^k)$ on $\Gamma[{\bs X}_H]$, it follows that (by using $\|\partial_t{\bs d}_H\|_{L^\infty(\widehat\Omega^{\rm s}_H)}\lesssim 1$, as shown in \eqref{ind-1-2}, and the inverse trace inequality) 
\begin{align*}
|\eqref{dteq-div-eu-a}| 
&\le Ch^{-1} \|q_h\|_{L^2(\Omega^{\rm f}_{\rm ext}[{\bs X}_H])} \|\grad{\bs e}_{\bs u}\|_{L^2(\Omega^{\rm f}_{\rm ext}[{\bs X}_H])} , \\
|\eqref{dteq-div-eu-b}| + |\eqref{dteq-div-eu-c}| 
&\le Ch^{k-1} \|q_h\|_{L^2(\Omega^{\rm f}_{\rm ext}[{\bs X}_H])} . 
\end{align*}
Moreover, by differentiating the expression of $R_{h,H}^*(\bs 0, q_h, \bs 0,\bs 0)$ in \eqref{def-RhH}, we have 
\begin{align}
\Big(\frac{\d}{\d t} R_{h,H}^*\Big) (\bs 0, q_h, \bs 0, \bs 0) 
= 
\,& \frac{\d}{\d t}\int_{\Omega^{\rm f}[\bs X_H^*]}  q_h \grad\cdot \bu_h^* -  \frac{\d}{\d t}\int_{\Omega^{\rm f}(t)}  q_h \grad\cdot \bu \notag\\
& + g_h^{\rm p} (\partial_t p_h^*,q_h ) \notag\\
= 
\,& \int_{\Omega^{\rm f}[\bs X_H^*]}  q_h \grad\cdot \partial_t\bu_h^* -  \int_{\Omega^{\rm f}(t)}  q_h \grad\cdot \partial_t\bu  \notag\\
\,& +\int_{\Gamma[\bs X_H^*]}  q_h (\grad\cdot \bu_h^*) (\vel_H^*\circ ({\bs X}_H^*)^{-1}) \cdot {\bs n}_{\Gamma[\bs X_H^*]} \notag\\
& - \int_{\Gamma(t)} q_h (\grad\cdot \bu) (\vel\circ {\bs X}^{-1}) \cdot {\bs n}_{\Gamma(t)} \notag\\
& + g_h^{\rm p} (\partial_t p_h^*,q_h ) . \label{dt-R-qh-c}
\end{align}
By noting that $\grad\cdot{\bs u}_h^* = \grad\cdot({\bs u}_h^*-{\bs u}) + \grad\cdot {\bs u}$ on $\Omega^{\rm f}[{\bs X}_H]$, with $|\grad\cdot({\bs u}_h^*-{\bs u})|=O(h^k)$ and $|\grad\cdot{\bs u}|=O(h^m)$ on $\Omega^{\rm f}[{\bs X}_H]$ for arbitrarily large $m$, as shown in \eqref{div-u-small-2}, it is straightforward to show the following result:   
\begin{align}
\Big| \Big(\frac{\d}{\d t} R_{h,H}^*\Big) (\bs 0, q_h, \bs 0, \bs 0) \Big|
\lesssim h^{k-1}\|q_h\|_{L^2(\Omega^{\rm f}_{\rm ext}[{\bs X}_H])} 
+ h^kg^{\rm p}_h(q_h,q_h)^{\frac12} . 
\end{align}
This proves that, by choosing $q_h=e_p$ in these estimates, 
\begin{align}\label{dteq-div-eu-estimates}
\eqref{err-g-dte}
\ge g_h^{\rm p} (\partial_t e_p, e_p) 
- \varepsilon \|e_p\|_{0,\Omega[{\bs X}_H]}^2 
- C_{\varepsilon}h^{2k-2} 
- C_{\varepsilon}h^{-2}\|\grad{\bs e}_{\bs u}\|_{L^2(\Omega^{\rm f}_{\rm ext}[{\bs X}_H])}^2 . 
\end{align}

\noindent{\bf Part 5: Estimates for (\ref{err-j-dte}).}\\
Similarly, we estimate \eqref{err-j-dte} by choosing  ${\bs v}_h = \bs 0$, $q_h=0$, ${\bs w}_H= \bs 0$ in \eqref{error-eq}, and use the abbreviations ${\bs\lambda}_H^\theta={\bs\lambda}_H\circ({\bs X}_H^\theta)^{-1}$, $\vel_H^{\,|\theta}=\partial_{t}{\bs d}_H\circ({\bs X}_H^\theta)^{-1}$, $\vel_H^{*|\theta}=\partial_{t}{\bs d}_H^*\circ({\bs X}_H^\theta)^{-1}$ and ${\bs\ddot{\bs d}}_H^{*|\theta}=\partial_{tt}{\bs d}_H^*\circ({\bs X}_H^\theta)^{-1}$. This leads to, using the notations in \eqref{notation-lambda},
\begin{align}\label{eq-eu-ded}
&\int_{\Gamma[\bs X_H]} \bs \lambda_H^1 \cdot ( {\bs e}_{\bs u} - {\bs\dot{\bs e}}_\dep^1 ) 
+ \int_{\Gamma[\bs X_H]} \bs \lambda_H^1  \cdot \big( \bu_h^*  - \vel_H^{*|1} \big) - \int_{\Gamma[\bs X_H^*]} \bs \lambda_H^0 \cdot \big( \bu_h^* - \vel_H^{*|0} \big) \notag\\
&= - R_{h,H}^*(0, 0, 0, {\bs\lambda}_H) .
\end{align}
Taking time derivative of \eqref{eq-eu-ded} yields the following relation: 
\begin{subequations}\label{dteq-eu-ded}
\begin{align}
&\int_{\Gamma[\bs X_H]} {\bs \lambda}_H^1 \cdot (\partial_t{\bs e}_{\bs u} -  {\bs\ddot{\bs e}}_\dep^1) \notag\\
&\quad\, 
+ \int_{\Gamma[\bs X_H]} {\bs \lambda}_H^1 \cdot [\vel_H^{\,|1}\cdot\grad {\bs e}_{\bs u} + ( {\bs e}_{\bs u} - {\bs\dot{\bs e}}_\dep^1 ) \grad_{\Gamma[\bs X_H]}\cdot \vel_H^{\,|1} ] \label{dteq-eu-ded-a}\\ 
&\quad\,  + \int_{\Gamma[\bs X_H]} {\bs \lambda}_H^1 \cdot 
[ {\bs\dot{\bs e}}_{\bs d}^{1}\cdot\grad {\bs u}_h^* + ({\bs u}_h^*  - \vel_H^{*|1} ) \grad_{\Gamma[\bs X_H]}\cdot {\bs\dot{\bs e}}_{\bs d}^{1}] \label{dteq-eu-ded-b}\\
&\quad\,  + \int_{\Gamma[\bs X_H]} {\bs \lambda}_H^1 \cdot 
[ (\partial_t{\bs u}_h^* + \vel_H^{*|1}\cdot\grad {\bs u}_h^* - {\bs\ddot\dep}_H^{*|1}) + ({\bs u}_h^*  - \vel_H^{*|1} ) \grad_{\Gamma[\bs X_H]}\cdot \vel_H^{*|1}] \label{dteq-eu-ded-c}\\ &\quad\, 
- \int_{\Gamma[\bs X_H^*]} {\bs \lambda}_H^0 \cdot [ (\partial_t{\bs u}_h^* + \vel_H^{*|0}\cdot\grad {\bs u}_h^* - {\bs\ddot\dep}_H^{*|0} ) + ({\bs u}_h^*  - \vel_H^{*|0}) \grad_{\Gamma[\bs X_H^*]}\cdot \vel_H^{*|0}] \label{dteq-eu-ded-d} \\
&= - \Big(\frac{\d}{\d t}\Big) R_{h,H}^*(\bs 0, 0, \bs 0, {\bs\lambda}_H) . \notag
\end{align}
\end{subequations}
Thus \eqref{err-j-dte}=$- (\frac{\d}{\d t} R_{h,H}^*)(\bs 0, 0, \bs 0, {\bs e}_{{\bs \sigma} }) $, which we estimate below by differentiating the expression of $R_{h,H}^*(\bs 0, 0, \bs 0, {\bs\lambda}_H)$ in \eqref{def-RhH}: 
\begin{subequations}\label{dt-RhH-lambda}
\begin{align}
&\Big(\frac{\d}{\d t} R_{h,H}^*\Big)(\bs 0, 0, \bs 0, {\bs\lambda}_H) \notag\\
= 
& \int_{\Gamma[\bs X_H^*]} {\bs\lambda}_H \circ ({\bs X}_H^*)^{-1} \cdot \big[ \partial_t\big(\bu_h^* \circ {\bs X}_H^*) \circ ({\bs X}_H^*)^{-1}- {\bs\ddot{\bs d}}_H^*  \circ ({\bs X}_H^*)^{-1} \big] \label{dt-RhH-lambda-a}\\
& + \int_{\Gamma[\bs X_H^*]} {\bs\lambda}_H \circ ({\bs X}_H^*)^{-1} \cdot \big( \bu_h^* - {\bs\dot{\bs d}}_H^* \circ ({\bs X}_H^*)^{-1} \big)\grad_{\Gamma[\bs X_H^*]}\cdot ( \vel_H^*\circ ({\bs X}_H^*)^{-1} )
\label{dt-RhH-lambda-b}\\
&- \int_{\Gamma(t)} \bs \lambda_H \circ \widehat{\bs\Phi}_H^{-1}  \circ {\bs X}^{-1}  \cdot \big[ \partial_t( \bu\circ{\bs X})\circ{\bs X}^{-1} -{\bs\ddot{\bs d}}\circ {\bs X}^{-1} \big) \label{dt-RhH-lambda-c}\\
&- \int_{\Gamma(t)} \bs \lambda_H \circ \widehat{\bs\Phi}_H^{-1}  \circ {\bs X}^{-1}  \cdot \big( \bu - \vel  \circ {\bs X}^{-1} \big)\grad_{\Gamma(t)}\cdot ( \vel\circ {\bs X}^{-1} ) , \label{dt-RhH-lambda-d}
\end{align}
\end{subequations} 
where \eqref{dt-RhH-lambda-d}=0 because $\bu - \vel  \circ {\bs X}^{-1}=0$ on $\Gamma(t)$. Since $\bu_h^* - {\bs\dot{\bs d}}_H^* \circ ({\bs X}_H^*)^{-1} $ differs from $\bu - {\bs\dot{\bs d}}\circ\widehat{\bs\Phi}_H \circ ({\bs X}_H^*)^{-1} $ by $O(h^{k+1}+H^{k+1})$ on $\Gamma[{\bs X}_H^*]$, and $\bu - {\bs\dot{\bs d}}\circ\widehat{\bs\Phi}_H \circ ({\bs X}_H^*)^{-1} $ differs from $\bu \circ{\bs X}\circ\widehat{\bs\Phi}_H \circ ({\bs X}_H^*)^{-1}- {\bs\dot{\bs d}}\circ\widehat{\bs\Phi}_H \circ ({\bs X}_H^*)^{-1} \equiv 0$ by $O(H^{k+1})$, it follows that $|\bu_h^* - {\bs\dot{\bs d}}_H^* \circ ({\bs X}_H^*)^{-1} | = O(h^{k+1}+H^{k+1})$ on $\Gamma[{\bs X}_H^*]$ and therefore 
\begin{align}
|\eqref{dt-RhH-lambda-b} |
&\lesssim \|{\bs\lambda}_H\|_{L^2(\widehat\Gamma_H)} \|\bu_h^* - {\bs\dot{\bs d}}_H^* \circ ({\bs X}_H^*)^{-1} \|_{L^2(\widehat\Gamma_H)}\notag\\
&\lesssim \|{\bs\lambda}_H\|_{L^2(\widehat\Gamma_H)} (h^{k+1}+H^{k+1}) . 
\end{align}
\eqref{dt-RhH-lambda-a} + \eqref{dt-RhH-lambda-c} can be estimated as follows:
\begin{subequations}
\begin{align}
& \eqref{dt-RhH-lambda-a} + \eqref{dt-RhH-lambda-c} \notag\\
= 
& \int_{\Gamma[\bs X_H^*]} {\bs\lambda}_H \circ ({\bs X}_H^*)^{-1} \cdot \big[ \partial_t\bu_h^* + \vel_H^*\circ({\bs X}_H^*)^{-1} \cdot \grad\bu_h^*  -  \partial_t\bu -  \vel \circ\widehat{\bs\Phi}_H\circ({\bs X}_H^*)^{-1} \cdot \grad\bu  \big] \label{dt-RhH-ac-a}\\
& - \int_{\Gamma[\bs X_H^*]} {\bs\lambda}_H \circ ({\bs X}_H^*)^{-1} \cdot \big[ {\bs\ddot{\bs d}}_H^*  \circ ({\bs X}_H^*)^{-1} -{\bs\ddot{\bs d}} \circ\widehat{\bs\Phi}_H \circ ({\bs X}_H^*)^{-1} \big] \label{dt-RhH-ac-b}\\
& + \int_{\Gamma[\bs X_H^*]} {\bs\lambda}_H \circ ({\bs X}_H^*)^{-1} \cdot \big[ \partial_t\bu + \vel\circ\widehat{\bs\Phi}_H\circ({\bs X}_H^*)^{-1} \cdot \grad\bu - {\bs\ddot{\bs d}} \circ\widehat{\bs\Phi}_H \circ ({\bs X}_H^*)^{-1}  \big] \label{dt-RhH-ac-c}\\
& - \int_{\Gamma(t)} \bs \lambda_H \circ \widehat{\bs\Phi}_H^{-1}  \circ {\bs X}^{-1}  \cdot \big[ \partial_t \bu + ({\bs\dot{\bs d}}\circ{\bs X}^{-1})\cdot\grad\bu - {\bs\ddot{\bs d}}\circ {\bs X}^{-1} \big] ,\label{dt-RhH-ac-d}
\end{align}
\end{subequations}
where 
\begin{align}
|\eqref{dt-RhH-ac-a}+\eqref{dt-RhH-ac-b}| &\lesssim \|{\bs\lambda}_H\|_{L^2(\widehat\Gamma_H)}(h^k + H^{k+1} )&&\mbox{(interpolation error)} \notag\\
|\eqref{dt-RhH-ac-c}+\eqref{dt-RhH-ac-d}| &\lesssim \|{\bs\lambda}_H\|_{L^2(\widehat\Gamma_H)}H^k  &&\mbox{(using formula \eqref{dt-int-Gamma-2})} . 
\end{align}
This shows that 
\begin{align}
|\eqref{dt-RhH-lambda-a} + \eqref{dt-RhH-lambda-c}| &\lesssim \|{\bs\lambda}_H\|_{L^2(\widehat\Gamma_H)}(h^k + H^{k} )
\end{align}
and therefore 
$$
\Big|\Big(\frac{\d}{\d t} R_{h,H}^*\Big)(\bs 0, 0, \bs 0, {\bs\lambda}_H)\Big|
\lesssim (h^k+H^k) \|{\bs\lambda}_H\|_{L^2(\widehat\Gamma_H)} 
\lesssim (h^k+H^k)H^{-\frac12} \|{\bs\lambda}_H\|_{H^{-\frac12}(\widehat\Gamma_H)}. 
$$ 
Since 
\begin{align*} 
\|\vel_H^{\,|1}\|_{L^\infty(\Gamma[\bs X_H])} 
&\lesssim \|\vel_H\|_{L^\infty(\widehat\Gamma_H)} \notag\\ 
&\lesssim \|\vel_H^*\|_{L^\infty(\widehat\Gamma_H)} 
+ \|{\bs\dot{\bs e}}_{\bs d}\|_{L^\infty(\widehat\Gamma_H)} &&\mbox{(norm equivalence in \eqref{norm-equiv-GH})}\\ 
&\lesssim 1 
&&\mbox{(here \eqref{ind-1-2} is used)} ,\\ 
\|\grad_{\Gamma[\bs X_H]} \vel_H^{\,|1}\|_{L^2(\Gamma[\bs X_H])} 
&\lesssim \|\vel_H\|_{H^1(\widehat\Gamma_H)} \notag\\ 
&\lesssim \|\vel_H^*\|_{H^1(\widehat\Gamma_H)} 
+ \|{\bs\dot{\bs e}}_{\bs d}\|_{H^1(\widehat\Gamma_H)} 
&&\mbox{(norm equivalence in \eqref{norm-equiv-GH})}\\ 
&\lesssim 1 
+ H^{-\frac32}\|{\bs\dot{\bs e}}_{\bs d}\|_{L^2(\widehat\Omega_H^{\rm s})} &&\mbox{(finite element inverse estimate)} \\ 
&\lesssim 1 &&\mbox{(here \eqref{ind-1-3} is used)} , 
\end{align*} 
by using \eqref{eu-dot-ed-Lp-H} with $p=\infty$, as well as the finite element inverse estimate $\|{\bs \lambda}_H\|_{L^2(\widehat\Gamma_H)}\lesssim H^{-\frac12}\|{\bs \lambda}_H\|_{H^{-\frac12}(\widehat\Gamma_H)}$ (see Lemma \ref{Lemma:inverse} in Appendix \ref{appendix:inverse}), we obtain
\begin{align} 
|\eqref{dteq-eu-ded-a}|
&\le C\|{\bs \lambda}_H\|_{L^2(\widehat\Gamma_H)} \|\grad {\bs e}_{\bs u} \|_{L^2(\Gamma[\bs X_H])} 
+ C\|{\bs \lambda}_H\|_{L^2(\widehat\Gamma_H)} \|{\bs e}_{\bs u} - {\bs\dot{\bs e}}_\dep  \|_{L^\infty(\Gamma[\bs X_H])} \notag\\
&\le CH^{-\frac12}\|{\bs \lambda}_H\|_{H^{-\frac12}(\widehat\Gamma_H)} 
(h^{-\frac12}\|{\bs e}_{\bs u} \|_{H^1(\Omega^{\rm f}_{\rm ext}[\bs X_H])} + \|{\bs e}_\dep  \|_{L^\infty(\widehat{\Omega}^{\rm s}_H)} +h^{k+1}+H^{k+1} )  \notag\\
&\quad\,+ CH^{-\frac12}\|{\bs \lambda}_H\|_{H^{-\frac12}(\widehat\Gamma_H)} \|{\bs e}_{\bs u} - {\bf P}_{\Gamma[\bs X_H]}{\bs e}_{\bs u} \|_{L^\infty(\Gamma[\bs X_H])}  \notag\\
&\le CH^{-\frac12}\|{\bs \lambda}_H\|_{H^{-\frac12}(\widehat\Gamma_H)} 
(h^{-\frac12}\|{\bs e}_{\bs u} \|_{H^1(\Omega^{\rm f}_{\rm ext}[\bs X_H])} + H^{-\frac12}\|{\bs e}_\dep  \|_{H^1(\widehat{\Omega}^{\rm s}_H)} +h^{k+1}+H^{k+1} ) \notag\\
&\le \varepsilon \|{\bs \lambda}_H\|_{H^{-\frac12}(\widehat\Gamma_H)}^2
+ C_{\varepsilon}( h^{-1}H^{-1} \| {\bs e}_{\bs u}  \|_{H^1(\Omega^{\rm f}_{\rm ext}[{\bs X}_H])}^2 + H^{-2}\|{\bs e}_\dep  \|_{H^1(\widehat{\Omega}^{\rm s}_H)}^2) \notag\\
&\quad\, + C_{\varepsilon}H^{-1}(h^{2k+2}+H^{2k+2}) ,
\end{align}
where the second-to-last inequality uses the $L^\infty$ stability of the $L^2$ projection on $\Gamma[\bs X_H]$, i.e., $\|{\bf P}_{\Gamma[\bs X_H]}{\bs e}_{\bs u} \|_{L^\infty(\Gamma[\bs X_H])}\lesssim \|{\bs e}_{\bs u} \|_{L^\infty(\Gamma[\bs X_H])}$ (see \cite[Appendix B]{Li-Tang-2025}), and the finite element inverse estimates in dimension $d=2,3$: 
\begin{align*} 
\| {\bs e}_\dep  \|_{L^\infty(\widehat{\Omega}^{\rm s}_H)}
&\lesssim H^{-\frac12}\|{\bs e}_\dep  \|_{H^1(\widehat{\Omega}^{\rm s}_H)} \\
\| {\bs e}_{\bs u} \|_{L^\infty(\Gamma[\bs X_H])} 
&\lesssim 
\| {\bs e}_{\bs u} \|_{L^\infty(\Omega^{\rm f}_{\rm ext}[{\bs X}_H])}
\lesssim 
h^{-\frac12}\| {\bs e}_{\bs u} \|_{H^1(\Omega^{\rm f}_{\rm ext}[{\bs X}_H])} . 
\end{align*} 

Since ${\bs u}\circ{\bs X}={\bs\dot{\bs d}}$ on $\widehat\Gamma$ and $\vel_H^{*|1}=\partial_t{\bs d}_H^*\circ{\bs X}_H^{-1}$ on $\Gamma[{\bs X}_H]$, it follows that, for the function ${\bs u}_h^*  - \vel_H^{*|1}$ on $\Gamma[{\bs X}_H]$,
\begin{align}\label{u-d-on-Gamma1}
|{\bs u}_h^*  - \vel_H^{*|1}|
&\le |{\bs u}_h^*  - {\bs u}| + |{\bs u} - {\bs u}\circ{\bs X}\circ\widehat{\bs\Phi}_H\circ{\bs X}_H^{-1}| + |\partial_t{\bs d}\circ\widehat{\bs\Phi}_H\circ{\bs X}_H^{-1} - \partial_t{\bs d}_H^*\circ{\bs X}_H^{-1}| \notag\\
&\lesssim h^{k+1} + |{\bs e}_{\bs d}^1| + H^{k+1} 
\end{align}
and therefore
\begin{align}
|\eqref{dteq-eu-ded-b}|
&\le \|{\bs \lambda}_H^1\|_{L^2(\Gamma[{\bs X}_H])} 
( \|{\bs\dot{\bs e}}_{\bs d}^1\|_{L^2(\Gamma[{\bs X}_H])} 
+ (h^{k+1}+H^{k+1}+\|{\bs e}_{\bs d}^1\|_{L^\infty(\Gamma[{\bs X}_H])}) \|{\bs\dot{\bs e}}_{\bs d}^1\|_{H^1(\Gamma[{\bs X}_H])} ) \notag\\
&\le \|{\bs \lambda}_H\|_{L^2(\widehat\Gamma_H)} ( \|{\bs\dot{\bs e}}_{\bs d}\|_{L^2(\widehat\Gamma_H)} 
+ H \|{\bs\dot{\bs e}}_{\bs d}\|_{H^1(\widehat\Gamma_H)} ) 
\quad\mbox{(here \eqref{ind-1-1} and $h^{k+1}\lesssim H$ are used)} \notag\\
&\le H^{-1} \|{\bs \lambda}_H\|_{H^{-\frac12}(\widehat\Gamma_H)} \|{\bs\dot{\bs e}}_{\bs d}\|_{L^2(\widehat\Omega_H^{\rm s })} \quad\mbox{(finite element inverse estimate)} \notag\\
&\le \varepsilon \|{\bs \lambda}_H\|_{H^{-\frac12}(\widehat\Gamma_H)}^2 
+ C_\varepsilon H^{-2} \|{\bs\dot{\bs e}}_{\bs d}\|_{L^2(\widehat\Omega_H^{\rm s })}^2 . 
\end{align}
Moreover, 
\begin{subequations}\label{dteq-eu-abc}
\begin{align}
&|\eqref{dteq-eu-ded-c} + \eqref{dteq-eu-ded-d}| \notag\\
&\le \Big| \int_{\Gamma[\bs X_H]} {\bs \lambda}_H^1 \cdot 
 (\partial_t{\bs u} + \vel_H^{*|1}\cdot\grad {\bs u} - {\bs\ddot\dep}_H^{*|1}) - \int_{\Gamma[\bs X_H^*]} {\bs \lambda}_H^0 \cdot  (\partial_t{\bs u} + \vel_H^{*|0}\cdot\grad {\bs u} - {\bs\ddot\dep}_H^{*|0} ) \Big| \label{dteq-eu-bc-a}\\
&\quad\, 
+ \Big| \int_{\Gamma[\bs X_H]}  {\bs \lambda}_H^1 \cdot [ ({\bs u}_h^*  - \vel_H^{*|1} ) \grad_{\Gamma[\bs X_H]}\cdot \vel_H^{*|1}  ]
-\int_{\Gamma[\bs X_H^*]} {\bs\lambda}_H^0 \cdot [({\bs u}_h^*  - \vel_H^{*|0} ) \grad_{\Gamma[\bs X_H^*]}\cdot \vel_H^{*|0}] \Big| \label{dteq-eu-bc-b}\\
&\quad\,
+ \Big| \int_{\Gamma[\bs X_H]} {\bs \lambda}_H^1 \cdot 
[ (\partial_t{\bs u}_h^*- \partial_t{\bs u}) + \vel_H^{*|1}\cdot\grad ({\bs u}_h^*-{\bs u})] \Big| \label{dteq-eu-bc-c}\\
&\quad\,
+ \Big| \int_{\Gamma[\bs X_H^*]} {\bs \lambda}_H^0 \cdot 
[ (\partial_t{\bs u}_h^*- \partial_t{\bs u}) + \vel_H^{*|0}\cdot\grad ({\bs u}_h^*-{\bs u})] \Big| \label{dteq-eu-bc-d},
\end{align}
\end{subequations}
where 
\begin{align}
|\eqref{dteq-eu-bc-a}|
&=\Big|\int_0^1\d\theta\int_{\Gamma[\bs X_H^\theta]} {\bs \lambda}_H^\theta \cdot 
\big[{\bs e}_{\bs d}^\theta\cdot\grad\partial_t{\bs u} + (\vel_H^{*|\theta})^{\top}(\grad^2 {\bs u}){\bs e}_{\bs d}^\theta \big] \notag\\
&\quad\,+\int_0^1\d\theta\int_{\Gamma[\bs X_H^\theta]} {\bs \lambda}_H^\theta \cdot \big(\partial_t{\bs u} + \vel_H^{*|\theta}\cdot\grad{\bs u} - {\bs\ddot\dep}_H^{*|\theta} \big) \grad_{\Gamma[\bs X_H^\theta]}\cdot{\bs e}_{\bs d}^\theta \Big|\notag\\
&\lesssim \int_0^1 \|{\bs \lambda}_H^\theta\|_{L^2(\widehat\Gamma_H)} \|{\bs e}_{\bs d}^\theta\|_{H^1(\widehat\Gamma_H)} \d\theta \notag\\
&\lesssim H^{-1}\|{\bs \lambda}_H\|_{H^{-\frac12}(\widehat\Gamma_H)} \|{\bs e}_{\bs d}\|_{H^1(\widehat\Omega_H^{\rm s})} \notag\\
&\le \varepsilon\|{\bs \lambda}_H\|_{H^{-\frac12}(\widehat\Gamma_H)}^2 + C_\varepsilon H^{-2}\|{\bs e}_{\bs d}\|_{H^1(\widehat\Omega_H^{\rm s})}^2 ,\\[5pt] 
|\eqref{dteq-eu-bc-b}|
&\lesssim \|{\bs \lambda}_H\|_{L^2(\widehat\Gamma_H)} 
( h^{k+1} + H^{k+1} + \|{\bs e}_{\bs d}\|_{L^2(\widehat\Gamma_H)} ) \notag\\
&\lesssim H^{-\frac12}\|{\bs \lambda}_H\|_{H^{-\frac12}(\widehat\Gamma_H)} 
( h^{k+1} + H^{k+1} + \|{\bs e}_{\bs d}\|_{H^1(\widehat\Gamma_H)} ) \notag\\
&\le \varepsilon\|{\bs\lambda}_H\|_{H^{-\frac12}(\widehat\Gamma_H)}^2 + C_\varepsilon H^{-1}\|{\bs e}_{\bs d}\|_{H^1(\widehat\Omega_H^{\rm s})}^2 
+ C_\varepsilon ( H^{-1} h^{2k+2} + H^{2k} ) ,  
\end{align}
and 
\begin{align}
|\eqref{dteq-eu-bc-c}+\eqref{dteq-eu-bc-d}|
&\lesssim \|{\bs \lambda}_H\|_{L^2(\widehat\Gamma_H)} h^k
\le \varepsilon \|{\bs \lambda}_H\|_{H^{-\frac12}(\widehat\Gamma_H)}^2 
+ C_{\varepsilon}H^{-1}h^{2k} .
\end{align}
Since  \eqref{err-j-dte}=$- (\frac{\d}{\d t} R_{h,H}^*)(0, 0, 0, {\bs e}_{{\bs \sigma} }) $ (see the text below \eqref{dteq-eu-ded}), choosing ${\bs \lambda}_H = {\bs e}_{{\bs \sigma} }$ in the above estimates yields the following result: 
\begin{align}\label{dteq-div-eu-estimates2}
\eqref{err-j-dte}
&\ge -\varepsilon \|{\bs e}_{{\bs \sigma} }\|_{H^{-\frac12}(\widehat\Gamma_H)}^2
- C_{\varepsilon}h^{-1}H^{-1}\| {\bs e}_{\bs u}  \|_{H^1(\Omega^{\rm f}_{\rm ext}[{\bs X}_H])}^2\notag\\
&\quad\, 
- C_{\varepsilon}H^{-2} (\|{\bs\dot{\bs e}}_{\bs d}\|_{L^2(\widehat\Omega_H^{\rm s })}^2 + \|{\bs e}_\dep\|_{H^1(\widehat{\Omega}^{\rm s}_H)}^2 )
- C_{\varepsilon} (H^{-1}h^{2k} + H^{2k-1} ). 
\end{align}

\noindent{\bf Part 6: Estimates for (\ref{err-n-dte}).}\\
From the definition of $R_{h,H}^*({\bs v}_h, q_h, \bs w_H, \bs \lambda_H) $ in \eqref{def-RhH} we see that 
\begin{subequations}
\begin{align}
R_{h,H}^*(\partial_t{\bs e}_{\bs u}, 0, \partial_{tt}{\bs e}_{\bs d}, \bs 0) 
= & 
\int_{\Omega^{\rm f}[\bs X_H^*]} \partial_t {\bs u}_h^* \cdot \partial_t{\bs e}_{\bs u}
- \int_{\Omega^{\rm f}(t)} \partial_t \bs u \cdot \partial_t{\bs e}_{\bs u} \label{RhHdt-a}\\
&+ \int_{\Omega^{\rm f}[\bs X_H^*]}   (\bu_h^* \cdot \grad) \bu_h^* \cdot \partial_t{\bs e}_{\bs u}
- \int_{\Omega^{\rm f}(t)}   (\bu \cdot \grad) \bu \cdot \partial_t{\bs e}_{\bs u} \label{RhHdt-b}\\
& +  \int_{\Omega^{\rm f}[\bs X_H^*]} \bs D(\bu_h^*): \bs D(\partial_t{\bs e}_{\bs u})  -  \int_{\Omega^{\rm f}(t)} \bs D(\bu): \bs D(\partial_t{\bs e}_{\bs u}) \label{RhHdt-c}\\ 
&  - \int_{\Omega^{\rm f}[\bs X_H^*]}  p_h^*\, \grad\cdot \partial_t{\bs e}_{\bs u} +  \int_{\Omega^{\rm f}(t)}  p \, \grad\cdot \partial_t{\bs e}_{\bs u} \label{RhHdt-d}\\
& +  \int_{\widehat{\Omega}^{\rm s}_H} \partial_{t}\vel_H^* \cdot \partial_{tt}{\bs e}_{\bs d} -  \int_{\widehat{\Omega}^{\rm s}} \partial_{t}\vel \cdot \partial_{tt}{\bs e}_{\bs d} \circ \widehat{\bs\Phi}_H^{-1} \label{RhHdt-e}\\
& +\int_{\widehat{\Omega}^{\rm s}_H} {\bs D}(\dep_H^*): {\bs D}(\partial_{tt}{\bs e}_{\bs d}) - \int_{\widehat{\Omega}^{\rm s}} {\bs D}(\dep) : {\bs D}(\partial_{tt}{\bs e}_{\bs d}\circ \widehat{\bs\Phi}_H^{-1})  \label{RhHdt-f}\\
& + \int_{\widehat{\Omega}^{\rm s}_H} (\grad\cdot \dep_H^*) (\grad\cdot \partial_{tt}{\bs e}_{\bs d}) - \int_{\widehat{\Omega}^{\rm s}} (\grad\cdot \dep)  \grad\cdot (\partial_{tt}{\bs e}_{\bs d}\circ \widehat{\bs\Phi}_H^{-1})  \label{RhHdt-f2}\\
 & - \int_{\Gamma[\bs X_H^*]} \widehat{\bs \sigma} _H^* \circ ({\bs X}_H^*)^{-1}  \cdot \big( \partial_{t}{\bs e}_{\bs u} - \partial_{tt}{\bs e}_{\bs d}  \circ ({\bs X}_H^*)^{-1} \big) \label{RhHdt-g}\\
& + \int_{\Gamma(t)} \widehat{\bs \sigma} \circ {\bs X}^{-1}  \cdot \big( \partial_t{\bs e}_{\bs u} - \partial_{tt}{\bs e}_{\bs d} \circ \widehat{\bs\Phi}_H^{-1} \circ {\bs X}^{-1} \big) \label{RhHdt-h}\\
& + \epsilon_h g_h^{\rm u} (\partial_t\bu_h^*,\partial_t{\bs e}_{\bs u} ) + g_h^{\rm u} (\bu_h^*,\partial_t{\bs e}_{\bs u} ) . \label{RhHdt-i}
\end{align}
\end{subequations} 
Note that 
\begin{subequations}
\begin{align}
\int_0^s \eqref{RhHdt-a} \d t 
&= \Big(\int_{\Omega^{\rm f}[\bs X_H^*]} \partial_t \bs u_h^* \cdot {\bs e}_{\bs u}
- \int_{\Omega^{\rm f}(t)} \partial_t \bs u \cdot {\bs e}_{\bs u} \Big)\Big|_{t=0}^{t=s} \label{RhHdt-a-a} \\
&\quad\, 
- \int_0^s \Big(\int_{\Omega^{\rm f}[\bs X_H^*]} \partial_{tt} \bs u_h^* \cdot {\bs e}_{\bs u}
- \int_{\Omega^{\rm f}(t)} \partial_{tt} \bs u \cdot {\bs e}_{\bs u} \Big) \d t \label{RhHdt-a-b} \\
&\quad\, 
- \int_0^s \int_{\Gamma[{\bs X}_H^*]} (\partial_{t} \bs u_h^* \cdot {\bs e}_{\bs u} ) \vel_H^*\circ ({\bs X}_H^*)^{-1}\cdot{\bs n}_{\Gamma[{\bs X}_H^*]} \d t \label{RhHdt-a-c} \\
&\quad\, + \int_0^s \int_{\Gamma(t)} (\partial_{t} \bs u \cdot {\bs e}_{\bs u})(\vel\circ {\bs X}^{-1})\cdot{\bs n}_{\Gamma(t)} \d t \label{RhHdt-a-d} , 
\end{align}
\end{subequations}
where \eqref{RhHdt-a-a} is the same as \eqref{RhH-a} and therefore has already been estimated in Lemma \ref{Lemma:RhH}, with 
\begin{align}
|\eqref{RhHdt-a-a}|
&\lesssim (h^k+H^k) (\|{\bs e}_{\bs u}(s)\|_{H^1(\Omega^{\rm f}_{\rm ext}[{\bs X}_H(\cdot,s)])} + \|{\bs e}_{\bs u}(0)\|_{H^1(\Omega^{\rm f}_{\rm ext}[{\bs X}_H(\cdot,s)])})\notag\\
&\le \varepsilon\|{\bs e}_{\bs u}(s)\|_{H^1(\Omega^{\rm f}_{\rm ext}[{\bs X}_H(\cdot,s)])}^2 + C_\varepsilon (h^{2k}+H^{2k}), 
\end{align}
where the last inequality follows from Young's inequality and ${\bs e}_{\bs u}|_{t=0}=({\bs u}_{0,h}-{\bs I}^{\rm f}_h{\bs u})|_{t=0}=O(h^k)$ in the $H^1$ norm as shown in \eqref{uh0-error}. The term \eqref{RhHdt-a-b} can be estimated similarly, i.e.,
\begin{align}
|\eqref{RhHdt-a-b}|
&\lesssim \int_0^s \big( \|{\bs e}_{\bs u}(t)\|_{H^1(\Omega^{\rm f}_{\rm ext}[{\bs X}_H(\cdot,t)])}^2 + h^{2k}+H^{2k} \big) \d t . 
\end{align}
Furthermore, by denoting 
$$
{\bs V}_{\Gamma[{\bs X}_H^{*,\theta}]}
\eqd[(1-\theta)\vel + \theta\vel_H^*\circ\widehat{\bs\Phi}_H^{-1}]\circ({\bs X}_H^{*,\theta})^{-1}
\quad\mbox{(velocity of interface $\Gamma[{\bs X}_H^{*,\theta}]$)} , 
$$
with ${\bs V}_{\Gamma[{\bs X}_H^{*,1}]}=\vel_H^*\circ ({\bs X}_H^*)^{-1}$ being the velocity of interface $\Gamma[{\bs X}_H^*]$, we have $\partial_\theta^\bullet {\bs V}_{\Gamma[{\bs X}_H^{*,\theta}]}={\bs\dot{\bs e}}_{\bs X}^{*,\theta}=O(H^{k+1})$ and therefore can estimate \eqref{RhHdt-a-c}+\eqref{RhHdt-a-d} as follows by using \eqref{dt-int-Gamma-1}: 
\begin{align}
\eqref{RhHdt-a-c}+\eqref{RhHdt-a-d}
&= - \int_0^s \int_0^1 \frac{\d}{\d\theta}\int_{\Gamma[{\bs X}_H^{*,\theta}]} (\partial_{t} \bs u \cdot {\bs e}_{\bs u} ) {\bs V}_{\Gamma[{\bs X}_H^{*,\theta}]}\cdot{\bs n}_{\Gamma[{\bs X}_H^{*,\theta}]} \d\theta\d t \notag\\
&\quad\,
-\int_0^s \int_{\Gamma[{\bs X}_H^*]} (\partial_{t}({\bs u}_h^*-{\bs u}) \cdot {\bs e}_{\bs u} ) {\bs V}_{\Gamma[{\bs X}_H^{*,1}]}\cdot{\bs n}_{\Gamma[{\bs X}_H^*]} \d t \notag\\
&\lesssim \int_0^s \int_0^1 \|{\bs e}_{\bs u}\|_{L^2(\Gamma[{\bs X}_H^{*,\theta}])} ( \|{\bs\dot{\bs e}}_{\bs X}^{*,\theta}\|_{L^2(\Gamma[{\bs X}_H^{*,\theta}])} 
+\|{\bs e}_{\bs X}^{*,\theta}\|_{H^1(\Gamma[{\bs X}_H^{*,\theta}])}) \d\theta\d t \notag\\
&\quad\,
+ \int_0^s \int_0^1 \|\grad{\bs e}_{\bs u}\|_{L^2(\Gamma[{\bs X}_H^{*,\theta}])} \|{\bs\dot{\bs e}}_{\bs X}^{*,\theta}\|_{L^2(\Gamma[{\bs X}_H^{*,\theta}])} \d\theta \d t \notag\\
&\quad\, + \int_0^sh^{k+1}\|{\bs e}_{\bs u}\|_{L^2(\Gamma[{\bs X}_H^{*,\theta}])} \d t \notag\\
&\lesssim \int_0^s\|{\bs e}_{\bs u}\|_{H^1(\Omega^{\rm f}_{\rm ext}[{\bs X}_H])}^2\d t + h^{2k} + H^{2k} + h^{-1}H^{2k+2} ,
\end{align}
where the last inequality uses the inverse trace estimate 
$$\|\grad{\bs e}_{\bs u}\|_{L^2(\Gamma[{\bs X}_H^{*,\theta}])}\lesssim h^{-\frac12}\|{\bs e}_{\bs u}\|_{H^1(\Omega^{\rm f}_{\rm ext}[{\bs X}_H])}$$ 
as well as 
$\|{\bs\dot{\bs e}}_{\bs X}^{*,\theta}\|_{L^2(\Gamma[{\bs X}_H^{*,\theta}])}=O(H^{k+1})$ and $\|{\bs e}_{\bs X}^{*,\theta}\|_{H^1(\Gamma[{\bs X}_H^{*,\theta}])}=O(H^{k})$. 
This proves that 
\begin{align}
\int_0^s \eqref{RhHdt-a} \d t 
&\lesssim 
\varepsilon \|{\bs e}_{\bs u}(s)\|_{H^1(\Omega^{\rm f}_{\rm ext}[{\bs X}_H(\cdot,s)])}^2 
+ C_\varepsilon\int_0^s\|{\bs e}_{\bs u}\|_{H^1(\Omega^{\rm f}_{\rm ext}[{\bs X}_H])}^2\d t \notag\\
&\quad\, + C_\varepsilon (h^{2k} + H^{2k}) \qquad\mbox{(under condition $H^2\lesssim h$)}.
\end{align}
All other terms \eqref{RhHdt-b}-\eqref{RhHdt-h} can be estimated similarly via integration by parts, with 
\begin{align}
&\Big| \int_0^s [\eqref{RhHdt-b} + \cdots + \eqref{RhHdt-h}]\d t \Big|  \notag\\ 
&\lesssim 
\varepsilon 
(\|{\bs e}_{\bs u}(s)\|_{H^1(\Omega^{\rm f}_{\rm ext}[{\bs X}_H(\cdot,s)])}^2 
+\|\partial_t{\bs e}_{\bs d}(s)\|_{L^2(\widehat\Omega^{\rm s}_H)}^2) \notag\\
&\quad\, 
+ C_\varepsilon\int_0^s
(\|{\bs e}_{\bs u}\|_{H^1(\Omega^{\rm f}_{\rm ext}[{\bs X}_H])}^2
+H^{-1}\|\partial_t{\bs e}_{\bs d}\|_{L^2(\widehat\Omega^{\rm s}_H)}^2)\d t \notag\\
&\quad\, + C_\varepsilon (h^{2k-1} + h^{-1}H^{2k} + H^{2k-1}) .
\end{align}
Since the ghost penalty vanishes on smooth functions,
\begin{align*}
\epsilon_h
\left|
g_h^{\rm u}(\partial_t{\bs u}_h^*,\partial_t{\bs e}_{\bs u})
\right|
&=
\epsilon_h
\left|
g_h^{\rm u}(\partial_t{\bs u}_h^*-\partial_t{\bs u},
             \partial_t{\bs e}_{\bs u})
\right|
\\
&\leq 
\varepsilon \epsilon_h^2 
g_h^{\rm u}(\partial_t{\bs e}_{\bs u},\partial_t{\bs e}_{\bs u})
+ C_\varepsilon  g_h^{\rm u}(\partial_t{\bs u}_h^*-\partial_t{\bs u},\partial_t{\bs u}_h^*-\partial_t{\bs u}) \\
&\leq \varepsilon \epsilon_h^2 g_h^{\rm u}(\partial_t{\bs e}_{\bs u},\partial_t{\bs e}_{\bs u})
+ C_\varepsilon h^{2k} ,
\end{align*}
it follows that
\begin{align}
&\Big| \int_0^s \eqref{RhHdt-i}\d t \Big| \notag\\
&\leq 
\varepsilon\int_0^s \epsilon_h^2
g_h^{\rm u}(\partial_t{\bs e}_{\bs u},\partial_t{\bs e}_{\bs u}) \d t  +
\varepsilon 
\|{\bs e}_{\bs u}(s)\|_{1,\Omega[{\bs X}_H(\cdot,s)]}^2 
+ C_\varepsilon\int_0^s
\|{\bs e}_{\bs u}\|_{1,\Omega[{\bs X}_H(\cdot,t)]}^2 \d t
+ C_\varepsilon h^{2k} .
\end{align}
Therefore, 
\begin{align}\label{dte-estimates-most}
\Big| \int_0^s \eqref{err-n-dte} \d t \Big| 
&= \Big| \int_0^s R_{h,H}^*(\partial_t{\bs e}_{\bs u}, 0, \partial_{tt}{\bs e}_{\bs d}, \bs 0) \d t \Big| \notag\\
&\leq
\varepsilon \int_0^s \epsilon_h^2
g_h^{\rm u}(\partial_t{\bs e}_{\bs u},\partial_t{\bs e}_{\bs u}) \d t 
+ \varepsilon 
(\|{\bs e}_{\bs u}(s)\|_{1,\Omega[{\bs X}_H(\cdot,s)]}^2 
+\|\partial_t{\bs e}_{\bs d}(s)\|_{L^2(\widehat\Omega^{\rm s}_H)}^2) \notag\\
&\quad\, 
+ C_\varepsilon\int_0^s
(\|{\bs e}_{\bs u}\|_{1,\Omega[{\bs X}_H]}^2 
+H^{-1}\|\partial_t{\bs e}_{\bs d}\|_{L^2(\widehat\Omega^{\rm s}_H)}^2)\d t \notag\\
&\quad\, + C_\varepsilon (h^{2k-1} + h^{-1}H^{2k} + H^{2k-1}) . 
\end{align}

Now, substituting the estimates of Parts 1--6 to \eqref{error-eq-dte}, and using Korn's inequality in \eqref{Korn-inequlaity}, we obtain the following estimate for all $s\in(0,t_*]$ after absorbing $\int_0^s \epsilon_h^2
g_h^{\rm u}(\partial_t{\bs e}_{\bs u},\partial_t{\bs e}_{\bs u}) \d t  $ into the left-hand side:
\begin{align} \label{dte-estimates}
& \int_0^s \int_{\Omega^{\rm f}[\bs X_H]} |\partial_t {\bs e}_{\bs u} |^2 \d t
+ \int_0^s \epsilon_h g_h^{\rm u} (\partial_t{\bs e}_{\bs u}, \partial_t{\bs e}_{\bs u}) \d t 
+ \int_0^s\int_{\widehat{\Omega}^{\rm s}_H} |\partial_{tt} {\bs e}_{\dep} |^2 \d t \notag\\
&\quad\, 
+\|{\bs e}_{\bs u}(s)\|_{1,\Omega[{\bs X}_H]}^2 
+ g^p_h(e_p(s),e_p(s)) \notag\\
&
\le 
\varepsilon
\big( \|\grad{\bs e}_{\bs d}(s)\|_{L^2(\widehat\Omega^{\rm s}_H)}^2 + \|\partial_t{\bs e}_{\bs d}(s)\|_{L^2(\widehat\Omega^{\rm s}_H)}^2 \big)\notag\\
&\quad\, 
+ \varepsilon \int_0^s ( \|e_p\|_{0,\Omega[{\bs X}_H]}^2 +\| {\bs e}_{{\bs \sigma} }\|_{H^{-\frac12}(\widehat\Gamma_H)}^2 ) \d t 
\notag\\
&\quad\, 
+ C_\varepsilon \int_0^s (h^{-2}+h^{-1}H^{-1})\| {\bs e}_{\bs u}  \|_{1,\Omega[{\bs X}_H]}^2  \d t \notag\\
&\quad\, 
+ C_\varepsilon \int_0^s ((h^{-1}H^{-1}+H^{-2})\|{\bs e}_\dep\|_{H^1(\widehat{\Omega}^{\rm s}_H)}^2 +H^{-1}\|\partial_t{\bs e}_\dep\|_{L^2(\widehat{\Omega}^{\rm s}_H)}^2) \d t \notag\\
&\quad\, 
+ C_\varepsilon (h^{2k-2} + H^{-1}h^{2k} + h^{-1}H^{2k} + H^{2k-1}) 
+ Cg^p_h(e_p(0),e_p(0)) ,
\end{align}
where $g^p_h(e_p(0),e_p(0))\lesssim h^{2k}$ is estimated by using \eqref{ghost-ep-0} (this is the reason we want to choose the initial value of ${\bs u}_h$ to be the Ritz projection). Therefore, substituting \eqref{ep-L2}, \eqref{esigma-H-half-hat} and \eqref{esigma-L4}  into \eqref{dte-estimates}, and noting that the term $\epsilon_h
g_h^{\rm u}(\partial_t{\bs e}_{\bs u},\partial_t{\bs e}_{\bs u})^{1/2}$ in \eqref{ep-L2} can be absorbed by the left-hand side of \eqref{dte-estimates}, we obtain 
\begin{subequations}\label{ep-esigma-estimates}
\begin{align} 
& \int_0^s \int_{\Omega^{\rm f}[\bs X_H]} |\partial_t {\bs e}_{\bs u} |^2 \d t
+ \int_0^s \epsilon_h g_h^{\rm u} (\partial_t{\bs e}_{\bs u}, \partial_t{\bs e}_{\bs u}) \d t 
+ \int_0^s\int_{\widehat{\Omega}^{\rm s}_H} |\partial_{tt} {\bs e}_{\dep} |^2 \d t \notag\\
&\quad\,
+\|{\bs e}_{\bs u}(s)\|_{1,\Omega[{\bs X}_H]}^2 
+ g^p_h(e_p(s),e_p(s))\notag\\
&
\le 
\varepsilon
\big( \|\grad{\bs e}_{\bs d}(s)\|_{L^2(\widehat\Omega^{\rm s}_H)}^2 + \|\partial_t{\bs e}_{\bs d}(s)\|_{L^2(\widehat\Omega^{\rm s}_H)}^2 \big)\notag\\
&\quad\, 
+ C_\varepsilon \int_0^s (h^{-2}+h^{-1}H^{-1})\| {\bs e}_{\bs u}  \|_{1,\Omega[{\bs X}_H]}^2  \d t \notag\\
&\quad\, 
+ C_\varepsilon \int_0^s ((h^{-1}H^{-1}+H^{-2})\|{\bs e}_\dep\|_{H^1(\widehat{\Omega}^{\rm s}_H)}^2 +H^{-1}\|\partial_t{\bs e}_\dep\|_{L^2(\widehat{\Omega}^{\rm s}_H)}^2) \d t \notag\\
&\quad\, 
+ C_\varepsilon (h^{2k-2} + H^{-1}h^{2k} + h^{-1}H^{2k} + H^{2k-1}) 
\end{align}
and
\begin{align} \label{ep-esigma-estimates-L4}
& \int_0^s \| e_p \|_{0,\Omega[{\bs X}_H]}^2 \d t
+ \int_0^s \big( \| {\bs e}_{{\bs \sigma} } \|_{H^{-\frac12}(\widehat\Gamma_H)}^2 + \| {\bs e}_{\bs\sigma} \|_{L^4(\widehat\Gamma_H)}^2 \big) \d t  + g^p_h(e_p(s),e_p(s)) \notag\\
&
\le 
\varepsilon
\big( \|\grad{\bs e}_{\bs d}(s)\|_{L^2(\widehat\Omega^{\rm s}_H)}^2 + \|\partial_t{\bs e}_{\bs d}(s)\|_{L^2(\widehat\Omega^{\rm s}_H)}^2 \big)\notag\\
&\quad\, 
+ C_\varepsilon \int_0^s (h^{-2}+h^{-1}H^{-1})\| {\bs e}_{\bs u}  \|_{1,\Omega[{\bs X}_H]}^2  \d t \notag\\
&\quad\, 
+ C_\varepsilon \int_0^s ((h^{-1}H^{-1}+H^{-2})\|{\bs e}_\dep\|_{H^1(\widehat{\Omega}^{\rm s}_H)}^2 +H^{-1}\|\partial_t{\bs e}_\dep\|_{L^2(\widehat{\Omega}^{\rm s}_H)}^2) \d t \notag\\
&\quad\, 
+ C_\varepsilon (h^{2k-2} + H^{-1}h^{2k} + h^{-1}H^{2k} + H^{2k-1}) .
\end{align}
\end{subequations}

\subsection{Stability estimates for \texorpdfstring{${\bs e}_{\bs u}$}{x}}
In this subsection, we present stability estimates for ${\bs e}_{\bs u}$ by choosing ${\bs v}_h={\bs e}_{\bs u}$,\,\, $q_h=e_p$,\,\, ${\bs w}_H = \partial_t{\bs e}_\dep$\,\, and\,\, ${\bs \lambda}_H ={\bs e}_{{\bs \sigma} }$\,\, in\, \eqref{error-eq}. We then estimate each term in \eqref{error-eq} accordingly as separated items in the following. 

\begin{itemize}[leftmargin=20pt]
\item
\eqref{error-eq-a} can be estimated as follows, using \eqref{dt-int-Omega-1} and the notation $\vel_H^1=\partial_{t}{\bs d}_H\circ {\bs X}_H^{-1}$:  
\begin{equation}\label{eq:estima1}
\begin{aligned}
\mbox{\eqref{error-eq-a}} 
&= \int_{\Omega^{\rm f}[\bs X_H]} \partial_t {\bs e}_{\bs u} \cdot {\bs e}_{\bs u} 
\\
&= \frac12 \int_{\Omega^{\rm f}[\bs X_H]} \partial_t  | {\bs e}_{\bs u} |^2 \\
&=  \frac{\d}{\d t} \int_{\Omega^{\rm f}[\bs X_H]} \frac12 | {\bs e}_{\bs u} |^2 
    - \int_{\Gamma[\bs X_H]} \frac12| {\bs e}_{\bs u} |^2 \, \vel_H^1\cdot {\bs n}_{\Gamma[\bs X_H]} 
\quad\mbox{(here \eqref{dt-int-Omega-1} is used)} \\
&\ge  \frac{\d}{\d t} \int_{\Omega^{\rm f}[\bs X_H]} \frac12 | {\bs e}_{\bs u} |^2 
    - C \| {\bs e}_{\bs u} \|_{L^2(\Gamma[\bs X_H])}^2 
    \quad\mbox{($\|\partial_{t}{\bs d}_H\|_{L^\infty(\widehat\Omega^{\rm s}_H)} \lesssim 1$, as a result of \eqref{ind-1-2})}\\
&\ge  \frac{\d}{\d t} \int_{\Omega^{\rm f}[\bs X_H]} \frac12 | {\bs e}_{\bs u} |^2 
    - C \| {\bs e}_{\bs u} \|_{L^2(\Omega^{\rm f}_{\rm ext}[\bs X_H])}\| {\bs e}_{\bs u} \|_{H^1(\Omega^{\rm f}_{\rm ext}[\bs X_H])} 
    \quad\mbox{(here \eqref{L2-f-trace} is used)} \\
&\ge  \frac{\d}{\d t} \int_{\Omega^{\rm f}[\bs X_H]} \frac12 | {\bs e}_{\bs u} |^2 
    - C_{\varepsilon}\| {\bs e}_{\bs u} \|_{L^2(\Omega^{\rm f}_{\rm ext}[\bs X_H])}^2
    -\varepsilon \| {\bs e}_{\bs u} \|_{H^1(\Omega^{\rm f}_{\rm ext}[\bs X_H])}^2 .
\end{aligned}
\end{equation} 

\item \eqref{error-eq-b} can be estimated as follows, using \eqref{dt-int-Omega-2} and the notation ${\bs e}_{\bs d}^\theta={\bs e}_{\bs d}\circ({\bs X}_H^\theta)^{-1}$:  
\begin{align}
\eqref{error-eq-b} &=\Big( \int_{\Omega^{\rm f}[\bs X_H]} - \int_{\Omega^{\rm f}[\bs X_H^*]} \Big) \partial_t {\bs u}_h^* \cdot {\bs e}_{\bs u} \notag\\
&= \int_0^{1} \Big( \frac{\d}{\d\theta}\int_{\Omega^{\rm f}[\bs X_H^\theta]}\partial_t {\bs u}_h^* \cdot{\bs e}_{\bs u} \Big) \d\theta \quad\mbox{(the following equality uses \eqref{dt-int-Omega-2})}\notag\\
&= \int_0^{1} \int_{\Omega^{\rm f}[\bs X_H^\theta]} \partial_\theta( \partial_t {\bs u}_h^* \cdot {\bs e}_{\bs u} ) \d\theta
+ \int_0^{1} \int_{\Gamma[\bs X_H^\theta]} (\partial_t {\bs u}_h^* \cdot {\bs e}_{\bs u} ) {\bs e}_\dep^\theta \cdot {\bs n}_{\Gamma[\bs X_H^\theta]} \d\theta \notag\\
&= \int_0^{1} \int_{\Gamma[\bs X_H^\theta]} (\partial_t {\bs u}_h^* \cdot {\bs e}_{\bs u} ) {\bs e}_\dep^\theta \cdot {\bs n}_{\Gamma[\bs X_H^\theta]} \d\theta \quad\mbox{(here $\partial_\theta( \partial_t {\bs u}_h^* \cdot {\bs e}_{\bs u} )=0$ is used)} \notag \\
&\ge - C\int_0^{1}  \|{\bs e}_\dep^\theta\|_{L^2(\Gamma[\bs X_H^\theta])} \|{\bs e}_{\bs u}\|_{L^2(\Gamma[\bs X_H^\theta])} \d\theta \notag\\
&\ge - C\int_0^{1}\big( \|{\bs e}_\dep^\theta\|_{L^2(\Gamma[\bs X_H^\theta])}^2 + \|{\bs e}_{\bs u}\|_{L^2(\Gamma[\bs X_H^\theta])}^2 \big) \d\theta \notag\\
&\ge - C \|{\bs e}_\dep\|_{L^2(\widehat{\Omega}^{\rm s}_H)} \|{\bs e}_\dep\|_{H^1(\widehat{\Omega}^{\rm s}_H)}  
- C \|{\bs e}_{\bs u}\|_{L^2(\Omega^{\rm f}_{\rm ext}[\bs X_H])} \|{\bs e}_{\bs u}\|_{H^1(\Omega^{\rm f}_{\rm ext}[\bs X_H])} \notag \\
&\ge - \varepsilon  \big( \|\grad {\bs e}_\dep\|_{L^2(\widehat{\Omega}^{\rm s}_H)}^2 +  \|\grad {\bs e}_{\bs u}\|_{L^2(\Omega^{\rm f}_{\rm ext}[\bs X_H])}^2 \big) 
- C_{\varepsilon} \big(\|{\bs e}_\dep\|_{L^2(\widehat{\Omega}^{\rm s}_H)}^2 + \|{\bs e}_{\bs u}\|_{L^2(\Omega^{\rm f}_{\rm ext}[\bs X_H])}^2 \big) , \label{eq:estim2}
\end{align}
where the second-to-last inequality follows from the norm equivalence $\|{\bs e}_\dep^\theta\|_{L^2(\Gamma[\bs X_H^\theta])}\simeq \|{\bs e}_\dep\|_{L^2(\widehat\Gamma_H)}$ and the trace inequality from $\widehat\Gamma_H=\partial\widehat{\Omega}^{\rm s}_H$ to $\widehat{\Omega}^{\rm s}_H$, as well as \eqref{L2-f-trace}. 

\item \eqref{error-eq-c} and \eqref{error-eq-d} can be estimated as follows, by decomposing $\grad{\bs u}_h$ into $\grad{\bs u}_h^* + \grad{\bs e}_{\bs u}$ and using the boundedness of $\|\grad{\bs u}_h^*\|_{L^\infty(\Omega^{\rm f}[\bs X_H])}$ and $\|{\bs e}_{\bs u}\|_{L^\infty(\Omega^{\rm f}[\bs X_H])}$ (the latter is given by \eqref{ind-1-2}):
\begin{align}
| \eqref{error-eq-c} + \eqref{error-eq-d} |
&\le C\|{\bs e}_{\bs u}\|_{L^2(\Omega^{\rm f}[\bs X_H])} \|{\bs e}_{\bs u}\|_{H^1(\Omega^{\rm f}[\bs X_H])} \notag\\
&\le \varepsilon \|\grad {\bs e}_{\bs u}\|_{L^2(\Omega^{\rm f}[\bs X_H])}^2 
+ C_{\varepsilon} \|{\bs e}_{\bs u}\|_{L^2(\Omega^{\rm f}[\bs X_H])}^2 . 
\end{align}

\item \eqref{error-eq-e} can be estimated in the same way as \eqref{error-eq-b}, with the same estimate: 
\begin{equation}\label{eq:estima3}
\begin{aligned}
\eqref{error-eq-e} & = \Big( \int_{\Omega^{\rm f}[\bs X_H]} - \int_{\Omega^{\rm f}[\bs X_H^*]} \Big) {\bs u}_h^* \cdot \grad {\bs u}_h^* \cdot {\bs e}_{\bs u} \\
&\ge - \varepsilon  \big( \|\grad {\bs e}_\dep\|_{L^2(\widehat{\Omega}^{\rm s}_H)}^2 +  \|\grad {\bs e}_{\bs u}\|_{L^2(\Omega^{\rm f}_{\rm ext}[\bs X_H])}^2 \big) 
- C_{\varepsilon} \big(\|{\bs e}_\dep\|_{L^2(\widehat{\Omega}^{\rm s}_H)}^2 + \|{\bs e}_{\bs u}\|_{L^2(\Omega^{\rm f}_{\rm ext}[\bs X_H])}^2 \big) . 
\end{aligned}
\end{equation}

\item For the term \eqref{error-eq-f}, we have 
$$
\eqref{error-eq-f} = \int_{\Omega^{\rm f}[\bs X_H]} | {\bs D}({\bs e}_{\bs u}) |^2 .
$$

\item For the term \eqref{error-eq-g} we proceed as follows: 
\begin{subequations}
\begin{align}
\eqref{error-eq-g} 
& = \int_{\Omega^{\rm f}[\bs X_H]} {\bs\sigma}(\bu_h^*,p_h^*): \grad {\bs e}_{\bs u}
- \int_{\Omega^{\rm f}[\bs X_H^*]} {\bs\sigma}(\bu_h^*,p_h^*): \grad {\bs e}_{\bs u} \notag\\
& = \int_{\Omega^{\rm f}[\bs X_H]} {\bs\sigma}(\bu,p): \grad {\bs e}_{\bs u}
- \int_{\Omega^{\rm f}[\bs X_H^*]} {\bs\sigma}(\bu,p): \grad {\bs e}_{\bs u} \notag\\
&\quad\, + \int_{\Omega^{\rm f}[\bs X_H]} [{\bs\sigma}(\bu_h^*,p_h^*) - {\bs\sigma}(\bu,p)] : \grad {\bs e}_{\bs u}
- \int_{\Omega^{\rm f}[\bs X_H^*]} [{\bs\sigma}(\bu_h^*,p_h^*)- {\bs\sigma}(\bu,p)]: \grad {\bs e}_{\bs u} \notag\\
& = - \int_{\Omega^{\rm f}[\bs X_H]} \grad\cdot {\bs\sigma}(\bu,p) \cdot {\bs e}_{\bs u}
+ \int_{\Omega^{\rm f}[\bs X_H^*]} \grad\cdot {\bs\sigma}(\bu,p) \cdot {\bs e}_{\bs u}
\label{sigma-gradv-a} \\
&\quad\, + \int_{\Gamma[\bs X_H]} {\bs\sigma}(\bu,p){\bs n}\cdot {\bs e}_{\bs u}
- \int_{\Gamma[\bs X_H^*]} {\bs\sigma}(\bu,p){\bs n}\cdot {\bs e}_{\bs u}
\label{sigma-gradv-b} \\
&\quad\, + \int_{\Omega^{\rm f}[\bs X_H]} [{\bs\sigma}(\bu_h^*,p_h^*) - {\bs\sigma}(\bu,p)] : \grad {\bs e}_{\bs u}
- \int_{\Omega^{\rm f}[\bs X_H^*]} [{\bs\sigma}(\bu_h^*,p_h^*)- {\bs\sigma}(\bu,p)]: \grad {\bs e}_{\bs u} ,
\label{sigma-gradv-c} 
\end{align}
\end{subequations}
where \eqref{sigma-gradv-a} can be estimated in the same way as \eqref{error-eq-b}, with the same estimate, i.e., 
\begin{align}\label{sigma-gradv-a-est}
|\eqref{sigma-gradv-a}| 
&\le \varepsilon \big( \|\grad {\bs e}_\dep\|_{L^2(\widehat{\Omega}^{\rm s}_H)}^2 +  \|\grad {\bs e}_{\bs u}\|_{L^2(\Omega^{\rm f}_{\rm ext}[\bs X_H])}^2 \big) \notag\\ 
&\quad\, + C_{\varepsilon} \big(\|{\bs e}_\dep\|_{L^2(\widehat{\Omega}^{\rm s}_H)}^2 + \|{\bs e}_{\bs u}\|_{L^2(\Omega^{\rm f}_{\rm ext}[\bs X_H])}^2 \big), 
\end{align}
and \eqref{sigma-gradv-c} can be estimated by using the estimates of interpolation error: 
\begin{align}\label{sigma-gradv-c-est}
|\eqref{sigma-gradv-c}| 
&\le Ch^k \|\grad{\bs e}_{\bs u}\|_{L^2(\Omega^{\rm f}_{\rm ext}[\bs X_H])} .
\end{align}
By using the notations ${\widehat{\bs\sigma}}_H^{*|\theta}={\widehat{\bs\sigma}}_H^*\circ({\bs X}_H^\theta)^{-1}$ and ${\bs\dot{\bs e}}_\dep^\theta = {\bs\dot{\bs e}}_\dep \circ({\bs X}_H^\theta)^{-1}$, we can estimate \eqref{sigma-gradv-b} in combination with \eqref{error-eq-k} and \eqref{error-eq-l}, i.e., 
\begin{subequations}\label{sigma-n-eu}
\begin{align}
\eqref{sigma-gradv-b} + \eqref{error-eq-k} + \eqref{error-eq-l} 
&= \int_{\Gamma[\bs X_H]} {\bs\sigma}(\bu,p){\bs n}\cdot {\bs e}_{\bs u} 
- \int_{\Gamma[\bs X_H]} {\widehat{\bs\sigma}}_H^{*|1}\cdot ( {\bs e}_{\bs u}  - {\bs\dot{\bs e}}_\dep^1 ) \notag\\
&\quad\, - \int_{\Gamma[\bs X_H^*]} {\bs\sigma}(\bu,p){\bs n}\cdot {\bs e}_{\bs u} 
+ \int_{\Gamma[\bs X_H^*]} {\widehat{\bs\sigma}}_H^{*|0} \cdot ( {\bs e}_{\bs u}  - {\bs\dot{\bs e}}_\dep^0 ) \notag\\
&= \int_{\Gamma[\bs X_H]} {\bs\sigma}(\bu,p){\bs n}\cdot {\bs e}_{\bs u} 
- \int_{\Gamma[\bs X_H]} {\widehat{\bs\sigma}}_H^{*|1}\cdot {\bs e}_{\bs u} 
\label{sigma-n-eu-1}\\
&\quad\, - \int_{\Gamma[\bs X_H^*]} {\bs\sigma}(\bu,p){\bs n}\cdot {\bs e}_{\bs u} 
+ \int_{\Gamma[\bs X_H^*]} {\widehat{\bs\sigma}}_H^{*|0} \cdot  {\bs e}_{\bs u} 
\label{sigma-n-eu-2}\\
&\quad\, + \int_{\Gamma[\bs X_H]} {\widehat{\bs\sigma}}_H^{*|1}\cdot {\bs\dot{\bs e}}_\dep^1  
- \int_{\Gamma[\bs X_H^*]} {\widehat{\bs\sigma}}_H^{*|0} \cdot {\bs\dot{\bs e}}_\dep^0 .
\label{sigma-n-eu-3}
\end{align}
\end{subequations}
We estimate \eqref{sigma-n-eu-1}--\eqref{sigma-n-eu-3} separately below. To this end, we consider the family of interfaces $\Gamma[{\bs X}_H^{\#,\theta}]$, $\theta\in[0,1]$, intermediate between $\Gamma(t)$ and $\Gamma[{\bs X}_H]$, as the image of the map ${\bs X}_H^{\#,\theta}:\widehat\Gamma \rightarrow \mathbb{R}^d$ defined in Section \ref{section:norm-equiv-2}. When $\theta$ moves from $0$ to $1$, the interface $\Gamma[{\bs X}_H^{\#,\theta}]$ moves from $\Gamma(t)$ to $\Gamma[{\bs X}_H]$ with velocity ${\bs e}_{\bs X}^{\#,\theta}$, which is given in \eqref{def-ej-theta}. 

Let ${\bs Y}_H^{\#,\theta} = {\bs X}_H^{\#,\theta} \circ \widehat{\bs\Phi}_H \circ {\bs X}_H^{-1}$ be the map from $\Gamma[{\bs X}_H]$ to $\Gamma[{\bs X}_H^{\#,\theta}]$, and consider the function
$$ 
{\bs\sigma}(\bu,p)^\theta = {\bs\sigma}(\bu,p)|_{\Gamma[{\bs X}_H^{\#,\theta}]} \circ {\bs Y}_H^{\#,\theta} , 
\quad\mbox{and}\quad
{\bs n}^\theta = {\bs n}_{\Gamma[{\bs X}_H^{\#,\theta}]} \circ {\bs Y}_H^{\#,\theta} ,
 $$ 
which are functions pulled back to $\Gamma[{\bs X}_H]$, satisfying the following relations:
\begin{align*}
\partial_\theta^\bullet {\bs\sigma}(\bu,p)^\theta = [ {\bs e}_{\bs X}^{\#,\theta} \cdot  \grad {\bs\sigma}(\bu,p)] \circ {\bs Y}_H^{\#,\theta}
\quad\mbox{and}\quad
\partial_\theta^\bullet {\bs n}^\theta = - [(\grad_{\Gamma[{\bs X}_H^{\#,\theta}]} {\bs e}_{\bs X}^{\#,\theta}) {\bs n}_{\Gamma[{\bs X}_H^{\#,\theta}]} ] \circ {\bs Y}_H^{\#,\theta} .
\end{align*}
Then \eqref{sigma-n-eu-1} can be estimated as follows, with the notations ${\widehat{\bs\sigma}} = [{\bs\sigma}(\bu,p){\bs n}]\circ {\bs X}$,  ${\widehat{\bs\sigma}}_H^*={\bs I}_H^{\rm s}({\widehat{\bs \sigma} } \circ \widehat{\bs\Phi}_H)$, and ${\bs\sigma}(\bu,p)^0{\bs n}^0 = [{\bs\sigma}(\bu,p){\bs n}]\circ {\bs X} \circ \widehat{\bs\Phi}_H \circ {\bs X}_H^{-1} = {\widehat{\bs\sigma}}\circ \widehat{\bs\Phi}_H \circ {\bs X}_H^{-1} $: 
\begin{subequations}\label{sigma-n-eu-1-123}
\begin{align}
\eqref{sigma-n-eu-1}  &= \int_{\Gamma[\bs X_H]} [{\bs\sigma}(\bu,p)^1 {\bs n}^1 - {\bs\sigma}(\bu,p)^0 {\bs n}^0] \cdot {\bs e}_{\bs u} \notag\\
&\quad\,+ \int_{\Gamma[\bs X_H]} \big[ ({\widehat{\bs\sigma}} \circ \widehat{\bs\Phi}_H  - {\widehat{\bs\sigma} }_H^* )\circ {\bs X}_H^{-1} \big] \cdot {\bs e}_{\bs u} \notag\\
&= \int_0^1 \int_{\Gamma[\bs X_H]} \partial_\theta^\bullet[{\bs\sigma}(\bu,p)^\theta {\bs n}^\theta] \cdot {\bs e}_{\bs u} \d\theta \notag\\
&\quad\,+ \int_{\Gamma[\bs X_H]} \big[ ({\widehat{\bs\sigma}}\circ \widehat{\bs\Phi}_H  - {\bs I}_H^s({\widehat{\bs\sigma}}\circ \widehat{\bs\Phi}_H) )\circ {\bs X}_H^{-1} \big] \cdot {\bs e}_{\bs u} \notag\\
&= \int_0^1 \int_{\Gamma[\bs X_H]} \Big( [({\bs e}_{\bs X}^{\#,\theta} \cdot \grad) {\bs\sigma}(\bu,p)] {\bs n}_{\Gamma[{\bs X}_H^{\#,\theta}]} \Big) \circ {\bs Y}_H^{\#,\theta} \cdot {\bs e}_{\bs u} \d\theta 
\label{sigma-n-eu-1-1} \\
&\quad\, - \int_0^1 \int_{\Gamma[\bs X_H]} {\bs\sigma}(\bu,p)^\theta [(\grad_{\Gamma[{\bs X}_H^{\#,\theta}]} {\bs e}_{\bs X}^{\#,\theta}) {\bs n}_{\Gamma[{\bs X}_H^{\#,\theta}]}]\circ {\bs Y}_H^{\#,\theta} \, \cdot {\bs e}_{\bs u} \d\theta 
\label{sigma-n-eu-1-2} \\ 
&\quad\,+ \int_{\Gamma[\bs X_H]} \big[ ({\widehat{\bs \sigma} } \circ \widehat{\bs\Phi}_H  - {\bs I}_H^s({\widehat{\bs\sigma}}\circ \widehat{\bs\Phi}_H) )\circ {\bs X}_H^{-1} \big] \cdot {\bs e}_{\bs u} .
\label{sigma-n-eu-1-3}
\end{align}
\end{subequations}
According to the trace inequality in \eqref{L2-f-trace} and the norm equivalence 
\begin{align*}
\| {\bs e}_{\bs X}^{\#,\theta} \circ {\bs Y}_H^{\#,\theta} \|_{L^2(\Gamma[\bs X_H])} \simeq \| {\bs e}_{\bs X}^{\#,\theta} \circ {\bs Y}_H^{\#,\theta} \circ {\bs X}_H \|_{L^2(\widehat\Gamma_H)} 
&\le C\| {\bs e}_{\bs d} \|_{L^2(\widehat\Gamma_H)} +  C\| {\bs e}_{\bs X}^* \|_{L^2(\widehat\Gamma)} \\
&\le C\| {\bs e}_{\bs d} \|_{L^2(\widehat\Gamma_H)} + CH^{k+1} ,
\end{align*}
we have 
\begin{align}\label{sigma-n-eu-1-1-est}
|\eqref{sigma-n-eu-1-1}|
&\le C\| {\bs e}_{\bs X}^{\#,\theta} \circ {\bs Y}_H^{\#,\theta} \|_{L^2(\Gamma[\bs X_H])} \| {\bs e}_{\bs u} \|_{L^2(\Gamma[\bs X_H])} \notag\\
&\le C(\| {\bs e}_{\bs d} \|_{L^2(\widehat\Gamma_H)} + H^{k+1} ) \| {\bs e}_{\bs u} \|_{L^2(\Gamma[\bs X_H])} \notag\\
&\le C\| {\bs e}_{\bs d} \|_{L^2(\widehat\Gamma_H)}^2  + CH^{2k+2}+ C\| {\bs e}_{\bs u} \|_{L^2(\Gamma[\bs X_H])}^2 \notag\\
&\le \varepsilon (\|\grad {\bs e}_{\bs u}\|_{L^2(\Omega^{\rm f}_{\rm ext}[\bs X_H])}^2 
+ \|\grad {\bs e}_{\bs d}\|_{L^2(\widehat{\Omega}^{\rm s}_H)}^2 ) \notag\\
&\quad\, + C_{\varepsilon}(\|{\bs e}_{\bs u}\|_{L^2(\Omega^{\rm f}_{\rm ext}[\bs X_H])}^2 + \| {\bs e}_{\bs d}\|_{L^2(\widehat{\Omega}^{\rm s}_H)}^2 ) + CH^{2k+2} . 
\end{align}
Similarly, 
\begin{align}\label{sigma-n-eu-1-3-est}
|\eqref{sigma-n-eu-1-3}|
&\le \varepsilon \|\grad {\bs e}_{\bs u}\|_{L^2(\Omega^{\rm f}_{\rm ext}[\bs X_H])}^2 
+ C_{\varepsilon} \|{\bs e}_{\bs u}\|_{L^2(\Omega^{\rm f}_{\rm ext}[\bs X_H])}^2 
+ CH^{2k+2} . 
\end{align}
Moreover, \eqref{sigma-n-eu-1-2} can be estimated as follows, by utilizing the fundamental theorem of calculus and the map ${\bs Z}^\theta = {\bs X} \circ ({\bs X}_H^{\#,\theta})^{-1}: \Gamma[{\bs X}_H^{\#,\theta}]\rightarrow \Gamma(t)$ which satisfies ${\bs Z}^1=  {\bs Y}_H^{\#,0}$: 
\begin{subequations}\label{sigma-n-eu-1-2-123}
\begin{align}
\eqref{sigma-n-eu-1-2} 
&= -  \int_{\Gamma[\bs X_H]} [{\bs\sigma}(\bu,p)(\grad_{\Gamma(t)} {\bs e}_{\bs X}^{\#,0})  {\bs n} ]\circ {\bs Z}^1 \, \cdot {\bs e}_{\bs u} \notag\\
&\quad\, - \int_0^1 \int_0^1 \frac{\d}{\d\alpha} \int_{\Gamma[\bs X_H]} [{\bs\sigma}(\bu,p)(\grad_{\Gamma[{\bs X}_H^{\#,\alpha\theta}]} {\bs e}_{\bs X}^{\#,\alpha\theta})  {\bs n}_{\Gamma[{\bs X}_H^{\#,\alpha\theta}]} ] \circ {\bs Y}_H^{\#,\alpha\theta}\, \cdot {\bs e}_{\bs u}\d\alpha \d\theta \notag\\
&= - \int_{\Gamma[\bs X]} [{\bs\sigma}(\bu,p)(\grad_{\Gamma(t)} {\bs e}_{\bs X}^{\#,0})  {\bs n} ]\, \cdot {\bs e}_{\bs u} \label{sigma-n-eu-1-2-1} \\
&\quad\, 
- \int_0^1\frac{\d}{\d\alpha}\int_{\Gamma[{\bs X}_H^{\#,\alpha}]} [{\bs\sigma}(\bu,p)(\grad_{\Gamma(t)} {\bs e}_{\bs X}^{\#,0})  {\bs n} ]\circ {\bs Z}^\alpha \, \cdot {\bs e}_{\bs u} \d\alpha 
\label{sigma-n-eu-1-2-2}\\
&\quad\, - \int_0^1 \int_0^1 \frac{\d}{\d\alpha} \int_{\Gamma[\bs X_H]} [{\bs\sigma}(\bu,p)(\grad_{\Gamma[{\bs X}_H^{\#,\alpha\theta}]} {\bs e}_{\bs X}^{\#,\alpha\theta})  {\bs n}_{\Gamma[{\bs X}_H^{\#,\alpha\theta}]} ] \circ {\bs Y}_H^{\#,\alpha\theta}\, \cdot {\bs e}_{\bs u}\d\alpha \d\theta 
\label{sigma-n-eu-1-2-3}
\end{align}
\end{subequations}
where \eqref{sigma-n-eu-1-2-1} can be estimated using integration by parts: 
\begin{align}\label{sigma-n-eu-1-2-1e1}
|\eqref{sigma-n-eu-1-2-1}|
&\le C\|{\bs e}_{\bs X}^{\#,0}\|_{H^{\frac12}(\Gamma(t))}\|{\bs e}_{\bs u}\|_{H^{\frac12}(\Gamma(t))} \notag\\
&\le C\|{\bs e}_{\bs X}^{\#,0} \circ {\bs X}\|_{H^{\frac12}(\widehat\Gamma)}\|{\bs e}_{\bs u}\|_{H^{\frac12}(\Gamma(t))} \notag\\
&\le C(\|{\bs e}_{\bs d}\circ\widehat{\bs\Phi}_H^{-1}\|_{H^{\frac12}(\widehat\Gamma)} + H^k)\|{\bs e}_{\bs u}\|_{H^{\frac12}(\Gamma(t))} 
\quad\mbox{(here \eqref{def-ej-theta} is used)} \notag\\
&\le C(\|{\bs e}_{\bs d}\circ\widehat{\bs\Phi}_H^{-1}\|_{H^1(\widehat\Omega^{\rm s})} + H^k)
    \|{\bs e}_{\bs u}\|_{H^1(\Omega^{\rm f}_{\rm ext}[{\bs X}_H])} 
\quad\mbox{(trace inequality)} \notag\\
&\le \varepsilon\|{\bs e}_{\bs u}\|_{H^1(\Omega^{\rm f}_{\rm ext}[{\bs X}_H])}^2 
     + C_{\varepsilon}\|{\bs e}_{\bs d}\|_{H^1(\widehat{\Omega}^{\rm s}_H)}^2 + C_{\varepsilon} H^{2k} ,
\end{align} 
where the second-to-last inequality uses the trace inequality in Appendix \ref{appendix:trace_inequlaity}. Moreover, \eqref{sigma-n-eu-1-2-2} can be estimated as follows after applying \eqref{dt-int-Gamma-1}--\eqref{dt-int-Gamma-1c}:
\begin{align}\label{sigma-n-eu-1-2-1e2}
|\eqref{sigma-n-eu-1-2-2}|
&\le C\int_0^1 \|\grad_{\Gamma(t)} {\bs e}_{\bs X}^{\#,0}\|_{L^2(\Gamma(t))} \|\grad_{\Gamma[{\bs X}_H^{\#,\alpha}]} {\bs e}_{\bs X}^{\#,\alpha}\|_{L^4(\Gamma[{\bs X}_H^{\#,\alpha}])} \|{\bs e}_{\bs u}\|_{L^4(\Gamma[{\bs X}_H^{\#,\alpha}])} \d\alpha\notag\\
&\quad\, 
+ C\int_0^1\|\grad_{\Gamma(t)} {\bs e}_{\bs X}^{\#,0}\|_{L^2(\Gamma(t))} \| {\bs e}_{\bs X}^{\#,\alpha}\|_{L^\infty(\Gamma[{\bs X}_H^{\#,\alpha}])} \|\grad {\bs e}_{\bs u}\|_{L^2(\Gamma[{\bs X}_H^{\#,\alpha}])} \d\alpha\notag\\
&\le 
\varepsilon\|\grad {\bs e}_{\bs u}\|_{L^2(\Omega^{\rm f}_{\rm ext}[{\bs X}_H])}^2
+ C_{\varepsilon} \|\grad {\bs e}_\dep \|_{L^2(\widehat{\Omega}^{\rm s}_H)}^2  + C_{\varepsilon} H^{2k} ,
\end{align}
where the last inequality uses 
\begin{align*}
&\|{\bs e}_{\bs u}\|_{L^4(\Gamma[{\bs X}_H^{\#,\alpha}])} 
\leq C\|{\bs e}_{\bs u}\|_{H^{\frac12}(\Gamma[{\bs X}_H^{\#,\alpha}])} 
\leq C\|{\bs e}_{\bs u}\|_{H^1(\Omega^{\rm f}_{\rm ext}[{\bs X}_H])} \notag\\ 
&\| \grad_{\Gamma[{\bs X}_H^{\#,\alpha}]}  {\bs e}_{\bs X}^{\#,\alpha}\|_{L^4(\Gamma[{\bs X}_H^{\#,\alpha}])} 
\le C\|\grad {\bs e}_{\bs d}\|_{L^4(\widehat\Gamma_H)} + CH^k \notag\\
&\hspace{125pt}\le CH^{-1}\|\grad {\bs e}_{\bs d}\|_{L^2(\widehat{\Omega}^{\rm s}_H)} + CH^k 
\le CH^{\frac12} \quad\mbox{(here \eqref{ind-1-3} is used)}\\
&\| \grad_{\Gamma[{\bs X}_H^{\#,\alpha}]}  {\bs e}_{\bs X}^{\#,\alpha}\|_{L^2(\Gamma[{\bs X}_H^{\#,\alpha}])} 
\le C\|\grad {\bs e}_{\bs d}\|_{L^2(\widehat\Gamma_H)} + CH^k 
\le CH^{-\frac12}\|\grad {\bs e}_{\bs d}\|_{L^2(\widehat{\Omega}^{\rm s}_H)} + CH^k \\
&\| {\bs e}_{\bs X}^{\#,\alpha}\|_{L^\infty(\Gamma[{\bs X}_H^{\#,\alpha}])} \le C\| {\bs e}_{\bs d}\|_{L^\infty(\widehat\Gamma_H)} + CH^k \le Ch^{\frac12}H^{\frac12} \quad\mbox{(here \eqref{ind-1-1} is used)}\\
&\|\grad {\bs e}_{\bs u} \|_{L^2(\Gamma[{\bs X}_H^{\#,\alpha}])} 
\le Ch^{-\frac12}\|\grad {\bs e}_{\bs u} \|_{L^2(\Omega^{\rm f}_{\rm ext}[{\bs X}_H])} \quad\mbox{(inverse trace estimate)} . 
\end{align*} 
The term \eqref{sigma-n-eu-1-2-3} can be estimated similarly after applying \eqref{dt-int-Gamma-1}--\eqref{dt-int-Gamma-1c}. This yields the following result: 
\begin{align}\label{sigma-n-eu-1-2-1e3}
|\eqref{sigma-n-eu-1-2-3}|
&\le 
\varepsilon\|\grad {\bs e}_{\bs u}\|_{L^2(\Omega^{\rm f}_{\rm ext}[{\bs X}_H])}^2 
+ C_{\varepsilon}  \|\grad {\bs e}_\dep \|_{L^2(\widehat{\Omega}^{\rm s}_H)}^2 + C_{\varepsilon} H^{2k} .
\end{align}
By substituting \eqref{sigma-n-eu-1-2-1e1}--\eqref{sigma-n-eu-1-2-1e3} into \eqref{sigma-n-eu-1-2-123}, we obtain 
\begin{align}\label{sigma-n-eu-1-2-est}
|\eqref{sigma-n-eu-1-2}|
&\le \varepsilon\|{\bs e}_{\bs u}\|_{H^1(\Omega^{\rm f}_{\rm ext}[{\bs X}_H])}^2 
+ C_{\varepsilon}  \|{\bs e}_\dep \|_{H^1(\widehat{\Omega}^{\rm s}_H)}^2 + C_{\varepsilon} H^{2k} .
\end{align}
Then, substituting \eqref{sigma-n-eu-1-1-est}--\eqref{sigma-n-eu-1-3-est} and \eqref{sigma-n-eu-1-2-est} into \eqref{sigma-n-eu-1-123}, we obtain 
\begin{align}\label{sigma-n-eu-1-f}
|\eqref{sigma-n-eu-1}| 
\le 
\varepsilon\|\grad {\bs e}_{\bs u}\|_{L^2(\Omega^{\rm f}_{\rm ext}[{\bs X}_H])}^2 
+ C_{\varepsilon} (\| {\bs e}_\dep \|_{H^1(\widehat{\Omega}^{\rm s}_H)}^2
+ \|{\bs e}_{\bs u} \|_{L^2(\Omega^{\rm f}_{\rm ext}[{\bs X}_H])}^2 ) + C_{\varepsilon} H^{2k} .
\end{align}
Note that \eqref{sigma-n-eu-2} can be estimated similarly as \eqref{sigma-n-eu-1}, with ${\bs e}_{\bs X}^{\#,0}= {\bs X}_H\circ \widehat{\bs\Phi}_H^{-1} - {\bs X}$ replaced by ${\bs e}_{\bs X}^*= {\bs X}_H^*\circ \widehat{\bs\Phi}_H^{-1} - {\bs X} $ in the estimate. Since $\|{\bs e}_{\bs X}^*\|_{H^1(\widehat{\Omega}^{\rm s}_H)}\le CH^k$, this leads to the following result: 
\begin{align}\label{sigma-n-eu-2-f}
|\eqref{sigma-n-eu-2}|
\le 
\varepsilon\|\grad {\bs e}_{\bs u}\|_{L^2(\Omega^{\rm f}_{\rm ext}[{\bs X}_H])}^2 
+ C_{\varepsilon} \|{\bs e}_{\bs u} \|_{L^2(\Omega^{\rm f}_{\rm ext}[{\bs X}_H])}^2  + C_{\varepsilon} H^{2k} .
\end{align}
Moreover, by denoting $({\bf P}_{\Gamma[\bs X_H]}{\bs e}_{\bs u})^\theta$, $({\bf P}_{\Gamma[\bs X_H]}{\bs e}_{\bs u})^{*,\theta}$ and $({\bf P}_{\Gamma[\bs X_H]}{\bs e}_{\bs u})^{\#,\theta}$ the push-forward of ${\bf P}_{\Gamma[\bs X_H]}{\bs e}_{\bs u}$ to the surfaces $\Gamma[\bs X_H^\theta]$ and $\Gamma[\bs X_H^{*,\theta}]$ and $\Gamma[\bs X_H^{\#,\theta}]$, respectively, so that 
\begin{subequations}\label{relation-0-1}
\begin{align}
({\bf P}_{\Gamma[\bs X_H]}{\bs e}_{\bs u})^{0}
&=({\bf P}_{\Gamma[\bs X_H]}{\bs e}_{\bs u})^{*,1}
\quad\mbox{and}\quad 
{\widehat{\bs\sigma}}_H^{*|0}={\widehat{\bs\sigma}}_H^{*,1}
&&\mbox{on}\,\,\,\Gamma[{\bs X}_H^*] ,\\
({\bf P}_{\Gamma[\bs X_H]}{\bs e}_{\bs u})^{*,0}
&=({\bf P}_{\Gamma[\bs X_H]}{\bs e}_{\bs u})^{\#,0} &&\mbox{on}\,\,\,\Gamma(t) , 
\end{align}
\end{subequations}
we can estimate \eqref{sigma-n-eu-3} as follows: 
\begin{subequations}\label{sigma-n-eu-3-to} 
\begin{align} 
\eqref{sigma-n-eu-3} 
&= \int_{\Gamma[\bs X_H]} {\widehat{\bs\sigma}}_H^{*|1}\cdot {\bs\dot{\bs e}}_\dep^1  
- \int_{\Gamma[\bs X_H^*]} {\widehat{\bs\sigma}}_H^{*|0} \cdot {\bs\dot{\bs e}}_\dep^0 \notag\\
&= \int_0^1 \frac{\d}{\d\theta} \int_{\Gamma[\bs X_H^\theta]} {\widehat{\bs\sigma}}_H^{*|\theta}\cdot {\bs\dot{\bs e}}_\dep^\theta \,\d\theta \notag\\
&= \int_0^1 \int_{\Gamma[\bs X_H^\theta]} {\widehat{\bs\sigma}}_H^{*|\theta}\cdot {\bs\dot{\bs e}}_\dep^\theta \,\grad_{\Gamma[\bs X_H^\theta]}\cdot {\bs e}_\dep^\theta\, \d\theta \notag\\
&= \int_0^1 \int_{\Gamma[\bs X_H^\theta]} {\widehat{\bs\sigma} }_H^{*|\theta}\cdot ({\bs\dot{\bs e}}_\dep^\theta - ({\bf P}_{\Gamma[\bs X_H]}{\bs e}_{\bs u})^\theta )\, \grad_{\Gamma[\bs X_H^\theta]}\cdot {\bs e}_\dep^\theta\, \d\theta \notag \\
&\quad\, 
+ \int_0^1 \int_{\Gamma[\bs X_H^\theta]} {\widehat{\bs\sigma} }_H^{*|\theta} \cdot ({\bf P}_{\Gamma[\bs X_H]}{\bs e}_{\bs u})^\theta\, \grad_{\Gamma[\bs X_H^\theta]}\cdot {\bs e}_\dep^\theta\, \d\theta \notag\\
&= \int_{\Gamma[\bs X_H^*]} {\widehat{\bs\sigma}}_H^{*|0} \cdot ({\bs\dot{\bs e}}_\dep^0 - ({\bf P}_{\Gamma[\bs X_H]}{\bs e}_{\bs u})^0 )\, \grad_{\Gamma[\bs X_H^*]}\cdot {\bs e}_\dep^0  \label{sigma-n-eu-3-t1} \\
&\quad\, + \int_0^1 \int_0^\theta \int_{\Gamma[\bs X_H^{\alpha}]} {\widehat{\bs \sigma} }_H^{*|\alpha} \cdot ({\bs\dot{\bs e}}_\dep^\alpha - ({\bf P}_{\Gamma[\bs X_H]}{\bs e}_{\bs u})^\alpha ) 
O(|\grad_{\Gamma[\bs X_H^{\alpha}]}\cdot {\bs e}_\dep^\alpha |^2) \, \d\alpha \d\theta  \label{sigma-n-eu-3-t2}  \\ 
&\quad\, 
+ \int_{\Gamma[\bs X_H^*]} {\widehat{\bs\sigma} }_H^{*|0} \cdot ({\bf P}_{\Gamma[\bs X_H]}{\bs e}_{\bs u})^0\, \grad_{\Gamma[\bs X_H^*]}\cdot {\bs e}_\dep^0  \label{sigma-n-eu-3-t3}  \\
&\quad\, 
+ \int_0^1\int_0^\theta \int_{\Gamma[\bs X_H^{\alpha}]} {\widehat{\bs\sigma}}_H^{*|\alpha} \cdot ({\bf P}_{\Gamma[\bs X_H]}{\bs e}_{\bs u})^\alpha \, O(|\grad_{\Gamma[\bs X_H^\alpha]}\cdot {\bs e}_\dep^\alpha |^2) \, \d\alpha \d\theta  \label{sigma-n-eu-3-t4} .
\end{align} 
\end{subequations} 
By using the notation in \eqref{notation-sigma} and the relation in \eqref{relation-0-1}, 
\begin{align}
\eqref{sigma-n-eu-3-t1} 
&= \int_{\Gamma[\bs X_H^{*,1}]} {\widehat{\bs\sigma}}_H^{*,1} \cdot ({\bs\dot{\bs e}}_\dep^{*,1} - ({\bf P}_{\Gamma[\bs X_H]}{\bs e}_{\bs u})^{*,1})\, \grad_{\Gamma[\bs X_H^*]}\cdot {\bs e}_\dep^{*,1} \notag\\
&= \int_{\Gamma[\bs X]} {\widehat{\bs\sigma}}_H^{*,0} \cdot [{\bs\dot{\bs e}}_\dep^{*,0} - ({\bf P}_{\Gamma[\bs X_H]}{\bs e}_{\bs u})^{*,0}] \grad_{\Gamma[\bs X]}\cdot {\bs e}_\dep^{*,0} 
\quad\mbox{(fundamental theorem of calculus)} \notag\\
&\quad\, + \int_0^1 \int_{\Gamma[\bs X_H^{*,\theta}]} {\widehat{\bs \sigma} }_H^{*,\theta} \cdot [{\bs\dot{\bs e}}_\dep^{*,\theta} - ({\bf P}_{\Gamma[\bs X_H]}{\bs e}_{\bs u})^{*,\theta} ] (\grad_{\Gamma[\bs X_H^{*,\theta}]}\cdot {\bs e}_\dep^{*,\theta})
O(|\grad_{\Gamma[\bs X_H^{*,\theta}]} {\bs e}_{\bs X}^*|) \, \d\theta  \notag\\
&\hspace{200pt} \mbox{(the formulas in \eqref{dt-int-Gamma-1}--\eqref{dt-int-Gamma-1c} are used)} \notag\\
&\lesssim \|{\bs\dot{\bs e}}_\dep^{*,0} - ({\bf P}_{\Gamma[\bs X_H]}{\bs e}_{\bs u})^{*,0}\|_{H^{\frac12}(\Gamma[\bs X])} \| {\bs e}_\dep^{*,0} \|_{H^{\frac12}(\Gamma[\bs X])}  \notag\\
&\quad\, 
+ \int_0^1 \|{\bs\dot{\bs e}}_\dep^{*,\theta} - ({\bf P}_{\Gamma[\bs X_H]}{\bs e}_{\bs u})^{*,\theta}\|_{L^2(\Gamma[\bs X_H^{*,\theta}])} \| {\bs e}_\dep^{*,\theta} \|_{H^1(\Gamma[\bs X_H^{*,\theta}])} H^k \d\theta \notag\\
&\le C[\| {\bs e}_\dep \|_{H^{\frac12}(\widehat\Gamma_H)}
+ H^{-\frac12}(h^{k+1}+H^{k+1})]\| {\bs e}_\dep \|_{H^{\frac12}(\widehat\Gamma_H)} \notag\\
&\quad\, 
+ C (\| {\bs e}_\dep \|_{L^2(\widehat\Gamma_H)} + h^{k+1}+H^{k+1}) \| {\bs e}_\dep \|_{H^1(\widehat\Gamma_H)} 
\quad\mbox{(here \eqref{eu-dot-ed-Lp-H} is used)} \notag\\
&\le C\| {\bs e}_\dep \|_{H^1(\widehat{\Omega}^{\rm s}_H)}^2 + C(h^{2k}+H^{2k})\quad\mbox{(under the condition $h^2\le H$)}, \\[10pt]
\eqref{sigma-n-eu-3-t2} 
&\le 
C\int_0^1 \|{\bs\dot{\bs e}}_\dep^\alpha - ({\bf P}_{\Gamma[\bs X_H]}{\bs e}_{\bs u})^\alpha\|_{L^\infty(\Gamma[\bs X_H^{\alpha}])} \| {\bs e}_\dep^\alpha \|_{H^1(\Gamma[\bs X_H^{\alpha}])} ^2 \d\alpha \notag\\
&\le C \| {\bs e}_\dep \|_{L^\infty(\widehat\Gamma_H)} \| {\bs e}_\dep \|_{H^1(\widehat\Gamma_H)}^2 
\quad\mbox{(here \eqref{eu-dot-ed-Lp-H} is used)} \notag\\
&\le CH \| {\bs e}_\dep \|_{H^1(\widehat\Gamma_H)}^2 \quad\mbox{(here \eqref{ind-1-1} is used)} \notag\\
&\le C\| {\bs e}_\dep \|_{H^1(\widehat{\Omega}^{\rm s}_H)}^2 \quad\mbox{(inverse trace estimate)},\\[10pt]
\eqref{sigma-n-eu-3-t3} 
&= \int_{\Gamma[\bs X_H^*]} {\widehat{\bs\sigma} }_H^{*,1} \cdot ({\bf P}_{\Gamma[\bs X_H]}{\bs e}_{\bs u})^{*,1}\, \grad_{\Gamma[\bs X_H^*]}\cdot {\bs e}_\dep^{*,1} 
\quad\mbox{(here relation \eqref{relation-0-1} is used)}\notag\\
&= \int_{\Gamma[\bs X]} {\widehat{\bs\sigma}}_H^{*,0} \cdot ({\bf P}_{\Gamma[\bs X_H]}{\bs e}_{\bs u})^{*,0} \grad_{\Gamma[\bs X]}\cdot {\bs e}_\dep^{*,0} 
\quad \mbox{(fundamental theorem of calculus)} \notag\\
&\quad\, + \int_0^1 \int_{\Gamma[\bs X_H^{*,\theta}]} {\widehat{\bs \sigma} }_H^{*,\theta} \cdot ({\bf P}_{\Gamma[\bs X_H]}{\bs e}_{\bs u})^{*,\theta} (\grad_{\Gamma[\bs X_H^{*,\theta}]}\cdot {\bs e}_\dep^{*,\theta})
O(|\grad_{\Gamma[\bs X_H^{*,\theta}]} {\bs e}_{\bs X}^*|) \, \d\theta  \notag\\
&\hspace{165pt} \mbox{(the formulas in \eqref{dt-int-Gamma-1}--\eqref{dt-int-Gamma-1c} are used)} \notag\\
&\le C\|({\bf P}_{\Gamma[\bs X_H]}{\bs e}_{\bs u})^{*,0}\|_{H^{\frac12}(\Gamma[\bs X])} \| {\bs e}_\dep^{*,0} \|_{H^{\frac12}(\Gamma[\bs X])} 
\quad\mbox{(integration by parts on $\Gamma(t)$)} \notag\\
&\quad\, 
+ C\int_0^1 \|({\bf P}_{\Gamma[\bs X_H]}{\bs e}_{\bs u})^{*,\theta}\|_{L^2(\Gamma[\bs X_H^{*,\theta}])} 
\| {\bs e}_\dep^{*,\theta} \|_{H^1(\Gamma[\bs X_H^{*,\theta}])} H^k \d\theta \notag\\
&\le C\|({\bf P}_{\Gamma[\bs X_H]}{\bs e}_{\bs u})^{*,0}\|_{H^{\frac12}(\Gamma[\bs X])} \| {\bs e}_\dep^{*,0} \|_{H^{\frac12}(\Gamma[\bs X])} 
\quad\mbox{(norm equiv + inverse estimate)} \notag\\
&\le C\| {\bs e}_{\bs u}\|_{H^1(\Omega^{\rm f}_{\rm ext}[\bs  X_H])} \| {\bs e}_\dep\|_{H^1(\widehat{\Omega}^{\rm s}_H)} 
\quad\mbox{(here \eqref{H12-PXeu} is used)}\notag\\
&\le \varepsilon\| {\bs e}_{\bs u}\|_{H^1(\Omega^{\rm f}_{\rm ext}[\bs  X_H])}^2
+ C_{\varepsilon} \| {\bs e}_\dep\|_{H^1(\widehat{\Omega}^{\rm s}_H)}^2 , \\
\eqref{sigma-n-eu-3-t4} 
&\le 
C\int_0^1 \|({\bf P}_{\Gamma[\bs X_H]}{\bs e}_{\bs u})^\alpha\|_{L^4(\Gamma[\bs X_H^{\alpha}])} \| {\bs e}_\dep^\alpha \|_{H^1(\Gamma[\bs X_H^{\alpha}])} \| {\bs e}_\dep^\alpha \|_{W^{1,4}(\Gamma[\bs X_H^{\alpha}])}  \d\alpha \notag\\
&\le C\|({\bf P}_{\Gamma[\bs X_H]}{\bs e}_{\bs u})^{\#,0}\|_{L^4(\Gamma[\bs X])}\| {\bs e}_\dep \|_{H^1(\widehat\Gamma_H)} H^{\frac12}
\quad\mbox{(use \eqref{ind-1-3} and inverse estimate)}\notag\\
&\le C\|({\bf P}_{\Gamma[\bs X_H]}{\bs e}_{\bs u})^{\#,0}\|_{H^{\frac12}(\Gamma[\bs X])}\| {\bs e}_\dep \|_{H^{\frac12}(\widehat\Gamma_H)}  \quad\mbox{(Sobolev embedding)}\notag\\
&\le \varepsilon\| {\bs e}_{\bs u}\|_{H^1(\Omega^{\rm f}_{\rm ext}[\bs  X_H])}^2
+ C_{\varepsilon} \| {\bs e}_\dep\|_{H^1(\widehat{\Omega}^{\rm s}_H)}^2 \quad\mbox{(here \eqref{H12-PXeu} is used)} .
\end{align}
Substituting these estimates into \eqref{sigma-n-eu-3-to}, and then substituting the obtained estimate of \eqref{sigma-n-eu-3} as well as \eqref{sigma-n-eu-1-f}--\eqref{sigma-n-eu-2-f} into \eqref{sigma-n-eu}, we obtain
\begin{align*} 
&|\eqref{sigma-gradv-b} + \eqref{error-eq-k} + \eqref{error-eq-l} | \notag\\ 
&\le 
\varepsilon\|\grad {\bs e}_{\bs u}\|_{L^2(\Omega^{\rm f}_{\rm ext}[{\bs X}_H])}^2 
+ C_{\varepsilon} (\|{\bs e}_\dep\|_{H^1(\widehat{\Omega}^{\rm s}_H)}^2+ \|{\bs e}_{\bs u} \|_{L^2(\Omega^{\rm f}_{\rm ext}[{\bs X}_H])}^2 ) + C_{\varepsilon} (h^{2k}+H^{2k}) .
\end{align*}
This, together with the estimates of \eqref{sigma-gradv-a} and \eqref{sigma-gradv-c} in \eqref{sigma-gradv-a-est}--\eqref{sigma-gradv-c-est}, yields the following result: 
\begin{align*}
& | \eqref{error-eq-g} + \eqref{error-eq-k} + \eqref{error-eq-l} | \notag\\ 
&\le 
\varepsilon\|\grad {\bs e}_{\bs u}\|_{L^2(\Omega^{\rm f}_{\rm ext}[{\bs X}_H])}^2 
+ C_{\varepsilon} (\|{\bs e}_\dep\|_{H^1(\widehat{\Omega}^{\rm s}_H)}^2+ \|{\bs e}_{\bs u} \|_{L^2(\Omega^{\rm f}_{\rm ext}[{\bs X}_H])}^2 ) 
+ C_{\varepsilon} (h^{2k}+H^{2k}) .
\end{align*}

\item
\begin{align}
| \eqref{error-eq-h} |
&= \Big| \Big( \int_{\Omega^{\rm f}[\bs X_H]} - \int_{\Omega^{\rm f}[\bs X_H^*]} \Big) e_p \grad\cdot {\bs u}_h^* \Big| \notag\\
&= \Big|\int_0^1 \int_{\Gamma[\bs X_H^\theta]} e_p  \grad\cdot {\bs u}_h^*\, ({\bs e}_\dep^\theta\cdot {\bs n}_{\Gamma[{\bs X}_H^\theta]}) \, \d\theta \Big| \notag\\
&= \Big|\int_0^1 \int_{\Gamma[\bs X_H^\theta]} e_p  [\grad\cdot ({\bs u}_h^*-{\bs u}) + \grad\cdot{\bs u} ] ({\bs e}_\dep^\theta\cdot {\bs n}_{\Gamma[{\bs X}_H^\theta]}) \,  \d\theta \Big| \notag\\ 
&\le \int_0^1 Ch^k \| e_p \|_{L^2(\Gamma[\bs X_H^\theta])} \|{\bs e}_\dep ^\theta \|_{L^2(\Gamma[\bs X_H^\theta])} \d\theta \notag\\
&\le C h^{k-\frac12} \| e_p \|_{L^2(\Omega^{\rm f}_{\rm ext}[{\bs X}_H])} \| {\bs e}_\dep\|_{H^1(\widehat{\Omega}^{\rm s}_H)} 
\quad\mbox{(norm equivalence)} \notag\\
&\le  \varepsilon h^{2k-1}\| e_p \|_{L^2(\Omega^{\rm f}_{\rm ext}[{\bs X}_H])}^2  
+ C_{\varepsilon} \| {\bs e}_\dep\|_{H^1(\widehat{\Omega}^{\rm s}_H)}^2 ,
\end{align}
where the second-to-last inequality follows from the finite element inverse trace inequality in the fluid region, and the trace inequality in the solid region. Therefore, by using the estimate of $\| e_p \|_{0,\Omega[{\bs X}_H]}$ in \eqref{ep-esigma-estimates} (which is stronger than $\| e_p \|_{L^2(\Omega^{\rm f}_{\rm ext}[{\bs X}_H])}$ here), we have 
\begin{align}
\int_0^s \eqref{error-eq-h} \d t 
&\le  
\varepsilon h^{2k-1}
\big( \|\grad{\bs e}_{\bs d}(s)\|_{L^2(\widehat\Omega^{\rm s}_H)}^2 + \|\partial_t{\bs e}_{\bs d}(s)\|_{L^2(\widehat\Omega^{\rm s}_H)}^2 \big)\notag\\
&\quad\, 
+ C_\varepsilon \int_0^s (h^{2k-3}+h^{2k-2}H^{-1})\| {\bs e}_{\bs u}  \|_{1,\Omega[{\bs X}_H]}^2  \d t \notag\\
&\quad\, 
+ C_\varepsilon \int_0^s (h^{2k-2}H^{-1}+h^{2k-1}H^{-2}+1)\|{\bs e}_\dep\|_{H^1(\widehat{\Omega}^{\rm s}_H)}^2 \d t \notag\\
&\quad\, 
+ C_\varepsilon \int_0^s h^{2k-1}H^{-1}\|\partial_t{\bs e}_\dep\|_{L^2(\widehat{\Omega}^{\rm s}_H)}^2 \d t \notag\\
&\quad\, 
+ C_\varepsilon (h^{2k} + H^{-1}h^{4k-1} + H^{2k}) . 
&\end{align}

\item
\begin{align}
\eqref{error-eq-i} 
& =  \int_{\widehat{\Omega}^{\rm s}_H} \partial_{tt} {\bs e}_{\dep} \cdot \partial_t{\bs e}_{\dep}
+ \int_{\widehat{\Omega}^{\rm s}_H} [{\bs D}({\bs e}_\dep) : {\bs D}(\partial_t{\bs e}_{\dep}) + (\grad\cdot {\bs e}_\dep)(\grad\cdot\partial_t{\bs e}_{\dep}) ] \notag\\
& = \frac{\d}{\d t}\Big(\frac12 \|\partial_t{\bs e}_{\dep}\|_{L^2(\widehat{\Omega}^{\rm s}_H)}^2 
+ \frac12 \|{\bs D}({\bs e}_\dep)\|_{L^2(\widehat{\Omega}^{\rm s}_H)}^2 + \frac12 \|\grad\cdot{\bs e}_\dep\|_{L^2(\widehat{\Omega}^{\rm s}_H)}^2 \Big) .
\end{align}

\item
\eqref{error-eq-j} is cancelled by \eqref{error-eq-m}, i.e., \eqref{error-eq-j}+\eqref{error-eq-m} = 0.

\item 
Let ${\bs e}_{{\bs\sigma}}^\theta={\bs e}_{{\bs\sigma}}\circ({\bs X}_H^\theta)^{-1}$ and $\vel_H^{*|\theta}=\vel_H^{*}\circ({\bs X}_H^\theta)^{-1}$. Then 
\begin{subequations}\label{error-eq-n-abcd}
\begin{align}
\eqref{error-eq-n}
& = \int_{\Gamma[\bs X_H]} {\bs e}_{{\bs\sigma}}^1 \cdot \big( \bu_h^*  - \vel_H^{*|1} \big) - \int_{\Gamma[\bs X_H^*]} {\bs e}_{{\bs\sigma}}^0 \cdot \big( \bu_h^* - \vel_H^{*|0} \big)  \notag \\
& = \int_0^1 \int_{\Gamma[\bs X_H^\theta]} \Big( 
{\bs e}_{{\bs\sigma}}^\theta \cdot \big({\bs e}_\dep^\theta\cdot\grad \bu_h^*\big) 
+ {\bs e}_{{\bs\sigma}}^\theta \cdot \big( \bu_h^*  - \vel_H^{*|\theta} \big) \grad_{\Gamma[\bs X_H^\theta]}\cdot {\bs e}_\dep^\theta \, \Big) \d\theta  \notag\\
& = \int_0^1 \int_{\Gamma[\bs X_H^\theta]} 
{\bs e}_{{\bs\sigma}}^\theta \cdot \big( {\bs e}_\dep^\theta\cdot\grad {\bs u} \big) \d\theta 
+ \int_0^1 \int_{\Gamma[\bs X_H^\theta]} 
{\bs e}_{{\bs\sigma}}^\theta \cdot \big( {\bs e}_\dep^\theta\cdot\grad ({\bs u}_h^* - {\bs u}) \big) \d\theta \notag\\
&\quad\, 
+ \int_0^1 \int_{\Gamma[\bs X_H^\theta]} {\bs e}_{{\bs\sigma}}^\theta \cdot \big( \bu_h^*  - \vel_H^{*|\theta} \big) \grad_{\Gamma[\bs X_H^\theta]}\cdot {\bs e}_\dep^\theta \,  \d\theta  \notag\\
& = \int_{\Gamma[\bs X_H]} {\bs e}_{{\bs\sigma}}^1 \cdot \big( {\bs e}_\dep^1\cdot\grad {\bs u} \big) 
\label{error-eq-n-a}\\
&\quad\, -  \int_0^1 \int_{\theta}^1\int_{\Gamma[\bs X_H^\alpha]} 
{\bs e}_{{\bs\sigma}}^\alpha \cdot \big[{\bs e}_\dep^\alpha\cdot ( {\bs e}_\dep^\alpha\cdot\grad^2 {\bs u} ) 
+( {\bs e}_\dep^\alpha\cdot\grad {\bs u} ) (\grad_{\Gamma[\bs X_H^\alpha]}\cdot {\bs e}_\dep^\alpha)\big] \d\alpha \d\theta
\label{error-eq-n-b}\\ 
&\quad\, 
+ \int_0^1 \int_{\Gamma[\bs X_H^\theta]} 
{\bs e}_{{\bs\sigma}}^\theta \cdot \big( {\bs e}_\dep^\theta\cdot\grad ({\bs u}_h^* - {\bs u}) \big) \d\theta 
\label{error-eq-n-c}\\
&\quad\, 
+ \int_0^1 \int_{\Gamma[\bs X_H^\theta]}  {\bs e}_{{\bs\sigma}}^\theta \cdot \big( \bu_h^*  - \vel_H^{*|\theta} \big) \grad_{\Gamma[\bs X_H^\theta]}\cdot {\bs e}_\dep^\theta \,  \d\theta .
\label{error-eq-n-d} 
\end{align}
\end{subequations}

Before estimating \eqref{error-eq-n-a}, we introduce the following notations: 
\begin{align*}
{\bs w}&={\bs e}_\dep\cdot[(\grad {\bs u})\circ{\bs X}_H],\qquad {\bs w}^*={\bs e}_\dep\cdot[(\grad {\bs u})\circ{\bs X}_H^*] \\
{\bs w}_H&={\bs I}_H^{\rm s} {\bs w}
\qquad\mbox{and}\qquad {\bs w}_H^*={\bs I}_H^{\rm s} {\bs w}^* \,\,\,\mbox{on}\,\,\,\widehat\Omega^{\rm s}_H\supset\widehat\Gamma_H.
\end{align*}
Then 
\begin{align}\label{w-wH-L2}
\|{\bs w}-{\bs w}_H\|_{L^2(\widehat\Gamma_H)}
&\le
\|{\bs w}-{\bs w}^*\|_{L^2(\widehat\Gamma_H)}
+\|{\bs w}^*-{\bs w}_H^*\|_{L^2(\widehat\Gamma_H)}
+\|{\bs w}_H^* - {\bs w}_H\|_{L^2(\widehat\Gamma_H)} \notag\\ 
&\lesssim \|{\bs e}_\dep\|_{L^4(\widehat\Gamma_H)}^2 
+ H^{2} \|{\bs e}_\dep\|_{H^1(\widehat\Gamma_H)} 
+ \|{\bs e}_\dep\|_{L^4(\widehat\Gamma_H)}^2  \notag\\
&\lesssim \|{\bs e}_\dep\|_{H^1(\widehat{\Omega}^{\rm s}_H)}^2 + H^{\frac32} \|{\bs e}_\dep\|_{H^1(\widehat{\Omega}^{\rm s}_H)} + \|{\bs e}_\dep\|_{H^1(\widehat{\Omega}^{\rm s}_H)}^2 \notag\\
&\lesssim H^{\frac32} \|{\bs e}_\dep\|_{H^1(\widehat{\Omega}^{\rm s}_H)} 
\quad\mbox{(here \eqref{ind-1-3} is used)}, 
\end{align}
where the second-to-last inequality uses the following super-approximation result with ${\bs\phi}=(\grad {\bs u})\circ{\bs X}_H^*$ in estimating $\|{\bs w}^*-{\bs w}_H^*\|_{L^2(\widehat\Gamma_H)}$:
\begin{align}\label{super-approx}
\|{\bs e}_\dep \cdot {\bs\phi} - {\bs I}_H^{\rm s} \big( {\bs e}_\dep \cdot {\bs\phi}\big)\|_{L^2(\widehat\Gamma_H)} 
\le CH^2\|{\bs e}_\dep \|_{H^1(\widehat\Gamma_H)} \|{\bs\phi}\|_{W^{2,\infty}_{\rm piecewise}(\widehat\Gamma_H)} 
\end{align}
where $\|\cdot\|_{W^{2,\infty}_{\rm piecewise}(\widehat\Gamma_H)}$ denotes the piecewise $W^{2,\infty}$ norm on each curved-triangular element of $\widehat\Gamma_H$. This super-approximation result can be proved as follows, by considering the estimate on a curved-triangular element $K\subset\widehat\Gamma_H$ (where $c_K$ and ${\bs\phi}_K$ denote the average value of ${\bs e}_{\bs d}$ and ${\bs\phi}$ on $K$, respectively):
\begin{align*}
&\|{\bs e}_\dep \cdot {\bs\phi} - {\bs I}_H^{\rm s} \big( {\bs e}_\dep \cdot {\bs\phi}\big)\|_{L^2(K)} \\
&\le \|({\bs e}_\dep-c_K) \cdot {\bs\phi} - {\bs I}_H^{\rm s} \big( ({\bs e}_\dep-c_K) \cdot {\bs\phi}\big)\|_{L^2(K)} 
+ \|c_K \cdot \big({\bs\phi} - {\bs I}_H^{\rm s} {\bs\phi}\big)\|_{L^2(K)} \\
&= \|({\bs e}_\dep-c_K) \cdot ({\bs\phi} - {\bs\phi}_K) - {\bs I}_H^{\rm s} \big( ({\bs e}_\dep-c_K) \cdot ({\bs\phi} - {\bs\phi}_K)\big)\|_{L^2(K)}  \\
&\quad\,
+ \|c_K \cdot \big({\bs\phi} - {\bs I}_H^{\rm s} {\bs\phi}\big)\|_{L^2(K)} \\
&\le CH(\|{\bs e}_\dep-c_K \|_{H^1(K)} \|{\bs\phi}-{\bs\phi}_K\|_{L^{\infty}(K)}
+ \|{\bs e}_\dep-c_K \|_{L^2(K)} \|{\bs\phi}-{\bs\phi}_K\|_{W^{1,\infty}(K)}) \\
&\quad\, + CH^2|c_K|\|{\bs\phi}\|_{W^{2,\infty}(K)}\\
&\le CH^2(\|{\bs e}_\dep \|_{H^1(K)} \|{\bs\phi}\|_{W^{1,\infty}(K)} 
+\|{\bs e}_\dep \|_{L^2(K)} \|{\bs\phi}\|_{W^{2,\infty}(K)}). 
\end{align*}
This proves the super-approximation result in \eqref{super-approx}. 

Moreover, since the nodal values of ${\bs w}_H={\bs I}_H^{\rm s} {\bs w}={\bs I}_H^{\rm s}\big({\bs e}_\dep\cdot[(\grad {\bs u})\circ{\bs X}_H]\big)$ are bounded by the nodal values of ${\bs e}_\dep$, it follows that 
\begin{subequations}\label{wH-L2-dtwH-L2}
\begin{align}
\|{\bs w}_H\|_{L^2(\widehat\Omega_H^{\rm s})} 
\lesssim 
\|{\bs e}_\dep\|_{L^2(\widehat\Omega_H^{\rm s})} .
\end{align}
Similarly,
\begin{align}
\|\partial_t{\bs w}_H\|_{L^2(\widehat\Omega_H^{\rm s})} 
\lesssim 
\|\partial_t{\bs e}_\dep\|_{L^2(\widehat\Omega_H^{\rm s})} 
+ \|{\bs e}_\dep\|_{L^2(\widehat\Omega_H^{\rm s})} .
\end{align}
\end{subequations}

In order to use the estimate in \eqref{w-wH-L2}, we rewrite \eqref{error-eq-n-a} into the following form by using relation \eqref{esigma-test}: 
\begin{subequations}
\begin{align}
\int_0^s \eqref{error-eq-n-a} \d t
=\, & \int_0^s \int_{\Gamma[\bs X_H]} {\bs e}_{{\bs\sigma}}^1 \cdot {\bs w}_H^1 \d t 
+ \int_0^s \int_{\Gamma[\bs X_H]} {\bs e}_{{\bs\sigma}}^1 \cdot ({\bs w}^1 - {\bs w}_H^1) \d t  \notag\\
=\, & \int_0^s \int_{\Gamma[\bs X_H]} {\bs e}_{{\bs\sigma}}^{\#,1} \cdot {\bs w}_H^{\#,1} \d t 
+ \int_0^s \int_{\Gamma[\bs X_H]} {\bs e}_{{\bs\sigma}}^1 \cdot ({\bs w}^1 - {\bs w}_H^1) \d t  \notag\\
=\, & -\int_0^s \int_{\widehat{\Omega}^{\rm s}_H} \partial_{tt} {\bs e}_{\dep} \cdot {\bs w}_H \d t \label{error-eq-n-a-a}\\
& 
- \int_0^s \int_{\widehat{\Omega}^{\rm s}_H} {\bs D}({\bs e}_\dep):{\bs D}(\bw_H) \d t- \int_0^s\int_{\widehat{\Omega}^{\rm s}_H} (\grad\cdot{\bs e}_\dep) (\grad\cdot \bw_H) \d t\label{error-eq-n-a-b}\\ 
& - \int_0^s\int_{\Gamma[\bs X_H]} {\widehat{\bs\sigma}}_H^{*|1}\cdot \bw_H^1 + \int_{\Gamma[\bs X_H^*]} {\widehat{\bs\sigma}}_H^{*|0} \cdot {\bs w}_H^0 \d t \label{error-eq-n-a-c} \\
& - \int_0^s R_{h,H}^*(0, 0, {\bs w}_H , 0) \d t \label{error-eq-n-a-d} \\
& + \int_0^s \int_{\Gamma[\bs X_H]} {\bs e}_{{\bs\sigma}}^1\cdot ({\bs w}^1 - {\bs w}_H^1) \d t 
\label{error-eq-n-a-e} , 
\end{align}
\end{subequations}
where \eqref{error-eq-n-a-a} can be estimated as follows by using \eqref{wH-L2-dtwH-L2}:
\begin{align}
|\eqref{error-eq-n-a-a}|
&= \Big| - \int_0^s \frac{\d}{\d t} \int_{\widehat{\Omega}^{\rm s}_H}\partial_t{\bs e}_{\dep} \cdot {\bs w}_H\d t + \int_0^s \int_{\widehat{\Omega}^{\rm s}_H} \partial_t{\bs e}_{\dep} \cdot \partial_t{\bs w}_H \d t \Big| \notag \\
&\le \|\partial_t{\bs e}_{\dep}(s)\|_{L^2(\widehat{\Omega}^{\rm s}_H)} \|{\bs w}_H(s)\|_{L^2(\widehat{\Omega}^{\rm s}_H)} 
+ \int_0^s \|\partial_t{\bs e}_{\dep}(t)\|_{L^2(\widehat{\Omega}^{\rm s}_H)} \| \partial_t{\bs w}_H(t)\|_{L^2(\widehat{\Omega}^{\rm s}_H)} \d t \notag \\
&\le \|\partial_t{\bs e}_{\dep}(s)\|_{L^2(\widehat{\Omega}^{\rm s}_H)}  \|{\bs e}_{\dep}(s)\|_{L^2(\widehat{\Omega}^{\rm s}_H)}
+ C\int_0^s (\|\partial_t{\bs e}_{\dep}(t)\|_{L^2(\widehat{\Omega}^{\rm s}_H)}^2 + \|{\bs e}_{\dep}(t)\|_{L^2(\widehat{\Omega}^{\rm s}_H)}^2) \d t 
\end{align}
and
\begin{align}
&|\eqref{error-eq-n-a-b}+\eqref{error-eq-n-a-c}+\eqref{error-eq-n-a-d}+\eqref{error-eq-n-a-e}| \notag\\ 
&\le  \int_0^s (\varepsilon H^2\|{\bs e}_{\bs\sigma}\|_{H^{-\frac12}(\Gamma[{\bs X}_H])}^2
+ C_{\varepsilon}\|{\bs e}_{\dep}\|_{H^1(\widehat{\Omega}^{\rm s}_H)}^2 ) \d t + C_{\varepsilon}(h^{2k}+H^{2k}), 
\end{align}
where the estimate of \eqref{error-eq-n-a-e} is obtained by using \eqref{w-wH-L2}. 
This proves that 
\begin{subequations} \label{error-eq-n-abcd-estimate}
\begin{align}\label{error-eq-n-a-estimate}
\int_0^s \eqref{error-eq-n-a} \d t
&\le \int_0^s (\varepsilon H^2\|{\bs e}_{\bs\sigma}\|_{H^{-\frac12}(\Gamma[{\bs X}_H])}^2
+ C_{\varepsilon}\|{\bs e}_{\dep}\|_{H^1(\widehat{\Omega}^{\rm s}_H)}^2) \d t \notag\\
&\quad\, 
+C\|\partial_t{\bs e}_{\dep}(s)\|_{L^2(\widehat{\Omega}^{\rm s}_H)}  \|{\bs e}_{\dep}(s)\|_{L^2(\widehat{\Omega}^{\rm s}_H)} \notag\\
&\quad\, 
+ C\int_0^s (\|\partial_t{\bs e}_{\dep}\|_{L^2(\widehat{\Omega}^{\rm s}_H)}^2 + \|{\bs e}_{\dep}\|_{L^2(\widehat{\Omega}^{\rm s}_H)}^2) \d t \notag\\
&\quad\, + C(h^{2k}+H^{2k}) \notag\\ 
&\le 
\varepsilon
\big( \|\grad{\bs e}_{\bs d}(s)\|_{L^2(\widehat\Omega^{\rm s}_H)}^2 + \|\partial_t{\bs e}_{\bs d}(s)\|_{L^2(\widehat\Omega^{\rm s}_H)}^2 \big) \notag\\
&\quad\, + \varepsilon \int_0^s ( \|{\bs e}_{\bs u}\|_{H^1(\Omega^{\rm f}_{\rm ext}[{\bs X}_H])}^2 + \|{\bs e}_\dep\|_{H^1(\widehat{\Omega}^{\rm s}_H)}^2 ) \d t \notag\\
&\hspace{50pt}\mbox{(where we use \eqref{ep-esigma-estimates} under condition $H\lesssim h$)} \notag\\
&\quad\, 
+ C_\varepsilon \int_0^s (\|\partial_t{\bs e}_{\dep}\|_{L^2(\widehat{\Omega}^{\rm s}_H)}^2 + \|{\bs e}_{\dep}\|_{L^2(\widehat{\Omega}^{\rm s}_H)}^2) \d t \notag\\
&\quad\, + C_\varepsilon (h^{2k}+H^{2k}) .
\end{align}
Moreover, since $\|{\bs e}_{\dep}\|_{H^1(\widehat\Gamma_H)}\lesssim H^{-\frac12}\|{\bs e}_{\dep}\|_{H^1(\widehat\Omega^{\rm s}_H)}\lesssim H$ as shown in \eqref{ind-1-3}, it follows that, by applying the estimate of $\|{\bs e}_{\bs\sigma}^1\|_{L^4(\Gamma[{\bs X}_H])} $ in \eqref{esigma-L4} and the equivalence of norms $\|{\bs e}_{\bs\sigma}^\alpha\|_{L^4(\Gamma[{\bs X}_H^\alpha])} $ for $\alpha\in[0,1]$, 
\begin{align}
\int_0^s \eqref{error-eq-n-b} \d t
&\lesssim \int_0^s \|{\bs e}_{\bs\sigma}^1\|_{L^4(\Gamma[{\bs X}_H])} \|{\bs e}_{\dep}\|_{L^4(\widehat\Gamma_H)}  \|{\bs e}_{\dep}\|_{H^1(\widehat\Gamma_H)} \d t \notag\\
&\lesssim \int_0^s \|{\bs e}_{\bs\sigma}^1\|_{L^4(\Gamma[{\bs X}_H])} \|{\bs e}_{\dep}\|_{H^{\frac12}(\widehat\Gamma_H)} H \d t \notag\\
&\le \int_0^s  \varepsilon H^2\|\partial_{tt}{\bs e}_{\dep}\|_{L^2(\widehat{\Omega}^{\rm s}_H)}^2 \d t
+ C_\varepsilon\int_0^s \|{\bs e}_{\dep}\|_{H^1(\widehat{\Omega}^{\rm s}_H)}^2 \d t 
+ C_\varepsilon (h^{2k}+H^{2k}) \notag\\
&\le 
\varepsilon
\big( \|\partial_t{\bs e}_{\bs d}(s)\|_{L^2(\widehat\Omega^{\rm s}_H)}^2 + \|\grad{\bs e}_{\bs d}(s)\|_{L^2(\widehat\Omega^{\rm s}_H)}^2 \big) + \int_0^s \varepsilon \|{\bs e}_{{\bs u}}\|_{1,\Omega[{\bs X}_H]}^2 \d t  \notag\\
&\quad\, 
+ C_\varepsilon \int_0^s (\|{\bs e}_{\dep}\|_{H^1(\widehat{\Omega}^{\rm s}_H)}^2 +\|\partial_t{\bs e}_{\dep}\|_{L^2(\widehat{\Omega}^{\rm s}_H)}^2)\d t 
+ C_\varepsilon (h^{2k}+H^{2k}) ,  \label{error-eq-n-b-estimate}
\end{align}
where the last inequality uses \eqref{ep-esigma-estimates}. Similarly, 
\begin{align}
\int_0^s \eqref{error-eq-n-c} \d t
&\lesssim \int_0^s h^k \|{\bs e}_{\bs\sigma}^1\|_{L^2(\Gamma[{\bs X}_H])} \|{\bs e}_{\dep}\|_{L^2(\widehat\Gamma_H)} \d t\notag\\
&\lesssim \int_0^s H \|{\bs e}_{\bs\sigma}^1\|_{L^4(\Gamma[{\bs X}_H])} \|{\bs e}_{\dep}\|_{H^1(\widehat{\Omega}^{\rm s}_H)} \d t
\quad\mbox{(under the condition $h^k\le H$)} \notag\\
&\le 
\varepsilon
\big( \|\partial_t{\bs e}_{\bs d}(s)\|_{L^2(\widehat\Omega^{\rm s}_H)}^2 + \|\grad{\bs e}_{\bs d}(s)\|_{L^2(\widehat\Omega^{\rm s}_H)}^2 \big) + \int_0^s \varepsilon \|{\bs e}_{{\bs u}}\|_{1,\Omega[{\bs X}_H]}^2 \d t  \notag\\
&\quad\, 
+ C_\varepsilon \int_0^s (\|{\bs e}_{\dep}\|_{H^1(\widehat{\Omega}^{\rm s}_H)}^2 +\|\partial_t{\bs e}_{\dep}\|_{L^2(\widehat{\Omega}^{\rm s}_H)}^2)\d t 
+ C_\varepsilon (h^{2k}+H^{2k}) ,  
\end{align}
where the last inequality is the same as \eqref{error-eq-n-b-estimate}. Similarly to \eqref{u-d-on-Gamma1}, we have 
\begin{align}\label{u-d-on-Gammatheta}
|{\bs u}_h^* - \vel_H^{*|\theta}| 
&\lesssim h^{k+1} + |{\bs e}_{\bs d}^\theta| + H^{k+1} 
\quad\mbox{on}\,\,\,\Gamma[{\bs X}_H^\theta] 
\end{align}
and therefore 
\begin{align*}
\|{\bs u}_h^*  - \vel_H^{*|\theta}\|_{L^4(\Gamma[{\bs X}_H^\theta])} 
&\le C\| {\bs e}_\dep\|_{L^4(\widehat\Gamma_H)} + C(h^{k+1}+H^{k+1}) .
\end{align*}
By using this result, we can estimate \eqref{error-eq-n-d} similarly as \eqref{error-eq-n-b} and \eqref{error-eq-n-c}. This would lead to the following result:  
\begin{align}
\int_0^s \eqref{error-eq-n-d} \d t
&\le 
\varepsilon
\big( \|\partial_t{\bs e}_{\bs d}(s)\|_{L^2(\widehat\Omega^{\rm s}_H)}^2 + \|\grad{\bs e}_{\bs d}(s)\|_{L^2(\widehat\Omega^{\rm s}_H)}^2 \big) + \int_0^s \varepsilon \|{\bs e}_{{\bs u}}\|_{1,\Omega[{\bs X}_H]}^2 \d t  \notag\\
&\quad\, 
+ C_\varepsilon \int_0^s (\|{\bs e}_{\dep}\|_{H^1(\widehat{\Omega}^{\rm s}_H)}^2 +\|\partial_t{\bs e}_{\dep}\|_{L^2(\widehat{\Omega}^{\rm s}_H)}^2)\d t 
+ C_\varepsilon (h^{2k}+H^{2k}) .  
\end{align}
\end{subequations}
Therefore, by substituting the estimates in \eqref{error-eq-n-abcd-estimate} into \eqref{error-eq-n-abcd}, we obtain 
\begin{align}
\int_0^s\eqref{error-eq-n} \d t
&\le 
\varepsilon
\big( \|\partial_t{\bs e}_{\bs d}(s)\|_{L^2(\widehat\Omega^{\rm s}_H)}^2 + \|\grad{\bs e}_{\bs d}(s)\|_{L^2(\widehat\Omega^{\rm s}_H)}^2 \big) + \int_0^s \varepsilon \|{\bs e}_{{\bs u}}\|_{1,\Omega[{\bs X}_H]}^2 \d t  \notag\\
&\quad\, 
+ C_\varepsilon \int_0^s (\|{\bs e}_{\dep}\|_{H^1(\widehat{\Omega}^{\rm s}_H)}^2 +\|\partial_t{\bs e}_{\dep}\|_{L^2(\widehat{\Omega}^{\rm s}_H)}^2)\d t 
+ C_\varepsilon (h^{2k}+H^{2k})  . 
\end{align}

We have estimated all the terms in \eqref{error-eq-a}--\eqref{error-eq-n}. By integrating \eqref{error-eq} over time interval $[0,s]$ and using the estimates for the terms in \eqref{error-eq-a}--\eqref{error-eq-n}, as well as the  consistency estimate in Lemma \ref{Lemma:RhH-2} for estimating \eqref{error-eq-p}, we have 
\begin{align} \label{aux2422}
&\| {\bs e}_{\bs u}(s) \|_{L^2(\Omega^{\rm f}[\bs X_H(s)])}^2 
+ \int_0^s \| {\bs D}({\bs e}_{\bs u}) \|_{L^2(\Omega^{\rm f}[\bs X_H])}^2 \d t \notag\\
&\quad\, 
+ \|\partial_t{\bs e}_\dep(s)\|_{L^2(\widehat{\Omega}^{\rm s}_H)}^2 
+ \|\grad {\bs e}_\dep(s)\|_{L^2(\widehat{\Omega}^{\rm s}_H)}^2 \notag\\
&\quad\, 
+ \epsilon_h g_h^{\rm u} ({\bs e}_{\bs u}(s), {\bs e}_{\bs u}(s)) 
+ \int_0^s [g_h^{\rm u} ({\bs e}_{\bs u}, {\bs e}_{\bs u} ) + g_h^{\rm p} (e_p, e_p )] \d t\notag\\ 
&\le \varepsilon \big(\|{\bs e}_{\bs u}(s)\|_{L^2(\Omega^{\rm f}[\bs X_H(s)])}^2 + \|\partial_t{\bs e}_\dep(s)\|_{L^2(\widehat{\Omega}^{\rm s}_H)}^2 + \|\grad{\bs e}_\dep(s)\|_{L^2(\widehat{\Omega}^{\rm s}_H)}^2 \big) \notag\\
&\quad\, 
+ \varepsilon \epsilon_h^2 g^{\rm u}_h({\bs e}_{\bs u}(s),{\bs e}_{\bs u}(s)) + \varepsilon \int_0^s \|{\bs e}_{\bs u}\|_{1, \Omega[{\bs X}_H]}^2 \d t \notag\\
&\quad\,
+ C_\varepsilon \int_0^s \Big( {\|{\bs e}_{\bs u}\|_{L^2(\Omega^{\rm f}[{\bs X}_H])}}^2 +  \|\partial_t{\bs e}_{\dep}\|_{L^2(\widehat{\Omega}^{\rm s}_H)}^2 + \|{\bs e}_\dep\|_{H^1(\widehat{\Omega}^{\rm s}_H)}^2  + h^2 g_h^{\rm u}({\bs e}_{\bs u},{\bs e}_{\bs u}) \Big) \d t \notag\\
&\quad\, 
+ C_\varepsilon(h^{2k}+H^{2k}) ,
\end{align}
where we have converted the extended-domain terms in \eqref{eq:estima1}, \eqref{eq:estima3} and \eqref{sigma-gradv-a-est}
into physical-domain quantities by using \eqref{0-0-norm}, i.e., 
$$
\begin{aligned}
\|{\bs e}_{\bs u}\|_{H^1(\Omega^{\rm f}_{\rm ext}[{\bs X}_H])}^2
\lesssim & \|{\bs e}_{\bs u}\|_{1,\Omega[{\bs X}_H]}^2 , \\
\|{\bs e}_{\bs u}\|_{L^2(\Omega^{\rm f}_{\rm ext}[{\bs X}_H])}^2
\le &  \|{\bs e}_{\bs u}\|_{L^2(\Omega^{\rm f}[{\bs X}_H])}^2
+ h^2\, g_h^{\rm u}({\bs e}_{\bs u},{\bs e}_{\bs u}) .
\end{aligned}
$$
For sufficiently small $\varepsilon$ and $h$, the terms $\varepsilon \epsilon_h^2 g^{\rm u}_h({\bs e}_{\bs u}(s),{\bs e}_{\bs u}(s))$, $\varepsilon\int_0^s\|{\bs e}_{\bs u}\|_{1,\Omega[{\bs X}_H]}^2\,\d t$ and $C_\varepsilon h^2\int_0^s g_h^{\rm u}({\bs e}_{\bs u},{\bs e}_{\bs u})\,\d t$ can be absorbed into the left-hand side of \eqref{aux2422}, while the term 
$C_\varepsilon\int_0^s\|{\bs e}_{\bs u}\|_{L^2(\Omega^{\rm f}[{\bs X}_H])}^2\,\d t$
can be retained on the right-hand side and handled by Gronwall's inequality below. Since
\begin{align*}
&\int_0^s \|{\bs e}_{\bs d}\|_{L^2(\widehat{\Omega}^{\rm s}_H)}^2 \d t \notag\\
&= s\|{\bs e}_{\bs d}(0)\|_{L^2(\widehat{\Omega}^{\rm s}_H)}^2
+ \int_0^s \int_0^t\frac{\d}{\d t'}\|{\bs e}_{\bs d}(t')\|_{L^2(\widehat{\Omega}^{\rm s}_H)}^2 \d t'\d t  \notag\\
&\le  s\|{\bs e}_{\bs d}(0)\|_{L^2(\widehat{\Omega}^{\rm s}_H)}^2
+ \int_0^s \int_0^t2\|{\bs e}_{\bs d}(t')\|_{L^2(\widehat{\Omega}^{\rm s}_H)}
\|\partial_t{\bs e}_{\bs d}(t')\|_{L^2(\widehat{\Omega}^{\rm s}_H)}  \d t'\d t\notag\\
&\le  s\|{\bs e}_{\bs d}(0)\|_{L^2(\widehat{\Omega}^{\rm s}_H)}^2
+ \int_0^s \left(\frac12\|{\bs e}_{\bs d}(t')\|_{L^2(\widehat{\Omega}^{\rm s}_H)}^2+
2(s-t')^2\|\partial_t{\bs e}_{\bs d}(t')\|_{L^2(\widehat{\Omega}^{\rm s}_H)}^2\right)\d t' , 
\end{align*}
it follows that, after absorbing the term $\int_0^s\frac12\|{\bs e}_{\bs d}(t')\|_{L^2(\widehat{\Omega}^{\rm s}_H)}^2\d t'$ into the left-hand side, 
\begin{align*}
&\int_0^s \|{\bs e}_{\bs d}\|_{L^2(\widehat{\Omega}^{\rm s}_H)}^2 \d t 
\le 2s\|{\bs e}_{\bs d}(0)\|_{L^2(\widehat{\Omega}^{\rm s}_H)}^2
+ 4s^2\int_0^s \|\partial_t{\bs e}_{\bs d}\|_{L^2(\widehat{\Omega}^{\rm s}_H)}^2 \d t 
\end{align*}
and therefore (since ${\bs e}_{\bs d}(0)=0$) 
\begin{align*}
&\int_0^s \|{\bs e}_{\bs d}\|_{H^1(\widehat{\Omega}^{\rm s}_H)}^2 \d t 
\le 4s^2\int_0^s \|\partial_t{\bs e}_{\bs d}\|_{L^2(\widehat{\Omega}^{\rm s}_H)}^2 \d t
+ \int_0^s \|\grad{\bs e}_{\bs d}\|_{L^2(\widehat{\Omega}^{\rm s}_H)}^2 \d t .
\end{align*}
By substituting this result into \eqref{aux2422}, choosing a sufficiently small $\varepsilon$, and using Gronwall's inequality, we obtain 
\begin{align}\label{Bound1}
& \max_{t\in[0,t_*]} \Big(\| {\bs e}_{\bs u} \|_{L^2(\Omega^{\rm f}[\bs X_H])}^2 + \|\partial_t{\bs e}_\dep\|_{L^2(\widehat{\Omega}^{\rm s}_H)}^2 
+ \|\grad {\bs e}_\dep\|_{L^2(\widehat{\Omega}^{\rm s}_H)}^2 
+ \epsilon_h g_h^{\rm u} ({\bs e}_{\bs u}, {\bs e}_{\bs u})  \Big) \notag\\
&\quad\,
+ \int_0^{t_*} \| {\bs e}_{\bs u}\|_{1, \Omega[\bs X_H]}^2 \d t \notag\\
&\le C_1(h^{2k}+H^{2k}), 
\end{align}
with $t_*$ defined in \eqref{ind-1}, and $C_1$ is some constant which is independent of $t_*$ (but may depend on $T$, as explained in the beginning of Section \ref{sect:interp}). The estimates for $\|{\bs e}_{\bs u}(s)\|_{1,\Omega[{\bs X}_H]}$, $\| e_p \|_{0,\Omega[{\bs X}_H]}$ and $\| {\bs e}_{{\bs \sigma} } \|_{H^{-\frac12}(\widehat\Gamma_H)}$ follow from substituting \eqref{Bound1} into \eqref{ep-esigma-estimates}, i.e., 
\begin{align}\label{Bound-ep-esigma}
H^2\max_{t\in[0,t_*]}\|{\bs e}_{\bs u}(t)\|_{1,\Omega[{\bs X}_H]}^2
+H^2\int_0^{t_*} \big( \| e_p \|_{0,\Omega[{\bs X}_H]}^2 
+  \| {\bs e}_{{\bs \sigma} } \|_{H^{-\frac12}(\widehat\Gamma_H)}^2 \big) \d t 
\le C_2(h^{2k}+H^{2k}) ,
\end{align}
\end{itemize}
where $C_2$ is some constant which is independent of $t_*$ (but may depend on $T$, as explained in the beginning of Section \ref{sect:interp}).

In view of the discussions in Section~\ref{section:outline}, this proves $t_*=T$ and therefore the estimates in \eqref{Bound1}--\eqref{Bound-ep-esigma} hold up to time $T$. After incorporating the interpolation error estimates, e.g., the results in \eqref{interpl-error-d}--\eqref{interpl-u}, we obtain the error estimate in Theorem \ref{THM} and complete the proof. 

It remains to prove the consistency estimate in Lemma \ref{Lemma:RhH-2}, which is used in deriving \eqref{aux2422}; see the text above \eqref{aux2422}.

\section{Consistency estimates: Part II}\label{section:consistency-2}

Let $({\bs u}_h^\#,p_h^\#)$ be the Ritz projection of $({\bs u},p)$ onto the finite element space $ {\bs V}_h^{\rm f} \times Q_h^{\rm f}$, determined by the following weak formulation (under the zero boundary condition ${\bs u}_h^\#=0$ on $\Gamma^{\rm f}$): Find $({\bs u}_h^\#,p_h^\#)\in {\bs V}_h^{\rm f} \times Q_h^{\rm f}$ such that
\begin{align}\label{def-Ritz}
& \int_{\Omega^{\rm f}[\bs X_H^*]} \bs D({\bs u}-{\bs u}_h^\#): \bs D(\bv_h)  - \int_{\Omega^{\rm f}[\bs X_H^*]} ( p - p_h^\# ) \grad\cdot\bv_h - \int_{\Omega^{\rm f}[\bs X_H^*]}  q_h \grad\cdot {\bs u}_h^\# \notag\\
& + g_h^{\rm u} ({\bs u}-{\bs u}_h^\#,\bv_h ) 
+ g_h^{\rm p} (p - p_h^\#,q_h ) = 0 ,
\end{align}
for all test functions $({\bs v}_h,q_h)\in {\bs V}_h^{\rm f} \times Q_h^{\rm f}$ (satisfying ${\bs v}_h=0$ on $\Gamma^{\rm f}$).  

\begin{lemma}\label{Lemma:Ritz}
The following error estimates hold for $t\in[0,T]$: 
\begin{subequations}
\begin{align}
\max_{t\in[0,T]} \big( \|{\bs u}-{\bs u}_h^\#\|_{1,\Omega[\bs X_H]} 
+ \|p-p_h^\#\|_{0,\Omega[\bs X_H]} \big)
&\le Ch^k , \label{Ritz-H1}\\
\max_{t\in[0,T]} \big( \|\partial_t{\bs u}-\partial_t{\bs u}_h^\#\|_{1,\Omega[\bs X_H]} 
+ \|\partial_tp-\partial_tp_h^\#\|_{0,\Omega[\bs X_H]} \big)
&\le Ch^{k-1} ,\label{Ritz-dt-H1}\\
\max_{\theta\in[0,1]} \|\partial_t{\bs u}-\partial_t{\bs u}_h^\#\|_{L^2(\Omega^{\rm f}[{\bs X}_H^{\#,\theta}])} 
\lesssim \|{\bs e}_{\bs d}\|_{H^1(\widehat\Omega^{\rm s}_H)}
+ h^k + H^k .  \label{Ritz-L2}
\end{align}
\end{subequations}
\end{lemma}
\begin{proof}
The $H^1$ error estimate in \eqref{Ritz-H1} can be obtained by choosing ${\bs v}_h={\bs u}_h^*-{\bs u}_h^\#$ and $q_h=p_h^*-p_h^\#$, using the results in \eqref{H1ext}--\eqref{Korn-inequlaity}, and applying the inf-sup condition to estimate $p_h^*-p_h^\#$. The details are omitted here.

The $H^1$ error estimate in \eqref{Ritz-dt-H1} can be obtained by taking time derivative of \eqref{def-Ritz}, which yields  
\begin{align}\label{dt-Ritz}
& \int_{\Omega^{\rm f}[\bs X_H^*]} \bs D(\partial_t{\bs u}-\partial_t{\bs u}_h^\#): \bs D(\bv_h) 
+ \int_{\Gamma[\bs X_H^*]} \bs D({\bs u}-{\bs u}_h^\#): \bs D(\bv_h) {\bs V}_{\Gamma[\bs X_H^*]}\cdot{\bs n}_{\Gamma[\bs X_H^*]} \notag\\
&\quad\, - \int_{\Omega^{\rm f}[\bs X_H^*]} ( \partial_tp - \partial_tp_h^\# ) (\grad\cdot\bv_h) 
- \int_{\Gamma[\bs X_H^*]} ( p - p_h^\# ) (\grad\cdot\bv_h) {\bs V}_{\Gamma[\bs X_H^*]}\cdot{\bs n}_{\Gamma[\bs X_H^*]} \notag\\
&\quad\, - \int_{\Omega^{\rm f}[\bs X_H^*]}  q_h \grad\cdot \partial_t{\bs u}_h^\# 
- \int_{\Gamma[\bs X_H^*]}  q_h \grad\cdot {\bs u}_h^\# \, {\bs V}_{\Gamma[\bs X_H^*]}\cdot{\bs n}_{\Gamma[\bs X_H^*]} \notag\\
&\quad\, + g_h^{\rm u} (\partial_t{\bs u}-\partial_t{\bs u}_h^\#,\bv_h ) + g_h^{\rm p} (\partial_tp - \partial_tp_h^\#,q_h ) = 0 ,
\end{align}
where ${\bs V}_{\Gamma[\bs X_H^*]}=\partial_t{\bs d}_H^*\circ({\bs X}_H^*)^{-1}$ is the velocity of the surface $\Gamma[\bs X_H^*]$. 
Choosing ${\bs v}_h=\partial_t{\bs u}_h^*-\partial_t{\bs u}_h^\#$ and $q_h=\partial_tp_h^*-\partial_tp_h^\#$, and applying the finite element inverse trace estimate for the integrals on $\Gamma[{\bs X}_H^*]$ (this produces a factor of $h^{-1}$) and \eqref{Ritz-H1}, we obtain  \eqref{Ritz-dt-H1}. The details are omitted. 

The $L^2$ error estimate in \eqref{Ritz-L2} follows via a duality argument. To this end, we define ${\bs v}$ to be the solution of the following continuous problem: 
\begin{align}\label{def-v-q}
\left\{
\begin{aligned}
-\grad\cdot({\bs D}(\bv) + q{\bs I}) &= \partial_t{\bs u}-\partial_t{\bs u}_h^\# &&\mbox{in}\,\,\,\Omega^{\rm f}(t) \\
\grad\cdot{\bs v}&=0 &&\mbox{in}\,\,\,\Omega^{\rm f}(t) \\
({\bs D}(\bv) + q{\bs I}){\bs n} &= \bs 0 &&\mbox{on}\,\,\Gamma(t) \\
{\bs v}&= \bs 0 &&\mbox{on}\,\,\Gamma^{\rm f} \\
\end{aligned}
\right.
\end{align}
which has the $H^2(\Omega^{\rm f}(t))^d\times H^1(\Omega^{\rm f}(t))$ regularity result as assumed in \eqref{def-v-q-0}--\eqref{H2-H1-estimate-Stokes}. We extend $({\bs v},q)$ from $H^2(\Omega^{\rm f}(t))^d\times H^1(\Omega^{\rm f}(t))$ to $H^2(\Omega)^d\times H^1(\Omega)$ boundedly. 
Testing the first equation of \eqref{def-v-q} with $\partial_t{\bs u}-\partial_t{\bs u}_h^\#$, we have
\begin{align}\label{dt-u-Ritz-L22}
\int_{\Omega^{\rm f}(t)} \grad(\partial_t{\bs u}-\partial_t{\bs u}_h^\#) : ({\bs D}(\bv) + q{\bs I}) 
= \|\partial_t{\bs u}-\partial_t{\bs u}_h^\#\|_{L^2(\Omega^{\rm f}(t))}^2 .
\end{align}
By using this relation, we can reformulate \eqref{dt-Ritz} as follows: 
\begin{subequations}\label{dt-Ritz-rf}
\begin{align}
& \int_{\Omega^{\rm f}[\bs X_H^*]} {\bs D}(\partial_t{\bs u}-\partial_t{\bs u}_h^\#): {\bs D}({\bs v}) \label{dt-Ritz-a}\\
&\quad\, 
+ \int_{\Omega^{\rm f}[\bs X_H^*]} {\bs D}(\partial_t{\bs u}-\partial_t{\bs u}_h^\#): {\bs D}({\bs v}_h-{\bs v}) \label{dt-Ritz-b}\\
&\quad\, 
+ \int_{\Gamma[\bs X_H^*]} {\bs D}({\bs u}-{\bs u}_h^\#): {\bs D}({\bs v}) {\bs V}_{\Gamma[\bs X_H^*]}\cdot{\bs n}_{\Gamma[\bs X_H^*]} \label{dt-Ritz-c}\\
&\quad\, 
+ \int_{\Gamma[\bs X_H^*]} {\bs D}({\bs u}-{\bs u}_h^\#): {\bs D}({\bs v}_h-{\bs v}) {\bs V}_{\Gamma[\bs X_H^*]}\cdot{\bs n}_{\Gamma[\bs X_H^*]} \label{dt-Ritz-d}\\
&\quad\, 
- \int_{\Omega^{\rm f}[\bs X_H^*]} ( \partial_tp - \partial_tp_h^\# ) \grad\cdot{\bs v} \label{dt-Ritz-e}\\
&\quad\, 
- \int_{\Omega^{\rm f}[\bs X_H^*]} ( \partial_tp - \partial_tp_h^\# ) \grad\cdot({\bs v}_h-{\bs v}) \label{dt-Ritz-f}\\
&\quad\, 
- \int_{\Gamma[\bs X_H^*]} ( p - p_h^\# ) \grad\cdot{\bs v}\, {\bs V}_{\Gamma[\bs X_H^*]}\cdot{\bs n}_{\Gamma[\bs X_H^*]} \label{dt-Ritz-g}\\ 
&\quad\, 
- \int_{\Gamma[\bs X_H^*]} ( p - p_h^\# ) \grad\cdot({\bs v}_h-{\bs v})\, {\bs V}_{\Gamma[\bs X_H^*]}\cdot{\bs n}_{\Gamma[\bs X_H^*]} \label{dt-Ritz-h}\\
&\quad\, 
- \int_{\Omega^{\rm f}[\bs X_H^*]}  q \grad\cdot \partial_t{\bs u}_h^\# \label{dt-Ritz-i}\\
&\quad\, 
- \int_{\Omega^{\rm f}[\bs X_H^*]}  (q_h-q) \grad\cdot \partial_t{\bs u}_h^\# \label{dt-Ritz-j}\\
&\quad\, 
- \int_{\Gamma[\bs X_H^*]} q \grad\cdot{\bs u}_h^\# \, {\bs V}_{\Gamma[\bs X_H^*]}\cdot{\bs n}_{\Gamma[\bs X_H^*]} \label{dt-Ritz-k}\\
&\quad\, 
- \int_{\Gamma[\bs X_H^*]} (q_h-q) \grad\cdot {\bs u}_h^\# \, {\bs V}_{\Gamma[\bs X_H^*]}\cdot{\bs n}_{\Gamma[\bs X_H^*]} \label{dt-Ritz-l}\\
&\quad\, + g_h^{\rm u} (\partial_t{\bs u}-\partial_t{\bs u}_h^\#,\bv_h ) + g_h^{\rm p} (\partial_tp - \partial_tp_h^\#,q_h ) \label{dt-Ritz-m} \\
&= 0 . \notag
\end{align}
\end{subequations}
Equation \eqref{dt-u-Ritz-L22} implies that 
\begin{subequations}
\begin{align}
&\quad\,\eqref{dt-Ritz-a} + \eqref{dt-Ritz-i} \notag\\
&= \|\partial_t{\bs u}-\partial_t{\bs u}_h^\#\|_{L^2(\Omega^{\rm f}(t))}^2 \label{dt-Ritz-a-i-a}\\
&\quad\, + \int_{\Omega^{\rm f}[\bs X_H^*]} {\bs D}(\partial_t{\bs u}-\partial_t{\bs u}_h^\#): {\bs D}({\bs v}) - \int_{\Omega^{\rm f}(t)} {\bs D}(\partial_t{\bs u}-\partial_t{\bs u}_h^\#): {\bs D}({\bs v}) \label{dt-Ritz-a-i-b}\\
&\quad\, + \int_{\Omega^{\rm f}[\bs X_H^*]}  q \grad\cdot(\partial_t{\bs u}-\partial_t{\bs u}_h^\#) - \int_{\Omega^{\rm f}(t)}  q \grad\cdot(\partial_t{\bs u}-\partial_t{\bs u}_h^\#) \label{dt-Ritz-a-i-c}\\
&\quad\, - \int_{\Omega^{\rm f}[\bs X_H^*]}  q \grad\cdot \partial_t{\bs u} . \label{dt-Ritz-a-i-d}
\end{align}
\end{subequations}
Moreover, \eqref{dt-Ritz-a-i-b} and \eqref{dt-Ritz-a-i-c} can be estimated as follows by using \eqref{dt-int-Omega-2} and the interpolation error estimate $\|{\bs X}_H^*-{\bs X} \circ\widehat{\bs\Phi}_H\|_{L^\infty(\widehat\Omega_H^{\rm s})}\lesssim H^{k+1}$: 
\begin{align*}
&\quad\, |\eqref{dt-Ritz-a-i-b} + \eqref{dt-Ritz-a-i-c}| \\
&\le CH^{k+1}\int_0^1 \|\grad(\partial_t{\bs u}-\partial_t{\bs u}_h^\#)\|_{L^2(\Gamma[\bs X_H^{*,\theta}])}
(\|\grad {\bs v}\|_{L^2(\Gamma[\bs X_H^{*,\theta}])}
+ \|q\|_{L^2(\Gamma[\bs X_H^{*,\theta}])}) \d\theta \\
&\le Ch^{-\frac12}H^{k+1} \|\partial_t{\bs u}-\partial_t{\bs u}_h^\#\|_{H^1(\Omega^{\rm f}_{\rm ext}[\bs X_H])}
(\|{\bs v}\|_{H^2(\Omega)}
+ \|q\|_{H^1(\Omega)}) \\
&\le Ch^{k-\frac32}H^{k+1} \|\partial_t{\bs u}-\partial_t{\bs u}_h^\#\|_{L^2(\Omega^{\rm f}(t))} ,
\end{align*}
where the second-to-last inequality uses the finite element inverse estimate and trace inequality in the extended fluid region, and the last inequality uses the result in \eqref{Ritz-dt-H1}. Moreover, \eqref{dt-Ritz-a-i-d} can be estimated similarly by using that $\grad\cdot \partial_t{\bs u} $ is arbitrarily high-order small on $\Omega^{\rm f}[\bs X_H^*]$. 
This proves that 
\begin{align}
\eqref{dt-Ritz-a} + \eqref{dt-Ritz-i} 
\ge \|\partial_t{\bs u}-\partial_t{\bs u}_h^\#\|_{L^2(\Omega^{\rm f}(t))}^2 - Ch^{k-\frac32}H^{k+1} \|\partial_t{\bs u}-\partial_t{\bs u}_h^\#\|_{L^2(\Omega^{\rm f}(t))} .
\end{align}

By choosing ${\bs v}_h$ and $q_h$ to be the $L^2$ projections of ${\bs v}$ and $q$ onto ${\bs V}_h^{\rm f}\subset H^1_0(\Omega)^d$ and $Q_h^{\rm f}$, respectively, and using \eqref{Ritz-dt-H1}, as well as that $\grad\cdot \partial_t{\bs u} $ is arbitrarily high-order small on $\Omega^{\rm f}[\bs X_H^*]$, we have
\begin{align*}
&| \eqref{dt-Ritz-b} + \eqref{dt-Ritz-f} + \eqref{dt-Ritz-j} | \\
&\le Ch(\|\partial_t{\bs u}-\partial_t{\bs u}_h^\#\|_{H^1(\Omega^{\rm f}_{\rm ext}[\bs X_H])} +\|\partial_tp - \partial_tp_h^\#\|_{L^2(\Omega^{\rm f}_{\rm ext}[\bs X_H])} + h^k) 
( \|{\bs v}\|_{H^2(\Omega)} + \|q\|_{H^1(\Omega)} )\\ 
&\le Ch^k \|\partial_t{\bs u}-\partial_t{\bs u}_h^\#\|_{L^2(\Omega^{\rm f}(t))} ,
\end{align*}
where the last inequality uses the result in \eqref{Ritz-dt-H1}.

Using \eqref{Ritz-H1} and the decompositions 
\begin{align*}
{\bs u}-{\bs u}_h^\# 
&= ({\bs u}-{\bs u}_h^*) + ({\bs u}_h^*-{\bs u}_h^\#) 
\quad\mbox{and}\quad 
p-p_h^\# = (p-p_h^*) + (p_h^*-p_h^\#) ,
\end{align*}
as well as the interpolation and projection error estimates 
(note that ${\bs v}_h$ and $q_h$ are the $L^2$ projections of ${\bs v}$ and $q$)
\begin{align*}
\|{\bs u}-{\bs u}_h^*\|_{W^{1,\infty}(\Omega^{\rm f}_{\rm ext}[{\bs X}_H])}
+ \|p-p_h^*\|_{L^{\infty}(\Omega^{\rm f}_{\rm ext}[{\bs X}_H])} &\lesssim h^k , \\
\|{\bs v}-{\bs v}_h\|_{H^1(\Gamma[{\bs X}_H^*])} +\|q-q_h\|_{L^2(\Gamma[{\bs X}_H^*])} 
&\lesssim h^{\frac12}(\|{\bs v}\|_{H^2(\Omega)}+\|q\|_{H^1(\Omega)}) , 
\end{align*}
we have 
\begin{align*}
&| \eqref{dt-Ritz-d} + \eqref{dt-Ritz-h} + \eqref{dt-Ritz-l} | \notag\\
&\le C(\|{\bs u}-{\bs u}_h^*\|_{H^1(\Gamma[{\bs X}_H^*])}+\|p-p_h^*\|_{L^2(\Gamma[{\bs X}_H^*])})h^{\frac12}(\|{\bs v}\|_{H^2(\Omega)}+\|q\|_{H^1(\Omega)})  \notag\\
&\quad\, 
+ Ch^{-\frac12} (\|{\bs u}_h^*-{\bs u}_h^\#\|_{H^1(\Omega^{\rm f}_{\rm ext}[{\bs X}_H])}+\|p_h^*-p_h^\#\|_{L^2(\Omega^{\rm f}_{\rm ext}[{\bs X}_H])})h^{\frac12}(\|{\bs v}\|_{H^2(\Omega)}+\|q\|_{H^1(\Omega)}) \notag\\
&\quad\, 
+ Ch^k(\|{\bs v}\|_{H^2(\Omega)}+\|q\|_{H^1(\Omega)}) \quad\mbox{(this term comes from $\grad\cdot \partial_t{\bs u} =O(h^k)$ on $\Gamma[\bs X_H^*]$)} \notag\\
&\le Ch^{k} (\|{\bs v}\|_{H^2(\Omega^{\rm f}(t))} +\|q\|_{H^1(\Omega^{\rm f}(t))} ) \notag\\
&\le Ch^{k} \|\partial_t{\bs u}-\partial_t{\bs u}_h^\#\|_{L^2(\Omega^{\rm f}(t))} .
\end{align*}
Since $\grad\cdot{\bs v}=0$ in $\Omega^{\rm f}(t)$, it follows that 
\begin{align*}
|\eqref{dt-Ritz-e}|
&=\Big|\int_{\Omega^{\rm f}[\bs X_H^*]} ( \partial_tp - \partial_tp_h^\# ) \grad\cdot{\bs v} - \int_{\Omega^{\rm f}[{\bs X}]} ( \partial_tp - \partial_tp_h^\# ) \grad\cdot{\bs v} \Big|\\
&= \Big|\int_0^1 \frac{\d}{\d\theta} \int_{\Omega^{\rm f}[\bs X_H^{*,\theta}]} ( \partial_tp - \partial_tp_h^\# ) \grad\cdot{\bs v}\Big| \\
&= \Big|\int_0^1 \int_{\Gamma[\bs X_H^{*,\theta}]} ( \partial_tp - \partial_tp_h^\# ) \grad\cdot{\bs v} \,\, {\bs e}_{\bs X}^{*,\theta}\cdot{\bs n}_{\Gamma[\bs X_H^{*,\theta}]} \Big|\\
&\le C\int_0^1 \|\partial_tp - \partial_tp_h^\# \|_{L^2(\Gamma[\bs X_H^{*,\theta}])} H^{k+1} \|\grad{\bs v}\|_{L^2(\Gamma[\bs X_H^{*,\theta}])} \d\theta \\
&\le C(\|\partial_tp - \partial_tp_h^* \|_{L^2(\Gamma[\bs X_H^{*,\theta}])} +h^{-\frac12} \|\partial_tp_h^* - \partial_tp_h^\# \|_{L^2(\Omega^{\rm f}_{\rm ext}[\bs X_H^{*,\theta}])} ) H^{k+1} \|{\bs v}\|_{H^2(\Omega)} \\
&\le Ch^{k-\frac32} H^{k+1} \|{\bs v}\|_{H^2(\Omega)} \\
&\le CH^k \|\partial_t{\bs u}-\partial_t{\bs u}_h^\#\|_{L^2(\Omega^{\rm f}(t))} .
\end{align*}
Similarly, since $\grad\cdot{\bs v}=0$ on $\Gamma(t)$, we can estimate $\|\grad\cdot{\bs v}\|_{L^2(\Gamma[\bs X_H^*])}$ by $\|{\bs e}_{\bs X}^*\|_{L^\infty(\widehat\Gamma)}^{\frac12}\|{\bs v}\|_{H^2(\Omega)}$. Therefore, 
\begin{align*}
|\eqref{dt-Ritz-g}|
&=\Big|\int_{\Gamma[{\bs X}_H^*]} ( p - p_h^\# ) \grad\cdot{\bs v}\, {\bs V}_{\Gamma[\bs X_H^*]}\cdot{\bs n}_{\Gamma[\bs X_H^*]} 
\Big| \\
&\lesssim \|p - p_h^\# \|_{L^2(\Gamma[\bs X_H^*])} \|\grad\cdot{\bs v}\|_{L^2(\Gamma[\bs X_H^{*}])} \\ 
&\lesssim \|p - p_h^\# \|_{L^2(\Gamma[\bs X_H^*])} \|{\bs e}_{\bs X}^*\|_{L^\infty(\widehat\Gamma)}^{\frac12} \|{\bs v}\|_{H^2(\Omega)}  \\ 
&\lesssim h^{k-\frac12} H^{\frac{k}{2}}\|{\bs v}\|_{H^2(\Omega)} \\
&\lesssim (h^k+H^k) \|{\bs v}\|_{H^2(\Omega)}
\quad\mbox{(when $k\ge 2$)} \notag\\
&\lesssim (h^k + H^k)  \|\partial_t{\bs u}-\partial_t{\bs u}_h^\#\|_{L^2(\Omega^{\rm f}(t))} . 
\end{align*}

Note that 
\begin{subequations}
\begin{align}
\eqref{dt-Ritz-c} + \eqref{dt-Ritz-k} 
& = \int_{\Gamma[\bs X_H^*]} {\rm tr}\big[ (\grad({\bs u}-{\bs u}_h^\#))^{\top} ({\bs D}({\bs v}) +q{\bs I} ) \big]{\bs V}_{\Gamma[\bs X_H^*]}\cdot{\bs n}_{\Gamma[\bs X_H^*]} \notag\\
&\quad\, 
- \int_{\Gamma[\bs X_H^*]} q \grad\cdot{\bs u} \, {\bs V}_{\Gamma[\bs X_H^*]}\cdot{\bs n}_{\Gamma[\bs X_H^*]} \notag\\ 
& = \int_{\Gamma[\bs X_H^*]} {\rm tr}\big[ (\grad_{\Gamma[\bs X_H^*]}({\bs u}-{\bs u}_h^\#))^{\top} ({\bs D}({\bs v}) +q{\bs I} ) \big]{\bs V}_{\Gamma[\bs X_H^*]}\cdot{\bs n}_{\Gamma[\bs X_H^*]} 
\label{dt-Ritz-c-k-a}\\ 
&\quad\, + \int_{\Gamma[\bs X_H^*]}  \partial_{\bs n}({\bs u}-{\bs u}_h^\#) \cdot ({\bs D}({\bs v}) +q{\bs I} ){\bs n} \, {\bs V}_{\Gamma[\bs X_H^*]}\cdot{\bs n}_{\Gamma[\bs X_H^*]} 
\label{dt-Ritz-c-k-b}\\
&\quad\, 
- \int_{\Gamma[\bs X_H^*]} q \grad\cdot{\bs u} \, {\bs V}_{\Gamma[\bs X_H^*]}\cdot{\bs n}_{\Gamma[\bs X_H^*]} \label{dt-Ritz-c-k-c} ,
\end{align} 
\end{subequations}
where the last equality follows from decomposing $\grad({\bs u}-{\bs u}_h^\#)$ into its tangential and normal components, i.e., 
$$
\grad({\bs u}-{\bs u}_h^\#)
= \grad_{\Gamma[\bs X_H^*]}({\bs u}-{\bs u}_h^\#) 
 + [\partial_{\bs n}({\bs u}-{\bs u}_h^\#)]{\bs n} . 
$$
We can estimate \eqref{dt-Ritz-c-k-a} by using integration by parts, after converting it to an integral on $\Gamma(t)$ with an error which can be estimated similarly as \eqref{dt-Ritz-g}, and estimate \eqref{dt-Ritz-c-k-b} by using the boundary condition $({\bs D}({\bs v}) +q{\bs I} ){\bs n}=0 $ on $\Gamma(t)$ similarly as we estimate \eqref{dt-Ritz-g}. Then we obtain 
\begin{align*}
|\eqref{dt-Ritz-c} + \eqref{dt-Ritz-k} |
\lesssim (h^k + H^k)  \|\partial_t{\bs u}-\partial_t{\bs u}_h^\#\|_{L^2(\Omega^{\rm f}(t))}.  
\end{align*} 

The following estimate follows from \eqref{Ritz-dt-H1} and interpolation property of $\bv_h$: 
\begin{align*}
\eqref{dt-Ritz-m}
&= 
g_h^{\rm u} (\partial_t{\bs u}-\partial_t{\bs u}_h^\#, {\bs v}_h ) + g_h^{\rm p} (\partial_tp - \partial_tp_h^\#, q_h ) \\
&= 
g_h^{\rm u} (\partial_t{\bs u}-\partial_t{\bs u}_h^\#, {\bs v}_h - {\bs v} ) + g_h^{\rm p} (\partial_tp - \partial_tp_h^\#, q_h -q ) \\
&\le Ch^k (\|{\bs v}\|_{H^2(\Omega)} +\|q\|_{H^1(\Omega)} ) 
\le Ch^k \|\partial_t{\bs u}-\partial_t{\bs u}_h^\#\|_{L^2(\Omega^{\rm f}(t))}.
\end{align*}

By collecting the estimates of \eqref{dt-Ritz-a}--\eqref{dt-Ritz-m}, we obtain 
$$
\|\partial_t{\bs u}-\partial_t{\bs u}_h^\#\|_{L^2(\Omega^{\rm f}(t))}\le C(h^k + H^k). 
$$

Finally, we consider
\begin{align*}
&\Big| \int_{\Omega^{\rm f}[{\bs X}_H^{\#,\theta}]} |\partial_t{\bs u}-\partial_t{\bs u}_h^\#|^2 
- \int_{\Omega^{\rm f}[{\bs X}]} |\partial_t{\bs u}-\partial_t{\bs u}_h^\#|^2 \Big| \\
&= \Big| \int_0^\theta\frac{\d}{\d\alpha}\int_{\Omega^{\rm f}[{\bs X}_H^{\#,\alpha}]} |\partial_t{\bs u}-\partial_t{\bs u}_h^\#|^2 \d\alpha \Big| \\
&= \Big| \int_0^\theta\int_{\Gamma[{\bs X}_H^{\#,\alpha}]} |\partial_t{\bs u}-\partial_t{\bs u}_h^\#|^2 {\bs e}_{\bs X}^{\#,\alpha} \cdot {\bs n}_{\Gamma[{\bs X}_H^{\#,\alpha}]} \d\alpha \Big| \\
&\lesssim \max_{\alpha\in[0,1]} \|\partial_t{\bs u}-\partial_t{\bs u}_h^\#\|_{L^4(\Gamma[{\bs X}_H^{\#,\alpha}])}^2 \|{\bs e}_{\bs X}^{\#,\alpha}\|_{L^2(\Gamma[{\bs X}_H^{\#,\alpha}])} \\
&\lesssim \|\partial_t{\bs u}-\partial_t{\bs u}_h^\#\|_{H^1(\Omega^{\rm f}_{\rm ext}[{\bs X}_H])}^2 
( \|{\bs e}_{\bs d}\|_{L^2(\widehat\Gamma_H)} + H^{k+1}) \\
&\lesssim h^{2k-2} 
( \|{\bs e}_{\bs d}\|_{H^1(\widehat\Omega^{\rm s}_H)} + H^{k+1}) \\
&\lesssim h^k\|{\bs e}_{\bs d}\|_{H^1(\widehat\Omega^{\rm s}_H)} + h^kH^{k+1} \\
&\lesssim \|{\bs e}_{\bs d}\|_{H^1(\widehat\Omega^{\rm s}_H)}^2 + h^{2k}+H^{2k} .
\end{align*}
This proves \eqref{Ritz-L2} via the triangle inequality. 
\end{proof}

\begin{lemma}\label{Lemma:RhH-2}
Under the assumptions in Theorem \ref{THM} and \eqref{ind-1}, the remainder $R_{h,H}^*(\cdot,\cdot,\cdot,\cdot)$ defined in \eqref{def-RhH} satisfies the following estimate: 
\begin{align}
&\Big|\int_0^s R_{h,H}^*({\bs e}_{\bs u}, e_p, \partial_t{\bs e}_{\bs d}, {\bs e}_{\bs\sigma}) \d t\Big| \notag\\
&\le \varepsilon
\big(\|{\bs e}_{\bs u}(s)\|_{L^2(\Omega^{\rm f}[\bs X_H])}^2 + \|\partial_t{\bs e}_\dep(s)\|_{L^2(\widehat{\Omega}^{\rm s}_H)}^2 + \|\grad{\bs e}_\dep(s)\|_{L^2(\widehat{\Omega}^{\rm s}_H)}^2 \big) \notag\\
&\quad\, 
+ \varepsilon \epsilon_h^2 g^{\rm u}_h({\bs e}_{\bs u}(s),{\bs e}_{\bs u}(s))  +  \varepsilon\int_0^s \| {\bs e}_{\bs u}  \|_{1,\Omega[{\bs X}_H]}^2  \d t \notag\\
&\quad\,
+ C_\varepsilon \int_0^s \Big( \|\partial_t{\bs e}_\dep\|_{L^2(\widehat{\Omega}^{\rm s}_H)}^2 + \|{\bs e}_\dep\|_{H^1(\widehat{\Omega}^{\rm s}_H)}^2 \Big)  \d t 
+ C_\varepsilon(h^{2k} + H^{2k}) .  
\end{align}
\end{lemma}
\begin{proof}
Choosing ${\bs v}_h={\bs e}_{\bs u}$, $q_h=e_p$, ${\bs w}_H=\partial_t{\bs e}_{\bs d}$ and ${\bs\lambda}_H={\bs e}_{{\bs \sigma} }$ in \eqref{def-RhH}, we obtain 
\begin{subequations}\label{RhH-eu-ep-ed-esigma}
\begin{align}
R_{h,H}^*({\bs e}_{\bs u}, e_p, \partial_t{\bs e}_{\bs d}, {\bs e}_{{\bs \sigma} }) = & 
\int_{\Omega^{\rm f}[\bs X_H^*]} \partial_t \bs u_h^* \cdot {\bs e}_{\bs u}  
- \int_{\Omega^{\rm f}(t)} \partial_t \bs u \cdot {\bs e}_{\bs u} \label{RhH3-a}\\
&+ \int_{\Omega^{\rm f}[\bs X_H^*]}   (\bu_h^* \cdot \grad) \bu_h^* \cdot {\bs e}_{\bs u} 
- \int_{\Omega^{\rm f}(t)}   (\bu \cdot \grad) \bu \cdot {\bs e}_{\bs u} \label{RhH3-b}\\
& +  \int_{\Omega^{\rm f}[\bs X_H^*]} \bs D(\bu_h^*): \bs D({\bs e}_{\bs u})  -  \int_{\Omega^{\rm f}(t)} \bs D(\bu): \bs D({\bs e}_{\bs u}) \label{RhH3-c}\\ 
&  - \int_{\Omega^{\rm f}[\bs X_H^*]}  p_h^*\, \nabla\cdot  {\bs e}_{\bs u} +  \int_{\Omega^{\rm f}(t)}  p \,\nabla\cdot  {\bs e}_{\bs u} \label{RhH3-d}\\
& + \int_{\Omega^{\rm f}[\bs X_H^*]}  e_p \nabla\cdot  \bu_h^* -   \int_{\Omega^{\rm f}(t)}  e_p\, \nabla\cdot  \bu \label{RhH3-e}\\
& +  \int_{\widehat{\Omega}^{\rm s}_H} \partial_t\vel_H^* \cdot \partial_t{\bs e}_{\bs d}- \int_{\widehat{\Omega}^{\rm s}} \partial_t\vel \cdot \partial_t{\bs e}_{\bs d}\circ \widehat{\bs\Phi}_H^{-1} \label{RhH3-f}\\
& + \int_{\widehat{\Omega}^{\rm s}_H} {\bs D}(\dep_H^*) : {\bs D}(\partial_t{\bs e}_{\bs d}) - \int_{\widehat{\Omega}^{\rm s}} {\bs D}(\dep) : {\bs D}(\partial_t{\bs e}_{\bs d} \circ \widehat{\bs\Phi}_H^{-1}) \notag\\
& + \int_{\widehat{\Omega}^{\rm s}_H} (\grad\cdot \dep_H^* ) (\grad\cdot \partial_t{\bs e}_{\bs d}) - \int_{\widehat{\Omega}^{\rm s}} (\grad\cdot \dep) (\grad\cdot \partial_t{\bs e}_{\bs d} \circ \widehat{\bs\Phi}_H^{-1}) \label{RhH3-g}\\
 & - \int_{\Gamma[\bs X_H^*]} \widehat{\bs \sigma} _H^* \circ ({\bs X}_H^*)^{-1}  \cdot \big( {\bs e}_{\bs u} - \partial_t{\bs e}_{\bs d} \circ ({\bs X}_H^*)^{-1} \big) \label{RhH3-h}\\
& + \int_{\Gamma(t)} \widehat{\bs \sigma} \circ {\bs X}^{-1}  \cdot \big( {\bs e}_{\bs u} - \partial_t{\bs e}_{\bs d} \circ \widehat{\bs\Phi}_H^{-1} \circ  {\bs X}^{-1} \big) \label{RhH3-i}\\
&+ \int_{\Gamma[\bs X_H^*]} {\bs e}_{{\bs \sigma} } \circ ({\bs X}_H^*)^{-1} \cdot \big( \bu_h^* - \vel_H^*  \circ ({\bs X}_H^*)^{-1} \big) \label{RhH3-j}\\
&- \int_{\Gamma(t)} {\bs e}_{{\bs \sigma} } \circ \widehat{\bs\Phi}_H^{-1}  \circ {\bs X}^{-1}  \cdot \big( \bu - \vel  \circ {\bs X}^{-1} \big) \label{RhH3-k}\\
& + \epsilon_h g_h^{\rm u} (\partial_t\bu_h^*,{\bs e}_{\bs u} ) 
+ g_h^{\rm u} (\bu_h^*,{\bs e}_{\bs u} ) + g_h^{\rm p} (p_h^*,e_p ) . \label{RhH3-l}
\end{align}
\end{subequations}

\noindent {\bf Part 1:} Estimate of \eqref{RhH3-a}--\eqref{RhH3-d}: 
In Lemma \ref{Lemma:RhH}, we have shown that 
$$
|\eqref{RhH3-a}+\eqref{RhH3-b}+\eqref{RhH3-c}+\eqref{RhH3-d}|
\le Ch^k \|{\bs e}_{\bs u}\|_{H^1(\Omega^{\rm f}_{\rm ext}[\bs X_H])} .
$$

\noindent {\bf Part 2:} Estimate of \eqref{RhH3-e} + \eqref{RhH3-l}: We use the following notations for $\theta\in[0,1] $ as before:  
\begin{align*}
& {\bs X}_H^{*,\theta} = (1-\theta) {\bs X} + \theta {\bs X}_H^*\circ \widehat{\bs\Phi}_H^{-1}: \widehat\Gamma\rightarrow \mathbb{R}^d \\
& {\bs e}_{\bs X}^{*,\theta} = ({\bs X}_H^*\circ \widehat{\bs\Phi}_H^{-1} - {\bs X})\circ ({\bs X}_H^{*,\theta})^{-1} . 
\end{align*}
Then \eqref{RhH3-e} can be written as, using the relation $\grad\cdot{\bs u}=0$ on $\Omega^{\rm f}(t)$, 
\begin{align*}
\eqref{RhH3-e}  
=& \int_{\Omega^{\rm f}[\bs X_H^*]}  e_p \grad\cdot({\bs u}_h^*-{\bs u}) + \int_{\Omega^{\rm f}[\bs X_H^*]}  e_p \grad\cdot{\bs u} .
\end{align*}
Choosing ${\bs v}_h=0$ and $q_h=e_p$ in \eqref{def-Ritz}, we have 
\begin{align}
-\int_{\Omega^{\rm f}[\bs X_H^*]} e_p \grad\cdot{\bs u}_h^\#
+ g_h^{\rm p} (p - p_h^\#, e_p ) = 0 .
\end{align}
Substituting this relation into \eqref{RhH3-e} yields 
\begin{subequations}\label{RhH3-e-l-abcd}
\begin{align}
&\eqref{RhH3-e} + \eqref{RhH3-l} \notag\\
=& \int_{\Omega^{\rm f}[\bs X_H^*]}  e_p \grad\cdot({\bs u}_h^*-{\bs u}_h^\#) \label{RhH3-e-l-a}\\
& + g_h^{\rm u} (\bu_h^*-{\bs u},{\bs e}_{\bs u} ) \label{RhH3-e-l-c}\\
& + g_h^{\rm p} (p_h^*- p_h^\#,e_p ) \label{RhH3-e-l-d} \\
& + \epsilon_h g_h^{\rm u} (\partial_t\bu_h^*-\partial_t{\bs u},{\bs e}_{\bs u} ) , \label{RhH3-e-l-e}
\end{align}
\end{subequations}
where 
\begin{align}\label{RhH3-e-l-bc}
|\eqref{RhH3-e-l-c}|
\le Ch^k\|{\bs e}_{\bs u}\|_{1,\Omega[\bs X_H]}  
\end{align}
and
\begin{align*}
|\eqref{RhH3-e-l-e}|
\lesssim
\epsilon_h h^k
g_h^{\rm u}({\bs e}_{\bs u},{\bs e}_{\bs u})^{1/2}
\le
\varepsilon\|{\bs e}_{\bs u}\|_{1,\Omega[X_H]}^2
+C_\varepsilon\epsilon_h^2h^{2k}.
\end{align*}
Choosing ${\bs v}_h={\bs u}_h^*-{\bs u}_h^\#$, $q_h=p_h^\#- p_h^*$, ${\bs w}_H=0$ and ${\bs\lambda}_H=0$ in the error equation \eqref{error-eq}, we have 
\begin{subequations}\label{ep-eq-a-m}
\begin{align}
\eqref{RhH3-e-l-a}+\eqref{RhH3-e-l-d} 
= & \int_{\Omega^{\rm f}[\bs X_H]} \partial_t {\bs e}_{\bs u} \cdot ({\bs u}_h^*-{\bs u}_h^\#) \label{ep-eq-a}\\
& + \Big( \int_{\Omega^{\rm f}[\bs X_H]} - \int_{\Omega^{\rm f}[\bs X_H^*]} \Big) \partial_t {\bs u}_h^* \cdot ({\bs u}_h^*-{\bs u}_h^\#) \label{ep-eq-b}\\
& + \int_{\Omega^{\rm f}[\bs X_H]} {\bs e}_{\bs u} \cdot \grad \bu_h \cdot ({\bs u}_h^*-{\bs u}_h^\#) \label{ep-eq-c}\\
& + \int_{\Omega^{\rm f}[\bs X_H]} {\bs u}_h^* \cdot \grad {\bs e}_{\bs u} \cdot ({\bs u}_h^*-{\bs u}_h^\#) \label{ep-eq-d}\\
& + \Big( \int_{\Omega^{\rm f}[\bs X_H]} - \int_{\Omega^{\rm f}[\bs X_H^*]} \Big) {\bs u}_h^* \cdot \grad {\bs u}_h^* \cdot ({\bs u}_h^*-{\bs u}_h^\#)\label{ep-eq-e}\\
&
+ \int_{\Omega^{\rm f}[\bs X_H]} \bs D({\bs e}_{\bs u}): \bs D({\bs u}_h^*-{\bs u}_h^\#) 
+ \int_{\Omega^{\rm f}[\bs X_H]} (p_h^\#- p_h^*) \nabla\cdot  {\bs e}_{\bs u}  \label{ep-eq-f}\\ 
& + \Big(\int_{\Omega^{\rm f}[\bs X_H]} -\int_{\Omega^{\rm f}[\bs X_H^*]} \Big) {\bs \sigma}(\bu_h^*,p_h^*): \grad ({\bs u}_h^*-{\bs u}_h^\#) 
\label{ep-eq-g} \\
& + \Big( \int_{\Omega^{\rm f}[\bs X_H]} - \int_{\Omega^{\rm f}[\bs X_H^*]} \Big) (p_h^\#- p_h^*) \grad\cdot \bu_h^* \label{ep-eq-h}\\
& - \int_{\Gamma[\bs X_H]} {\bs e}_{{\bs\sigma}}\circ {\bs X}_H^{-1} \cdot ({\bs u}_h^*-{\bs u}_h^\#)  \label{ep-eq-i}\\
& - \int_{\Gamma[\bs X_H]} {\widehat{\bs\sigma}}_H^*\circ {\bs X}_H^{-1}\cdot ({\bs u}_h^*-{\bs u}_h^\#)
\label{ep-eq-j}\\
& + \int_{\Gamma[\bs X_H^*]} {\widehat{\bs\sigma}}_H^*\circ({\bs X}_H^*)^{-1} \cdot ({\bs u}_h^*-{\bs u}_h^\#)  
\label{ep-eq-k}\\
& + \epsilon_h g_h^{\rm u} (\partial_t{\bs e}_{\bs u},{\bs u}_h^*-{\bs u}_h^\# )  + g_h^{\rm u} ({\bs e}_{\bs u},{\bs u}_h^*-{\bs u}_h^\# ) \label{ep-eq-l}\\
& + R_{h,H}^*( {\bs u}_h^*-{\bs u}_h^\# , p_h^\#- p_h^*, 0, 0)  \label{ep-eq-m}\\
& - \Big( \int_{\Omega^{\rm f}[\bs X_H]} - \int_{\Omega^{\rm f}[\bs X_H^*]} \Big) e_p \grad\cdot({\bs u}_h^*-{\bs u}_h^\#) . 
\label{ep-eq-n}
\end{align}
\end{subequations}
Let ${\bs V}_{\Gamma[\bs X_H]}=\partial_t{\bs d}_H\circ({\bs X}_H^{-1})$ denote the velocity of the interface $\Gamma[\bs X_H]$. Then 
\begin{align*}
\int_0^s \eqref{ep-eq-a} \d t
&= \int_0^s \int_{\Omega^{\rm f}[\bs X_H]} \partial_t {\bs e}_{\bs u} \cdot ({\bs u}_h^*-{\bs u}_h^\#) \d t \\
&= \int_0^s \frac{\d}{\d t}\int_{\Omega^{\rm f}[\bs X_H]}{\bs e}_{\bs u} \cdot ({\bs u}_h^*-{\bs u}_h^\#) \d t 
- \int_0^s \int_{\Omega^{\rm f}[\bs X_H]}{\bs e}_{\bs u} \cdot \partial_t({\bs u}_h^*-{\bs u}_h^\#) \d t \notag\\
&\quad\, 
- \int_0^s \int_{\Gamma[\bs X_H]}{\bs e}_{\bs u} \cdot ({\bs u}_h^*-{\bs u}_h^\#) {\bs V}_{\Gamma[\bs X_H]}\cdot {\bs n}_{\Gamma[\bs X_H]} \, \d t \\
&\le 
C\|{\bs e}_{\bs u}(s)\|_{L^2(\Omega^{\rm f}[\bs X_H(s)])} \| {\bs u}_h^*(s)-{\bs u}_h^\#(s)\|_{L^2(\Omega^{\rm f}[\bs X_H(s)])} \\
&\quad\, + C\|{\bs e}_{\bs u}(0)\|_{L^2(\Omega^{\rm f}[\bs X_H(0)])} \| {\bs u}_h^*(0)-{\bs u}_h^\#(0)\|_{L^2(\Omega^{\rm f}[\bs X_H(0)])} \\
&\quad\,
+ \int_0^s \|{\bs e}_{\bs u}\|_{H^1(\Omega^{\rm f}[\bs X_H])}
\|\partial_t({\bs u}_h^*-{\bs u}_h^\#)\|_{L^2(\Omega^{\rm f}[\bs X_H])} \d t \\
&\quad\,
+ \int_0^s \|{\bs e}_{\bs u}\|_{H^1(\Omega^{\rm f}[\bs X_H])}
\|{\bs u}_h^*-{\bs u}_h^\#\|_{H^1(\Omega^{\rm f}_{\rm ext}[\bs X_H])} \d t \\
&\le 
Ch^k\|{\bs e}_{\bs u}(s)\|_{L^2(\Omega^{\rm f}[\bs X_H(s)])} 
+ Ch^{2k}
+ \int_0^s \|{\bs e}_{\bs u}\|_{H^1(\Omega^{\rm f}[\bs X_H])}
(\|{\bs e}_{\bs d}\|_{H^1(\widehat\Omega^{\rm s}_H)} + h^k + H^k) \d t ,
\end{align*}
where the last inequality uses the results in Lemma \ref{Lemma:Ritz}. Moreover, the following results can be obtained by using the formula in \eqref{dt-int-Omega-2}: 
\begin{align*}
|\eqref{ep-eq-b}+\eqref{ep-eq-e}|
&\le Ch^k\|{\bs e}_{\bs d}\|_{H^1(\widehat{\Omega}_H^{\rm s})} ,\\
|\eqref{ep-eq-c}+\eqref{ep-eq-d}+\eqref{ep-eq-f}|
&\le Ch^k\|{\bs e}_{\bs u}\|_{H^1(\Omega^{\rm f}_{\rm ext}[\bs X_H])} .
\end{align*}
\eqref{ep-eq-g}+\eqref{ep-eq-j}+\eqref{ep-eq-k} can be estimated together as follows: 
\begin{subequations}\label{ep-eq-gjk}
\begin{align}
  & \eqref{ep-eq-g}+\eqref{ep-eq-j}+\eqref{ep-eq-k} \notag\\
=\,& \Big(\int_{\Omega^{\rm f}[\bs X_H]} -\int_{\Omega^{\rm f}[\bs X_H^*]} \Big) {\bs \sigma}(\bu_h^*-{\bs u},p_h^*-p): \grad ({\bs u}_h^*-{\bs u}_h^\#) \notag\\
  & - \Big(\int_{\Omega^{\rm f}[\bs X_H]} -\int_{\Omega^{\rm f}[\bs X_H^*]} \Big) \grad\cdot{\bs \sigma}({\bs u},p) \cdot ({\bs u}_h^*-{\bs u}_h^\#) \notag\\
  & + \Big(\int_{\Gamma[\bs X_H]} -\int_{\Gamma[\bs X_H^*]} \Big) {\bs \sigma}({\bs u},p){\bs n} \cdot ({\bs u}_h^*-{\bs u}_h^\#) \notag\\
  & - \int_{\Gamma[\bs X_H]} ({\widehat{\bs \sigma} }_H^*\circ {\bs X}_H^{-1}) \cdot ({\bs u}_h^*-{\bs u}_h^\#) + \int_{\Gamma[\bs X_H^*]} ({\widehat{\bs \sigma} }_H^*\circ ({\bs X}_H^*)^{-1}) \cdot ({\bs u}_h^*-{\bs u}_h^\#) \notag\\
=\,& \Big(\int_{\Omega^{\rm f}[\bs X_H]} -\int_{\Omega^{\rm f}[\bs X_H^*]} \Big) {\bs \sigma}(\bu_h^*-{\bs u},p_h^*-p): \grad ({\bs u}_h^*-{\bs u}_h^\#) \label{ep-eq-gjk-a}\\
  & - \Big(\int_{\Omega^{\rm f}[\bs X_H]} -\int_{\Omega^{\rm f}[\bs X_H^*]} \Big) \grad\cdot{\bs \sigma}({\bs u},p) \cdot ({\bs u}_h^*-{\bs u}_h^\#) \label{ep-eq-gjk-b}\\
  & + \int_0^1\int_{\Gamma[\bs X_H^\theta]} [{\bs e}_{\bs d}^\theta\cdot\grad{\bs \sigma}({\bs u},p)]{\bs n}_{\Gamma[{\bs X}_H^\theta]}\cdot ({\bs u}_h^*-{\bs u}_h^\#) \d\theta \label{ep-eq-gjk-c}\\
  & - \int_0^1\int_{\Gamma[\bs X_H^\theta]} {\bs \sigma}({\bs u},p)(\grad_{\Gamma[\bs X_H^\theta]}{\bs e}_{\bs d}^\theta){\bs n}_{\Gamma[{\bs X}_H^\theta]} \cdot ({\bs u}_h^*-{\bs u}_h^\#) \d\theta \label{ep-eq-gjk-d}
  \quad\mbox{(here \eqref{dt-int-Gamma-1b} is used)}\\
  & + \int_0^1\int_{\Gamma[\bs X_H^\theta]} {\bs \sigma}({\bs u},p){\bs n}_{\Gamma[{\bs X}_H^\theta]} \cdot [{\bs e}_{\bs d}^\theta\cdot\grad({\bs u}_h^*-{\bs u}_h^\#)] \d\theta \label{ep-eq-gjk-e}\\
  & + \int_0^1\int_{\Gamma[\bs X_H^\theta]} {\bs \sigma}({\bs u},p){\bs n}_{\Gamma[{\bs X}_H^\theta]} \cdot ({\bs u}_h^*-{\bs u}_h^\#) \grad_{\Gamma[\bs X_H^\theta]}\cdot {\bs e}_{\bs d}^\theta \d\theta \label{ep-eq-gjk-f}\\
  & - \int_0^1\int_{\Gamma[{\bs X}_H^\theta]} {\widehat{\bs\sigma}}_H^{*|\theta} \cdot [{\bs e}_{\bs d}^\theta\cdot\grad({\bs u}_h^*-{\bs u}_h^\#)] \d\theta \label{ep-eq-gjk-g}\\
  & - \int_0^1\int_{\Gamma[{\bs X}_H^\theta]} {\widehat{\bs\sigma}}_H^{*|\theta} \cdot ({\bs u}_h^*-{\bs u}_h^\#) \grad_{\Gamma[\bs X_H^\theta]}\cdot {\bs e}_{\bs d}^\theta \d\theta , \label{ep-eq-gjk-h} 
\end{align}
\end{subequations}
where 
\begin{align*}
|\eqref{ep-eq-gjk-a}| &\le Ch^{2k} \quad\mbox{(interpolation error and Lemma \ref{Lemma:Ritz})},\\ 
|\eqref{ep-eq-gjk-b}+\eqref{ep-eq-gjk-c}| &\le Ch^{k} \|{\bs e}_{\bs d}\|_{H^1(\widehat{\Omega}_H^{\rm s})} ,
\end{align*}

\begin{subequations}
\begin{align}
\eqref{ep-eq-gjk-d} 
&= - \int_{\Gamma[{\bs X}_H^*]} {\bs\sigma}({\bs u},p)( \grad_{\Gamma[{\bs X}_H^*]}{\bs e}_{\bs d}^0 ){\bs n}_{\Gamma[{\bs X}_H^*]} \cdot ({\bs u}_h^*-{\bs u}_h^\#) \notag\\
&\quad\, - \int_0^1 \int_0^\theta \frac{\d}{\d\alpha}\int_{\Gamma[{\bs X}_H^\alpha]} {\bs\sigma}({\bs u},p)( \grad_{\Gamma[{\bs X}_H^\alpha]}{\bs e}_{\bs d}^\alpha ){\bs n}_{\Gamma[{\bs X}_H^\alpha]} \cdot ({\bs u}_h^*-{\bs u}_h^\#) \d\alpha \d\theta \notag\\
&= - \int_0^1\int_{\Gamma(t)} {\bs\sigma}({\bs u},p)( \grad_{\Gamma(t)}{\bs e}_{\bs d}^{*,0} ){\bs n}_{\Gamma(t)} \cdot ({\bs u}_h^*-{\bs u}_h^\#) \d\theta \label{ep-eq-gjk-d-a} \\
&\quad\, - \int_0^1 \frac{\d}{\d\theta}\int_{\Gamma[{\bs X}_H^{*,\theta}]} {\bs\sigma}({\bs u},p)( \grad_{\Gamma[{\bs X}_H^{*,\theta}]}{\bs e}_{\bs d}^{*,\theta}){\bs n}_{\Gamma[{\bs X}_H^{*,\theta}]} \cdot ({\bs u}_h^*-{\bs u}_h^\#) \d\theta \label{ep-eq-gjk-d-b} \\
&\quad\, - \int_0^1 \int_0^\theta \frac{\d}{\d\alpha}\int_{\Gamma[{\bs X}_H^\alpha]} {\bs\sigma}({\bs u},p)( \grad_{\Gamma[{\bs X}_H^\alpha]}{\bs e}_{\bs d}^\alpha ){\bs n}_{\Gamma[{\bs X}_H^\alpha]} \cdot ({\bs u}_h^*-{\bs u}_h^\#) \d\alpha \d\theta , \label{ep-eq-gjk-d-c} 
\end{align}
\end{subequations}
where \eqref{ep-eq-gjk-d-a} can be estimated via integration by parts, i.e., 
\begin{align*}
|\eqref{ep-eq-gjk-d-a}| 
&\lesssim \|{\bs e}_{\bs d}^{*,0}\|_{H^{\frac12}(\Gamma[\bs X])}\|{\bs u}_h^*-{\bs u}_h^\#\|_{H^{\frac12}(\Gamma[\bs X])} \\
&\lesssim \|{\bs e}_{\bs d}\|_{H^1(\widehat\Gamma_H)}\|{\bs u}_h^*-{\bs u}_h^\#\|_{H^1(\Omega^{\rm f}[\bs X])} \\
&\lesssim\|{\bs e}_{\bs d}\|_{H^1(\widehat{\Omega}_H^{\rm s})} 
(\|{\bs e}_{\bs d}\|_{H^1(\widehat{\Omega}_H^{\rm s})} + h^k + H^k) ,
\end{align*}
and \eqref{ep-eq-gjk-d-b} can be estimated by using the formulas in \eqref{dt-int-Gamma-1}--\eqref{dt-int-Gamma-1c}, i.e., 
\begin{align*}
|\eqref{ep-eq-gjk-d-b}| 
&\lesssim \|{\bs e}_{\bs d}^{*,\theta}\|_{H^1(\Gamma[{\bs X}_H^{*,\theta}])}
\|{\bs e}_{\bs d}^{*,\theta}\|_{W^{1,4}(\Gamma[{\bs X}_H^{*,\theta}])}
\|{\bs u}_h^*-{\bs u}_h^\#\|_{L^4(\Gamma[{\bs X}_H^{*,\theta}])} \\
&\quad\, 
+ \|{\bs e}_{\bs d}^{*,\theta}\|_{H^1(\Gamma[{\bs X}_H^{*,\theta}])}
\|{\bs e}_{\bs d}^{*,\theta}\|_{L^4(\Gamma[{\bs X}_H^{*,\theta}])}
\|{\bs u}_h^*-{\bs u}_h^\#\|_{W^{1,4}(\Gamma[{\bs X}_H^{*,\theta}])} \\
&\lesssim H^{-\frac32}\|{\bs e}_{\bs d}\|_{H^1(\widehat\Omega^{\rm s}_H)}^2\|{\bs u}_h^*-{\bs u}_h^\#\|_{H^1(\Omega^{\rm f}_{\rm ext}[{\bs X}_H])} \notag\\
&\quad\,
+ H^{-\frac12}h^{-1}\|{\bs e}_{\bs d}\|_{H^1(\widehat\Omega^{\rm s}_H)}^2\|{\bs u}_h^*-{\bs u}_h^\#\|_{H^1(\Omega^{\rm f}_{\rm ext}[{\bs X}_H])} \quad\mbox{(inverse estimate)} \notag\\
&\lesssim (h^k + Hh^{k-1}) \|{\bs e}_{\bs d}\|_{H^1(\widehat\Omega^{\rm s}_H)} 
\quad\mbox{(Lemma \ref{Lemma:Ritz} and \eqref{ind-1-3} are used)} \notag\\
&\lesssim (h^k + H^k) \|{\bs e}_{\bs d}\|_{H^1(\widehat\Omega^{\rm s}_H)}.
\end{align*}
Note that \eqref{ep-eq-gjk-d-c} can be estimated in the same way as \eqref{ep-eq-gjk-d-b}. This gives an estimate of \eqref{ep-eq-gjk-d}. Since \eqref{ep-eq-gjk-f} and \eqref{ep-eq-gjk-h} can be estimated in the same way as \eqref{ep-eq-gjk-d} (details are omitted), we present their estimates together:  
\begin{align}
|\eqref{ep-eq-gjk-d}+\eqref{ep-eq-gjk-f}+\eqref{ep-eq-gjk-h}| 
&\lesssim h^{2k}+H^{2k} + \|{\bs e}_{\bs d}\|_{H^1(\widehat{\Omega}_H^{\rm s})}^2  . \end{align}

Note that \eqref{ep-eq-gjk-e}$+$\eqref{ep-eq-gjk-g} contains ${\bs\sigma}({\bs u},p){\bs n}|_{\Gamma[{\bs X}_H^\theta]} 
- \widehat{\bs\sigma}_H^{*|\theta} $, which has the following decomposition: 
\begin{align*}
{\bs\sigma}({\bs u},p){\bs n}|_{\Gamma[{\bs X}_H^\theta]} 
- \widehat{\bs\sigma}_H^{*|\theta} 
&= {\bs\sigma}({\bs u},p){\bs n}|_{\Gamma[{\bs X}_H^\theta]} 
- {\bs\sigma}({\bs u},p){\bs n}|_{\Gamma(t)} \circ {\bs X} \circ \widehat{\bs\Phi}_H \circ ({\bs X}_H^\theta)^{-1} \notag\\
&\quad\, 
+ ({\bs\sigma}({\bs u},p){\bs n}|_{\Gamma(t)} \circ {\bs X} \circ \widehat{\bs\Phi}_H
- \widehat{\bs\sigma}_H^*) \circ ({\bs X}_H^\theta)^{-1} . 
\end{align*}
Since $\widehat{\bs\sigma}_H^* = {\bs I}_H^{\rm s}[{\bs\sigma}({\bs u},p){\bs n}|_{\Gamma(t)} \circ {\bs X} \circ \widehat{\bs\Phi}_H]$, it follows that 
\begin{align*}
\| {\bs\sigma}({\bs u},p){\bs n}|_{\Gamma[{\bs X}_H^\theta]} 
- \widehat{\bs\sigma}_H^{*|\theta} \|_{L^\infty(\Gamma[{\bs X}_H^\theta])}
\lesssim \| {\bs e}_{\bs d} \|_{L^\infty(\widehat\Gamma_H)} + H^{k+1} 
\lesssim h \quad\mbox{(under condition $H^{k+1}\lesssim h$)} . 
\end{align*}
This can be used to estimate $\eqref{ep-eq-gjk-e}+\eqref{ep-eq-gjk-g}$ as follows: 
\begin{align}
|\eqref{ep-eq-gjk-e}+\eqref{ep-eq-gjk-g}| 
&\lesssim h \| {\bs e}_{\bs d} \|_{L^2(\widehat\Gamma_H)}
\| \grad ({\bs u}_h^*-{\bs u}_h^\#) \|_{L^2(\Gamma[{\bs X}_H^\theta])} \notag\\
&\lesssim h^{\frac12} \| {\bs e}_{\bs d} \|_{H^1(\widehat\Omega^{\rm s}_H)} \| \grad ({\bs u}_h^*-{\bs u}_h^\#) \|_{L^2(\Omega^{\rm f}_{\rm ext}[{\bs X}_H])} \notag\\
&\lesssim h^{k} \|{\bs e}_{\bs d}\|_{H^1(\widehat{\Omega}_H^{\rm s})} \quad\mbox{(Lemma \ref{Lemma:Ritz} is used)} .
\end{align}

Substituting the estimates of \eqref{ep-eq-gjk-a}--\eqref{ep-eq-gjk-h} into \eqref{ep-eq-gjk}, we obtain 
\begin{align}
|\eqref{ep-eq-g}+\eqref{ep-eq-j}+\eqref{ep-eq-k}|
\lesssim h^{2k}+H^{2k} + \|{\bs e}_{\bs d}\|_{H^1(\widehat{\Omega}_H^{\rm s})}^2 . 
\end{align}

Moreover, by using the formula in \eqref{dt-int-Omega-2}, we have
\begin{align*}
\eqref{ep-eq-h}
&= \Big( \int_{\Omega^{\rm f}[\bs X_H]} - \int_{\Omega^{\rm f}[\bs X_H^*]} \Big) (p_h^\#- p_h^*) \grad\cdot ({\bs u}_h^* - {\bs u}) \\
&\quad\, + \Big( \int_{\Omega^{\rm f}[\bs X_H]} - \int_{\Omega^{\rm f}[\bs X_H^*]} \Big) (p_h^\#- p_h^*) \grad\cdot {\bs u}\\
&= \Big( \int_{\Omega^{\rm f}[\bs X_H]} - \int_{\Omega^{\rm f}[\bs X_H^*]} \Big) (p_h^\#- p_h^*) \grad\cdot ({\bs u}_h^* - {\bs u}) \\
&\quad\, + \int_0^1\int_{\Gamma[\bs X_H^\theta]}  (p_h^\#- p_h^*) (\grad\cdot {\bs u}) ({\bs e}_{\bs d}^\theta \cdot{\bs n}_{\Gamma[\bs X_H^\theta]}) \d\theta \\
&\lesssim h^{2k} + h^k \|{\bs e}_{\bs d}\|_{H^1(\widehat{\Omega}_H^{\rm s})} , 
\end{align*}
where the last inequality uses Lemma \ref{Lemma:Ritz} and the finite element inverse trace estimates, as well as that $\grad\cdot{\bs u} = O(h^m)$ for arbitrarily large $m$ on $\Gamma[{\bs X}_H^\theta]$, as shown in \eqref{div-u-small-2}. 

The first term in \eqref{ep-eq-l} can be estimated as follows (using integration by parts in time):  
\begin{align*}
&\int_0^s \epsilon_h g_h^{\rm u}(\partial_t{\bs e}_{\bs u},{\bs u}_h^*-{\bs u}_h^\#)\d t \\
&= \epsilon_h g_h^{\rm u}({\bs e}_{\bs u}(s),{\bs u}_h^*(s)-{\bs u}_h^\#(s)) - \epsilon_h g_h^{\rm u}({\bs e}_{\bs u}(0),{\bs u}_h^*(0)-{\bs u}_h^\#(0)) 
- \int_0^s \epsilon_h g_h^{\rm u}({\bs e}_{\bs u},\partial_t{\bs u}_h^*-\partial_t{\bs u}_h^\#)\d t \\
&\leq C\epsilon_h g^{\rm u}_h({\bs e}_{\bs u}(s),{\bs e}_{\bs u}(s))^{\frac12} h^k
+ C\epsilon_h h^{2k}
+ \varepsilon \int_0^s g_h^{\rm u}({\bs e}_{\bs u},{\bs e}_{\bs u}) \d t
+ C_\varepsilon \int_0^s \epsilon_h^2 g_h^{\rm u}(\partial_t{\bs u}_h^*-\partial_t{\bs u}_h^\#,\partial_t{\bs u}_h^*-\partial_t{\bs u}_h^\#) \d t \\
&\leq \varepsilon \epsilon_h^2 g^{\rm u}_h({\bs e}_{\bs u}(s),{\bs e}_{\bs u}(s)) 
+ \varepsilon \int_0^s g_h^{\rm u}({\bs e}_{\bs u},{\bs e}_{\bs u}) \d t + C_\varepsilon h^{2k} \quad\mbox{(here Lemma \ref{Lemma:Ritz} and $\epsilon_h\leq h$ are used)} \\
&\leq \varepsilon \epsilon_h^2 g^{\rm u}_h({\bs e}_{\bs u}(s),{\bs e}_{\bs u}(s)) 
+ \varepsilon \int_0^s \|{\bs e}_{\bs u}\|_{1,\Omega[{\bs X}_H]}^2 \d t + C_\varepsilon h^{2k} ,  
\end{align*}
where we have estimated $g_h^{\rm u}({\bs e}_{\bs u},{\bs e}_{\bs u})$ by $\|{\bs e}_{\bs u}\|_{1,\Omega[{\bs X}_H]}^2$. Therefore, using Lemma \ref{Lemma:RhH} and Lemma \ref{Lemma:Ritz}, we have 
\begin{align}
\Big|\int_0^s \Big( \eqref{ep-eq-l}+\eqref{ep-eq-m} \Big) \d t\Big|
&\leq
\varepsilon \epsilon_h^2 g^{\rm u}_h({\bs e}_{\bs u}(s),{\bs e}_{\bs u}(s)) 
+ \varepsilon \int_0^s \|{\bs e}_{\bs u}\|_{1,\Omega[{\bs X}_H]}^2 \d t
+ C_\varepsilon (h^{2k}+H^{2k}) .
\end{align}

Moreover, \eqref{ep-eq-n} can be estimated by using the fundamental theorem of calculus and the transport formula in \eqref{dt-int-Omega-2}, i.e., 
\begin{align}\label{requive-kgeq3}
|\eqref{ep-eq-n}|
&= \Big| \int_0^1 \Big( \int_{\Omega^{\rm f}[\bs X_H]} - \int_{\Omega^{\rm f}[\bs X_H^*]} \Big) e_p \grad\cdot({\bs u}_h^*-{\bs u}_h^\#) \d\theta \Big| \notag\\
&= \Big| \int_0^1 \int_{\Gamma[\bs X_H^\theta]} e_p \grad\cdot({\bs u}_h^*-{\bs u}_h^\#) {\bs e}_{\bs d}^\theta\cdot{\bs n}_{\Gamma[\bs X_H^\theta]} \d\theta \Big| \notag\\
&\lesssim \int_0^1 \|e_p\|_{L^2(\Gamma[\bs X_H^\theta])} \|\grad({\bs u}_h^*-{\bs u}_h^\#) \|_{L^4(\Gamma[\bs X_H^\theta])} \|{\bs e}_{\bs d}^\theta\|_{L^4(\Gamma[\bs X_H^\theta])} \d\theta \notag\\
&\lesssim h^{-\frac34}\|e_p\|_{L^2(\Omega^{\rm f}_{\rm ext}[{\bs X}_H])} \|\grad({\bs u}_h^*-{\bs u}_h^\#) \|_{L^4(\Omega^{\rm f}_{\rm ext}[{\bs X}_H])} \|{\bs e}_{\bs d}\|_{L^4(\widehat\Gamma_H)} \notag\\
&\hspace{85pt}\mbox{(here the inverse trace estimates are used)} \notag\\
&\lesssim h^{-\frac32}\|e_p\|_{L^2(\Omega^{\rm f}_{\rm ext}[{\bs X}_H])} \|\grad({\bs u}_h^*-{\bs u}_h^\#) \|_{L^2(\Omega^{\rm f}_{\rm ext}[{\bs X}_H])} \|{\bs e}_{\bs d}\|_{L^4(\widehat\Gamma_H)} \notag\\
&\hspace{85pt}\mbox{(here the inverse estimates are used)} \notag\\
&\lesssim h^{k-\frac32}\|e_p\|_{L^2(\Omega^{\rm f}_{\rm ext}[{\bs X}_H])} \|{\bs e}_{\bs d}\|_{H^1(\widehat\Omega^{\rm s}_H)} 
\quad\mbox{(\eqref{Ritz-H1} and}\, \|{\bs e}_{\bs d}\|_{L^4(\widehat\Gamma_H)}\lesssim \|{\bs e}_{\bs d}\|_{H^1(\widehat\Omega^{\rm s}_H)}) \notag\\
&\lesssim \varepsilon h^2 \|e_p\|_{L^2(\Omega^{\rm f}_{\rm ext}[{\bs X}_H])}^2 + C_\varepsilon \|{\bs e}_{\bs d}\|_{H^1(\widehat\Omega^{\rm s}_H)}^2 ,
\end{align}
where the last inequality uses $k\ge 3$ and Young's inequality. Then, using the estimate of relation between $\|e_p\|_{L^2(\Omega^{\rm f}_{\rm ext}[{\bs X}_H])}$ and $\| e_p \|_{0,\Omega[{\bs X}_H]}$ as shown in \eqref{0-0-norm-b}, as well as the estimate of $\int_0^s \| e_p \|_{0,\Omega[{\bs X}_H]}^2 \d t$ in \eqref{ep-esigma-estimates-L4} (with $\varepsilon=1$ therein), we obtain the following result under the condition $h\lesssim H$: 
\begin{align} 
\Big| \int_0^s \eqref{ep-eq-n} \d t \Big| 
&
\le 
\varepsilon h^2
\big( \|\grad{\bs e}_{\bs d}(s)\|_{L^2(\widehat\Omega^{\rm s}_H)}^2 + \|\partial_t{\bs e}_{\bs d}(s)\|_{L^2(\widehat\Omega^{\rm s}_H)}^2 \big)
+ \varepsilon \int_0^s \| {\bs e}_{\bs u}  \|_{1,\Omega[{\bs X}_H]}^2  \d t \notag\\
&\quad\, 
+ C_\varepsilon  \int_0^s (\|{\bs e}_\dep\|_{H^1(\widehat{\Omega}^{\rm s}_H)}^2 + \|\partial_t{\bs e}_\dep\|_{L^2(\widehat{\Omega}^{\rm s}_H)}^2) \d t 
+ C_\varepsilon  (h^{2k} + H^{2k} ) . 
\end{align}

To conclude, we have presented estimates for \eqref{ep-eq-a}-\eqref{ep-eq-h} and \eqref{ep-eq-j}-\eqref{ep-eq-n}. The only missing term which has not been estimated is \eqref{ep-eq-i}. Therefore, in view of the expression of \eqref{RhH3-e-l-a}+\eqref{RhH3-e-l-d} in \eqref{ep-eq-a-m}, we have 
\begin{align}
\int_0^s( \eqref{RhH3-e-l-a}+\eqref{RhH3-e-l-d} )\d t  
&\le \varepsilon \big( \|{\bs e}_{\bs u}(s)\|_{L^2(\Omega^{\rm f}[\bs X_H])}^2 
+ h^2
\|\grad{\bs e}_{\bs d}(s)\|_{L^2(\widehat\Omega^{\rm s}_H)}^2 + h^2 \|\partial_t{\bs e}_{\bs d}(s)\|_{L^2(\widehat\Omega^{\rm s}_H)}^2 \big) \notag\\
&\quad\, + \varepsilon \epsilon_h^2 g^{\rm u}_h({\bs e}_{\bs u}(s),{\bs e}_{\bs u}(s)) + \varepsilon \int_0^s \|{\bs e}_{\bs u}\|_{1,\Omega[{\bs X}_H]}^2 \d t \notag\\
&\quad\, + C_\varepsilon  \int_0^s \Big(\|{\bs e}_\dep\|_{H^1(\widehat{\Omega}^{\rm s}_H)}^2 + \|\partial_t{\bs e}_\dep\|_{L^2(\widehat{\Omega}^{\rm s}_H)}^2\Big) \d t \notag\\
&\quad\, 
+ C_\varepsilon (h^{2k} + H^{2k}) + \int_0^s \eqref{ep-eq-i} \d t .
\end{align}
This, together with \eqref{RhH3-e-l-bc} and the expression of \eqref{RhH3-e}$+$\eqref{RhH3-l} in \eqref{RhH3-e-l-abcd}, shows that 
\begin{align}
\int_0^s[ \eqref{RhH3-e} + \eqref{RhH3-l} ] \d t 
&\le \varepsilon \big( \|{\bs e}_{\bs u}(s)\|_{L^2(\Omega^{\rm f}[\bs X_H])}^2 
+ h^2
\|\grad{\bs e}_{\bs d}(s)\|_{L^2(\widehat\Omega^{\rm s}_H)}^2 + h^2 \|\partial_t{\bs e}_{\bs d}(s)\|_{L^2(\widehat\Omega^{\rm s}_H)}^2 \big) \notag\\
&\quad\, + \varepsilon \epsilon_h^2 g^{\rm u}_h({\bs e}_{\bs u}(s),{\bs e}_{\bs u}(s)) + \varepsilon \int_0^s \Big(\|{\bs e}_{\bs u}\|_{1,\Omega[{\bs X}_H]}^2 
+ H^2\|e_p\|_{L^2(\Omega^{\rm f}[\bs X_H^*])}^2 \Big) \d t \notag\\
&\quad\, + C_\varepsilon  \int_0^s \Big(\|{\bs e}_\dep\|_{H^1(\widehat{\Omega}^{\rm s}_H)}^2 + \|\partial_t{\bs e}_\dep\|_{L^2(\widehat{\Omega}^{\rm s}_H)}^2 \Big) \d t \notag\\
&\quad\, 
+ C_\varepsilon (h^{2k} + H^{2k}) + \int_0^s \eqref{ep-eq-i} \d t \notag\\ 
&\le \varepsilon \big( \|{\bs e}_{\bs u}(s)\|_{L^2(\Omega^{\rm f}[\bs X_H])}^2 
+ 
\|\grad{\bs e}_{\bs d}(s)\|_{L^2(\widehat\Omega^{\rm s}_H)}^2 + \|\partial_t{\bs e}_{\bs d}(s)\|_{L^2(\widehat\Omega^{\rm s}_H)}^2 \big) \notag\\
&\quad\, + \varepsilon \epsilon_h^2 g^{\rm u}_h({\bs e}_{\bs u}(s),{\bs e}_{\bs u}(s)) + \varepsilon \int_0^s \|{\bs e}_{\bs u}\|_{1,\Omega[{\bs X}_H]}^2  \d t \notag\\
&\quad\, + C_\varepsilon  \int_0^s (\|{\bs e}_\dep\|_{H^1(\widehat{\Omega}^{\rm s}_H)}^2 + \|\partial_t{\bs e}_\dep\|_{L^2(\widehat{\Omega}^{\rm s}_H)}^2) \d t \notag\\
&\quad\, 
+ C_\varepsilon (h^{2k} + H^{2k}) + \int_0^s \eqref{ep-eq-i} \d t ,
\end{align}
where the last inequality follows from using the estimate of $\|e_p\|_{L^2(\Omega_{\rm ext}^{\rm f}[\bs X_H])}$ in \eqref{ep-esigma-estimates} (with $\varepsilon=1$ therein) under condition $H\lesssim h$.

\noindent {\bf Part 3:} Estimate of \eqref{RhH3-f}. Clearly, using the estimates of the interpolation errors, we have 
\begin{align}
|\eqref{RhH3-f}|
\le CH^k \|\partial_t{\bs e}_{\bs d}\|_{L^2(\widehat{\Omega}_H^{\rm s})}. 
\end{align}

\noindent {\bf Part 4:} Estimate of \eqref{RhH3-g}. Using integration by parts, we have 
\begin{align*}
\int_0^s\eqref{RhH3-g} \d t
&= \int_0^s\int_{\widehat{\Omega}^{\rm s}_H} {\bs D}(\dep_H^*) : {\bs D}(\partial_t{\bs e}_{\bs d}) \d t - \int_0^s\int_{\widehat{\Omega}^{\rm s}} {\bs D}(\dep) : {\bs D}(\partial_t{\bs e}_{\bs d} \circ \widehat{\bs\Phi}_H^{-1}) \d t \notag\\
&\quad\, + \int_0^s\int_{\widehat{\Omega}^{\rm s}_H} (\grad\cdot \dep_H^* ) (\grad\cdot \partial_t{\bs e}_{\bs d}) \d t - \int_0^s\int_{\widehat{\Omega}^{\rm s}} (\grad\cdot \dep) (\grad\cdot \partial_t{\bs e}_{\bs d} \circ \widehat{\bs\Phi}_H^{-1}) \d t \\
&= \Big[\int_{\widehat{\Omega}^{\rm s}_H} {\bs D}(\dep_H^*) : {\bs D}({\bs e}_{\bs d}) - \int_{\widehat{\Omega}^{\rm s}} {\bs D}(\dep) : {\bs D}({\bs e}_{\bs d} \circ \widehat{\bs\Phi}_H^{-1}) \Big]\Big|_{t=0}^{t=s} \\
&\quad\, - \int_0^s\int_{\widehat{\Omega}^{\rm s}_H} {\bs D}(\partial_t\dep_H^*) : {\bs D}({\bs e}_{\bs d}) \d t  + \int_0^s \int_{\widehat{\Omega}^{\rm s}} {\bs D}( \partial_t\dep) : {\bs D}({\bs e}_{\bs d} \circ \widehat{\bs\Phi}_H^{-1}) \d t \\
&\quad\, + \Big[\int_{\widehat{\Omega}^{\rm s}_H} (\grad\cdot \dep_H^*) (\grad \cdot{\bs e}_{\bs d}) - \int_{\widehat{\Omega}^{\rm s}} (\grad\cdot \dep) (\grad\cdot {\bs e}_{\bs d} \circ \widehat{\bs\Phi}_H^{-1}) \Big]\Big|_{t=0}^{t=s} \\
&\quad\, - \int_0^s\int_{\widehat{\Omega}^{\rm s}_H} (\grad\cdot\partial_t\dep_H^*) (\grad\cdot {\bs e}_{\bs d}) \d t  + \int_0^s \int_{\widehat{\Omega}^{\rm s}} (\grad\cdot \partial_t\dep)   \grad\cdot ({\bs e}_{\bs d} \circ \widehat{\bs\Phi}_H^{-1}) \d t 
\end{align*}
and therefore
\begin{align*}
\Big| \int_0^s\eqref{RhH3-g} \d t \Big| 
\le CH^k \| \grad{\bs e}_{\bs d}(s) \|_{L^2(\widehat{\Omega}^{\rm s}_H)} 
+ \int_0^s CH^k \| \grad {\bs e}_{\bs d} \|_{L^2(\widehat{\Omega}^{\rm s}_H)} \d t .
\end{align*}

\noindent {\bf Part 5:} Estimate of \eqref{RhH3-h} + \eqref{RhH3-i}. Similarly to Part 4, using integration by parts in time, we have  
\begin{align*}
\Big|\int_0^s [\eqref{RhH3-h} + \eqref{RhH3-i}]\d t \Big|
\lesssim H^k \|{\bs e}_{\bs d}(s) \|_{H^1(\widehat{\Omega}^{\rm s}_H)} 
+ \int_0^s H^k (\| {\bs e}_{\bs u} \|_{H^1({\Omega}^{\rm f}_{\rm ext}[{\bs X}_H])} 
+ \| {\bs e}_{\bs d} \|_{H^1(\widehat{\Omega}^{\rm s}_H)} ) \d t .
\end{align*}

\noindent {\bf Part 6:} Estimate of \eqref{RhH3-j} + \eqref{RhH3-k}. Clearly, using the estimates of the interpolation error, we have 
\begin{align}
&\quad\, \Big|\int_0^s [\eqref{RhH3-j} + \eqref{RhH3-k}] \d t \Big| \notag\\
&\lesssim (h^{k+1} +H^{k+1}) \int_0^s \|{\bs e}_{\bs\sigma}\|_{L^2(\widehat\Gamma_H)} \d t \notag\\
&\le C_\varepsilon (h^{2k} +H^{2k})
+\varepsilon (h^2+H^2) \int_0^s \|{\bs e}_{\bs\sigma}\|_{L^4(\widehat\Gamma_H)}^2 \d t \notag\\
&\le C_\varepsilon (h^{2k} +H^{2k}) 
+ \varepsilon \big(\|\partial_t{\bs e}_\dep(s)\|_{L^2(\widehat{\Omega}^{\rm s}_H)}^2 + \|{\bs e}_\dep(s)\|_{H^1(\widehat{\Omega}^{\rm s}_H)}^2 \big) \notag\\
&\quad\, + \varepsilon \int_0^s \big(\| {\bs e}_{\bs u} \|_{1,\Omega[{\bs X}_H]}^2 + \|\partial_t{\bs e}_\dep\|_{L^2(\widehat{\Omega}^{\rm s}_H)}^2 + \|{\bs e}_\dep\|_{H^1(\widehat{\Omega}^{\rm s}_H)}^2 \big) \d t ,
\end{align}
where we have used the estimate of $\| {\bs e}_{\bs\sigma} \|_{L^4(\widehat\Gamma_H)}$ in \eqref{ep-esigma-estimates-L4}, by choosing $\varepsilon=O(1)$ therein and using the mesh-size condition $h\simeq H$.

In summary, by combining the estimates in Part 1 -- Part 6, we have proved the following result:
\begin{align}\label{R-estimate-1}
&\Big| \int_0^s R_{h,H}^*({\bs e}_{\bs u}, e_p, \partial_t{\bs e}_{\bs d}, {\bs e}_{{\bs \sigma} }) \d t\Big| \notag\\
&\le \Big| \int_0^s \eqref{ep-eq-i} \d t\Big| 
+ C_\varepsilon(h^{2k} +H^{2k}) \notag\\
&\quad\,
+ \varepsilon \epsilon_h^2 g^{\rm u}_h({\bs e}_{\bs u}(s),{\bs e}_{\bs u}(s)) + \varepsilon \int_0^s \| {\bs e}_{\bs u} \|_{1,\Omega[{\bs X}_H]}^2 \d t 
\notag\\ 
&\quad\, 
+ \varepsilon \big(\|{\bs e}_{\bs u}(s)\|_{L^2(\Omega^{\rm f}[\bs X_H])}^2 + \|\partial_t{\bs e}_\dep(s)\|_{L^2(\widehat{\Omega}^{\rm s}_H)}^2 + \|{\bs e}_\dep(s)\|_{H^1(\widehat{\Omega}^{\rm s}_H)}^2 \big) \notag\\ 
&\quad\, 
+ C_\varepsilon \int_0^s \big( \|\partial_t{\bs e}_\dep\|_{L^2(\widehat{\Omega}^{\rm s}_H)}^2 + \|{\bs e}_\dep\|_{H^1(\widehat{\Omega}^{\rm s}_H)}^2 \big) \d t  .
\end{align}

\noindent {\bf Part 7:} Estimate of $|\int_0^s \eqref{ep-eq-i} \d t|$. It remains to estimate $|\int_0^s \eqref{ep-eq-i} \d t|$. To this end, we denote by ${\bs w}$ an extension of $({\bs u}_h^*-{\bs u}_h^\#)\circ {\bs X}_H \circ \widehat{\bs\Phi}_H^{-1}$ from the smooth interface $\widehat\Gamma\times[0,t_*]$ to the smooth domain $\widehat{\Omega}^{\rm s}\times[0,t_*]$, and define ${\bs w}_H$ as a finite element approximation of ${\bs w}$ in the bulk domain $\widehat{\Omega}^{\rm s}_H\times[0,t_*]$, satisfying ${\bs w}_H = [{\bf P}_{\Gamma[{\bs X}_H]}({\bs w}\circ \widehat{\bs\Phi}_H\circ {\bs X}_H^{-1})] \circ {\bs X}_H $ on $\widehat\Gamma_H$, where ${\bf P}_{\Gamma[{\bs X}_H]}$ denotes the $L^2$ projection onto the finite element space on $\Gamma[{\bs X}_H]$. The construction of such extensions ${\bs w}$ and ${\bs w}_H$ is discussed in Part 8 of the proof, and the constructed ${\bs w}$ and ${\bs w}_H$ will have the following properties: 
\begin{subequations}\label{def-w-wH}
\begin{align}
&{\bs w}=({\bs u}_h^*-{\bs u}_h^\#)\circ {\bs X}_H \circ \widehat{\bs\Phi}_H^{-1}\,\,\,\mbox{on}\,\,\,\widehat\Gamma \times [0,t_*], \label{def-w-w}\\[5pt]
&{\bs w}_H \circ {\bs X}_H^{-1} = {\bf P}_{\Gamma[{\bs X}_H]}({\bs w}\circ \widehat{\bs\Phi}_H\circ {\bs X}_H^{-1}) 
= {\bf P}_{\Gamma[{\bs X}_H]}({\bs u}_h^*-{\bs u}_h^\#) \,\,\,\mbox{on}\,\,\,\Gamma[{\bs X}_H], \label{def-wH-wH}\\[5pt]
& \|\bw\|_{H^1(\widehat{\Omega}^{\rm s})} \lesssim h^k , \label{estimate-w-wH-1} \\
&\|\partial_t\bw\|_{L^2(\widehat{\Omega}^{\rm s})} 
\lesssim \|{\bs e}_{\bs d}\|_{H^1(\widehat\Omega^{\rm s}_H)}
+ h^k + H^k , \label{estimate-dtw-L2}\\[5pt]
& \|\bw_H-\bw\circ\widehat{\bs\Phi}_H\|_{L^2(\widehat{\Omega}_H^{\rm s})}
\lesssim H \|\bw\|_{H^1(\widehat{\Omega}^{\rm s})} 
\lesssim Hh^k , \label{estimate-w-wH-2}\\[5pt]
&\|\bw_H\|_{H^1(\widehat{\Omega}_H^{\rm s})}
\lesssim \|\bw\|_{H^1(\widehat{\Omega}^{\rm s})} 
\lesssim h^k . \label{estimate-w-wH-3}
\end{align}
\end{subequations}
Then, by using relation \eqref{esigma-test} with this ${\bs w}_H$, we have for $s\in(0,t_*]$, 
\begin{subequations}\label{ep-eq-i-e}
\begin{align}
\int_0^s \eqref{ep-eq-i} \d t 
&= -\int_0^s\int_{\Gamma[\bs X_H]} ({\bs e}_{{\bs\sigma}}\circ {\bs X}_H^{-1} ) \cdot {\bf P}_{\Gamma[{\bs X}_H]}({\bs u}_h^*-{\bs u}_h^\#) \d t \notag\\
&= -\int_0^s\int_{\Gamma[\bs X_H]} ({\bs e}_{{\bs\sigma}}\circ {\bs X}_H^{-1} ) \cdot (\bw_H \circ{\bs X}_H^{-1}) \d t \notag\\
&= \int_0^s \int_{\widehat{\Omega}^{\rm s}_H} \partial_{tt}{\bs e}_{\dep} \cdot \bw_H \d t \label{ep-eq-i-e-a}\\
&\quad\, 
+ \int_0^s\int_{\widehat{\Omega}^{\rm s}_H} {\bs D}({\bs e}_\dep):{\bs D}(\bw_H) \d t 
+\int_0^s\int_{\widehat{\Omega}^{\rm s}_H} (\grad\cdot{\bs e}_\dep) (\grad\cdot \bw_H) \d t\label{ep-eq-i-e-b}\\ 
&\quad\, +\Big( \int_0^s \int_{\Gamma[\bs X_H]} {\widehat{\bs\sigma} }_H^{*|1} \cdot \bw_H^1 \d t - \int_0^s \int_{\Gamma[\bs X_H^*]} {\widehat{\bs\sigma} }_H^{*|0} \cdot \bw_H^0 \d t \Big) \label{ep-eq-i-e-c} \\
&\quad\, + \int_0^s R_{h,H}^*(0, 0, \bs w_H, 0) \d t. \label{ep-eq-i-e-d}
\end{align}
\end{subequations}

Using the results in \eqref{estimate-w-wH-1} and \eqref{estimate-w-wH-2}, we have  
\begin{align}\label{ep-eq-i-e-a-b}
|\eqref{ep-eq-i-e-a}| 
&= \Big| \int_0^s \int_{\widehat{\Omega}^{\rm s}_H} \partial_{tt}{\bs e}_{\dep} \cdot (\bw_H - {\bs w}\circ \widehat{\bs\Phi}_H) \d t \Big| 
+ \Big| \int_0^s \int_{\widehat{\Omega}^{\rm s}_H} \partial_{tt}{\bs e}_{\dep} \cdot ({\bs w}\circ \widehat{\bs\Phi}_H) \d t \Big| \notag\\ 
&\lesssim Hh^k \|\partial_{tt}{\bs e}_{\dep} \|_{L^2(0,s;L^2(\widehat\Omega^{\rm s}_H))} 
+ \Big| \int_0^s \int_{\widehat{\Omega}^{\rm s}_H} \partial_{tt}{\bs e}_{\dep} \cdot ({\bs w}\circ \widehat{\bs\Phi}_H) \d t \Big| , 
\end{align}
where the first term on the right-hand side of \eqref{ep-eq-i-e-a-b} can be estimated by using the first result of \eqref{ep-esigma-estimates} under the condition $H\lesssim h$, i.e., 
\begin{align}
Hh^k \|\partial_{tt}{\bs e}_{\dep} \|_{L^2(0,s;L^2(\widehat\Omega^{\rm s}_H))}  
&\lesssim \varepsilon^{-1}(h^{2k} + H^{2k} )
+ \varepsilon H^2 \int_0^s \|\partial_{tt}{\bs e}_{\dep} \|_{L^2(\widehat{\Omega}^{\rm s}_H)}^2 \d t  \notag\\
&\lesssim \varepsilon^{-1}(h^{2k} + H^{2k}) 
+ \varepsilon H^2\big( \|\grad{\bs e}_{\bs d}(s)\|_{L^2(\widehat\Omega^{\rm s}_H)}^2 + \|\partial_t{\bs e}_{\bs d}(s)\|_{L^2(\widehat\Omega^{\rm s}_H)}^2 \big)\notag\\
&\quad\, 
+  \varepsilon\int_0^s \big(\| {\bs e}_{\bs u}  \|_{1,\Omega[{\bs X}_H]}^2 + H\|\partial_t{\bs e}_\dep\|_{L^2(\widehat{\Omega}^{\rm s}_H)}^2 
+ \|{\bs e}_\dep\|_{H^1(\widehat{\Omega}^{\rm s}_H)}^2\big) \d t . \notag
\end{align}
The second term on the right-hand side of \eqref{ep-eq-i-e-a-b} can be estimated via integration by parts in time, i.e.,
\begin{align}
\Big| \int_0^s \int_{\widehat{\Omega}^{\rm s}_H} \partial_{tt}{\bs e}_{\dep} \cdot ({\bs w}\circ \widehat{\bs\Phi}_H) \d t \Big|
&\le \Big|\int_{\widehat{\Omega}^{\rm s}_H} \partial_t {\bs e}_{\dep}(s) \cdot ({\bs w}(s)\circ\widehat{\bs\Phi}_H) - \int_{\widehat{\Omega}^{\rm s}_H} \partial_t {\bs e}_{\dep}(0) \cdot  ({\bs w}(0)\circ\widehat{\bs\Phi}_H) \Big| \notag\\
&\quad\, + \Big| \int_0^s \int_{\widehat{\Omega}^{\rm s}_H} \partial_t {\bs e}_{\dep} \cdot (\partial_t {\bs w}\circ\widehat{\bs\Phi}_H) \d t \Big| \notag\\
&\lesssim h^k (\|\partial_t {\bs e}_{\dep}(s) \|_{L^2(\widehat{\Omega}^{\rm s}_H)} 
+\|\partial_t{\bs e}_{\dep}(0) \|_{L^2(\widehat{\Omega}^{\rm s}_H)} ) \notag\\
&\quad\, + (\|{\bs e}_{\bs d}\|_{L^2(0,s;H^1(\widehat\Omega^{\rm s}_H))} 
+ h^k + H^k)\|\partial_t {\bs e}_{\dep}\|_{L^2(0,s;L^2(\widehat{\Omega}^{\rm s}_H))} , \notag
\end{align}
where we have used \eqref{estimate-w-wH-1} and \eqref{estimate-dtw-L2} in deriving the last inequality. Since $\|\partial_t{\bs e}_{\dep}(0)\|_{L^2(\widehat\Omega_H^{\rm s})}\lesssim h^k$ as shown in \eqref{dtdh0-error}, combining the last two estimates above, we obtain 
\begin{align}
|\eqref{ep-eq-i-e-a}|
&\le \varepsilon \big( \|\grad{\bs e}_{\bs d}(s)\|_{L^2(\widehat\Omega^{\rm s}_H)}^2 + \|\partial_t{\bs e}_{\bs d}(s)\|_{L^2(\widehat\Omega^{\rm s}_H)}^2 \big) \notag\\
&\quad\, 
+  \varepsilon\int_0^s \big(\| {\bs e}_{\bs u}  \|_{1,\Omega[{\bs X}_H]}^2 + \|\partial_t{\bs e}_\dep\|_{L^2(\widehat{\Omega}^{\rm s}_H)}^2 \big) \d t\notag\\
&\quad\,
+ C_\varepsilon \int_0^s \|{\bs e}_\dep\|_{H^1(\widehat{\Omega}^{\rm s}_H)}^2 \d t 
+ C_\varepsilon(h^{2k} + H^{2k}) . 
\end{align}
From \eqref{estimate-w-wH-3} we see that
\begin{align}
|\eqref{ep-eq-i-e-b}|
&\lesssim \int_0^s \|\grad{\bs e}_{\bs d}\|_{L^2(\widehat{\Omega}_H^{\rm s})} \|{\bs w}_H\|_{H^1(\widehat{\Omega}_H^{\rm s})} \d t
\lesssim h^{k} \|\grad{\bs e}_{\bs d}\|_{L^2(0,s;L^2(\widehat{\Omega}_H^{\rm s}))}  .
\end{align}
We can estimate \eqref{ep-eq-i-e-c} by using the formulas in \eqref{dt-int-Gamma-1}--\eqref{dt-int-Gamma-1c}, with the notations in \eqref{def-eXs} and \eqref{notation-sigma}, i.e.,
\begin{align}
\eqref{ep-eq-i-e-c} 
&= \int_0^s \int_0^1 \int_{\Gamma[\bs X_H^\theta]} {\widehat{\bs\sigma} }_H^{*|\theta}\cdot \bw_H^\theta \grad_{\Gamma[\bs X_H^\theta]}\cdot {\bs e}_{\bs d}^\theta \,\d\theta \d t \notag\\
&=  \int_0^s  \int_{\Gamma[\bs X]} ({\widehat{\bs \sigma} }_H^*\circ\widehat{\bs\Phi}_H^{-1}\circ{\bs X}^{-1})\cdot (\bw_H\circ\widehat{\bs\Phi}_H^{-1}\circ{\bs X}^{-1}) \grad_{\Gamma[\bs X]}\cdot ({\bs e}_{\bs d}\circ\widehat{\bs\Phi}_H^{-1}\circ{\bs X}^{-1}) \, \d t \notag\\
&\quad\, + \int_0^s \int_0^1\int_0^1 \int_{\Gamma[\bs X_H^{\alpha\theta}]} {\widehat{\bs \sigma} }_H^{*|\alpha\theta}\cdot \bw_H^{\alpha\theta} \grad_{\Gamma[\bs X_H^{\alpha\theta}]}\cdot {\bs e}_{\bs d}^{\alpha\theta}\, O(|\grad_{\Gamma[\bs X_H^{\alpha\theta}]} {\bs e}_{\bs d}^{\alpha\theta}|)\,\d\alpha\d\theta \d t \notag\\
&\quad\, + \int_0^s \int_0^1 \int_{\Gamma[\bs X_H^{*,\theta}]} {\widehat{\bs \sigma} }_H^{*,\theta}\cdot \bw_H^{*,\theta} \grad_{\Gamma[\bs X_H^{*,\theta}]}\cdot {\bs e}_{\bs d}^{*,\theta}\, O(|\grad_{\Gamma[\bs X_H^{*,\theta}]} {\bs e}_{\bs X}^{*,\theta}|)\,\d\theta \d t \notag\\
&\le \int_0^s C\|\bw_H\circ\widehat{\bs\Phi}_H^{-1}\circ{\bs X}^{-1}\|_{H^{\frac12}(\Gamma[\bs X])} \|{\bs e}_{\bs d}\circ\widehat{\bs\Phi}_H^{-1}\circ{\bs X}^{-1}\|_{H^{\frac12}(\Gamma[\bs X])} \, \d t \notag\\
&\quad\, 
+ \int_0^s \int_0^1 \int_0^1  C\|\bw_H^{\alpha\theta}\|_{L^4(\Gamma[\bs X_H^{\alpha\theta}])} 
\|{\bs e}_{\bs d}^{\alpha\theta}\|_{H^1(\Gamma[\bs X_H^{\alpha\theta}])} 
\|{\bs e}_{\bs d}^{\alpha\theta}\|_{W^{1,4}(\Gamma[\bs X_H^{\alpha\theta}])}\, \d\alpha\d\theta\d t\notag\\
&\quad\, 
+ \int_0^s \int_0^1 C\|\bw_H^{*,\theta}\|_{L^4(\Gamma[\bs X_H^{*,\theta}])} 
\|{\bs e}_{\bs d}^{*,\theta}\|_{H^1(\Gamma[\bs X_H^{*,\theta}])} 
\|{\bs e}_{\bs X}^{*,\theta}\|_{W^{1,4}(\Gamma[\bs X_H^{*,\theta}])}\, \d\theta\d t\notag\\
&\le \int_0^s C\|\bw_H\|_{H^{\frac12}(\widehat\Gamma_H)} \|{\bs e}_{\bs d}\|_{H^{\frac12}(\widehat\Gamma_H)} \d t 
\quad\mbox{(norm equivalence)} \notag\\
&\quad\, 
+ \int_0^sC\|\bw_H\|_{L^4(\widehat\Gamma_H)} 
H^{-\frac12}\|{\bs e}_{\bs d}\|_{H^1(\widehat{\Omega}_H^{\rm s})} 
H^{-\frac{d-1}{4}-\frac12}\|{\bs e}_{\bs d}\|_{H^1(\widehat{\Omega}_H^{\rm s})}\, \d t \notag\\
&\quad\, 
+ \int_0^s C\|\bw_H\|_{L^4(\widehat\Gamma_H)} 
H^{-\frac12}\|{\bs e}_{\bs d}\|_{H^1(\widehat{\Omega}_H^{\rm s})} H^k \, \d t \quad\mbox{(inverse estimates)}\notag\\
&\le \int_0^s C\|\bw_H\|_{H^1(\widehat{\Omega}_H^{\rm s})} \|{\bs e}_{\bs d}\|_{H^1(\widehat{\Omega}_H^{\rm s})} \d t \quad\mbox{(we have used \eqref{ind-1-3} with $2\le d\le 3$)}\notag\\
&\le Ch^{2k} + C\int_0^s \|{\bs e}_{\bs d}\|_{H^1(\widehat{\Omega}_H^{\rm s})}^2 \d t.
\end{align}
From Lemma \ref{Lemma:RhH} we see that 
\begin{align}
|\eqref{ep-eq-i-e-d}|
&\le C(h^k+H^k) \|{\bs w}_H\|_{H^1(\widehat{\Omega}_H^{\rm s})} 
\le C(h^{2k}+H^{2k}) .
\end{align}
Now, substituting the estimates of \eqref{ep-eq-i-e-a}--\eqref{ep-eq-i-e-d} into \eqref{ep-eq-i-e}, we obtain 
\begin{align}
\Big| \int_0^s \eqref{ep-eq-i} \d t \Big| 
&\le \varepsilon \big( \|\grad{\bs e}_{\bs d}(s)\|_{L^2(\widehat\Omega^{\rm s}_H)}^2 + \|\partial_t{\bs e}_{\bs d}(s)\|_{L^2(\widehat\Omega^{\rm s}_H)}^2 \big) \notag\\
&\quad\, 
+  \varepsilon\int_0^s \big(\| {\bs e}_{\bs u}  \|_{1,\Omega[{\bs X}_H]}^2 + \|\partial_t{\bs e}_\dep\|_{L^2(\widehat{\Omega}^{\rm s}_H)}^2 \big) \d t\notag\\
&\quad\,
+ C_\varepsilon \int_0^s \|{\bs e}_\dep\|_{H^1(\widehat{\Omega}^{\rm s}_H)}^2 \d t 
+ C_\varepsilon(h^{2k} + H^{2k}) .  
\end{align}
Then, substituting this result into \eqref{R-estimate-1}, we have 
\begin{align}\label{R-estimate-2}
&\Big| \int_0^s R_{h,H}^*({\bs e}_{\bs u}, e_p, \partial_t{\bs e}_{\bs d}, {\bs e}_{{\bs \sigma} }) \d t\Big| \notag\\
&\le \varepsilon \big(\|{\bs e}_{\bs u}(s)\|_{L^2(\Omega^{\rm f}[\bs X_H])}^2 + \|\partial_t{\bs e}_\dep(s)\|_{L^2(\widehat{\Omega}^{\rm s}_H)}^2 + \|{\bs e}_\dep(s)\|_{H^1(\widehat{\Omega}^{\rm s}_H)}^2 \big) \notag\\
&\quad\, 
+  \varepsilon\int_0^s \| {\bs e}_{\bs u}  \|_{1,\Omega[{\bs X}_H]}^2  \d t \notag\\
&\quad\,
+ C_\varepsilon \int_0^s \big( \|\partial_t{\bs e}_\dep\|_{L^2(\widehat{\Omega}^{\rm s}_H)}^2 + \|{\bs e}_\dep\|_{H^1(\widehat{\Omega}^{\rm s}_H)}^2 \big)  \d t 
+ C_\varepsilon(h^{2k} + H^{2k}) .  
\end{align}
Since
\begin{align*} 
&\|{\bs e}_{\bs d}(s)\|_{L^2(\widehat{\Omega}^{\rm s}_H)}^2 \notag\\
&= \|{\bs e}_{\bs d}(0)\|_{L^2(\widehat{\Omega}^{\rm s}_H)}^2
+ \int_0^s \frac{\d}{\d t}\|{\bs e}_{\bs d}(t)\|_{L^2(\widehat{\Omega}^{\rm s}_H)}^2 \d t \notag\\
&\le \|{\bs e}_{\bs d}(0)\|_{L^2(\widehat{\Omega}^{\rm s}_H)}^2
+ \int_0^s 2\|{\bs e}_{\bs d}(t)\|_{L^2(\widehat{\Omega}^{\rm s}_H)}
\|\partial_t{\bs e}_{\bs d}(t)\|_{L^2(\widehat{\Omega}^{\rm s}_H)} \d t \notag\\
&\le \|{\bs e}_{\bs d}(0)\|_{L^2(\widehat{\Omega}^{\rm s}_H)}^2
+ \int_0^s \left(\frac12\|{\bs e}_{\bs d}(t)\|_{L^2(\widehat{\Omega}^{\rm s}_H)}^2+
2(s-t)^2\|\partial_t{\bs e}_{\bs d}(t)\|_{L^2(\widehat{\Omega}^{\rm s}_H)}^2\right)\d t , 
\end{align*}
it follows that, by applying Gronwall's inequality to absorb the term $\int_0^s\frac12\|{\bs e}_{\bs d}(t)\|_{L^2(\widehat{\Omega}^{\rm s}_H)}^2\d t$ on the left-hand side and using ${\bs e}_{\bs d}(0)=0$, 
\begin{align*}
&\|{\bs e}_{\bs d}(s)\|_{L^2(\widehat{\Omega}^{\rm s}_H)}^2 
\lesssim \int_0^s \|\partial_t{\bs e}_{\bs d}\|_{L^2(\widehat{\Omega}^{\rm s}_H)}^2 \d t .
\end{align*}  
Therefore, the term $\|{\bs e}_\dep(s)\|_{H^1(\widehat{\Omega}^{\rm s}_H)}^2$ in \eqref{R-estimate-2} can be replaced by $\|\grad{\bs e}_\dep(s)\|_{L^2(\widehat{\Omega}^{\rm s}_H)}^2$. This proves the result of Lemma \ref{Lemma:RhH-2} based on the hypothesis of the existence of ${\bs w}$ and ${\bs w}_H$ satisfying the conditions in \eqref{def-w-wH}. This hypothesis is proved in the next part.\bigskip

\noindent {\bf Part 8:} Existence of ${\bs w}$ and ${\bs w}_H$ satisfying the conditions in \eqref{def-w-wH}. We define $\bs x=\bs y+\tau{\bs n}_{\widehat\Gamma}(\bs y)$ for $\bs y\in\widehat\Gamma$ and $\tau\in[0,\delta]$, with ${\bs n}_{\widehat\Gamma}(\bs y)$ denoting the outward unit normal vector of $\widehat\Gamma$ at $\bs y\in\widehat\Gamma$. For a sufficiently small $\delta$, the map from $(\bs y,\tau)\in\widehat\Gamma\times[0,\delta]$ to $\bs x=\bs y+\tau{\bs n}_{\widehat\Gamma}(\bs y)$ is invertible and bounded in $W^{1,\infty}$ norm, and the inverse map is also bounded in the $W^{1,\infty}$ norm. We denote by $\widehat{\Omega}^{\rm f}_\delta$ a one-sided neighborhood of $\widehat\Gamma$ in $\Omega\backslash\widehat{\Omega}^{\rm s}$, defined as  
$$
\widehat{\Omega}^{\rm f}_\delta =\{\bs  x: \bs x= \bs y+\tau{\bs n}_{\widehat\Gamma}(\bs y)\,\,\mbox{for some}\,\, (\bs y,\tau)\in\widehat\Gamma\times[0,\delta] \} 
$$ 
and extend the map ${\bs X}_H(t)\circ\widehat{\bs\Phi}_H^{-1}:\widehat\Gamma\rightarrow\Gamma[{\bs X}_H]$ to $\widehat{\bs X}_H(t):\widehat{\Omega}^{\rm f}_\delta\rightarrow\Omega^{\rm f}[{\bs X}_H]$ as follows: 
\begin{align}\label{def-map-XH}
\widehat{\bs X}_H(\bs x,t)
&\eqd {\bs X}_H(\widehat{\bs\Phi}_H^{-1}(\bs y),t)  + \tau {\bs n}_{\Gamma[{\bs X}(t)]}({\bs X}(\bs y,t)) . 
\end{align}
Let $\Omega^{\rm f}_\delta(t)$ be the image of $\widehat\Omega_\delta^f$ under the map $\widehat{\bs X}_H(t)$. For sufficiently small $\delta$ and $H$, the normal offset in the definition of $\widehat{\bs X}_H$ stays on the fluid side of $\Gamma[{\bs X}_H]$, whence $\Omega^{\rm f}_\delta(t)\subset\Omega^{\rm f}[{\bs X}_H]$. The discussions in Appendix~\ref{appendix:map_fluid_region} show that the map $\widehat{\bs X}_H(t):\widehat{\Omega}^{\rm f}_\delta \rightarrow \Omega^{\rm f}_\delta(t)\subset\Omega^{\rm f}[{\bs X}_H]$ is invertible, and the pulled-back function $({\bs u}_h^*-{\bs u}_h^\#)\circ\widehat{\bs X}_H$ satisfies the following estimates: 
\begin{align}
\|({\bs u}_h^*-{\bs u}_h^\#)\circ\widehat{\bs X}_H\|_{H^1(\widehat{\Omega}^{\rm f}_\delta)}
&\lesssim \|{\bs u}_h^*-{\bs u}_h^\#\|_{H^1(\Omega^{\rm f}[{\bs X}_H])} 
\lesssim h^k, \\
\|\partial_t[({\bs u}_h^*-{\bs u}_h^\#)\circ\widehat{\bs X}_H]\|_{L^2(\widehat{\Omega}^{\rm f}_\delta)}\big) 
&\lesssim \|{\bs u}_h^*-{\bs u}_h^\#\|_{H^1(\Omega^{\rm f}[{\bs X}_H])}
+ \|\partial_t({\bs u}_h^*-{\bs u}_h^\#)\|_{L^2(\Omega^{\rm f}[{\bs X}_H])} \notag\\
&\lesssim \|{\bs e}_{\bs d}\|_{H^1(\widehat\Omega^{\rm s}_H)}
+ h^k + H^k ,
\end{align} 
where the last inequality is obtained by using \eqref{Ritz-H1} and \eqref{Ritz-L2}. 

We define $\widehat{\Omega}^{\rm f,-}_\delta\eqd \{\bs x= \bs y-\tau{\bs n}_{\widehat\Gamma}(\bs y): (\bs y,\tau)\in\widehat\Gamma\times[0,\delta] \}$ to be the reflection of $\widehat{\Omega}^{\rm f}_\delta$ via $\widehat\Gamma$, and define the value of ${\bs w}$ on $\widehat{\Omega}^{\rm f,-}_\delta$ as follows: 
\begin{align}
{\bs w}(\bs x,t) 
\eqd ({\bs u}_h^*-{\bs u}_h^\#)\circ\widehat{\bs X}_H (\bs x',t) , 
\quad
\bs x= \bs y-\tau{\bs n}_{\widehat\Gamma}(\bs y),\quad \bs x'= \bs y+\tau{\bs n}_{\widehat\Gamma}(\bs y).
\end{align}
This definition (via reflection) guarantees that \eqref{def-w-w} and the following estimate are satisfied: 
\begin{align}
\|{\bs w}\|_{H^1(\widehat{\Omega}^{\rm f,-}_\delta)}
&\lesssim \|({\bs u}_h^*-{\bs u}_h^\#)\circ\widehat{\bs X}_H\|_{H^1(\widehat{\Omega}^{\rm f}_\delta)} 
\lesssim h^k , \\
\|\partial_t{\bs w}\|_{L^2(\widehat{\Omega}^{\rm f,-}_\delta)} 
&\lesssim \|\partial_t[({\bs u}_h^*-{\bs u}_h^\#)\circ\widehat{\bs X}_H]\|_{L^2(\widehat{\Omega}^{\rm f}_\delta)}
\lesssim \|{\bs e}_{\bs d}\|_{H^1(\widehat\Omega^{\rm s}_H)}
+ h^k + H^k .
\label{dtwL2Omegadelta}
\end{align}

Then we can extend ${\bs w}$ to all of $\widehat\Omega^{\rm s}$ by multiplying ${\bs w}$ by a smooth cut-off function $\chi$ that satisfies $\chi=1$ on $\widehat\Gamma$ and $\chi=0$ in $\widehat\Omega^{\rm s}\backslash\widehat{\Omega}^{\rm f,-}_\delta$. This proves \eqref{def-w-w}, \eqref{estimate-w-wH-1} and \eqref{estimate-dtw-L2}. 

Next, we define ${\bs w}_H$ on $\widehat\Gamma_H$ as  
${\bs w}_H = [{\bf P}_{\Gamma[{\bs X}_H]}({\bs w}\circ \widehat{\bs\Phi}_H\circ {\bs X}_H^{-1})] \circ {\bs X}_H $. This definition and the equivalence of norms for functions transported between $\widehat\Gamma_H$ and $\Gamma[{\bs X}_H]$ guarantee that 
$$
\| {\bs w}_H \|_{L^2(\widehat\Gamma_H)} 
\le C\| {\bs w} \|_{L^2(\widehat\Gamma)} . 
$$
Then we extend ${\bs w}_H$ from $\widehat\Gamma_H$ to $\widehat{\Omega}^{\rm s}_H$ as follows: For every node $x$ in the interior of $\widehat{\Omega}^{\rm s}_H$, we define ${\bs w}_H(x,t)$ to be the value of the local $L^2$ projection of ${\bs w}\circ\widehat{\bs\Phi}_H$ onto a patch of size $O(H)$ sharing $x$. This definition guarantees the properties in \eqref{estimate-w-wH-2} and \eqref{estimate-w-wH-3} (see Definition \ref{L2-L2-proj2} and Lemma \ref{Lemma:PH2} in Appendix \ref{appendix:projection}). This proves the existence of ${\bs w}$ and ${\bs w}_H$ satisfying the conditions in \eqref{def-w-wH}.  
\end{proof}

\section{Numerical tests}
\label{sect:num}

We test the convergence of the proposed CutFEM by using the analytical solution of Problem~II in \cite{wang-pironneau-26}, which describes a radially oscillating elastic annulus coupled to an incompressible Newtonian fluid. 
We set $\widehat\Omega^{\rm s} =\{\widehat{\bm x}\in{\mathbb R}^2:\ R_0<|\widehat{\bm x}|<R_1\}$
with $R_0=3$ and $R_1=4$,
and let the fixed computational fluid background domain be $\Omega=(-5,5)^2$. The inner boundary $\widehat\Gamma_0\eqd \{\widehat{\bm x}:|\widehat{\bm x}|=R_0\}$
is clamped, while the outer solid boundary
$
  \widehat\Gamma=\{\widehat{\bm x}:|\widehat{\bm x}|=R_1\}
$
denotes the reference configuration of  the fluid-solid interface. 
It should be noted that the present setting slightly differs from the one of Section~\ref{sect:setting}, as the inner solid boundary $\widehat\Gamma_0$ is clamped. The discrete formulation is modified accordingly by imposing the homogeneous Dirichlet condition on $\widehat\Gamma_0$.
\begin{figure}[h!]
\centering
\includegraphics[width=0.35\textwidth]{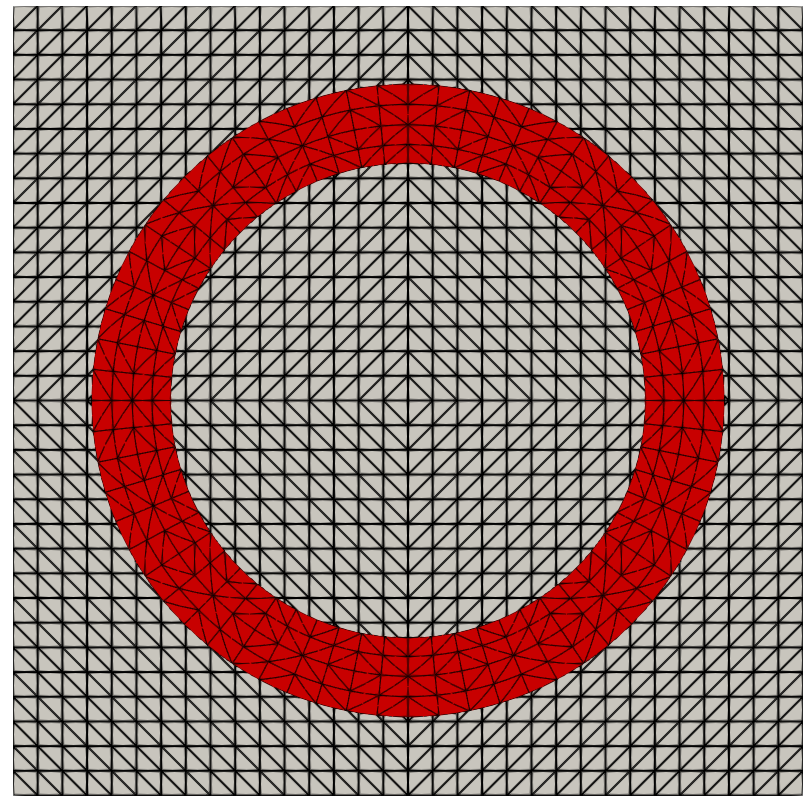}
\caption{
Overlapping meshes of  the fluid and solid computational domains.
}
\label{fig:mesh}
\end{figure}
The current configuration of the interface is the circle
$$
  \Gamma(t)=\{\bm x\in{\mathbb R}^2:\ |\bm x|=r_\Gamma(t)\},
  \qquad
  r_\Gamma(t)\eqd R_1+\bar d\sin(\omega t),
$$
and the fluid occupies the exterior domain
$$
  \Omega^{\rm f}(t)=\{\bm x\in\Omega:\ |\bm x|>r_\Gamma(t)\}.
$$
The interpolated exact velocity is imposed strongly on the outer boundary of the box $\partial \Omega$, while the discrete initial data are defined as in Section~\ref{section:FEM}. 
The expressions of the exact solutions can be found in \cite{wang-pironneau-26}, where we choose the following parameters:
$$
\begin{gathered}
  \bar d=0.05,
  \quad
  \omega=2\pi,\quad T=0.1,\\
  \rho^{\rm f}=1,
  \qquad
  \mu^{\rm f}=2,
  \qquad
  \rho^{\rm s}=2,
  \qquad
  \mu^{\rm s}=10,
  \qquad
  \lambda^{\rm s}=100 .   
\end{gathered}
$$

The overlapping computational meshes are  shown in Figure~\ref{fig:mesh}. The background fluid mesh and the solid mesh are refined simultaneously, with $H\simeq h$. 


Although Theorem~\ref{THM} is established only for $k\geq3$, numerical  experiments  are conducted with $k=2$ and $k=3$  to investigate whether the technical restriction discussed in Remark~\ref{remark0_THM} manifests in practice. We set $\epsilon_h=h^2$ and employ a two-step backward differentiation formula for the time discretization, with time step size $\Delta t=h^2$ for $k=2$ and $\Delta t=h^{2.5}$ for $k=3$. The resulting temporal discretization errors are expected to be $O(h^4)$ and $O(h^5)$, respectively, so that the observed convergence rates are dominated by the spatial discretization error. 


\begin{figure}[h!]
\centering
\vspace{15pt}
\begin{minipage}[t]{0.45\textwidth}
\centering
\includegraphics[width=\textwidth]{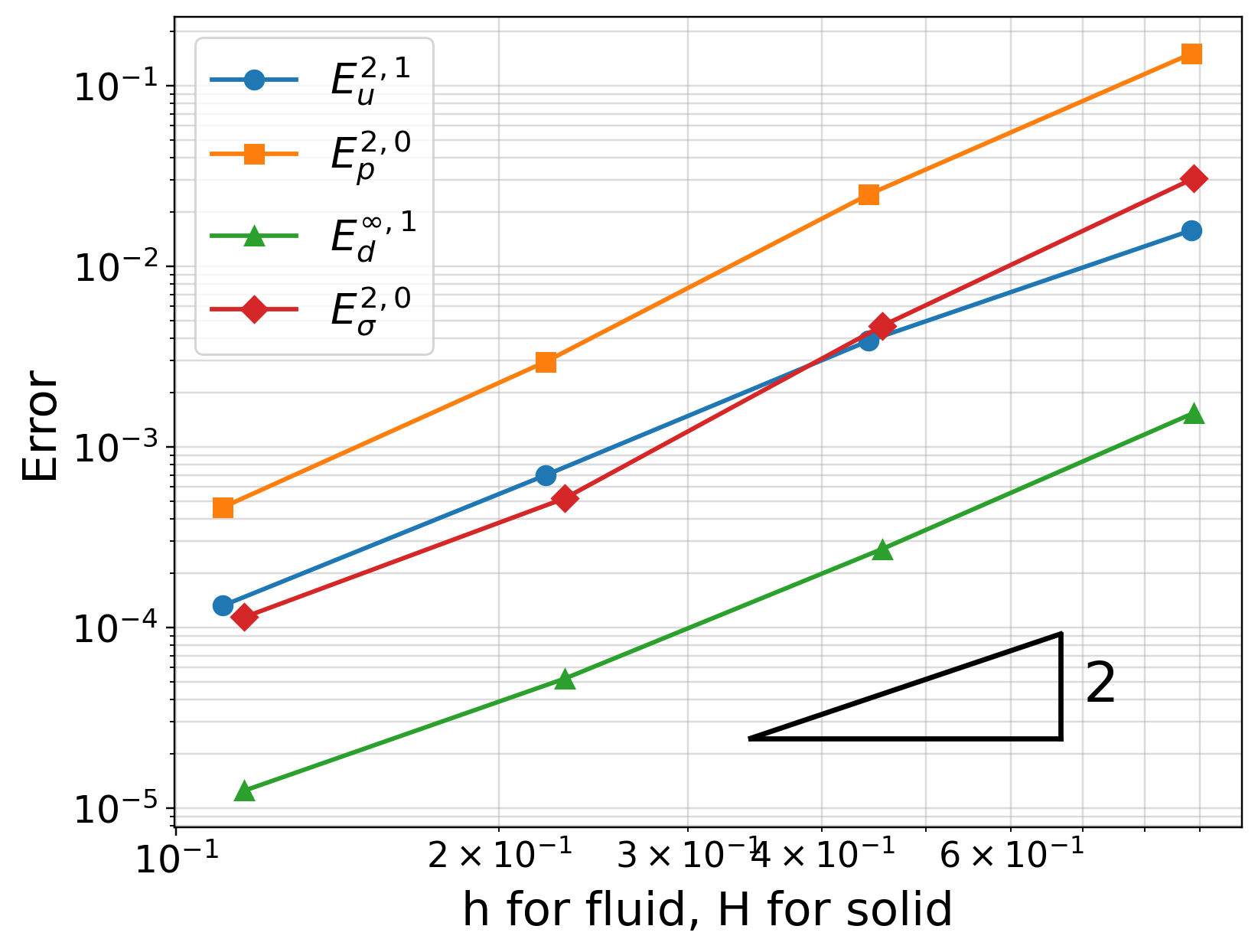}
\end{minipage}
\hfill
\begin{minipage}[t]{0.45\textwidth}
\centering
\includegraphics[width=\textwidth]{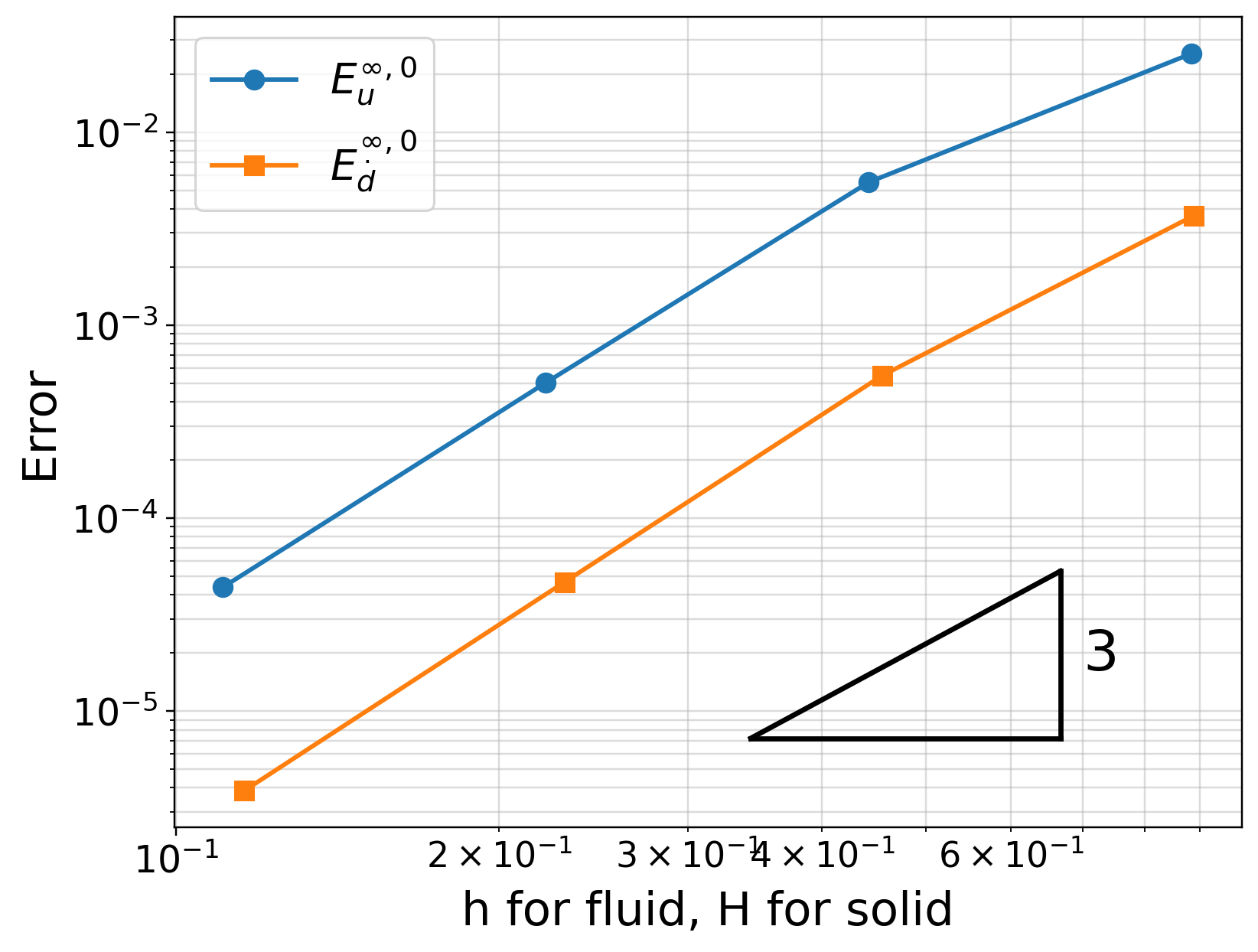}
\end{minipage}
\caption{
Spatial convergence with finite elements of degree $k=2$.  Left: errors in the natural norms of Theorem~\ref{THM}, showing 2nd-order convergence.  Right:
$L^\infty(0,T;L^2)$-type errors, showing an observed 3rd-order convergence.
}
\label{fig:convergence_k=2}
\end{figure}
\begin{figure}[h!]
\centering
\vspace{15pt}
\begin{minipage}[t]{0.45\textwidth}
\centering
\includegraphics[width=\textwidth]{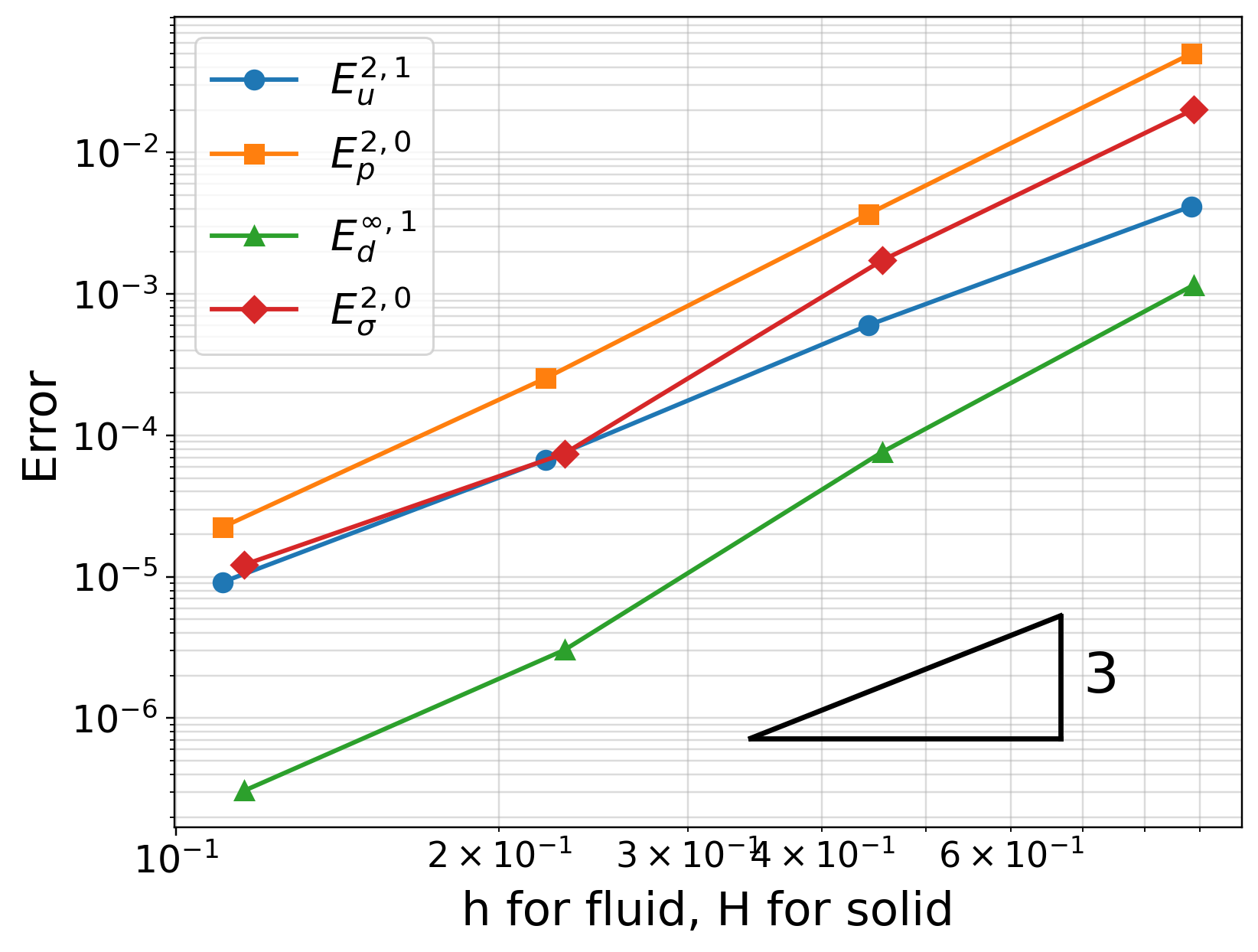}
\end{minipage}
\hfill
\begin{minipage}[t]{0.45\textwidth}
\centering
\includegraphics[width=\textwidth]{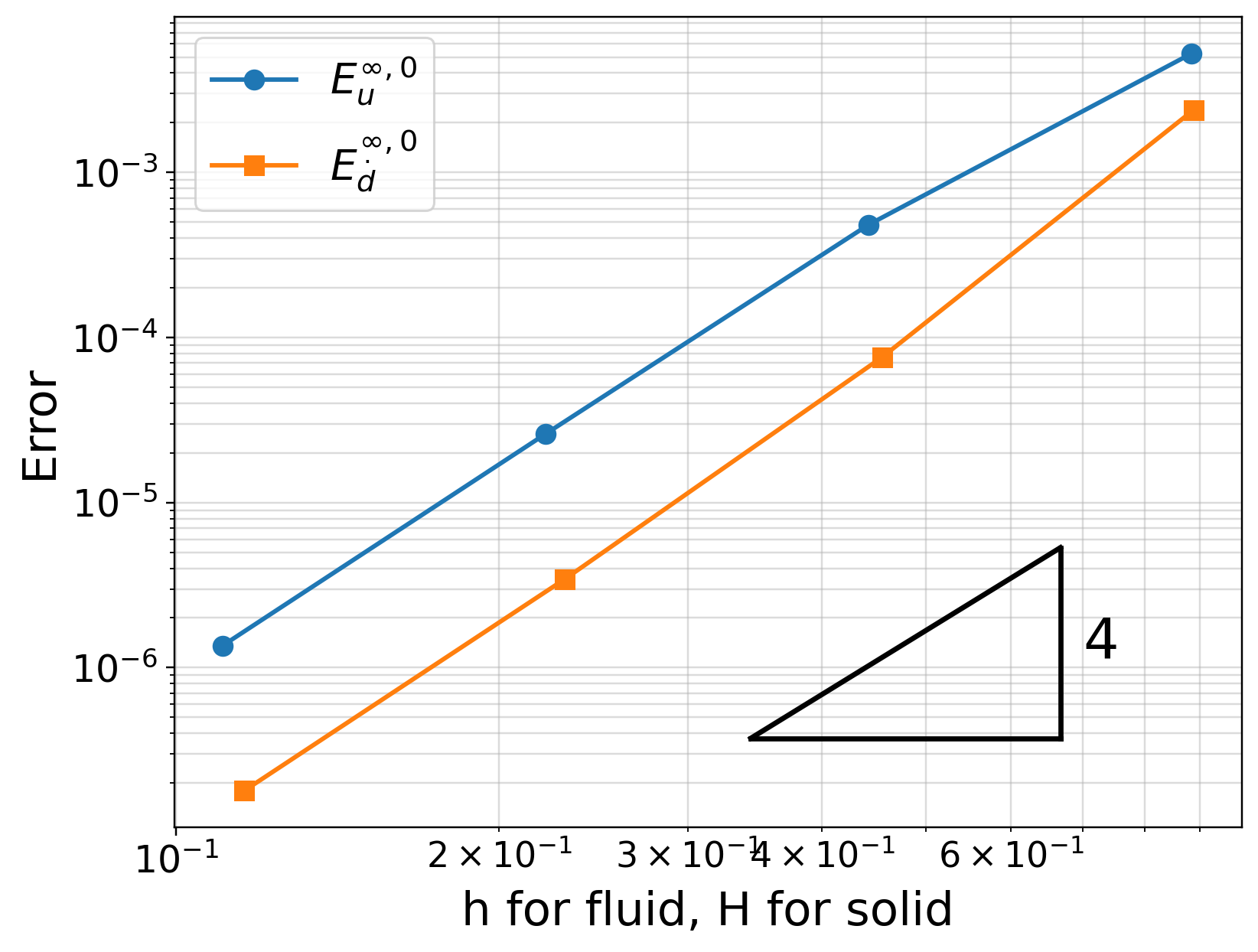}
\end{minipage}
\caption{
Spatial convergence with finite elements of degree $k=3$.  Left: errors in the natural norms of Theorem~\ref{THM}, showing 3rd-order convergence.  Right:
$L^\infty(0,T;L^2)$-type errors, showing an observed 4th-order convergence.
}
\label{fig:convergence_k=3}
\end{figure}
The convergence results for $k=2$ and $k=3$ are shown in Figures \ref{fig:convergence_k=2} and \ref{fig:convergence_k=3}, respectively, where we examine the following norms of the errors appearing in Theorem~\ref{THM}, i.e., 
\begin{align*}
E_u^{\infty,0}
&\eqd
\max_{0\le n\le N}
\|\bm u(t_n)-\bm u_h^n\|_{L^2(\Omega^{\rm f}(t_n))},
&&
E_u^{2,1}
\eqd
\Big(
\Delta t\sum_{n=1}^N
\|\bm u(t_n)-\bm u_h^n\|_{H^1(\Omega^{\rm f}(t_n))}^2
\Big)^{\frac12},\\
E_p^{2,0}
&\eqd
\Big(
\Delta t\sum_{n=1}^N
\|p(t_n)-p_h^n\|_{L^2(\Omega^{\rm f}(t_n))}^2
\Big)^{\frac12}, 
&&
E_d^{\infty,1}
\eqd
\max_{0\le n\le N}
\|\bm d(t_n)-\bm d_H^n\|_{H^1(\widehat\Omega_H^{\rm s})},\\
E_d^{\infty,0}
&\eqd
\max_{0\le n\le N}
\|\bm d(t_n)-\bm d_H^n\|_{L^2(\widehat\Omega_H^{\rm s})}, 
&&
E_{\dot d}^{\infty,0}
\eqd
\max_{0\le n\le N}
\|\partial_t\bm d(t_n)-\dot{\bm d}_H^n\|_{L^2(\widehat\Omega_H^{\rm s})},\\
E_\sigma^{2,0}
&\eqd
\Big(
\Delta t\sum_{n=1}^N
\|\widehat{\bm\sigma}(t_n)-\widehat{\bm\sigma}_H^n\|_{L^2(\widehat\Gamma_H)}^2
\Big)^{\frac12} ,
\end{align*} 
where we use the $L^2(\widehat\Gamma_H)$-norm in $E_{\sigma}^{2,0}$ because it is difficult to calculate the $H^{-\frac12}(\widehat\Gamma_H)$ norm numerically.

From Figures \ref{fig:convergence_k=2} and \ref{fig:convergence_k=3}  we see that the velocity error and solid-displacement error exhibit the expected $k$th-order convergence in the energy norms provided by Theorem~\ref{THM}. The pressure and traction errors also display $k$th-order convergence in this experiment, which is one order higher than the rate directly implied by Theorem~\ref{THM}.
The results for the case $k=2$ indicate that the technical finite element degree restriction $k\ge 3$ may be relaxed; see the discussions in Remark \ref{remark0_THM}. We additionally observe $(k+1)$th-order convergence for the $L^\infty(0,T;L^2)$-type errors $E_u^{\infty,0}$ and $E_d^{\infty,0}$. These convergence rates are not covered by the present analysis.

\section*{Acknowledgement}
The authors would like to thank Dr. Guotao Lin for the computer implementation of  the fully discrete numerical scheme and for providing the numerical results of Section~\ref{sect:num}. 

\appendix

\section{Local well-posedness of the numerical scheme}
\label{section:wellposed}

In this section, we show that the spatial semidiscrete scheme  \eqref{eq:discrete-semi-fsi} can be equivalently formulated as a system of ordinary differential equations (ODEs), which guarantees the local existence and uniqueness of solution. For this purpose, we let $Q_h^{\rm f,\perp}$ be the orthogonal complement of $\mathbb{P}_{k-1}(\Omega)$ in $Q_h^{\rm f}$ with respect to the $L^2(\Omega)$ inner product, i.e., 
$$
Q_h^{\rm f,\perp} \eqd  \bigg\{ q_h\in Q_h^{\rm f}:
\int_{\Omega} q_h\chi_h= 0
\quad \forall  \chi_h\in \mathbb{P}_{k-1}(\Omega) \bigg\} .
$$
Then $g_h^{\rm p} (p_h,q_h)=0$ for all $q_h\in \mathbb{P}_{k-1}(\Omega) $, and $\mathbb{P}_{k-1}(\Omega)$ is the kernel of $g_h^{\rm p}$, so that $g_h^{\rm p}$ is an inner product on $Q_h^{\rm f,\perp}$ (see, e.g., \cite{burman-10,massing-et-al-14,GuzmOlsh}).  For $p_h \in Q_h^{\rm f}$, we let $p_h= p_h^\perp + \hat p_h$ be its direct sum decomposition, with $p_h^\perp\in Q_h^{\rm f,\perp}$ and $\hat p_h\in \mathbb{P}_{k-1}(\Omega) $. 

Setting $q_h=0$ and ${\bs\lambda}_H=\bs 0$ in \eqref{eq:discrete-semi-fsi} yields (using the notation ${\bs\dot{\bs d}}_H\eqd \partial_{t}{\bs d}_H$)
\begin{align}\label{eq:discrete-semi-fsi2}
&\int_{\Omega^{\rm f}[\bs X_H]} \partial_t \bs u_h \cdot {\bs v}_h 
+ \epsilon_h g_h^{\rm u} (\partial_t\bu_h,\bv_h ) 
-\int_{\Omega^{\rm f}[\bs X_H]} \hat p_h \grad \cdot {\bs v}_h \notag\\
&\quad\,
+ \int_{\widehat{\Omega}^{\rm s}_H} \partial_{t}{\bs\dot{\bs d}}_H \cdot \bw_H - \int_{\Gamma[\bs X_H]} \widehat{\bs \sigma} _H \circ {\bs X}_H^{-1} \cdot ( \bv_h - \bw_H \circ {\bs X}_H^{-1}) \notag\\
&= 
{\bs F}_1({\bs v}_h) + {\bs F}_2({\bs w}_H) 
\qquad\forall\, ({\bs v}_h,{\bs w}_H)\in {\bs V}_h^{\rm f} \times \widehat{\bs V}_H^{\rm s} ,
\end{align}
where ${\bs F}_1$ and  ${\bs F}_2$ are linear functionals   defined by 
\begin{align*}
{\bs F}_1({\bs v}_h)
&\eqd - \int_{\Omega^{\rm f}[\bs X_H]}   (\bs u_h \cdot \grad) \bu_h \cdot {\bs v}_h 
-   \int_{\Omega^{\rm f}[\bs X_H]} \bs D(\bu_h): \bs D(\bv_h)  \notag\\
&\quad\, + \int_{\Omega^{\rm f}[\bs X_H]} p_h^\perp \grad\cdot \bv_h - g_h^{\rm u} (\bu_h,\bv_h ) ,\\
{\bs F}_2({\bs w}_H)
&\eqd  - \int_{\widehat{\Omega}^{\rm s}_H} [{\bs D}({\bs d}_H) : {\bs D}({\bs w}_H) + (\grad\cdot{\bs d}_H)(\grad\cdot{\bs w}_H) ] .
\end{align*}

If $\big(\bu_h,p_h, {\bs d}_H, \widehat{\bs\sigma} _H\big)\in {\bs V}_h^{\rm f} \times  Q_h^{\rm f} \times \widehat{\bs V}_H^{\rm s} \times \widehat{\bs \Lambda}_H$ is a temporally smooth solution of \eqref{eq:discrete-semi-fsi}, then choosing ${\bs v}_h=\bs 0$, ${\bs w}_H=\bs 0$ and ${\bs \lambda}_H=\bs 0$ in \eqref{eq:discrete-semi-fsi} yields 
\begin{align}\label{eq:semi-qh}
\int_{\Omega^{\rm f}[\bs X_H]} q_h \grad\cdot \bu_h + g_h^{\rm p} (p_h,q_h) = 0 ,
\end{align}
and differentiating in time \eqref{eq:semi-qh} yields 
\begin{equation}\label{dteq:semi-qh}
\begin{aligned}
\int_{\Omega^{\rm f}[\bs X_H]} q_h \grad\cdot \partial_t\bu_h +
g_h^{\rm p} (\partial_tp_h,q_h ) 
= {\bs F}_3(q_h) \quad\forall\, q_h\in Q_h^{\rm f}, 
\end{aligned}
\end{equation}
where
$$
{\bs F}_3(q_h) \eqd  - \int_{\Gamma[\bs X_H]} q_h (\grad\cdot\bu_h){\bs\dot{\bs d}}_H\cdot{\bs n}_{\Gamma[\bs X_H]} . 
$$
If we choose $q_h=\hat q_h\in \mathbb{P}_{k-1}(\Omega)$ in \eqref{dteq:semi-qh}, then $g_h^{\rm p} (\partial_tp_h,\hat q_h) =0$ and therefore
\begin{equation}\label{dteq:semi-qh2}
\begin{aligned}
\int_{\Omega^{\rm f}[\bs X_H]} \hat q_h \grad\cdot \partial_t\bu_h
= {\bs F}_3(\hat q_h) \quad\forall\, \hat q_h\in \mathbb{P}_{k-1}(\Omega) .
\end{aligned}
\end{equation}
On the other hand, if we restrict $q_h=q_h^\perp\in Q_h^{\rm f,\perp}$ in \eqref{dteq:semi-qh}, then $g_h^{\rm p} (\partial_tp_h,q_h^\perp) = g_h^{\rm p} (\partial_tp_h^\perp,q_h^\perp)$ and therefore
\begin{equation}\label{dteq:semi-qh3}
\begin{aligned}
g_h^{\rm p} (\partial_t p_h^\perp,q_h^\perp) = 
- \int_{\Omega^{\rm f}[\bs X_H]} q_h^\perp \grad\cdot \partial_t\bu_h
- {\bs F}_3(q_h^\perp) \quad\forall\, q_h^\perp\in Q_h^{\rm f,\perp}.
\end{aligned}
\end{equation}

Similarly, choosing ${\bs v}_h=\bs 0$, ${\bs w}_H=\bs 0$ and $q_h=0$ in \eqref{eq:discrete-semi-fsi} yields 
\begin{align}\label{eq:semi-lambdah}
\int_{\Gamma[\bs X_H]} \bs \lambda_H \circ {\bs X}_H^{-1} \cdot ( \bu_h - {\bs\dot{\bs d}}_H \circ {\bs X}_H^{-1}) = 0 .
\end{align}
and differentiating in time \eqref{eq:semi-lambdah} yields 
\begin{equation}\label{dteq:semi-lambdah}
\int_{\Gamma[\bs X_H]} {\bs\lambda}_H \circ {\bs X}_H^{-1} \cdot ( \partial_t{\bs u}_h - (\partial_t{\bs\dot{\bs d}}_H) \circ {\bs X}_H^{-1}) = {\bs F}_4({\bs\lambda}_H) 
\quad\forall\, {\bs\lambda}_H\in \widehat{\bs \Lambda}_H,
\end{equation}
with 
$$
{\bs F}_4({\bs\lambda}_H)
\eqd  - \int_{\Gamma[\bs X_H]} {\bs\lambda}_H \circ {\bs X}_H^{-1} \cdot ( {\bs\dot{\bs d}}_H\cdot\grad {\bs u}_h) -\int_{\Gamma[\bs X_H]} {\bs\lambda}_H \circ {\bs X}_H^{-1} \cdot ({\bs u}_h - {\bs\dot{\bs d}}_H \circ {\bs X}_H^{-1}) \grad_{\Gamma[\bs X_H]}\cdot {\bs\dot{\bs d}}_H .
$$
Therefore, we have shown that if $\big(\bu_h,p_h, {\bs d}_H, \widehat{\bs\sigma} _H\big)\in {\bs V}_h^{\rm f} \times  Q_h^{\rm f} \times \widehat{\bs V}_H^{\rm s} \times \widehat{\bs \Lambda}_H$ is a temporally smooth solution of \eqref{eq:discrete-semi-fsi}, then $\big(\partial_t\bu_h, \hat p_h,\partial_{t}{\bs\dot{\bs d}}_H,\widehat{\bs \sigma} _H\big)\in {\bs V}_h^{\rm f} \times \mathbb{P}_{k-1}(\Omega) \times \widehat{\bs V}_H^{\rm s}\times \widehat{\bs \Lambda}_H$ satisfies \eqref{eq:discrete-semi-fsi2}, \eqref{dteq:semi-qh2} and \eqref{dteq:semi-lambdah}, and $\partial_tp_h^\perp\in Q_h^{\rm f,\perp}$ satisfies \eqref{dteq:semi-qh3}. 

Conversely, for $\epsilon_h>0$ and given $\big(\bu_h, p_h^\perp, {\bs d}_H, {\bs\dot{\bs d}}_H\big)\in {\bs V}_h^{\rm f} \times  Q_h^{\rm f} \times \widehat{\bs V}_H^{\rm s} \times \widehat{\bs V}_H^{\rm s}$ (with $\Omega^{\rm f}[{\bs X}_H]$ well defined such as at $t=0$), we can uniquely determine 
$(\partial_t\bu_h, \hat p_h,\partial_{t}{\bs\dot{\bs d}}_H,\widehat{\bs \sigma} _H ) \in {\bs V}_h^{\rm f} \times \mathbb{P}_{k-1}(\Omega) \times \widehat{\bs V}_H^{\rm s}\times \widehat{\bs \Lambda}_H$ 
from \eqref{eq:discrete-semi-fsi2}, \eqref{dteq:semi-qh2} and \eqref{dteq:semi-lambdah}. This can be shown by considering whether the corresponding homogeneous linear system has only a trivial zero solution. Indeed, for the homogeneous linear system corresponding to ${\bs F}_j=0$, $j=1,\ldots,4$, in \eqref{eq:discrete-semi-fsi2}, \eqref{dteq:semi-qh2} and \eqref{dteq:semi-lambdah}, choosing $({\bs v}_h, \hat q_h,{\bs w}_H,{\bs\lambda}_H)=( \partial_t\bu_h, \hat p_h ,\partial_{t}{\bs\dot{\bs d}}_H,\widehat{\bs \sigma} _H)$ immediately gives
$$
\int_{\Omega^{\rm f}[\bs X_H]} |\partial_t {\bs u}_h|^2 
+ \epsilon_h g_h^{\rm u} (\partial_t{\bs u}_h,\partial_t{\bs u}_h) 
+ \int_{\widehat{\Omega}^{\rm s}_H} |\partial_{t}{\bs\dot{\bs d}}_H |^2 = 0 .
$$
Here $g_h^{\rm u} (\partial_t{\bs u}_h,\partial_t{\bs u}_h) =0$ enforces $\partial_t{\bs u}_h$ to be a polynomial of degree $k$ (see \cite{burman-10,massing-et-al-14,GuzmOlsh}) and $\int_{\Omega^{\rm f}[\bs X_H]} |\partial_t {\bs u}_h|^2 =0$ further enforces this polynomial to be zero. Hence $\partial_t {\bs u}_h$ and $\partial_{t}{\bs\dot{\bs d}}_H$ must be both zero. Then choosing $\bw_H \circ {\bs X}_H^{-1}$ to be the $L^2$ projection of ${\bs v}_h$ onto the finite element space ${\bs \Lambda}_H(\Gamma[{\bs X}_H])$ on $\Gamma[{\bs X}_H]$, we have 
\begin{align*}
\int_{\Gamma[\bs X_H]} \widehat{\bs \sigma} _H \circ {\bs X}_H^{-1} \cdot ( \bv_h - \bw_H \circ {\bs X}_H^{-1}) = 0 
\end{align*}
and substituting this into \eqref{eq:discrete-semi-fsi2} (with ${\bs F}_1=0$ and ${\bs F}_2=0$ therein)
\begin{align*}
&\int_{\Omega^{\rm f}[\bs X_H]} \hat p_h \grad \cdot {\bs v}_h = 0 
\qquad\forall\, {\bs v}_h\in {\bs V}_h^{\rm f} . 
\end{align*}
This and the inf-sup condition (see \eqref{inf-sup-ep-L2} in Section \ref{section:ep}), as well as $g^{\rm p}_h(\hat p_h,\hat p_h)=0$ because of $\hat p_h\in \mathbb{P}_{k-1}(\Omega)$, imply that $\hat p_h=0$. Then \eqref{eq:discrete-semi-fsi2} (with ${\bs F}_1=0$ and ${\bs F}_2=0$ therein) reduces to 
\begin{align*}
\int_{\Gamma[\bs X_H]} \widehat{\bs \sigma}_H \circ {\bs X}_H^{-1} \cdot{\bs w}_H \circ {\bs X}_H^{-1} = 0 \quad\forall\, {\bs w}_H\in \widehat{\bs \Lambda}_H ,
\end{align*}
which gives $\widehat{\bs \sigma}_H=\bs 0$. 

This proves that the finite-dimensional homogeneous linear system corresponding to ${\bs F}_j=0$, $j=1,\ldots, 4$, in \eqref{eq:discrete-semi-fsi2}, \eqref{dteq:semi-qh2} and \eqref{dteq:semi-lambdah} has only the zero solution $(\partial_t\bu_h, \hat p_h, \partial_t{\bs\dot{\bs d}}_H,\widehat{\bs\sigma}_H)=(\bs 0,0,\bs 0,\bs 0)$, and therefore the inhomogeneous linear system \eqref{eq:discrete-semi-fsi2}, \eqref{dteq:semi-qh2} and \eqref{dteq:semi-lambdah} has a unique solution $(\partial_t\bu_h, \hat p_h, \partial_t{\bs\dot{\bs d}}_H,\widehat{\bs\sigma}_H)$. Then $\partial_tp_h^\perp$ can be determined uniquely from \eqref{dteq:semi-qh3} because $g_h^{\rm p} (\cdot,\cdot)$ is an inner product on $Q_h^{\rm f,\perp}$. 

Therefore, $(\partial_t\bu_h, \partial_t p_h^\perp, \partial_t{\bs d}_H, \partial_t{\bs\dot{\bs d}}_H)$ can be determined by $(\bu_h, p_h^\perp, {\bs d}_H, {\bs\dot{\bs d}}_H)$ from \eqref{eq:discrete-semi-fsi2}, \eqref{dteq:semi-qh2}, \eqref{dteq:semi-lambdah} and \eqref{dteq:semi-qh3}. This is essentially a system of ODEs, i.e., 
$$
(\partial_t\bu_h, \partial_tp_h^\perp, \partial_t{\bs d}_H, \partial_t{\bs\dot{\bs d}}_H)
= {\bs F} (\bu_h,p_h^\perp, {\bs d}_H, {\bs\dot{\bs d}}_H) . 
$$
As long as ${\bs X}_H(\cdot,t)$ remains an orientation-preserving
diffeomorphism (such as at time $t=0$), the geometry-dependent finite-dimensional operators are locally Lipschitz. Standard finite-dimensional ODE theory
therefore gives a unique local solution (existence and uniqueness in a short time period). 
The initial data satisfy both the discrete divergence constraint and
the discrete interface constraint as shown in \eqref{def-Ritz-0} (choosing ${\bs v}_h=\bs 0$ therein) and \eqref{initial-constraint}, while equations \eqref{dteq:semi-qh2} and
\eqref{dteq:semi-lambdah} guarantee that the time derivatives of
the corresponding constraint residuals vanish. Hence the constraints
remain valid throughout the interval of existence, and the ODE
solution is a solution of the original semidiscrete scheme. 


\section{A map from a one-sided neighborhood of $\widehat\Gamma$ into \texorpdfstring{$\Omega^{\rm f}[{\bs X}_H]$}{x}}\label{appendix:map_fluid_region}

For a sufficiently small $\delta$, the map from $(\bs y,\tau)\in\widehat\Gamma\times[0,\delta]$ to $\bs x=\bs y+\tau{\bs n}_{\widehat\Gamma}(\bs y) \in \mathbb{R}^d$, with ${\bs n}_{\widehat\Gamma}$ being the outward unit normal vector of $\widehat\Gamma$ at $y\in\widehat\Gamma$, is invertible and bounded in the $W^{1,\infty}$ norm, with the inverse map also bounded in the $W^{1,\infty}$ norm. We denote by $\widehat{\Omega}^{\rm f}_\delta$ a one-sided tubular neighborhood of $\widehat\Gamma$ on the exterior side of $\widehat{\Omega}^{\rm s}$, defined as  
$$
\widehat{\Omega}^{\rm f}_\delta \eqd \{ \bs x= \bs y+\tau{\bs n}_{\widehat\Gamma}(\bs y): (\bs y,\tau)\in\widehat\Gamma\times[0,\delta] \} 
$$ 
and extend the map ${\bs X}_H(\cdot,t)\circ\widehat{\bs\Phi}_H^{-1}:\widehat\Gamma\rightarrow\Gamma[{\bs X}_H]$ to $\widehat{\bs X}_H(\cdot,t):\widehat{\Omega}^{\rm f}_\delta\rightarrow\Omega^{\rm f}[{\bs X}_H]$ 
as follows: 
\begin{align}\label{def-map-XH-appd}
\widehat{\bs X}_H(\bs x,t)
\eqd &  {\bs X}_H(\widehat{\bs\Phi}_H^{-1}(\bs y),t)  + \tau {\bs n}_{\Gamma[{\bs X}(\cdot,t)]}({\bs X}(\bs y,t)) \notag\\
=& \underbrace{[{\bs X}(\bs y,t)  + \tau {\bs n}_{\Gamma[{\bs X}(\cdot,t)]}({\bs X}(\bs y,t))]}_{\eqd {\bs F}(\bs x,t)} + \underbrace{[{\bs X}_H(\widehat{\bs\Phi}_H^{-1}(\bs y),t) - {\bs X}(\bs y,t)]}_{\eqd {\bs E}(\bs x,t)},
\end{align}
where ${\bs F}(\cdot,t) $ is an invertible map with $\|{\bs F}(\cdot,t)\|_{W^{1,\infty}(\widehat{\Omega}^{\rm f}_\delta)}$ and $\|1/\det \!\grad {\bs F}(\cdot,t) \|_{L^\infty(\widehat{\Omega}^{\rm f}_\delta)}$ being uniformly bounded for $t\in[0,T]$, 
and ${\bs E}(\cdot,t) $ is $O(H^{\frac14})$ in the $W^{1,\infty}$ norm for $t\in[0,t_*]$ in view of the estimates in \eqref{aux865} and \eqref{ind-1-2}. Therefore, for sufficiently small $H$, the total map $\widehat{\bs X}_H(\cdot,t) = {\bs F}(\cdot,t) + {\bs E}(\cdot,t)$ is invertible for $t\in[0,t_*]$ and satisfies the following estimate:
\begin{align}\label{W1inf-hat-XH}
\| \widehat{\bs X}_H(\cdot,t)\|_{W^{1,\infty}(\widehat{\Omega}^{\rm f}_\delta)} 
+ \| \widehat{\bs X}_H(\cdot,t)^{-1} \|_{W^{1,\infty}(\Omega^{\rm f}_\delta(t))} 
\le C ,
\end{align}
where $\Omega^{\rm f}_\delta(t)$ is the image of $\widehat{\Omega}^{\rm f}_\delta $ under the map $\widehat{\bs X}_H(\cdot,t)$. 
The estimate of $\|\partial_t{\bs e}_{\bs d}\|_{L^\infty(\widehat\Omega^{\rm s}_H)}\lesssim H^{\frac14}$ in \eqref{ind-1-2} and the finite element inverse estimate, together with the expression of $\widehat{\bs X}_H(x,t)$ in \eqref{def-map-XH-appd}, guarantee that 
\begin{align}\label{Wtinf-hat-XH}
\max_{t\in[0,t_*]}\|\partial_t\widehat{\bs X}_H(\cdot,t)\|_{L^{\infty}(\widehat{\Omega}^{\rm f}_\delta)}
\le C .
\end{align} 

Since $\widehat{\bs X}_H(\cdot,t)$ maps $\widehat\Gamma$ to $\Gamma[{\bs X}_H]$, it follows that 
$\Omega^{\rm f}_\delta(t)\subset\Omega^{\rm f}[{\bs X}_H]$. 
Therefore, by using the chain rule of partial differentiation and the estimates in \eqref{W1inf-hat-XH}--\eqref{Wtinf-hat-XH}, as well as the estimates in \eqref{Ritz-H1} and \eqref{Ritz-L2}, we see that the pulled-back function $({\bs u}_h^*-{\bs u}_h^\#)\circ\widehat{\bs X}_H$ satisfies the following estimates:
\begin{align}
\|({\bs u}_h^*-{\bs u}_h^\#)\circ\widehat{\bs X}_H\|_{H^1(\widehat{\Omega}^{\rm f}_\delta)}
&\lesssim \|{\bs u}_h^*-{\bs u}_h^\#\|_{H^1(\Omega^{\rm f}[{\bs X}_H])} 
\lesssim h^k, \\
\|\partial_t[({\bs u}_h^*-{\bs u}_h^\#)\circ\widehat{\bs X}_H]\|_{L^2(\widehat{\Omega}^{\rm f}_\delta)}
&\lesssim \|{\bs u}_h^*-{\bs u}_h^\#\|_{H^1(\Omega^{\rm f}[{\bs X}_H])}
+ \|\partial_t({\bs u}_h^*-{\bs u}_h^\#)\|_{L^2(\Omega^{\rm f}[{\bs X}_H])} \notag\\
&\lesssim \|{\bs e}_{\bs d}\|_{H^1(\widehat\Omega^{\rm s}_H)}
+ h^k + H^k .
\end{align}

\section{Trace inequality from \texorpdfstring{$\Gamma[{\bs X}_H^\theta]$}{x} uniformly for $\theta\in[0,1]$ and $t\in[0,t_*]$}
\label{appendix:trace_inequlaity}

Using the mapping $\widehat{\bs X}_H(\cdot,t): \widehat{\Omega}^{\rm f}_\delta \to \Omega^{\rm f}_\delta(t)\subset \Omega^{\rm f}_{\rm ext}[{\bs X}_H]$, we can pull back a function $f$ defined on $\Omega^{\rm f}_{\rm ext}[{\bs X}_H]$ to $\widehat{\Omega}^{\rm f}_\delta$. By \eqref{W1inf-hat-XH}, the pullback $f\circ \widehat{\bs X}_H(\cdot,t)$ satisfies the following norm-equivalence relations:
\begin{align*}
\begin{aligned}
 \|f\circ \widehat{\bs X}_H(\cdot,t)\|_{W^{1,p}(\widehat{\Omega}^{\rm f}_\delta)}
&\simeq \|f\|_{W^{1,p}(\Omega^{\rm f}_\delta(t))} , \\
 \|f\circ \widehat{\bs X}_H(\cdot,t)\|_{L^p(\widehat\Gamma)}
&\simeq \|f\|_{L^p(\Gamma[{\bs X}_H])} , \\
\|f\circ \widehat{\bs X}_H(\cdot,t)\|_{W^{1,p}(\widehat\Gamma)}
&\simeq \|f\|_{W^{1,p}(\Gamma[{\bs X}_H])}
\end{aligned}
\end{align*}
for all $1 \le p\le \infty$.
The last two inequalities imply that, through complex interpolation between the $L^p$ and $W^{1,p}$ spaces, 
$$
\|f\circ \widehat{\bs X}_H(\cdot,t)\|_{W^{s,p}(\widehat\Gamma)}
\simeq \|f\|_{W^{s,p}(\Gamma[{\bs X}_H])}
\quad \forall  s \in [0,1], \quad  p \in [1,\infty] . 
$$
Because of these norm equivalence relations, the trace inequality from $\widehat\Gamma$ to $\widehat{\Omega}^{\rm f}_\delta$ carries over to the corresponding trace inequality from $\Gamma[{\bs X}_H]$ to $\Omega^{\rm f}_\delta(t)$, with equivalence constants that are uniform for all $t\in[0,t_*]$. Namely, we have 
$$
\|f\|_{W^{1-1/p,p}(\Gamma[{\bs X}_H])}
 \le C\|f\|_{W^{1,p}(\Omega^{\rm f}_{\rm ext}[{\bs X}_H])} 
$$
for all $ 1\le p\le \infty$ and $t\in[0,t_*]$. 
Similarly, the same result also holds for the intermediate interfaces $\Gamma[{\bs X}_H^\theta]$ with $\theta\in[0,1]$ (the argument is similar and therefore omitted): 
$$
\|f\|_{W^{1-1/p,p}(\Gamma[{\bs X}_H^\theta])}
 \le C\|f\|_{W^{1,p}(\Omega^{\rm f}_{\rm ext}[{\bs X}_H])} 
$$
for all $ 1\le p\le \infty$ and $t\in[0,t_*]$. 

\section{Boundedness of \texorpdfstring{${\bf P}_{\Gamma[{\bs X}_H]}$}{x} on \texorpdfstring{$H^s(\Gamma[{\bs X}_H])$}{x}}
\label{appendix:L2-projection}

Recall that ${\bf P}_{\Gamma[\bs X_H]}$ denotes the $L^2$-orthogonal projection onto the finite element space ${\bs\Lambda}_H(\Gamma[\bs X_H])$. At time $t=0$, it is known that the $L^2$ projection ${\bf P}_{\widehat\Gamma_H}$ is bounded in both $L^2(\widehat\Gamma_H)$  and $H^1(\widehat\Gamma_H)$, and has the optimal-order approximation property (see \cite[Appendix B]{Li-Tang-2025}), i.e., 
\begin{align}
&\|{\bf P}_{\widehat\Gamma_H} f\|_{L^2(\widehat\Gamma_H)} 
\le C\|f\|_{L^2(\widehat\Gamma_H)},\quad  
\|{\bf P}_{\widehat\Gamma_H}f\|_{H^1(\widehat\Gamma_H)} 
\le C\|f\|_{H^1(\widehat\Gamma_H)} ,\\
&\|f - {\bf P}_{\widehat\Gamma_H} f\|_{L^2(\widehat\Gamma_H)} 
\le CH\|f\|_{H^1(\widehat\Gamma_H)} .
\end{align}
The norm equivalence relations in \eqref{norm-equiv-interpl-1} and \eqref{norm-equiv-2} imply the following estimates for $t\in[0,t_*]$: 
\begin{subequations}\label{norm-equiv-GXH-G0}
\begin{align}
\|f\|_{L^p(\Gamma[{\bs X}_H])} &\simeq \|f\circ {\bs X}_H\|_{L^p(\widehat\Gamma_H )} , \\
\|\grad_{\Gamma[{\bs X}_H]} f\|_{L^p(\Gamma[{\bs X}_H])} &\simeq \|\grad_{\widehat\Gamma_H} (f\circ {\bs X}_H)\|_{L^p(\widehat\Gamma_H)} .
\end{align}
This implies, through complex interpolation between $L^2(\Gamma[{\bs X}_H])$ and $H^1(\Gamma[{\bs X}_H])$, the following norm-equivalence relation:
\begin{align}\label{norm-equivl-Hs}
\|f\|_{H^s(\Gamma[{\bs X}_H])} &\simeq \|f\circ {\bs X}_H\|_{H^s(\widehat\Gamma_H )} 
\end{align}
\end{subequations}
for $s\in[0,1].$
Therefore, if we denote $f_h= [{\bf P}_{\widehat\Gamma_H} (f\circ{\bs X}_H)]\circ{\bs X}_H^{-1}$, then 
\begin{align}\label{L2-f-Phf}
\|f - {\bf P}_{\Gamma[\bs X_H]} f\|_{L^2(\Gamma[\bs X_H])}
&\le \|f - f_h\|_{L^2(\Gamma[\bs X_H])} \notag\\
&\simeq \|f\circ{\bs X}_H - {\bf P}_{\widehat\Gamma_H} (f\circ{\bs X}_H)\|_{L^2(\widehat\Gamma_H)} \notag\\
&\le CH \|f\circ{\bs X}_H \|_{H^1(\widehat\Gamma_H)} \notag\\
&\le CH \|f\|_{H^1(\Gamma[\bs X_H])} \quad\mbox{for}\,\,\,t\in[0,t_*], 
\end{align}
where the last inequality uses $\|{\bs X}_H \|_{W^{1,\infty}(\widehat\Gamma_H)} \le C$, which follows from \eqref{aux964}. 

Similarly, since $({\bf P}_{\Gamma[\bs X_H]} f)\circ {\bs X}_H$ is a finite element function in ${\bs\Lambda}_H(\widehat\Gamma_H)$, it follows that 
\begin{align}\label{H1-f-Phf}
&\|f - {\bf P}_{\Gamma[\bs X_H]} f\|_{H^1(\Gamma[\bs X_H])} \notag\\
&\simeq 
\|f\circ{\bs X}_H - ({\bf P}_{\Gamma[\bs X_H]} f)\circ{\bs X}_H\|_{H^1(\widehat\Gamma_H)} \notag\\
&\le \|f \circ{\bs X}_H - {\bf P}_{\widehat\Gamma_H}(f\circ{\bs X}_H)\|_{H^1(\widehat\Gamma_H)} + 
\|{\bf P}_{\widehat\Gamma_H}(f\circ{\bs X}_H) - ({\bf P}_{\Gamma[\bs X_H]} f)\circ{\bs X}_H\|_{H^1(\widehat\Gamma_H)} \notag\\
&\le C\|f \circ{\bs X}_H \|_{H^1(\widehat\Gamma_H)} + CH^{-1}\|{\bf P}_{\widehat\Gamma_H}(f \circ{\bs X}_H) - ({\bf P}_{\Gamma[\bs X_H]} f)\circ {\bs X}_H\|_{L^2(\widehat\Gamma_H)} \,\,\mbox{(inverse inequality)}\notag\\
&\le C\|f \circ{\bs X}_H \|_{H^1(\widehat\Gamma_H)} + CH^{-1}\|f\circ{\bs X}_H - ({\bf P}_{\Gamma[{\bs X}_H]} f)\circ {\bs X}_H\|_{L^2(\widehat\Gamma_H)} \quad\mbox{($L^2$ stability of ${\bf P}_{\widehat\Gamma_H}$)} \notag\\
&\le C\|f\|_{H^1(\Gamma[{\bs X}_H])} + CH^{-1}\|f - {\bf P}_{\Gamma[\bs X_H]} f\|_{L^2(\Gamma[{\bs X}_H])} \quad\mbox{(norm equivalence in \eqref{norm-equiv-GH})} \notag\\
&\le C \|f\|_{H^1(\Gamma[\bs X_H])} ,
\end{align}
where the last inequality follows from \eqref{L2-f-Phf}. These two inequalities imply that 
\begin{align*}
\|{\bf P}_{\Gamma[\bs X_H]} f\|_{L^2(\Gamma[\bs X_H])}
&\le C\|f\|_{L^2(\Gamma[\bs X_H])} &&\mbox{(by the definition of $L^2$ projection)}\\
\|{\bf P}_{\Gamma[\bs X_H]} f\|_{H^1(\Gamma[\bs X_H])}
&\le C\|f\|_{H^1(\Gamma[\bs X_H])} &&\mbox{(by \eqref{H1-f-Phf} and the triangle inequality)}.
\end{align*}
By the complex interpolation between $L^2(\Gamma[\bs X_H])$ and $H^1(\Gamma[\bs X_H])$, we obtain 
\begin{align}\label{Hs-Phf}
\|{\bf P}_{\Gamma[\bs X_H]} f\|_{H^s(\Gamma[\bs X_H])}
&\le C\|f\|_{H^s(\Gamma[\bs X_H])} \quad\mbox{for all}\,\,\,s\in[0,1].
\end{align}

Since ${\bf P}_{\Gamma[\bs X_H]}$ is self-adjoint, i.e., 
$$
\int_{\Gamma[\bs X_H]} ({\bf P}_{\Gamma[\bs X_H]} f)g = 
\int_{\Gamma[\bs X_H]} f ({\bf P}_{\Gamma[\bs X_H]} g) 
$$
for all $f,g\in L^2(\Gamma[\bs X_H])$, 
it follows that 
\begin{align*} 
\Big|\int_{\Gamma[\bs X_H]} ({\bf P}_{\Gamma[\bs X_H]} f)g\Big|
&=\Big|\int_{\Gamma[\bs X_H]} f({\bf P}_{\Gamma[\bs X_H]} g)\Big| \notag\\
&\le \|f\|_{H^{-s}(\Gamma[\bs X_H])} \|{\bf P}_{\Gamma[\bs X_H]} g\|_{H^{s}(\Gamma[\bs X_H])} \notag\\
&\le C\|f\|_{H^{-s}(\Gamma[\bs X_H])} \|g\|_{H^{s}(\Gamma[\bs X_H])} \quad\mbox{for all}\,\,\, s\in[0,1] , 
\end{align*}
where the last inequality follows from \eqref{Hs-Phf}. 
This implies that, via the duality argument, 
\begin{align}\label{H-s-Phf}
\|{\bf P}_{\Gamma[\bs X_H]} f\|_{H^{-s}(\Gamma[\bs X_H])}
&\le C\|f\|_{H^{-s}(\Gamma[\bs X_H])} 
\end{align}
for all $s\in[0,1]$. 

\section{A local ${\bf L^2}$ projection in the bulk while being the ${\bf L^2}$ projection on the boundary} \label{appendix:projection}

Let $K$ be a closed simplex in the triangulation of $\widehat\Omega^{\rm s}_H$, and let $N_K$ be the set of nodes on $K$. For any node $x\in N_K$, let $\widetilde K_x$ be the set of closed simplices which contain the node $x$. Moreover, let ${\bf P}_{\widetilde K_x}$ be the $L^2$-orthogonal projection onto the finite element spaces on $\widetilde K_x$, and let $\widetilde K=\bigcup_{x\in N_K}\widetilde K_x$. 

\begin{definition}\label{L2-L2-proj}
For ${\bs w}\in H^1(\widehat\Omega^{\rm s}_H)^d$, we define $\widehat{\bf P}_H{\bs w}\in \widehat{\bs V}^{\rm s}_H$ to be a finite element function determined by the following conditions: 
\begin{enumerate}
\item 
$(\widehat{\bf P}_H{\bs w})|_{\widehat\Gamma_H}$ is defined as ${\bf P}_{\widehat\Gamma_H}({\bs w}|_{\widehat\Gamma_H})$, i.e., the $L^2$ projection of ${\bs w}|_{\widehat\Gamma_H}$ on $\widehat\Gamma_H$. 

\item 
$(\widehat{\bf P}_H{\bs w})(x_j)=({\bf P}_{\widetilde K_{x_j}}{\bs w})(x_j)$ for every interior node $x_j$ {\rm(}excluding nodes on $\widehat\Gamma_H${\rm)}.
\end{enumerate}
\end{definition}

\begin{lemma}\label{Lemma:PH}
The projection $\widehat{\bf P}_H:H^1(\widehat\Omega^{\rm s}_H)^d\rightarrow \widehat{\bs V}^{\rm s}_H$ defined in Definition \ref{L2-L2-proj} has the following properties: 
\begin{align*}
\|\widehat{\bf P}_H{\bs w} - {\bs w}\|_{L^2(\widehat\Omega^{\rm s}_H)}
&\lesssim H\|{\bs w}\|_{H^1(\widehat\Omega^{\rm s}_H)} \\
\|\widehat{\bf P}_H{\bs w}\|_{H^1(\widehat\Omega^{\rm s}_H)}
&\lesssim \|{\bs w}\|_{H^1(\widehat\Omega^{\rm s}_H)} .
\end{align*}
\end{lemma}
\begin{proof}
If $K$ does not intersect the boundary $\widehat\Gamma_H$, then 
$\widehat{\bf P}_H{\bs w}(x) = ({\bf P}_{\widetilde K_x}{\bs w})(x)$ for every $x\in N_K$, and therefore
\begin{align}\label{L2-wH-w}
\|\widehat{\bf P}_H{\bs w}\|_{L^2(K)}^2&
\simeq H^d\|\widehat{\bf P}_H{\bs w}\|_{L^\infty(K)}^2 
\simeq H^d\max_{x\in N_K}|({\bf P}_{\widetilde K_x}{\bs w})(x)|^2 
\lesssim H^d \max_{x\in N_K} \|{\bf P}_{\widetilde K_x}{\bs w}\|_{L^\infty(\widetilde K_x)}^2 \notag\\
&\lesssim \max_{x\in N_K} \|{\bf P}_{\widetilde K_x}{\bs w}\|_{L^2(\widetilde K_x)}^2 
\lesssim \max_{x\in N_K} \|{\bs w}\|_{L^2(\widetilde K_x)}^2  \lesssim \|{\bs w}\|_{L^2(\widetilde K)}^2 . 
\end{align}

If $K$ intersects the boundary $\widehat\Gamma_H$, then we denote by $\mathring N_K$ and $\partial N_K$ the set of interior and boundary nodes on $K$, respectively. Thus 
\begin{align*}
\|\widehat{\bf P}_H{\bs w}\|_{L^2(K)}^2
&\simeq H^d\|\widehat{\bf P}_H{\bs w}\|_{L^\infty(K)}^2 \\
&\simeq H^d\max_{x\in \mathring N_K}|({\bf P}_{\widetilde K_x}{\bs w})(x)|^2 + H^d\max_{x\in \partial N_K}|({\bf P}_{\widehat\Gamma_H}{\bs w})(x)|^2 \\
&\lesssim \|{\bs w}\|_{L^2(\widetilde K)}^2 + H^d\max_{x\in \partial N_K}|({\bf P}_{\widehat\Gamma_H}{\bs w})(x)|^2 \\
&\lesssim \|{\bs w}\|_{L^2(\widetilde K)}^2 + H \|{\bf P}_{\widehat\Gamma_H}{\bs w}\|_{L^2(\partial K \cap \widehat\Gamma_H)}^2 .
\end{align*}
Therefore, 
\begin{align*}
\sum_{K\cap\widehat\Gamma_H\neq\emptyset} \|\widehat{\bf P}_H{\bs w}\|_{L^2(K)}^2
&\lesssim \sum_{K\cap\widehat\Gamma_H\neq\emptyset}\|{\bs w}\|_{L^2(\widetilde K)}^2 + 
H \|{\bf P}_{\widehat\Gamma_H}{\bs w}\|_{L^2(\widehat\Gamma_H)}^2 \\
&\lesssim \sum_{K\cap\widehat\Gamma_H\neq\emptyset}\|{\bs w}\|_{L^2(\widetilde K)}^2 + 
H \|{\bs w}\|_{L^2(\widehat\Gamma_H)}^2 .
\end{align*}
Combining this result with the interior estimates in \eqref{L2-wH-w}, we obtain  
\begin{align*}
\|\widehat{\bf P}_H{\bs w}\|_{L^2(\widehat\Omega^{\rm s}_H)}^2
&\lesssim \|{\bs w}\|_{L^2(\widehat\Omega^{\rm s}_H)}^2  + 
H \|{\bs w}\|_{L^2(\widehat\Gamma_H)}^2 .
\end{align*}
Let ${\bf P}_{\widehat\Omega^{\rm s}_H}{\bs w}$ be the $L^2$-orthogonal projection of ${\bs w}$ onto $\widehat{\bs V}^{\rm s}_H$. Then 
\begin{align}\label{PHw-POmegaw}
\|\widehat{\bf P}_H{\bs w}-{\bf P}_{\widehat\Omega^{\rm s}_H}{\bs w}\|_{L^2(\widehat\Omega^{\rm s}_H)}^2
&= 
\|\widehat{\bf P}_H({\bs w}-{\bf P}_{\widehat\Omega^{\rm s}_H}{\bs w})\|_{L^2(\widehat\Omega^{\rm s}_H)}^2 \notag\\
&\lesssim \|{\bs w}-{\bf P}_{\widehat\Omega^{\rm s}_H}{\bs w}\|_{L^2(\widehat\Omega^{\rm s}_H)}^2  + 
H \|{\bs w}-{\bf P}_{\widehat\Omega^{\rm s}_H}{\bs w}\|_{L^2(\widehat\Gamma_H)}^2 \notag\\
&\lesssim H^2\|{\bs w}\|_{H^1(\widehat\Omega^{\rm s}_H)}^2  .
\end{align}
By applying the triangle inequality, we obtain 
\begin{align*}
\|\widehat{\bf P}_H{\bs w}-{\bs w}\|_{L^2(\widehat\Omega^{\rm s}_H)}
\le 
\|\widehat{\bf P}_H{\bs w}-{\bf P}_{\widehat\Omega^{\rm s}_H}{\bs w}\|_{L^2(\widehat\Omega^{\rm s}_H)}
+ \|{\bs w}-{\bf P}_{\widehat\Omega^{\rm s}_H}{\bs w}\|_{L^2(\widehat\Omega^{\rm s}_H)} 
&\lesssim H\|{\bs w}\|_{H^1(\widehat\Omega^{\rm s}_H)}.
\end{align*}
This proves the first result of Lemma \ref{Lemma:PH}. 

By applying the finite element inverse estimate to \eqref{PHw-POmegaw} and using the $H^1$ stability of the $L^2$ projection ${\bf P}_{\widehat\Omega^{\rm s}_H}$, we obtain the second result of Lemma \ref{Lemma:PH}. 
\end{proof}

Similarly, we can define a projection ${\bf P}_H(t){\bs w}$ which uses the $L^2$ projection on $\Gamma[{\bs X}_H(\cdot,t)]$ instead of the $L^2$ projection on $\widehat\Gamma_H$. 

\begin{definition}\label{L2-L2-proj2}
Let $t\in[0,t_*]$. For ${\bs w}\in H^1(\widehat\Omega^{\rm s}_H)^d$, we define ${\bf P}_H(t){\bs w}\in \widehat{\bs V}^{\rm s}_H$ to be a finite element function determined by the following conditions: 
\begin{enumerate}
\item 
$({\bf P}_H(t){\bs w})|_{\widehat\Gamma_H}$ is defined as $[{\bf P}_{\Gamma[{\bs X}_H(\cdot,t)]}({\bs w}|_{\widehat\Gamma_H}\circ{\bs X}_H(\cdot,t)^{-1})]\circ{\bs X}_H(\cdot,t)$. 

\item 
$({\bf P}_H(t){\bs w})(x_j)=({\bf P}_{\widetilde K_{x_j}}{\bs w})(x_j)$ for every interior node $x_j$ {\rm(}excluding nodes on $\widehat\Gamma_H${\rm)}.
\end{enumerate}
\end{definition}

The same stability and approximation property also hold for the projection ${\bf P}_H(t){\bs w}$, as shown in the following lemma. The proof is similar and therefore omitted. 
\begin{lemma}\label{Lemma:PH2}
The projection ${\bf P}_H(t):H^1(\widehat\Omega^{\rm s}_H)^d\rightarrow \widehat{\bs V}^{\rm s}_H$ defined in Definition \ref{L2-L2-proj2} has following properties: 
\begin{align*}
\|{\bf P}_H(t){\bs w} - {\bs w}\|_{L^2(\widehat\Omega^{\rm s}_H)}
&\lesssim H\|{\bs w}\|_{H^1(\widehat\Omega^{\rm s}_H)} \\
\|{\bf P}_H(t){\bs w}\|_{H^1(\widehat\Omega^{\rm s}_H)}
&\lesssim \|{\bs w}\|_{H^1(\widehat\Omega^{\rm s}_H)} .
\end{align*}
\end{lemma}

\section{Duality between FE subspaces of \texorpdfstring{$H^s(\Gamma[{\bs X}_H])$}{x} and \texorpdfstring{$H^{-s}(\Gamma[{\bs X}_H])$}{x}}
\label{appendix:duality}

The following result is frequently used. 
\begin{lemma}\label{Lemma-A-duality}
Let $s\in[0,1]$. Then the following relation holds for $t\in[0,t_*]$: 
\begin{align}
\|{\bs w}_H\|_{H^{-s}(\Gamma[{\bs X}_H])}
\simeq
\sup_{\substack{{\bs\lambda}_H\in {\bs\Lambda}_H(\Gamma[{\bs X}_H])\\ \|{\bs\lambda}_H\|_{H^{s}(\Gamma[{\bs X}_H])}\le 1}} \int_{\Gamma[{\bs X}_H]} {\bs w}_H\cdot {\bs\lambda}_H 
\quad\mbox{for}\,\,\,{\bs w}_H\in {\bs\Lambda}_H(\Gamma[{\bs X}_H]) ,\\
\|{\bs w}_H\|_{L^p(\Gamma[{\bs X}_H])}
\simeq
\sup_{\substack{{\bs\lambda}_H\in {\bs\Lambda}_H(\Gamma[{\bs X}_H])\\ \|{\bs\lambda}_H\|_{L^{p'}(\Gamma[{\bs X}_H])}\le 1}} \int_{\Gamma[{\bs X}_H]} {\bs w}_H\cdot {\bs\lambda}_H 
\quad\mbox{for}\,\,\,{\bs w}_H\in {\bs\Lambda}_H(\Gamma[{\bs X}_H]) ,
\end{align}
for $1\le p,p'\leq \infty$ such that $\frac{1}{p}+\frac{1}{p'}=1$.
\end{lemma}

\begin{proof}
Since $\|{\bf P}_{\Gamma[{\bs X}_H]}{\bs\lambda}\|_{H^s(\Gamma[{\bs X}_H])} \lesssim \|{\bs\lambda}\|_{H^s(\Gamma[{\bs X}_H])}$ for $s\in[-1,1]$, as shown in \eqref{Hs-Phf} and \eqref{H-s-Phf}, it follows that for ${\bs w}_H\in {\bs\Lambda}_H(\Gamma[{\bs X}_H])$:
\begin{align*}
\|{\bs w}_H\|_{H^{-s}(\Gamma[{\bs X}_H])}
&\simeq \sup_{\substack{{\bs\lambda}\in H^{s}(\Gamma[{\bs X}_H])\\ \|{\bs\lambda}\|_{H^s(\Gamma[{\bs X}_H])}\le 1}} 
\int_{\Gamma[{\bs X}_H]} {\bs w}_H\cdot {\bs\lambda} \notag\\
&= \sup_{\substack{{\bs\lambda}\in H^{s}(\Gamma[{\bs X}_H])\\ \|{\bs\lambda}\|_{H^s(\Gamma[{\bs X}_H])}\le 1}} 
\int_{\Gamma[{\bs X}_H]} {\bs w}_H\cdot {\bf P}_{\Gamma[{\bs X}_H]}{\bs\lambda} 
&&\mbox{(by the definition of ${\bf P}_{\Gamma[{\bs X}_H]}$)}\notag\\
&\le \sup_{\substack{{\bs\lambda}_H\in {\bs\Lambda}_H(\Gamma[{\bs X}_H])\\ \|{\bs\lambda}_H\|_{H^s(\Gamma[{\bs X}_H])}\le C}} 
\int_{\Gamma[{\bs X}_H]} {\bs w}_H\cdot {\bs\lambda}_H 
&&\mbox{(here \eqref{Hs-Phf} and \eqref{H-s-Phf} are used)} \notag\\
&\le C\sup_{\substack{{\bs\lambda}_H\in {\bs\Lambda}_H(\Gamma[{\bs X}_H])\\ \|{\bs\lambda}_H\|_{H^s(\Gamma[{\bs X}_H])}\le 1}} 
\int_{\Gamma[{\bs X}_H]} {\bs w}_H\cdot {\bs\lambda}_H .
\end{align*}
The converse inequality holds naturally, i.e., 
\begin{align*}
\|{\bs w}_H\|_{H^{-s}(\Gamma[{\bs X}_H])}
&\simeq \sup_{\substack{{\bs\lambda}\in H^{s}(\Gamma[{\bs X}_H])\\ \|{\bs\lambda}\|_{H^s(\Gamma[{\bs X}_H])}\le 1}} 
\int_{\Gamma[{\bs X}_H]} {\bs w}_H\cdot {\bs\lambda} 
\ge \sup_{\substack{{\bs\lambda}_H\in {\bs\Lambda}_H(\Gamma[{\bs X}_H])\\ \|{\bs\lambda}_H\|_{H^s(\Gamma[{\bs X}_H])}\le 1}} 
\int_{\Gamma[{\bs X}_H]} {\bs w}_H\cdot {\bs\lambda}_H .
\end{align*}
This proves the first result of Lemma \ref{Lemma-A-duality}. 

Since the $L^2$ projection ${\bf P}_{\Gamma[\bs X_H]}$ is also bounded on $L^p(\Gamma[\bs X_H])$ for all $1\le p\le \infty$ (see \cite[Appendix B]{Li-Tang-2025}, where $p\ge 2$ is proved and $1\le p\le 2$ follows from duality), the same argument as above can be used to show the second result of Lemma \ref{Lemma-A-duality}. 
\end{proof}

Since ${\bs w}_H\circ{\bs X}_H^{-1}\in {\bs\Lambda}_H(\Gamma[{\bs X}_H])$ iff ${\bs w}_H\in \widehat{\bs\Lambda}_H = {\bs\Lambda}_H(\widehat\Gamma_H)$, and $\|{\bs\lambda}_H\circ{\bs X}_H^{-1}\|_{H^s(\Gamma[{\bs X}_H])} \simeq \|{\bs\lambda}_H\|_{H^s(\widehat\Gamma_H )}$ when $s\in[0,1]$ in view of \eqref{norm-equivl-Hs}, Lemma \ref{Lemma-A-duality} implies the following result. 

\begin{lemma}[Duality estimates on {$\Gamma[{\bs X}_H]$}]\label{Lemma-A-duality0}
Let $s\in[0,1]$. Then the following relation holds for ${\bs w}_H\in \widehat{\bs\Lambda}_H $ and $t\in[0,t_*]$: 
\begin{align*}
\|{\bs w}_H\circ{\bs X}_H^{-1}\|_{H^{-s}(\Gamma[{\bs X}_H])}
\simeq
\sup_{\substack{{\bs\lambda}_H\in \widehat{\bs\Lambda}_H\\ \|{\bs\lambda}_H\|_{H^{s}(\widehat\Gamma_H)}\le 1}} \int_{\Gamma[{\bs X}_H]} ({\bs w}_H\circ{\bs X}_H^{-1}) \cdot ({\bs\lambda}_H\circ{\bs X}_H^{-1} ) .
\end{align*}
\end{lemma}

It is known that there exists a standard extension operator $\mathcal{E}:H^{\frac12}(\widehat\Gamma_H)^d\rightarrow H^1(\widehat\Omega^{\rm s}_H)^d$; see \cite[Theorem 1.5.1.3]{Grisvard-2011}. 
For any ${\bs w}_H\in \widehat{\bs\Lambda}_H$, let $\mathcal{E}_H{\bs w}_H\eqd \widehat{\bf P}_H\mathcal{E}{\bs w}_H$, where $\widehat{\bf P}_H$ is defined in Definition \ref{L2-L2-proj} of Appendix \ref{appendix:projection}. Then $\mathcal{E}_H{\bs w}_H={\bs w}_H$ on $\widehat\Gamma_H$, and Lemma \ref{Lemma:PH} implies that 
$$
\|\mathcal{E}_H{\bs w}_H\|_{H^1(\widehat\Omega^{\rm s}_H)}
\lesssim \|\mathcal{E}{\bs w}_H\|_{H^1(\widehat\Omega^{\rm s}_H)}
\lesssim \|{\bs w}_H\|_{H^{\frac12}(\widehat\Gamma_H)} . 
$$
This proves the following result.
\begin{lemma}\label{Lemma:extension}
There exists an extension operator $\mathcal{E}_H: \widehat{\bs\Lambda}_H \rightarrow \widehat{\bs V}^{\rm s}_H $ such that 
\begin{align*}
\|\mathcal{E}_H {\bs w}_H\|_{H^1(\widehat\Omega^{\rm s}_H)}
\lesssim \|{\bs w}_H\|_{H^{\frac12}(\widehat\Gamma_H)} . 
\end{align*}
\end{lemma}

By choosing $s=\frac12$ in Lemma \ref{Lemma-A-duality0} and using Lemma \ref{Lemma:extension}, we obtain the following relation for ${\bs w}_H\in \widehat{\bs\Lambda}_H $ and $t\in[0,t_*]$: 
\begin{align*}
\|{\bs w}_H\circ{\bs X}_H^{-1}\|_{H^{-\frac12}(\Gamma[{\bs X}_H])}
&\simeq
\sup_{\substack{{\bs\lambda}_H\in \widehat{\bs\Lambda}_H\\ \|{\bs\lambda}_H\|_{H^{\frac12}(\widehat\Gamma_H)}\le 1}} \int_{\Gamma[{\bs X}_H]} ({\bs w}_H\circ{\bs X}_H^{-1}) \cdot ({\bs\lambda}_H\circ{\bs X}_H^{-1} ) \\
&\lesssim
\sup_{\substack{\mathcal{E}_H{\bs\lambda}_H\in \widehat{\bs V}^{\rm s}_H\\\ \|\mathcal{E}_H{\bs\lambda}_H\|_{H^1(\widehat\Omega^{\rm s}_H)}\le C}} \int_{\Gamma[{\bs X}_H]} ({\bs w}_H\circ{\bs X}_H^{-1}) \cdot (\mathcal{E}_H{\bs\lambda}_H\circ{\bs X}_H^{-1} ) \\
&\lesssim
\sup_{\substack{{\bs\lambda}_H\in \widehat{\bs V}^{\rm s}_H\\ \|{\bs\lambda}_H\|_{H^1(\widehat\Omega^{\rm s}_H)}\le 1}} \int_{\Gamma[{\bs X}_H]} ({\bs w}_H\circ{\bs X}_H^{-1}) \cdot ({\bs\lambda}_H\circ{\bs X}_H^{-1} ) .
\end{align*}
The converse inequality holds naturally, i.e., 
\begin{align*}
&\hspace{-20pt}\sup_{\substack{{\bs\lambda}_H\in \widehat{\bs V}^{\rm s}_H\\ \|{\bs\lambda}_H\|_{H^1(\widehat\Omega^{\rm s}_H)}\le 1}} \int_{\Gamma[{\bs X}_H]} ({\bs w}_H\circ{\bs X}_H^{-1}) \cdot ({\bs\lambda}_H\circ{\bs X}_H^{-1} ) \\
&\lesssim
\sup_{\substack{{\bs\lambda}_H\in \widehat{\bs V}^{\rm s}_H\\ \|{\bs\lambda}_H\|_{H^1(\widehat\Omega^{\rm s}_H)}\le 1}} \|{\bs w}_H\circ{\bs X}_H^{-1}\|_{H^{-\frac12}(\Gamma[{\bs X}_H])}
\|{\bs \lambda}_H\circ{\bs X}_H^{-1}\|_{H^{\frac12}(\Gamma[{\bs X}_H])} \\
&\lesssim
\sup_{\substack{{\bs\lambda}_H\in \widehat{\bs V}^{\rm s}_H\\ \|{\bs\lambda}_H\|_{H^1(\widehat\Omega^{\rm s}_H)}\le 1}} \|{\bs w}_H\circ{\bs X}_H^{-1}\|_{H^{-\frac12}(\Gamma[{\bs X}_H])}
\|{\bs \lambda}_H\|_{H^1(\widehat\Omega^{\rm s}_H)} 
\quad\mbox{(see \eqref{norm-equivl-Hs})}\\
&\lesssim
\|{\bs w}_H\circ{\bs X}_H^{-1}\|_{H^{-\frac12}(\Gamma[{\bs X}_H])} ,
\end{align*}
where the second-to-last inequality follows from the norm-equivalence relation in \eqref{norm-equivl-Hs} and the trace inequality on $\widehat\Omega^{\rm s}_H$. This proves the following result. 

\begin{lemma}\label{Lemma:H-half}
For $t\in[0,t_*]$, the following relation holds: 
\begin{align*}
\|{\bs w}_H\circ{\bs X}_H^{-1}\|_{H^{-\frac12}(\Gamma[{\bs X}_H])}
&\simeq 
\sup_{\substack{{\bs\lambda}_H\in \widehat{\bs V}^{\rm s}_H\\ \|{\bs\lambda}_H\|_{H^1(\widehat\Omega^{\rm s}_H)}\le 1}} \int_{\Gamma[{\bs X}_H]} ({\bs w}_H\circ{\bs X}_H^{-1}) \cdot ({\bs\lambda}_H\circ{\bs X}_H^{-1} ) .
\end{align*}
\end{lemma}

\section{Finite element inverse estimates on \texorpdfstring{$\Gamma[{\bs X}_H]$}{x}}
\label{appendix:inverse}

The following results are frequently used in this paper.

\begin{lemma}[Inverse estimates on {$\Gamma[{\bs X}_H]$}]\label{Lemma:inverse}
Let $s\in[0,1]$. Then the following inverse inequalities hold for $t\in[0,t_*]$: 
\begin{align}\label{wHL2-H-s}
\begin{aligned}
\|{\bs w}_H\|_{H^s(\Gamma[{\bs X}_H])}
&\lesssim H^{-s}\|{\bs w}_H\|_{L^2(\Gamma[{\bs X}_H])} 
&&\mbox{for}\,\,\, {\bs w}_H\in {\bs\Lambda}_H(\Gamma[{\bs X}_H]) ,\\
\|{\bs w}_H\|_{L^2(\Gamma[{\bs X}_H])}
&\lesssim H^{-s}\|{\bs w}_H\|_{H^{-s}(\Gamma[{\bs X}_H])} 
&&\mbox{for}\,\,\, {\bs w}_H\in {\bs\Lambda}_H(\Gamma[{\bs X}_H]) . 
\end{aligned}
\end{align}
\end{lemma}

\begin{proof}
For $s\in[0,1]$ and $t\in[0,t_*]$, the equivalence of norms of functions transported between $\Gamma[{\bs X}_H]$ and $\widehat\Gamma_H$ and the inverse inequality on $\widehat\Gamma_H$ imply that 
\begin{align}\label{inverse-Hs-L2}
\| {\bs\lambda}_H \circ {\bs X}_H^{-1} \|_{H^s(\Gamma[{\bs X}_H])}
\simeq \|{\bs\lambda}_H \|_{H^s(\widehat\Gamma_H)} 
&\lesssim H^{-s}\|{\bs\lambda}_H \|_{L^2(\widehat\Gamma_H)} \notag\\
&\simeq H^{-s}\|{\bs\lambda}_H \circ {\bs X}_H^{-1} \|_{L^2(\Gamma[{\bs X}_H])} . 
\end{align} 
This proves the first result of Lemma \ref{Lemma:inverse}. By using Lemma \ref{Lemma-A-duality} and \eqref{inverse-Hs-L2}, we have 
\begin{align}\label{wHL2-H-s1}
\|{\bs w}_H\|_{L^2(\Gamma[{\bs X}_H])}
&\simeq \sup_{\substack{{\bs\lambda}_H\in {\bs\Lambda}_H(\Gamma[{\bs X}_H])\\ \|{\bs\lambda}_H\|_{L^2(\Gamma[{\bs X}_H])}\le 1}} 
\int_{\Gamma[{\bs X}_H]} {\bs w}_H\cdot {\bs\lambda}_H \\
&\lesssim \|{\bs w}_H\|_{H^{-s}(\Gamma[{\bs X}_H])} \|{\bs\lambda}_H \|_{H^s(\Gamma[{\bs X}_H])} \notag\\
&\lesssim H^{-s}\|{\bs w}_H\|_{H^{-s}(\Gamma[{\bs X}_H])} \|{\bs\lambda}_H \|_{L^2(\Gamma[{\bs X}_H])} \quad\mbox{for}\,\,\, s\in[0,1] . \notag
\end{align}
Since $\|{\bs\lambda}_H\|_{L^2(\Gamma[{\bs X}_H])}\le 1$ in \eqref{wHL2-H-s1}, we obtain the second result of Lemma \ref{Lemma:inverse}. 
\end{proof}

\section{Uniform in time inf-sup condition}\label{appendix:inf-sup}

Since the numerical domains $\Omega^{\rm f}[{\bs X}_H]$ are Lipschitz, for any given 
$t\in[0,t_*]$ the inf--sup condition \eqref{inf-sup-ep-L2} with a constant $c_0$ 
independent of $h$ and of the position of $\Gamma[\bX_H]$ in the background fluid mesh 
follows from \cite[Corollary~1]{GuzmOlsh} for all $h\le h_0$, once $h_0>0$ is chosen 
sufficiently small and the pressure function satisfies $\int_{\Omega^{\rm f}_{\rm i}[{\bs X}_H]}e_p=0$, where 
$\Omega^{\rm f}_{\rm i}[{\bs X}_H]\eqd \Omega^{\rm f}[{\bs X}_H]\setminus \Omega^{\rm f}_{\Gamma}[{\bs X}_H]$, the domain consisting of simplices completely inside the fluid domain. We  show that the constants $c_0$ and $h_0$ in 
\eqref{inf-sup-ep-L2} can be taken \emph{uniform in time} $t\in[0,t_*]$.  We also extend \eqref{inf-sup-ep-L2} to $e_p$ without the zero mean condition. 

The uniformity  in time  follows from the arguments in \cite{GuzmOlsh} provided that the three 
conditions formulated below are satisfied. Let ${\mathcal T}_h^{\rm i}(t)$ denote the set of 
simplices from ${\mathcal T}_h$ lying entirely in the interior of 
$\Omega^{\rm f}[{\bs X}_H]$. The following conditions are required:
\begin{description}
  \item[1] The constant $\beta(\Omega^{\rm f}[{\bs X}_H])$ in the Bogovski\u{\i} inequality 
  (the continuous inf--sup inequality) is uniformly bounded away from zero for all 
  $t\in[0,t_*]$, $h\le h_0$, and $H\le H_0$.

  \item[2] For any $T\in {\mathcal T}_h$ intersected by $\Gamma[\bX_H]$, we can assign a simplex $K_T\in {\mathcal T}_h^{\rm i}(t)$ (a ``base'' simplex for $T$) such that $K_T$ can be 
  reached from $T$ by crossing at most $L$ simplices from ${\mathcal T}_h^{\rm i}(t)$, where $L$ 
  is uniformly bounded in $h$ and in $t$. Moreover, any $K\in {\mathcal T}_h^{\rm i}(t)$ can 
  serve as a base for only a uniformly bounded number of simplices intersected by 
  $\Gamma[\bX_H]$, independently of $h$ and~$t$.

  \item[3] For any adjacent $T_1,T_2\in {\mathcal T}_h$ intersected by $\Gamma[\bX_H]$, there 
  exists a path from $K_{T_1}$ to $K_{T_2}$ crossing at most $M$ simplices from 
  ${\mathcal T}_h^{\rm i}(t)$, where $M$ is uniformly bounded in $h$ and~$t$.
\end{description}

To verify Condition~1, 
it suffices to build a diffeomorphism ${\bs\Psi}(t):\,\Omega^{\rm f}(t)\to\Omega^{\rm f}[{\bs X}_H(t)]$ such 
that 
\begin{equation}\label{eq:aux5191}
\|{\bs\Psi}(t)-\mathbf{Id}\|_{W^{1,\infty}(\Omega^{\rm f}(t))} \le \eps_0,
\end{equation}
with sufficiently small $\eps_0>0$ independent of $t$ and $h,H$.  Then  \cite[Theorem~4.4]{bernardi2016continuity} implies the estimate
\begin{equation}\label{eq:aux5196}
|\beta(\Omega^{\rm f}[{\bs X}_H])-\beta(\Omega^{\rm f}(t))|\le c\,\eps_0,
\end{equation}
with some absolute constant $c$. Since the evolution of $\Omega^{\rm f}(t)$ is smooth, by the same theorem $\beta(\Omega^{\rm f}(t))$ depends continuously on $t\in[0,t_*]$ and stays positive, which implies $\inf_{t\in[0,t_*]}\beta(\Omega^{\rm f}(t))>0$. Therefore, \eqref{eq:aux5196} yields Condition~1.  

The required map ${\bs\Psi}(t)$ can be constructed as follows. First, for $u\in W^{1,\infty}({\Omega}^s(t))$ we define an extension $\mathcal{E} u\in W^{1,\infty}(\R^d)$ as follows: Since ${\Omega}^s(t)$ is  $C^2$-smooth with curvatures bounded uniformly in time, there is a band neighborhood $O(\Gamma(t))$ of width $\rho>0$ (where $\rho$ can be chosen independent of $t$) such that the closest point projection $\bp:\,O(\Gamma(t))\to \Gamma(t)$ is well defined. We let  $\mathcal{E} u(\bx) =u(\bp(\bx))$ for $\bx\in O(\Gamma(t))\setminus{\Omega}^s(t) $.   By construction it holds that $\mathcal{E} u\in W^{1,\infty}(O(\Gamma(t))\cup{\Omega}^s(t))$
and $\|\mathcal{E} u\|_{W^{1,\infty}}\le C\|u\|_{W^{1,\infty}}$ with $C$ depending only on principal curvatures of $\Gamma(t)$. Hence $C$ is uniformly bounded in time. The extension to $\R^d$ can be defined by applying a smooth cut-off function. Then we have 
\[
\|\mathcal{E} u\|_{W^{1,\infty}(\R^d)}\le C\|u\|_{W^{1,\infty}(\Omega^s(t))}\quad \forall u\in W^{1,\infty}({\Omega}^s(t)),
\]
with a finite $C$ independent of $u$ and $t$. 

For the mapping $\widetilde{\bs\Psi}\eqd  \bX_H\circ\widehat{\bs\Phi}_H^{-1}\circ\bX^{-1}:\,\Omega^{\rm s}(t)\to \Omega^{\rm s}[{\bs X}_H]$, the condition in \eqref{ind-1-2} guarantees that 
\[
\|\widetilde{\bs\Psi}-\mathbf{Id}\|_{W^{1,\infty}(\Omega^{\rm s}(t))}\lesssim
\|\bX_H\circ\widehat{\bs\Phi}_H^{-1}-\bX\|_{W^{1,\infty}(\widehat\Omega^{\rm s})}\lesssim H^{\frac14}.
\]

Let $\eta$ be a smooth cut-off function, compactly supported in $\Omega$, such that $0\le \eta\le 1$ and $\eta\equiv1$ on ${\Omega}^s(t)$. We may assume $\|\eta\|_{W^{1,\infty}(\Omega\times[0,t^\ast])}\le C$. Then
\[
{\bs\Psi} = \mathbf{Id} + \eta\,\mathcal{E} (\widetilde{\bs\Psi}-\mathbf{Id})
\]
is the desired diffeomorphism for $H\le H_0$ when $H_0$ is sufficiently small.\smallskip

Conditions~2 and~3 follow for sufficiently small $h$ from the time uniform  bounds on 
the principal curvatures of $\Gamma(t)$ and from the observation that \eqref{ind-1} 
implies that $\Gamma[\bX_H]$ remains within a $\kappa h$-wide narrow band around 
$\Gamma(t)$.

\medskip 

To extend  \eqref{inf-sup-ep-L2} for $e_p$ without zero mean condition, consider the (time-dependent) decomposition of $\Omega$ into three non-overlapping domains $\overline{\Omega}=\overline{\Omega^{\rm f}_{\rm i}[{\bs X}_H]} \cup \overline{\Omega^{\rm f}_{\Gamma}[{\bs X}_H]} \cup \overline{\Omega^{\rm f}_{\rm e}[{\bs X}_H]}$, representing the fluid interior, fluid boundary, and fluid exterior subregions, respectively, with each of the subregions being unions of elements, satisfying the following properties:
\begin{enumerate}
\item 
$\Gamma[{\bs X}_H]\subset\Omega^{\rm f}_{\Gamma}[{\bs X}_H]$ and 
$|\Omega^{\rm f}_{\Gamma}[{\bs X}_H]|\simeq h$;

\item
$\overline{\Omega^{\rm f}_{\rm i}[{\bs X}_H]}\cap \overline{\Omega^{\rm f}_{\rm e}[{\bs X}_H]}=\emptyset$ 
and ${\rm dist}(\Omega^{\rm f}_{\rm i}[{\bs X}_H], \Omega^{\rm f}_{\rm e}[{\bs X}_H])=O(h)$. 
\end{enumerate}
Let $m_1 \eqd |\Omega^{\rm f}_{\rm i}[{\bs X}_H]|^{-1} \int_{\Omega^{\rm f}_{\rm i}[{\bs X}_H]}e_p$ and consider $e^{\rm ext}\in Q_h^{\rm f}$ such that 
\[
e^{\rm ext} \eqd \left\{\begin{aligned}
     m_1& \text{ in }\Omega^{\rm f}_{\rm i}[{\bs X}_H] , \\
     m_2& \text{ in }\Omega^{\rm f}_{\rm e}[{\bs X}_H] ,\\
     0 & \text{ at the remaining nodes in }\Omega^{\rm f}_{\Gamma}[{\bs X}_H] 
\end{aligned}\right.
\]
and $\int_{\Omega}e^{\rm ext}=0$,
with $m_2\in\mathbb{R}$ satisfying $m_2\simeq -m_1$ and therefore 
$$
\|e^{\rm ext}\|_{L^\infty(\Omega)}
\lesssim |m_1| 
\lesssim \|e_p\|_{L^2(\Omega^{\rm f}_i[{\bs X}_H])}  .
$$
This yields a decomposition $e_p = e_p^0 + e^{\rm ext}$ with an $e_p^0$ satisfying $\int_{\Omega^{\rm f}_{\rm i}[{\bs X}_H]}e_p^0=0$. 

Since $\int_{\Omega^{\rm f}_{\rm i}[{\bs X}_H]}e_p^0=0$ and $e^{\rm ext}$ is constant on $\Omega^{\rm f}_{\rm i}[{\bs X}_H]$, it follows that $\int_{\Omega^{\rm f}_{\rm i}[{\bs X}_H]} e_p^0e^{\rm ext} = 0$ and therefore 
\begin{equation}\label{aG1}
\begin{split}
\|e_p\|^2_{L^2(\Omega^{\rm f}_i[{\bs X}_H])} &= \|e_p^0\|^2_{L^2(\Omega^{\rm f}_i[{\bs X}_H])} 
+ \|e^{\rm ext}\|_{L^2(\Omega^{\rm f}_i[{\bs X}_H])}^2 .
\end{split}
\end{equation}
Since $e_p^0$ satisfies the proper zero mean condition, i.e., $\int_{\Omega^{\rm f}_{\rm i}[{\bs X}_H]}e_p^0=0$, it has been shown in \cite[Lemma 4 and Section 6]{GuzmOlsh} that the first term on the right-hand side of \eqref{aG1} can be estimated by 
\begin{equation}\label{aG1-1}
\|e_p^0\|^2_{L^2(\Omega^{\rm f}_i[{\bs X}_H])}  
\lesssim \int_{\Omega^{\rm f}_i[{\bs X}_H]} e_p^0 \,\grad\cdot {\bs v}_h^0
= \int_{\Omega^{\rm f}_i[{\bs X}_H]} e_p \,\grad\cdot {\bs v}_h^0
= \int_{\Omega^{\rm f}[{\bs X}_H]} e_p \,\grad\cdot {\bs v}_h^0 ,
\end{equation}
where $\bv_h^0\in\bV_h^{\rm f}$ is a finite element function such that $\text{supp}(\bv_h^0)\subset \overline{\Omega^{\rm f}_{\rm i}[{\bs X}_H]}$ and 
\begin{align}
\|\bv_h^0\|_{H^1(\Omega^{\rm f}_i[{\bs X}_H])}
\lesssim \|e_p^0\|_{L^2(\Omega^{\rm f}_i[{\bs X}_H])} 
&\lesssim \|e_p\|_{L^2(\Omega^{\rm f}_i[{\bs X}_H])} + \|e^{\rm ext}\|_{L^2(\Omega^{\rm f}_i[{\bs X}_H])} \notag\\
&\lesssim \|e_p\|_{L^2(\Omega^{\rm f}_i[{\bs X}_H])} . 
\end{align}
Therefore, using the result in \eqref{aG1-1}, we have 
\begin{align}\label{aG1-2}
\|e_p^0\|^2_{L^2(\Omega^{\rm f}[{\bs X}_H])}  
&\lesssim 
\|e_p^0\|^2_{L^2(\Omega^{\rm f}_i[{\bs X}_H])}  
+ g_h^{\rm p}(e_p^0,e_p^0)
\quad\mbox{(see \cite[inequality (4.2)]{GuzmOlsh})}\notag\\
&\lesssim 
\int_{\Omega^{\rm f}[{\bs X}_H]} e_p \,\grad\cdot {\bs v}_h^0
+ g_h^{\rm p}(e_p,e_p) + g_h^{\rm p}(e^{\rm ext},e^{\rm ext})\notag\\
&\lesssim 
\int_{\Omega^{\rm f}[{\bs X}_H]} e_p \,\grad\cdot {\bs v}_h^0
+ g_h^{\rm p}(e_p,e_p) + h \|e^{\rm ext}\|_{L^\infty(\Omega)}^2 \notag\\
&\lesssim 
\int_{\Omega^{\rm f}[{\bs X}_H]} e_p \,\grad\cdot {\bs v}_h^0
+ g_h^{\rm p}(e_p,e_p) + h \|e_p\|_{L^2(\Omega^{\rm f}_i[{\bs X}_H])}^2 , 
\end{align}
where the second-to-last inequality follows from the property of $e^{\rm ext}$, i.e., the jump of $e^{\rm ext}$ is nonzero only on the faces contained in $\overline{\Omega^{\rm f}_{\Gamma}[{\bs X}_H]}$. Substituting \eqref{aG1-2} into \eqref{aG1} yields
\begin{equation*}
\|e_p\|^2_{L^2(\Omega^{\rm f}_i[{\bs X}_H])} 
\lesssim 
\int_{\Omega^{\rm f}[{\bs X}_H]} e_p \,\grad\cdot {\bs v}_h^0 
+ g_h^{\rm p}(e_p,e_p) 
+ \|e^{\rm ext}\|_{L^2(\Omega^{\rm f}_i[{\bs X}_H])}^2 .
\end{equation*}
Therefore, for some constant $c_1>0$, 
\begin{align}\label{aG2}
\|e_p\|^2_{L^2(\Omega^{\rm f}[{\bs X}_H])} 
&\lesssim 
\|e_p\|^2_{L^2(\Omega^{\rm f}_i[{\bs X}_H])} + g_h^{\rm p}(e_p,e_p) \quad\mbox{(see \cite[inequality (4.2)]{GuzmOlsh})} \notag\\
&\le 
c_1 \Big( \int_{\Omega^{\rm f}[{\bs X}_H]} e_p \,\grad\cdot {\bs v}_h^0 
+ g_h^{\rm p}(e_p,e_p) 
+ \|e^{\rm ext}\|_{L^2(\Omega^{\rm f}_i[{\bs X}_H])}^2 \Big) .
\end{align}

The standard inf-sup condition on the domain $\Omega$ guarantees the existence of $\bw_h\in \bV_h^{\rm f}$, with $\bw_h=0$ on $\partial \Omega$, such that 
\begin{equation}\label{aG3}
\|\bw_h\|_{H^1(\Omega)}\lesssim \|e^{\rm ext}\|_{L^2(\Omega)} \lesssim \|e_p\|_{L^2(\Omega^{\rm f}_i[{\bs X}_H])} ,\quad 
\|e^{\rm ext}\|_{L^2(\Omega)}^2\lesssim \int_\Omega e^{\rm ext} \nabla\cdot \bw_h .
\end{equation}
Noting that $\int_\Omega\nabla\cdot\bw_h=0$ holds, we also have
\begin{align*}
\int_\Omega e^{\rm ext} \nabla\cdot \bw_h=& \int_{\Omega^{\rm f}_i[{\bs X}_H]\cup \Omega^{\rm f}_{\Gamma}[{\bs X}_H]}e^{\rm ext} \nabla\cdot \bw_h  -
m_2\int_{\Omega^{\rm f}_i[{\bs X}_H]\cup \Omega^{\rm f}_{\Gamma}[{\bs X}_H]} \nabla\cdot \bw_h\\
=& (m_1-m_2)\int_{\Omega^{\rm f}_i[{\bs X}_H]} \nabla\cdot \bw_h
+\int_{\Omega^{\rm f}_{\Gamma}[{\bs X}_H]}(e^{\rm ext}-m_2) \nabla\cdot \bw_h\\ 
=&\frac{m_1-m_2}{m_1}\int_{\Omega^{\rm f}[{\bs X}_H]} e^{\rm ext} 
\nabla\cdot \bw_h
+\int_{\Omega^{\rm f}_{\Gamma}[{\bs X}_H]}(e^{\rm ext}-m_2) \nabla\cdot \bw_h
\\ &- \frac{m_1-m_2}{m_1}\int_{\Omega^{\rm f}[{\bs X}_H]\cap \Omega^{\rm f}_\Gamma[{\bs X}_H]}e^{\rm ext} \nabla\cdot \bw_h \\
\le &\frac{m_1-m_2}{m_1}\int_{\Omega^{\rm f}[{\bs X}_H]}e^{\rm ext} 
\nabla\cdot \bw_h + c\, h^{\frac12} \|e^{\rm ext}\|_{L^2(\Omega)}\|\nabla\cdot \bw_h\|_{L^2(\Omega)} ,
\end{align*}
where the last inequality uses $\|e^{\rm ext}\|_{L^\infty(\Omega)}\simeq |m_1|\simeq \|e^{\rm ext}\|_{L^2(\Omega)}$. 
Hence, for sufficiently small $h$, we conclude that
\begin{align}\label{aG4}
\|e^{\rm ext}\|_{L^2(\Omega)}^2
&\le c_2\int_{\Omega^{\rm f}[{\bs X}_H]} e^{\rm ext} \nabla\cdot \bw_h \notag\\
&=c_2\int_{\Omega^{\rm f}[{\bs X}_H]} e_p \nabla\cdot \bw_h 
-c_2\int_{\Omega^{\rm f}[{\bs X}_H]} e_p^0 \nabla\cdot \bw_h \notag\\
&\le c_2\int_{\Omega^{\rm f}[{\bs X}_H]} e_p \nabla\cdot \bw_h 
+ c_2 \| e_p^0 \|_{L^2(\Omega^{\rm f}[{\bs X}_H])}\| e_p \|_{L^2(\Omega^{\rm f}[{\bs X}_H])}
\quad\mbox{(here \eqref{aG3} is used)}\notag\\
&\le c_2\int_{\Omega^{\rm f}[{\bs X}_H]} e_p \nabla\cdot \bw_h 
+\frac{1}{2c_1} \|e_p\|_{L^2(\Omega^{\rm f}[{\bs X}_H])}^2 \notag\\
&\quad\, + c_3\Big(\int_{\Omega^{\rm f}[{\bs X}_H]} e_p \,\grad\cdot {\bs v}_h^0
+ g_h^{\rm p}(e_p,e_p) + h\|e_p\|_{L^2(\Omega^{\rm f}[{\bs X}_H])}^2 \Big) ,
\end{align}
where the last inequality follows from Young's inequality and \eqref{aG1-2}. 
For sufficiently small $h$, substituting this result into \eqref{aG2} yields 
\begin{equation}\label{aG5}
\|e_p\|^2_{L^2(\Omega^{\rm f}[{\bs X}_H])} 
\lesssim \int_{\Omega^{\rm f}[{\bs X}_H]} e_p \,\grad\cdot {\bs v}_h
+ g_h^{\rm p}(e_p,e_p) 
\end{equation}
where ${\bs v}_h = c_1c_2 \bw_h + c_1(1+c_3){\bs v}_h^0 $ satisfies the following estimate in view of \eqref{aG3} and that ${\bs v}_h^0$ has compact support on $\overline{\Omega^{\rm f}_i[{\bs X}_H]}$: 
\begin{align}\label{vh-H1-ep-L2}
\|{\bs v}_h\|_{H^1(\Omega)}
\lesssim
\|e_p\|_{L^2(\Omega^{\rm f}[{\bs X}_H])}  .
\end{align}
By the finite element inverse trace inequality and the finite element inverse estimates,
\begin{align*}
g_h^{\rm u}({\bs v}_h,{\bs v}_h)
=
\sum_{F\in\mathcal F_h}
\sum_{j=1}^k
h^{2j-1}
\bigl\|
[\partial_{n_F}^j{\bs v}_h]
\bigr\|_{L^2(F)}^2
&\lesssim
\sum_{K\in\mathcal T_h}
\sum_{j=1}^k
h^{2j-2}
\|\nabla^j{\bs v}_h\|_{L^2(K)}^2
\\
&\lesssim
\sum_{K\in\mathcal T_h}
\|\nabla {\bs v}_h\|_{L^2(K)}^2\\
&\lesssim
\|{\bs v}_h\|_{H^1(\Omega)}^2
\lesssim \|e_p\|_{L^2(\Omega^{\rm f}[{\bs X}_H])}^2 ,
\end{align*}
where the last inequality follows from \eqref{vh-H1-ep-L2}. Combining this estimate with the restriction of ${\bs v}_h$ to the physical fluid domain, we obtain 
\[
\|{\bs v}_h\|_{1,\Omega[{\bs X}_H]}
\lesssim
\|e_p\|_{L^2(\Omega^{\rm f}[{\bs X}_H])} .
\]
This result and the estimate in \eqref{aG5} give, using the notation in \eqref{1-0-norm}, 
\begin{align*}
\|{\bs v}_h\|_{1,\Omega[{\bs X}_H]}
&\lesssim
\|e_p\|_{L^2(\Omega^{\rm f}[{\bs X}_H])},
\\
\|e_p\|_{0,\Omega[{\bs X}_H]}^2
&\lesssim
\int_{\Omega^{\rm f}[{\bs X}_H]}
e_p\nabla\cdot {\bs v}_h
+
g_h^{\rm p}(e_p,e_p).
\end{align*}
This proves the results in \eqref{inf-sup-ep-L2}.

\section{Summary of the notation}
\label{appendix:notation-summary}

This appendix summarizes the principal notation used in the paper. For a time-dependent function $f$, we often write $f(t)=f(\cdot,t)$. Unless explicitly stated otherwise, the transport of a scalar- or vector-valued function between interfaces means composition with the corresponding parametrization.

\begingroup
\small

\newlength{\NotationColOne}
\newlength{\NotationColTwo}
\newlength{\NotationColThree}

\setlength{\NotationColOne}{0.32\linewidth}
\setlength{\NotationColThree}{0.14\linewidth}
\setlength{\NotationColTwo}{0.47\linewidth}   

\setlength{\LTleft}{\fill}
\setlength{\LTright}{\fill}

\begin{longtable}{
  |>{\raggedright\arraybackslash}p{\NotationColOne}
  |>{\raggedright\arraybackslash}p{\NotationColTwo}
  |>{\raggedright\arraybackslash}p{\NotationColThree}|}
\caption{Summary of the notation used in the paper.}
\label{tab:notation-summary}\\
\hline
\textbf{Notation} & \textbf{Description} & \textbf{Definition}\\
\hline
\endfirsthead

\multicolumn{3}{c}{\tablename~\thetable\ (continued)}\\
\hline
\textbf{Notation} & \textbf{Description} & \textbf{Definition}\\
\hline
\endhead

\hline
\multicolumn{3}{r}{Continued on the next page}\\
\endfoot

\hline
\endlastfoot

\multicolumn{3}{|l|}{\textbf{Domains, interfaces, and exact configuration}}\\
\hline

$\Omega\subset\R^d$, $d=2,3$
&
Fixed ambient computational domain.
&
Section~\ref{sect:setting}
\\
\hline

$\widehat{\Omega}^{\rm s}$
&
Reference configuration of the solid.
&
Section~\ref{sect:setting}
\\
\hline

$\widehat\Gamma=\partial\widehat{\Omega}^{\rm s}$
&
Reference fluid--solid interface.
&
Section~\ref{sect:setting}
\\
\hline

$\Gamma^{\rm f}=\partial\Omega$
&
Fixed exterior boundary of the fluid computational domain.
&
Section~\ref{sect:setting}
\\
\hline

${\bs X}(\widehat{\bs x},t)$
&
Exact solid motion map from the reference solid to its current
configuration.
&
Section~\ref{sect:setting}
\\
\hline

$\dep={\bs X}-{\bs I}$
&
Exact solid displacement.
&
Section~\ref{sect:setting}
\\
\hline

$\vel=\partial_t\dep$
&
Exact solid velocity in the reference configuration.
&
Section~\ref{sect:interp}
\\
\hline

$\Omega^{\rm s}(t)=\Omega^{\rm s}[{\bs X}(t)]$
&
Current exact solid domain,
$\Omega^{\rm s}(t)={\bs X}(\widehat{\Omega}^{\rm s},t)$.
&
Section~\ref{sect:setting}
\\
\hline

$\Omega^{\rm f}(t)=\Omega^{\rm f}[{\bs X}(t)]$
&
Current exact fluid domain,
$\Omega^{\rm f}(t)=\Omega\setminus
\overline{\Omega^{\rm s}(t)}$.
&
\eqref{eq:omegaft}
\\
\hline

$\Gamma(t)=\Gamma[{\bs X}(t)]$
&
Current exact fluid--solid interface,
$\Gamma(t)={\bs X}(\widehat\Gamma,t)$.
&
Section~\ref{sect:setting}
\\
\hline

${\bs n}_{\widehat\Gamma}$
&
Unit normal on the reference interface $\widehat\Gamma$.
&
\eqref{eq:coupling-nl}
\\
\hline

${\bs n}_{\Gamma[{\bs Y}]}$
&
Unit normal on the interface determined by a configuration map
${\bs Y}$, where ${\bs Y}$ can be ${\bs X}$, ${\bs X}_H$ or ${\bs X}_H^*$. In particular, ${\bs n}_{\Gamma(t)}
={\bs n}_{\Gamma[{\bs X}]}$.
&
Section~\ref{sect:prelim}
\\
\hline

\multicolumn{3}{|l|}{\textbf{Exact solution and stress variables}}\\
\hline

$\bu$, $p$
&
Exact fluid velocity and pressure, smoothly extended to
$\Omega\times[0,T]$ when needed.
&
\eqref{eq:coupling-PDE}
\\
\hline

${\bs D}(\bu)$
&
Symmetric velocity gradient,
${\bs D}(\bu)=\frac12(\grad\bu+(\grad\bu)^{\rm T})$.
&
Section~\ref{sect:setting}
\\
\hline

$\stress(\bu,p)$
&
Fluid Cauchy stress,
$\stress(\bu,p)=2\mu^{\rm f}{\bs D}(\bu)
-p{\bs I}_{d\times d}$.
&
Section~\ref{sect:setting}
\\
\hline

${\bs\sigma}^{\rm s}(\dep)$
&
Solid stress,
${\bs\sigma}^{\rm s}(\dep)
=2\mu^{\rm s}{\bs D}(\dep)
+\lambda^{\rm s}(\grad\cdot\dep){\bs I}_{d\times d}$.
&
Section~\ref{sect:setting}
\\
\hline

$\widehat{\bs\sigma}$
&
Fluid traction pulled back to the exact reference interface:
$\widehat{\bs\sigma}
=[{\bs\sigma}(\bu,p){\bs n}_{\Gamma(t)}]\circ{\bs X}$.
&
\eqref{def-hat-sigma}
\\
\hline

$\bs\lambda$
&
Test variable for the kinematic interface constraint in the continuous
Lagrange multiplier formulation.
&
\eqref{eq:weak-fsi}
\\
\hline

\multicolumn{3}{|l|}{\textbf{Meshes and finite element spaces}}\\
\hline

${\mathcal T}_h$, $h$
&
Background fluid triangulation and its mesh size.
&
Section~\ref{section:mesh}
\\
\hline

$\widehat{\mathcal T}^{\rm s}_H$, $H$
&
Isoparametric solid triangulation and its mesh size.
&
Section~\ref{section:mesh}
\\
\hline

$\widehat{\Omega}^{\rm s}_H$
&
Discrete reference solid domain,
$\widehat{\Omega}^{\rm s}_H
=\bigcup_{\widehat K\in\widehat{\mathcal T}^{\rm s}_H}\widehat K$.
&
Section~\ref{section:mesh}
\\
\hline

$\widehat\Gamma_H$
&
Discrete reference interface,
$\widehat\Gamma_H=\partial\widehat{\Omega}^{\rm s}_H$.
&
Section~\ref{section:mesh}
\\
\hline

$\widehat{\bs\Phi}_H$
&
Geometry approximation map
$\widehat{\bs\Phi}_H:\widehat{\Omega}^{\rm s}_H
\to\widehat{\Omega}^{\rm s}$ satisfying
$\widehat{\bs\Phi}_H(\widehat\Gamma_H)=\widehat\Gamma$.
&
\eqref{GeomApprox}
\\
\hline

${\bs V}^{\rm f}_h$
&
Continuous piecewise-$\mathbb P_k$ finite element space for the fluid
velocity on the background mesh.
&
Section~\ref{section:mesh}
\\
\hline

$Q^{\rm f}_h$
&
Continuous piecewise-$\mathbb P_{k-1}$ finite element space for the
fluid pressure.
&
Section~\ref{section:mesh}
\\
\hline

$\widehat{\bs V}^{\rm s}_H$
&
Isoparametric piecewise-$\mathbb P_k$ finite element space for the
solid displacement.
&
Section~\ref{section:mesh}
\\
\hline

$\widehat{\bs\Lambda}_H$
&
Reference-interface multiplier space,
$\widehat{\bs\Lambda}_H
=\{{\bs w}_H|_{\widehat\Gamma_H}:
{\bs w}_H\in\widehat{\bs V}^{\rm s}_H\}$.
&
Section~\ref{section:mesh}
\\
\hline

${\bs\Lambda}_H(\Gamma[{\bs X}_H])$
&
Multiplier space transported to the numerical interface:
$\{{\bs\lambda}_H\circ{\bs X}_H^{-1}:
{\bs\lambda}_H\in\widehat{\bs\Lambda}_H\}$.
&
Section~\ref{sect:error}
\\
\hline

\multicolumn{3}{|l|}{\textbf{Numerical solution and numerical geometry}}\\
\hline

$\bu_h$, $p_h$
&
Semidiscrete fluid velocity and pressure.
&
\eqref{eq:discrete-semi-fsi}
\\
\hline

$\dep_H$
&
Semidiscrete solid displacement on
$\widehat{\Omega}^{\rm s}_H$.
&
\eqref{eq:discrete-semi-fsi}
\\
\hline

$\vel_H=\partial_t\dep_H$
&
Semidiscrete solid velocity.
&
Section~\ref{sect:interp}
\\
\hline

$\widehat{\bs\sigma}_H$
&
Semidiscrete traction multiplier on $\widehat\Gamma_H$.
&
\eqref{eq:discrete-semi-fsi}
\\
\hline

${\bs X}_H={\bs I}+\dep_H$
&
Numerical solid motion map.
&
Section~\ref{section:mesh}
\\
\hline

$\Omega^{\rm s}[{\bs X}_H]$
&
Numerical current solid domain,
${\bs X}_H(\widehat{\Omega}^{\rm s}_H)$.
&
Section~\ref{section:mesh}
\\
\hline

$\Omega^{\rm f}[{\bs X}_H]$
&
Numerical fluid domain,
$\Omega\setminus\overline{\Omega^{\rm s}[{\bs X}_H]}$.
&
Section~\ref{section:mesh}
\\
\hline

$\Gamma[{\bs X}_H]$
&
Numerical fluid--solid interface,
${\bs X}_H(\widehat\Gamma_H)$.
&
Section~\ref{section:mesh}
\\
\hline

${\bf P}_{\Gamma[{\bs X}_H]}$
&
$L^2$-orthogonal projection onto
${\bs\Lambda}_H(\Gamma[{\bs X}_H])$.
&
Section~\ref{sect:error}
\\
\hline

\multicolumn{3}{|l|}{\textbf{Interpolated exact solution}}\\
\hline

$\dep_H^*$
&
Lagrange interpolant of the exact displacement:
$\dep_H^*={\bs I}^{\rm s}_H
(\dep\circ\widehat{\bs\Phi}_H)$.
&
\eqref{def-sigma-s}
\\
\hline

$\vel_H^*=\partial_t\dep_H^*$
&
Lagrange interpolant of the exact solid velocity:
$\vel_H^*={\bs I}^{\rm s}_H
(\partial_t\dep\circ\widehat{\bs\Phi}_H)$.
&
Section~\ref{sect:interp}
\\
\hline

$\widehat{\bs\sigma}_H^*$
&
Lagrange interpolant of the pulled-back exact traction:
$\widehat{\bs\sigma}_H^*
={\bs I}^{\rm s}_H
(\widehat{\bs\sigma}\circ\widehat{\bs\Phi}_H)$.
&
\eqref{def-sigma-s}
\\
\hline

${\bs X}_H^*={\bs I}+\dep_H^*$
&
Interpolated solid configuration.
&
\eqref{interpl-error-X}
\\
\hline

$\bu_h^*$, $p_h^*$
&
Lagrange interpolants
$\bu_h^*={\bs I}^{\rm f}_h\bu$ and
$p_h^*=I^{\rm f}_hp$.
&
\eqref{interpl-u}
\\
\hline

$\Omega^{\rm s}[{\bs X}_H^*]$
&
Solid domain determined by the interpolated configuration.
&
\eqref{def-Omega*}
\\
\hline

$\Omega^{\rm f}[{\bs X}_H^*]$
&
Fluid domain determined by the interpolated configuration.
&
\eqref{def-Omega*}
\\
\hline

$\Gamma[{\bs X}_H^*]$
&
Interface determined by the interpolated configuration.
&
\eqref{def-Omega*}
\\
\hline

${\bs\Phi}_H(t)$
&
Map from the interpolated current solid domain to the exact current
solid domain:
${\bs X}(t)\circ\widehat{\bs\Phi}_H
\circ({\bs X}_H^*(t))^{-1}$.
&
\eqref{eq:map-phiHt}
\\
\hline

$\Omega^{\rm f}_{\rm ext}[{\bs X}_H]$
&
Extended fluid region consisting of elements within distance
$2\kappa h$ of $\Omega^{\rm f}[{\bs X}_H]$.
&
\eqref{Extended-region}
\\
\hline

\multicolumn{3}{|l|}{\textbf{Basic finite element errors}}\\
\hline

${\bs e}_{\bs u}=\bu_h-\bu_h^*$
&
Fluid velocity error relative to the Lagrange interpolant, defined on
the background domain $\Omega$.
&
\eqref{def-eu-ed}
\\
\hline

$e_p=p_h-p_h^*$
&
Fluid pressure error relative to the Lagrange interpolant.
&
\eqref{def-eu-ed}
\\
\hline

${\bs e}_{\bs d}=\dep_H-\dep_H^*$
&
Solid displacement error relative to the Lagrange interpolant, defined
on $\widehat{\Omega}^{\rm s}_H$.
&
\eqref{def-eu-ed}
\\
\hline

$\partial_t{\bs e}_{\bs d}
=\vel_H-\vel_H^*$
&
Solid velocity error on $\widehat{\Omega}^{\rm s}_H$.
&
\eqref{def-eu-ed}
\\
\hline

${\bs e}_{\bs\sigma}
=\widehat{\bs\sigma}_H-\widehat{\bs\sigma}_H^*$
&
Traction multiplier error on the discrete reference interface
$\widehat\Gamma_H$.
&
\eqref{def-eu-ed}
\\
\hline

${\bs e}_{\bs X}^*
={\bs X}_H^*\circ\widehat{\bs\Phi}_H^{-1}-{\bs X}$
&
Geometric interpolation error on the exact reference configuration.
&
\eqref{def-eXs}
\\
\hline

\multicolumn{3}{|l|}{\textbf{Consistency errors}}\\
\hline

$R_{h,H}^*({\bs v}_h,q_h,{\bs w}_H,{\bs\lambda}_H)$
&
Consistency remainder obtained by inserting the interpolated exact
solution into the semidiscrete formulation.
&
\eqref{def-RhH}
\\
\hline

$({\bs u}_h^\#,p_h^\#)$
&
Cut-domain Ritz projection of the exact fluid velocity and pressure.
&
\eqref{def-Ritz}
\\
\hline

${\bs u}_h^*-{\bs u}_h^\#$,
$p_h^*-p_h^\#$
&
Differences between the Lagrange interpolants and the Ritz projection.
&
Section~\ref{section:consistency-2}
\\
\hline

\multicolumn{3}{|l|}{\textbf{Intermediate configurations and interfaces}}\\
\hline

${\bs X}_H^\theta
=(1-\theta){\bs X}_H^*+\theta{\bs X}_H$
&
Intermediate configuration between the interpolated and numerical
configurations; ${\bs X}_H^0={\bs X}_H^*$ and
${\bs X}_H^1={\bs X}_H$.
&
Section~\ref{section:norm-equiv-2}
\\
\hline

$\Gamma[{\bs X}_H^\theta]$
&
Intermediate interface between $\Gamma[{\bs X}_H^*]$ and
$\Gamma[{\bs X}_H]$.
&
Section~\ref{section:norm-equiv-2}
\\
\hline

${\bs X}_H^{*,\theta}
=(1-\theta){\bs X}
+\theta{\bs X}_H^*\circ\widehat{\bs\Phi}_H^{-1}$
&
Intermediate configuration between the exact and interpolated
configurations.
&
\eqref{def-XH-star-theta}
\\
\hline

$\Gamma[{\bs X}_H^{*,\theta}]$
&
Intermediate interface between $\Gamma(t)$ and
$\Gamma[{\bs X}_H^*]$.
&
\eqref{def-XH-star-theta}
\\
\hline

${\bs X}_H^{\#,\theta}
=(1-\theta){\bs X}
+\theta{\bs X}_H\circ\widehat{\bs\Phi}_H^{-1}$
&
Intermediate configuration between the exact and numerical
configurations.
&
Section~\ref{section:norm-equiv-2}
\\
\hline

$\Gamma[{\bs X}_H^{\#,\theta}]$
&
Intermediate interface between $\Gamma(t)$ and
$\Gamma[{\bs X}_H]$.
&
Section~\ref{section:norm-equiv-2}
\\
\hline

${\bs e}_{\bs d}^\theta
={\bs e}_{\bs d}\circ({\bs X}_H^\theta)^{-1}$
&
Velocity of $\Gamma[{\bs X}_H^\theta]$ with respect to the parameter
$\theta$; equivalently, the displacement error transported from
$\widehat\Gamma_H$ to $\Gamma[{\bs X}_H^\theta]$.
&
\eqref{dt-int-Omega}
\\
\hline

${\bs e}_{\bs X}^{*,\theta}
={\bs e}_{\bs X}^*\circ({\bs X}_H^{*,\theta})^{-1}$
&
Velocity of the exact-to-interpolated interface family with respect to
$\theta$.
&
\eqref{def-eXs}
\\
\hline

${\bs e}_{\bs X}^{\#,\theta}$
&
Velocity of the exact-to-numerical interface family:
$({\bs X}_H\circ\widehat{\bs\Phi}_H^{-1}-{\bs X})
\circ({\bs X}_H^{\#,\theta})^{-1}$.
&
\eqref{def-ej-theta}
\\
\hline

$\vel_H^\theta$
&
Physical-time velocity of $\Gamma[{\bs X}_H^\theta]$:
$[(1-\theta)\vel_H^*+\theta\vel_H]
\circ({\bs X}_H^\theta)^{-1}$.
&
Section~\ref{sect:prelim}
\\
\hline

\multicolumn{3}{|l|}{\textbf{Transported functions on $\Gamma[{\bs X}_H^\theta]$}}\\
\hline

${\bs\lambda}_H^\theta
={\bs\lambda}_H\circ({\bs X}_H^\theta)^{-1}$
&
Reference multiplier transported to
$\Gamma[{\bs X}_H^\theta]$.
&
\eqref{notation-lambda}
\\
\hline

${\bs e}_{\bs d}^\theta
={\bs e}_{\bs d}\circ({\bs X}_H^\theta)^{-1}$
&
Solid displacement error transported to
$\Gamma[{\bs X}_H^\theta]$.
&
\eqref{notation-lambda}
\\
\hline

${\bs\dot{\bs e}}_{\bs d}^\theta
=\partial_t{\bs e}_{\bs d}
\circ({\bs X}_H^\theta)^{-1}$
&
Solid velocity error transported to
$\Gamma[{\bs X}_H^\theta]$.
&
\eqref{notation-lambda}
\\
\hline

${\bs d}_H^{*|\theta}
={\bs d}_H^*\circ({\bs X}_H^\theta)^{-1}$
&
Interpolated solid displacement transported using
${\bs X}_H^\theta$.
&
\eqref{notation-lambda}
\\
\hline

${\bs\dot{\bs d}}_H^{*|\theta}
=\partial_t{\bs d}_H^*
\circ({\bs X}_H^\theta)^{-1}$
&
Interpolated solid velocity transported using
${\bs X}_H^\theta$.
&
\eqref{notation-lambda}
\\
\hline

$\vel_H^{\,|\theta}
=\partial_t\dep_H\circ({\bs X}_H^\theta)^{-1}$
&
Numerical solid velocity represented on
$\Gamma[{\bs X}_H^\theta]$.
&
Section~\ref{section:dteu}
\\
\hline

${\widehat{\bs\sigma}}_H^{*|\theta}
=\widehat{\bs\sigma}_H^*
\circ({\bs X}_H^\theta)^{-1}$
&
Interpolated traction transported to
$\Gamma[{\bs X}_H^\theta]$.
&
Section~\ref{sect:error}
\\
\hline

${\bs e}_{\bs\sigma}^{\theta}
={\bs e}_{\bs\sigma}
\circ({\bs X}_H^\theta)^{-1}$
&
Traction error transported to
$\Gamma[{\bs X}_H^\theta]$.
&
\eqref{error-eq-n-abcd}
\\
\hline


\multicolumn{3}{|l|}{\textbf{Transported functions on
$\Gamma[{\bs X}_H^{*,\theta}]$}}\\
\hline

${\bs\lambda}_H^{*,\theta}
={\bs\lambda}_H\circ\widehat{\bs\Phi}_H^{-1}
\circ({\bs X}_H^{*,\theta})^{-1}$
&
Discrete reference multiplier transported to the
exact-to-interpolated interface.
&
\eqref{notation-lambda}
\\
\hline

${\bs e}_{\bs d}^{*,\theta}
={\bs e}_{\bs d}\circ\widehat{\bs\Phi}_H^{-1}
\circ({\bs X}_H^{*,\theta})^{-1}$
&
Solid displacement error transported to
$\Gamma[{\bs X}_H^{*,\theta}]$.
&
\eqref{notation-sigma}
\\
\hline

${\bs\dot{\bs e}}_{\bs d}^{*,\theta}
=\partial_t{\bs e}_{\bs d}
\circ\widehat{\bs\Phi}_H^{-1}
\circ({\bs X}_H^{*,\theta})^{-1}$
&
Solid velocity error transported to
$\Gamma[{\bs X}_H^{*,\theta}]$.
&
Section~\ref{sect:error}
\\
\hline

${\bs d}^{*,\theta}
=\dep\circ({\bs X}_H^{*,\theta})^{-1}$
&
Exact solid displacement represented on
$\Gamma[{\bs X}_H^{*,\theta}]$.
&
\eqref{notation-lambda}
\\
\hline

${\bs\dot{\bs d}}^{*,\theta}
=\partial_t\dep\circ({\bs X}_H^{*,\theta})^{-1}$
&
Exact solid velocity represented on
$\Gamma[{\bs X}_H^{*,\theta}]$.
&
\eqref{notation-lambda}
\\
\hline

${\bs\dot{\bs d}}_H^{*,\theta}
=\partial_t\dep_H^*\circ\widehat{\bs\Phi}_H^{-1}
\circ({\bs X}_H^{*,\theta})^{-1}$
&
Interpolated solid velocity represented on
$\Gamma[{\bs X}_H^{*,\theta}]$.
&
\eqref{notation-lambda}
\\
\hline

${\widehat{\bs\sigma}}_H^{*,\theta}
=\widehat{\bs\sigma}_H^*
\circ\widehat{\bs\Phi}_H^{-1}
\circ({\bs X}_H^{*,\theta})^{-1}$
&
Interpolated traction transported to
$\Gamma[{\bs X}_H^{*,\theta}]$.
&
\eqref{notation-sigma}
\\
\hline

${\bs w}_H^{*,\theta}
={\bs w}_H\circ\widehat{\bs\Phi}_H^{-1}
\circ({\bs X}_H^{*,\theta})^{-1}$
&
Solid finite element function transported to
$\Gamma[{\bs X}_H^{*,\theta}]$.
&
Section~\ref{section:consistency}
\\
\hline

\multicolumn{3}{|l|}{\textbf{Transported functions on
$\Gamma[{\bs X}_H^{\#,\theta}]$}}\\
\hline

${\bs\lambda}_H^{\#,\theta}
={\bs\lambda}_H\circ\widehat{\bs\Phi}_H^{-1}
\circ({\bs X}_H^{\#,\theta})^{-1}$
&
Discrete reference multiplier transported to the
exact-to-numerical interface.
&
Section~\ref{section:PHeu}
\\
\hline

${\bs e}_{\bs d}^{\#,\theta}
={\bs e}_{\bs d}\circ\widehat{\bs\Phi}_H^{-1}
\circ({\bs X}_H^{\#,\theta})^{-1}$
&
Solid displacement error transported to
$\Gamma[{\bs X}_H^{\#,\theta}]$.
&
\eqref{notation-sigma}
\\
\hline

${\bs\dot{\bs e}}_{\bs d}^{\#,\theta}
=\partial_t{\bs e}_{\bs d}
\circ\widehat{\bs\Phi}_H^{-1}
\circ({\bs X}_H^{\#,\theta})^{-1}$
&
Solid velocity error transported to
$\Gamma[{\bs X}_H^{\#,\theta}]$.
&
\eqref{eu-dot-ed-Lp-HH}
\\
\hline

${\bs e}_{\bs\sigma}^{\#,\theta}
={\bs e}_{\bs\sigma}\circ\widehat{\bs\Phi}_H^{-1}
\circ({\bs X}_H^{\#,\theta})^{-1}$
&
Traction error transported to
$\Gamma[{\bs X}_H^{\#,\theta}]$.
&
\eqref{notation-sigma}
\\
\hline

${\bs e}_{\bs\sigma}^{\#,1}
={\bs e}_{\bs\sigma}\circ{\bs X}_H^{-1}$
&
Traction error represented on $\Gamma[{\bs X}_H]$.
&
Section~\ref{section:esigma}
\\
\hline

${\bs w}_H^{\#,\theta}
={\bs w}_H\circ\widehat{\bs\Phi}_H^{-1}
\circ({\bs X}_H^{\#,\theta})^{-1}$
&
Solid finite element function transported to
$\Gamma[{\bs X}_H^{\#,\theta}]$.
&
\eqref{notation-sigma}
\\
\hline

${\widehat{\bs\sigma}}_H^{\#,\theta}
=\widehat{\bs\sigma}_H^*
\circ\widehat{\bs\Phi}_H^{-1}
\circ({\bs X}_H^{\#,\theta})^{-1}$
&
Interpolated traction transported to
$\Gamma[{\bs X}_H^{\#,\theta}]$.
&
\eqref{notation-sigma}
\\
\hline

$({\bf P}_{\Gamma[{\bs X}_H]}
{\bs e}_{\bs u})^{\#,\theta}$
&
Projected fluid velocity error transported from
$\Gamma[{\bs X}_H]$ to $\Gamma[{\bs X}_H^{\#,\theta}]$.
&
\eqref{eu-dot-ed-Lp-HH}
\\
\hline

$({\bf P}_{\Gamma[{\bs X}_H]}
{\bs e}_{\bs u})^{\#,0}$
&
Projected fluid velocity error transported from the numerical
interface to the exact interface $\Gamma(t)$.
&
Section~\ref{section:PHeu}
\\
\hline

\multicolumn{3}{|l|}{\textbf{Frequently used endpoint identities}}\\
\hline

${\bs X}_H^0={\bs X}_H^*$,
${\bs X}_H^1={\bs X}_H$
&
Endpoints of the interpolated-to-numerical interface family.
&
Section~\ref{section:norm-equiv-2}
\\
\hline

${\bs X}_H^{*,0}={\bs X}$,
${\bs X}_H^{*,1}
={\bs X}_H^*\circ\widehat{\bs\Phi}_H^{-1}$
&
Endpoints of the exact-to-interpolated interface family.
&
\eqref{def-XH-star-theta}
\\
\hline

${\bs X}_H^{\#,0}={\bs X}$,
${\bs X}_H^{\#,1}
={\bs X}_H\circ\widehat{\bs\Phi}_H^{-1}$
&
Endpoints of the exact-to-numerical interface family.
&
Section~\ref{section:norm-equiv-2}
\\
\hline

${\bs\lambda}_H^0={\bs\lambda}_H^{*,1}$
&
The two representations of the multiplier coincide on
$\Gamma[{\bs X}_H^*]$.
&
\eqref{notation-lambda}
\\
\hline

$\vel_H^{*|0}=\vel_H^{*,1}$
&
The two representations of the interpolated solid velocity coincide on
$\Gamma[{\bs X}_H^*]$.
&
\eqref{notation-lambda}
\\
\hline

${\bs e}_{\bs d}^{1}
={\bs e}_{\bs d}\circ{\bs X}_H^{-1}$
&
Displacement error represented on the numerical interface.
&
\eqref{notation-lambda}
\\
\hline

${\bs\dot{\bs e}}_{\bs d}^{1}
=\partial_t{\bs e}_{\bs d}\circ{\bs X}_H^{-1}$
&
Solid velocity error represented on the numerical interface.
&
\eqref{notation-lambda}
\\
\hline

$({\bf P}_{\Gamma[{\bs X}_H]}
{\bs e}_{\bs u})^{*,0}
=
({\bf P}_{\Gamma[{\bs X}_H]}
{\bs e}_{\bs u})^{\#,0}$
&
Coinciding representations of the projected velocity error on the
exact interface.
&
\eqref{relation-0-1}
\\
\hline

\multicolumn{3}{|l|}{\textbf{Projected kinematic mismatch}}\\
\hline

${\bf P}_{\Gamma[{\bs X}_H]}{\bs e}_{\bs u}
-\partial_t{\bs e}_{\bs d}\circ{\bs X}_H^{-1}$
&
Projected mismatch between the fluid velocity error and the solid
velocity error on the numerical interface.
&
Section~\ref{sect:error}
\\
\hline

${\bf P}_{\Gamma[{\bs X}_H]}{\bs e}_{\bs u}
-{\bs\dot{\bs e}}_{\bs d}^{1}$
&
Abbreviated form of the projected kinematic mismatch on
$\Gamma[{\bs X}_H]$.
&
\eqref{eu-dot-ed-Lp-H}
\\
\hline

$({\bf P}_{\Gamma[{\bs X}_H]}
{\bs e}_{\bs u})^{\#,\theta}
-{\bs\dot{\bs e}}_{\bs d}^{\#,\theta}$
&
Projected kinematic mismatch transported to
$\Gamma[{\bs X}_H^{\#,\theta}]$.
&
\eqref{eu-dot-ed-Lp-HH}
\\
\hline

\multicolumn{3}{|l|}{\textbf{Transport maps used in interface estimates}}\\
\hline

${\bs Y}_H^{\#,\theta}
={\bs X}_H^{\#,\theta}
\circ\widehat{\bs\Phi}_H\circ{\bs X}_H^{-1}$
&
Map from the numerical interface $\Gamma[{\bs X}_H]$ to the
intermediate interface $\Gamma[{\bs X}_H^{\#,\theta}]$.
&
Section~\ref{sect:error}
\\
\hline

${\bs Z}^{\theta}
={\bs X}\circ({\bs X}_H^{\#,\theta})^{-1}$
&
Map from $\Gamma[{\bs X}_H^{\#,\theta}]$ to the exact interface
$\Gamma(t)$.
&
Section~\ref{sect:error}
\\
\hline

${\bs v}_h^{\#,\theta}
={\bs v}_h\circ{\bs X}_H\circ\widehat{\bs\Phi}_H^{-1}
\circ({\bs X}_H^{\#,\theta})^{-1}$
&
Transport of a finite element function from the numerical interface to
$\Gamma[{\bs X}_H^{\#,\theta}]$.
&
Section~\ref{section:ep}
\\
\hline






\multicolumn{3}{|l|}{\textbf{Ghost penalties and mesh-dependent norms}}\\
\hline



$g_h^{\rm u}(\cdot,\cdot)$
&
Velocity ghost-penalty stabilization and extension form.
&
\eqref{eqn:GPForms}
\\
\hline

$g_h^{\rm p}(\cdot,\cdot)$
&
Pressure ghost-penalty stabilization and extension form.
&
\eqref{eqn:GPForms}
\\
\hline

$\epsilon_h\in(0,h]$
&
Stabilization parameter multiplying the velocity time-derivative ghost
penalty.
&
\eqref{eq:discrete-semi-fsi}
\\
\hline

$\|{\bs v}_h\|_{1,\Omega[{\bs X}_H]}$
&
Mesh-dependent velocity $H^1$ norm:
$
\|{\bs v}_h\|_{H^1(\Omega^{\rm f}[{\bs X}_H])}^2
+g_h^{\rm u}({\bs v}_h,{\bs v}_h).
$
&
\eqref{1-0-norm}
\\
\hline

$\|{\bs v}_h\|_{0,\Omega[{\bs X}_H]}$
&
Mesh-dependent velocity $L^2$ norm:
$
\|{\bs v}_h\|_{L^2(\Omega^{\rm f}[{\bs X}_H])}^2
+h^2g_h^{\rm u}({\bs v}_h,{\bs v}_h).
$
&
\eqref{0-0-norm-a2}
\\
\hline

$\|q_h\|_{0,\Omega[{\bs X}_H]}$
&
Mesh-dependent pressure norm:
$
\|q_h\|_{L^2(\Omega^{\rm f}[{\bs X}_H])}^2
+g_h^{\rm p}(q_h,q_h).
$
&
\eqref{1-0-norm}
\\
\hline

\hline
\end{longtable}
\endgroup


\bibliographystyle{plain}
\bibliography{biblio}
 
\end{document}